\documentclass[a4paper, reqno]{amsart}

\usepackage{tikz} 
\usetikzlibrary{cd,patterns,arrows,positioning,decorations.pathmorphing,decorations.markings,calc,intersections,through,backgrounds,matrix,fit,shapes}
\newcommand{\enumber}[1]{%
    \begin{tikzpicture}[remember picture]
        \node[inner sep=0pt](a){#1};
    \end{tikzpicture}%
    \begin{tikzpicture}[overlay, remember picture]
        \node[draw, black, fit=(a), ellipse, inner sep=1pt, line width=0.8pt]{};
    \end{tikzpicture}%
    }

\usepackage{xparse} 
\usepackage[normalem]{ulem} 
\usepackage{tensor}
\usepackage{pgfplots}
\pgfplotsset{compat=1.14}
\usepgfplotslibrary{fillbetween}
\pgfdeclarelayer{bg}
\pgfsetlayers{bg,main}
\tikzset{
    hatch distance/.store in=\hatchdistance,
    hatch distance=10pt,
    hatch thickness/.store in=\hatchthickness,
    hatch thickness=0.3pt
}

\makeatletter
\pgfdeclarepatternformonly[\hatchdistance,\hatchthickness]{northwest}
{\pgfqpoint{0pt}{0pt}}
{\pgfqpoint{\hatchdistance}{\hatchdistance}}
{\pgfpoint{\hatchdistance-1pt}{\hatchdistance-1pt}}%
{
    \pgfsetcolor{\tikz@pattern@color}
    \pgfsetlinewidth{\hatchthickness}
    \pgfpathmoveto{\pgfqpoint{\hatchdistance}{0pt}}
    \pgfpathlineto{\pgfqpoint{0pt}{\hatchdistance}}
    \pgfusepath{stroke}
}

\usepackage{amsthm} 
\usepackage{amsmath} 
\usepackage{hyperref} 
\usepackage{amssymb} 
\usepackage{longtable} 
\usepackage{array} 
\usepackage{graphicx} 
\usepackage{mathtools} 
\usepackage{enumitem} 
\usepackage{bm} 
\usepackage{graphicx}
\usepackage{subcaption}
\usepackage{changepage}
\usepackage{imakeidx}

\newcommand{\Sources}{\mathcal{S}}

\newcommand{\La}{\Lambda}
\newcommand{\tn}{\tau_n}
\newcommand{\tno}{\tau_n^-}
\newcommand{\longto}{\longrightarrow}
\newcommand{\longfrom}{\longleftarrow}
\newcommand{\To}{\Rightarrow}
\newcommand{\comp}{\circ}
\newcommand{\om}{\Omega}
\newcommand{\isom}{\cong}
\newcommand{\equivalent}{\simeq}

\newcommand{\cA}{\mathcal{A}}

\newcommand{\cC}{\mathcal{C}}
\newcommand{\cD}{\mathcal{D}}
\newcommand{\cF}{\mathcal{F}}
\newcommand{\cG}{\mathcal{G}}
\newcommand{\ccG}{\mathcal{G}}
\newcommand{\cI}{\mathcal{I}}

\newcommand{\cL}{\mathcal{L}}
\newcommand{\cM}{\mathcal{M}}
\newcommand{\cP}{\mathcal{P}}
\newcommand{\cR}{\mathcal{R}}

\newcommand{\cU}{\mathcal{U}}
\newcommand{\cV}{\mathcal{V}}

\newcommand{\NN}{\mathbb{N}}
\newcommand{\ZZ}{\mathbb{Z}}

\newcommand{\dirI}{\mathbb{I}}
\newcommand{\K}{\mathbf{k}}

\newcommand{\sinks}{\mathbf{t}}
\newcommand{\sources}{\mathbf{s}}
\newcommand{\D}{D}

\newcommand{\Cks}{\dot{\mathcal{C}}}
\newcommand{\Pks}{\dot{\Phi}}
\newcommand{\Tks}{\dot{\Theta}}
\newcommand{\Qks}{\dot{Q}}
\newcommand{\Rks}{\dot{\cR}}
\newcommand{\Mks}{\dot{\mathcal{M}}}

\newcommand\restr[2]{{
  \left.\kern-\nulldelimiterspace 
  #1 
  \vphantom{\big|} 
  \right|_{#2} 
  }}

\NewDocumentCommand\glue{ O{} O{}}{\mathbin{\raisebox{0.2ex}{${}^{\mbox{\tiny$#1$}}$\rotatebox[origin=c]{90}{$\triangleright$}$^{\mbox{\tiny$#2$}}$}}}
\NewDocumentCommand\qtwo{m O{}}{\begin{smallmatrix} #1 \\ #2 \end{smallmatrix}}
\NewDocumentCommand\qthree{m O{} O{}}{\begin{smallmatrix} #1 \\ #2 \\ #3 \end{smallmatrix}} 
\NewDocumentCommand\qfour{m O{} O{} O{}}{\begin{smallmatrix} #1 \\ #2 \\ #3 \\ #4 \end{smallmatrix}}
\NewDocumentCommand\qfive{m O{} O{} O{} O{}}{\begin{smallmatrix} #1 \\ #2 \\ #3 \\ #4 \\ #5 \end{smallmatrix}}
\NewDocumentCommand\tinyqthree{m O{} O{}}{\mbox{\tiny$\begin{smallmatrix} #1 \\ #2 \\ #3 \end{smallmatrix}$}} 
\newcommand{\nsup}[2]{{^{#1}\!{#2}}}
\newcommand{\supn}[2]{{{#1}\!^{#2}}}
\newcommand{\PD}{\nsup{P}{\triangle}}
\newcommand{\DI}{\supn{\triangle}{I}}
\newcommand{\ABP}{\cP\strut^{\text{ab}}}
\newcommand{\IAB}{\cI\strut^{\text{ab}}}
\newcommand{\MABP}{\cP\strut^{\text{mab}}}
\newcommand{\MIAB}{\cI\strut^{\text{mab}}}

\newcommand{\sABP}{\cP\rule[-.45\baselineskip]{0pt}{\baselineskip}^{\text{ab}}}
\newcommand{\sMABPind}{\mathcal{P}^{\text{mab}}_{\text{ind}}}
\newcommand{\sMABIind}{\mathcal{I}^{\text{mab}}_{\text{ind}}}

\newcommand{\nospacepunct}[1]{\makebox[0pt][l]{\,#1}}
\newcommand{\tilvert}{\prescript{\scriptscriptstyle P}{}{\sim}^{\scriptscriptstyle I}_{\scriptscriptstyle 0}}
\newcommand{\tilarr}{\prescript{\scriptscriptstyle P}{}{\sim}^{\scriptscriptstyle I}_{\scriptscriptstyle 1}}

\newcommand{\tilQ}{\tilde{Q}}

\newcommand{\tilR}{\tilde{\cR}}

\newcommand{\tilM}{\tilde{\cM}}

\newcommand{\tilLa}{\tilde{\Lambda}}

\newcommand\nct[2]{\fill ({#1},{#2})  circle (0.05cm);} 
\newcommand\mct[2]{\draw ({#1},{#2})  circle (0.05cm);} 

\DeclareMathOperator{\kAlg}{\K-\textbf{Alg}}

\DeclareMathOperator{\Cat}{\textbf{Cat}}
\DeclareMathOperator{\op}{op}

\DeclareMathOperator{\m}{mod}
\DeclareMathOperator{\rep}{rep}
\DeclareMathOperator{\M}{Mod}
\DeclareMathOperator{\Hom}{Hom}
\DeclareMathOperator{\End}{End}
\DeclareMathOperator{\Ext}{Ext}
\DeclareMathOperator{\add}{add}
\DeclareMathOperator{\Id}{id}
\DeclareMathOperator{\proj}{proj}
\DeclareMathOperator{\inj}{inj}
\DeclareMathOperator{\ind}{ind}
\DeclareMathOperator{\obj}{obj}

\DeclareMathOperator{\supp}{supp}

\DeclareMathOperator{\rad}{rad}
\DeclareMathOperator{\coker}{coker}
\DeclareMathOperator{\height}{ht}
\DeclarePairedDelimiter\abs{\lvert}{\rvert}%
\DeclareMathOperator{\Sub}{Sub}
\DeclareMathOperator{\Fac}{Fac}
\DeclareMathOperator{\gldim}{gl.dim}
\DeclareMathOperator{\domdim}{dom.dim}
\DeclareMathOperator{\pdim}{proj.dim}
\DeclareMathOperator{\idim}{inj.dim}
\DeclareMathOperator{\Aut}{Aut}

\tikzset{
  symbol/.style={
    draw=none,
    every to/.append style={
      edge node={node [sloped, allow upside down, auto=false]{$#1$}}}
  }
}

\theoremstyle{plain}
\newtheorem*{theorem*}{Theorem}
\newtheorem{theorem}{Theorem}[section] 

\theoremstyle{definition}
\newtheorem{definition}[theorem]{Definition} 
\newtheorem{definition-proposition}[theorem]{Definition-Proposition} 
\newtheorem{example}[theorem]{Example} 
\newtheorem{corollary}[theorem]{Corollary} 
\newtheorem{lemma}[theorem]{Lemma} 
\newtheorem{proposition}[theorem]{Proposition} 
\newtheorem{remark}[theorem]{Remark} 
\numberwithin{equation}{section} 
\newtheorem{question}{Question} 

\title[$n$-cluster tilting from gluing systems of representation-directed algebras]{$n$-cluster tilting subcategories from gluing systems of representation-directed algebras}
\author{Laertis Vaso}


\makeindex[name=definitions, columns=2, title=List of definitions]
\makeindex[name=symbols, columns=2, title=List of symbols, options={-s index_style.ist}]
\indexsetup{level=\section*} 

\begin{document}
\maketitle

\begin{abstract}
We present a new way to construct $n$-cluster tilting subcategories of abelian categories. Our method takes as input a direct system of abelian categories $\cA_i$ with certain subcategories and, under reasonable conditions, outputs an $n$-cluster tilting subcategory of an admissible target $\cA$ of the direct system. We apply this general method to a direct system of module categories $\m\La_i$ of rep\-re\-sen\-ta\-tion-di\-rect\-ed algebras $\La_i$ and obtain an $n$-cluster tilting subcategory $\cM$ of a module category $\m\cC$ of a locally bounded Krull--Schmidt category $\cC$. In certain cases we also construct an admissible $\ZZ$-action of $\cC$. Using a result of Darp{\"o}--Iyama, we obtain an $n$-cluster tilting subcategory of $\m(\cC/\ZZ)$ where $\cC/\ZZ$ is the corresponding orbit category. We show that in this case $\m(\cC/\ZZ)$ is equivalent to the module category of a fi\-nite-di\-men\-sion\-al algebra. In this way we construct many new families of rep\-re\-sen\-ta\-tion-fi\-nite algebras whose module categories admit $n$-cluster tilting modules.
\end{abstract}

\tableofcontents

\section{Introduction}

\subsection{Motivation} Let $\K$ be a field and $\La$ be a fi\-nite-di\-men\-sion\-al $\K$-algebra. One of the main aims of representation theory is to describe the category $\m\La$ of fi\-nite-di\-men\-sion\-al (right) $\La$-modules. Since it is generally impossible to achieve this aim, one usually restricts either to specific classes of algebras or to specific subcategories of the module category. One such restriction on the side of algebras is to assume that there only exist finitely many indecomposable $\La$-modules up to isomorphism; in this case $\La$ is called \emph{rep\-re\-sen\-ta\-tion-fi\-nite}\index[definitions]{rep\-re\-sen\-ta\-tion-fi\-nite algebra}. In particular, if $\{M_1,\dots,M_k\}$ is a complete and irredundant collection of representatives of indecomposable $\La$-modules and $M=\bigoplus_{i=1}^k M_i$, then $\m\La$ is equivalent to the additive closure $\add(M)$ of $M$ in $\m\La$. A fundamental result in this case is the \emph{Auslander correspondence} \cite[Section III.4]{AUS} which gives a bijection between rep\-re\-sen\-ta\-tion-fi\-nite algebras up to Morita equivalence and \emph{Auslander algebras}, that is algebras $\Gamma$ with $\gldim(\Gamma)\leq 2\leq \domdim(\Gamma)$, up to Morita equivalence, where $\gldim$ is the \emph{global dimension} and $\domdim$ is the \emph{dominant dimension}. 

The bijection of the Auslander correspondence is given by $M\mapsto \End_{\La}(M)$. Hence Auslander algebras can be seen as algebras given by the additive generator of a module category. In this sense Aus\-lan\-der--Rei\-ten theory can be seen as generalizing some of the properties of Auslander algebras to general module categories. In particular in the case of rep\-re\-sen\-ta\-tion-fi\-nite algebras, Aus\-lan\-der--Rei\-ten theory gives a complete description of $\m\La$, see for example \cite{ARS, ASS}. 

One way to restrict to a subcategory of the whole module category is via Iyama's higher dimensional Aus\-lan\-der--Rei\-ten theory \cite{IYA2, IYA1}, which can be seen as a generalization of the above situation. In this theory one changes focus from $\m\La$ to a functorially finite subcategory $\cM\subseteq \m\La$ such that
\begin{align*}
    \cM &= \{ X \in \m\La \mid \Ext^{i}_{\La}\left(X, \cM\right)=0\text{ for all $1\leq i \leq n-1$}\} \\
    & = \{ X \in \m\La \mid \Ext^{i}_{\La}\left(\cM,X\right)=0\text{ for all $1\leq i \leq n-1$}\}, 
\end{align*}
called an \emph{$n$-cluster tilting subcategory}\index[definitions]{n-cluster tilting subcategory@$n$-cluster tilting subcategory}. In this setting one can construct a higher dimensional analogue of Aus\-lan\-der--Rei\-ten theory inside $\cM$. In particular, there exists a higher Auslander correspondence \cite[Theorem 0.2]{IYA2} between algebras whose module categories admit an $n$-cluster tilting subcategory with finitely many indecomposable modules and algebras $\Gamma$ with $\gldim(\Gamma) \leq n+1 \leq \domdim(\Gamma)$. Both the higher Auslander correspondence and higher dimensional Aus\-lan\-der--Rei\-ten theory restrict to their classical versions when $n=1$. Notice also that a module category may admit $n$-cluster tilting subcategories for different values of $n$; in fact $\m\La$ itself is always the unique $1$-cluster tilting subcategory.

Let $\cM\subseteq \m\La$ be an $n$-cluster tilting subcategory. If $\cM\equivalent \add(M)$ for some $M\in\m\La$ then $M$ is called an \emph{$n$-cluster tilting module}\index[definitions]{n-cluster tilting module@$n$-cluster tilting module}. Clearly if $n=1$ then a $1$-cluster tilting module exists if and only if $\La$ is rep\-re\-sen\-ta\-tion-fi\-nite. If $n\geq 2$ it is an open question whether there exists an $n$-cluster tilting subcategory with no additive generator. Furthermore, due to the homological nature of $n$-cluster tilting subcategories, an important role is played by the global dimension $d\coloneqq\gldim(\La)$. In particular if $\cM=\add(M)$ is given by a $d$-cluster tilting module $M$, then $\cM$ is unique and given by 
\[\cM=\add \{\tau_d^{-i}(\La) \mid i\geq 0\},\]
where $\tau_d^-=\tau^-\om^{-(d-1)}$ denotes the \emph{inverse $d$-Aus\-lan\-der--Rei\-ten translation}, see Section \ref{subsec:background and notation}.

Finding an algebra $\La$ whose module category admits an $n$-cluster tilting module is a challenging problem, even in the much more studied case of $n=1$ where many families of rep\-re\-sen\-ta\-tion-fi\-nite algebras have been classified (e.g. \cite{Gab1, Rid, Fis, BR}). The case $n=d$ has attracted the most attention (e.g. \cite{IO, HI, CIM, VAS}), but the cases $n\in d\ZZ$ (e.g. \cite{JK}) and $d=\infty$ (e.g. \cite{EH, DI}) have also been studied. The case $n\nmid d$ was studied in \cite{VAS2}. We briefly recall the methods of \cite{DI} and \cite{VAS2} which are particularly important for this paper.

\subsection{Previous work} By abstracting the above situation further, one defines $n$-cluster tilting subcategories for any abelian category. In particular, let $\cC$ be a locally bounded $\K$-linear Krull--Schmidt category. Then $\m\cC$ is an abelian category which in general is not equivalent to a category of modules over an algebra. By applying the classical theory of Galois covering in this setting, a method of constructing $n$-cluster tilting subcategories was introduced in \cite{DI}. This method takes as input a locally bounded $\K$-linear Krull--Schmidt category $\cC$ admitting an $n$-cluster tilting subcategory $\cM$ and an admissible group of automorphisms $\cG$ of $\cC$ and, under reasonable compatibility conditions between $\cG$ and $\cM$, produces a new $n$-cluster tilting subcategory of the module category $\m(\cC/\cG)$ of the orbit category $\cC/\cG$ \cite[Theorem 2.14]{DI}. Notice that although $\m\cC$ may not be equivalent to a module category of an algebra, it may happen that $\m(\cC/\cG)$ is equivalent to $\m\La$ for some fi\-nite-di\-men\-sion\-al algebra $\La$, thus landing back to the world of algebras.

Next, let $\La$ be a \emph{rep\-re\-sen\-ta\-tion-di\-rect\-ed}\index[definitions]{rep\-re\-sen\-ta\-tion-di\-rect\-ed algebra} algebra, that is a rep\-re\-sen\-ta\-tion-fi\-nite algebra with no oriented cycles in its Aus\-lan\-der--Rei\-ten quiver, see \cite[2.4]{RIN}. Since $\La$ is rep\-re\-sen\-ta\-tion-fi\-nite, $n$-cluster tilting modules and $n$-cluster tilting subcategories coincide. It is not difficult to see that if there exists an $n$-cluster tilting subcategory of $\La$, then it is unique and it is given by 
\[\cM=\add\{\tau_n^{-i}\left(\La\right)\mid i\geq 0\}.\] 
A characterization of $n$-cluster tilting modules of rep\-re\-sen\-ta\-tion-di\-rect\-ed algebras was proved in \cite[Theorem 1]{VAS}, based on, among other things, results in \cite{IYA4}. This characterization states that $\cM$ is an $n$-cluster tilting subcategory if and only if the following two conditions are satisfied.
\begin{enumerate}
    \item $\tn:\cM_{\cP}\to \cM_{\cI}$ and $\tno:\cM_{\cI}\to\cM_{\cP}$ induce mutually inverse bijections, where $\cM_{\cP}$ respectively $\cM_{\cI}$ denote the set of isomorphism classes of indecomposable nonprojective respectively noninjective $\La$-modules in $\cM$.
    \item For all $i\in\{1,\dots,n-1\}$ and all $M\in\cM_{\cP}$ and $N\in\cM_{\cI}$ we have that the $i$-th syzygy $\om^i (M)$ and the $i$-th cosyzygy $\om^{-i}(N)$ are indecomposable. 
\end{enumerate}

Based on this result, another method of constructing $n$-cluster tilting modules for rep\-re\-sen\-ta\-tion-di\-rect\-ed algebras was introduced in \cite{VAS2} called \emph{gluing}. First, based on the above characterization, one generalizes the notion of $n$-cluster tilting modules and $n$-cluster tilting subcategories to the notion of \emph{$n$-fractured modules}\index[definitions]{n-fractured module@$n$-fractured module} and \emph{$n$-fractured subcategories}, see \cite[Definition 3.10]{VAS2}. This notion is based on considering suitable subcategories $\cP^L,\cI^R\subseteq \m\La$ which substitute for the classes of  projective and injective $\La$-modules. More precisely, if $\cP^L=\add(T^L)$ and $\cI^R=\add(T^R)$ for some suitably chosen modules $T^L$ and $T^R$, then $\add\{\tau_n^{-i}\left(T^L\right) \mid i\geq 0\}$ is an $n$-fractured subcategory if and only if the following two conditions are satisfied.
\begin{enumerate}
    \item $\tn:\cM_{\cP^L}\to \cM_{\cI^R}$ and $\tno:\cM_{\cI^R}\to \cM_{\cP^L}$ induce mutually inverse bijections, where $\cM_{\cP^L}$ respectively $\cM_{\cI^R}$ is the set of isomorphism classes of indecomposable $\La$-modules not in $\cP^L$ respectively not in $\cI^R$.
    \item For all $i\in\{1,\dots,n-1\}$ and all $M\in\cM_{\cP^L}$ and $N\in\cM_{\cI^R}$ we have that $\om^i(M)$ and $\om^{-i}(N)$ are indecomposable. 
\end{enumerate}
Then an $n$-fractured subcategory is an $n$-cluster tilting subcategory if and only if $T^L\isom\La$ and $T^R\isom\D(\La)$, where $\D=\Hom_{\K}(-,\K)$ is the standard duality between $\m\La$ and $\m\La^{\text{op}}$.

Next one introduces a gluing procedure which takes as input two rep\-re\-sen\-ta\-tion-di\-rect\-ed algebras $A$ and $B$, together with a certain $A$-module and a certain $B$-module and produces a new rep\-re\-sen\-ta\-tion-di\-rect\-ed algebra $B\glue A$. The representation theory of $B\glue A$ can be completely described by the representation theory of $A$ and $B$. If moreover $\m A$ and $\m B$ each admit an $n$-fractured module $M_A$ and $M_B$, and some reasonable compatibility restrictions between $M_A$ and $M_B$ are satisfied, then $\m (B\glue A)$ also admits an $n$-fractured module $M_{B{\scriptscriptstyle \triangle} A}$ \cite[Theorem 3.16]{VAS2}. It is shown in \cite{VAS2} that in many cases $M_{B{\scriptscriptstyle \triangle} A}$ is an actual $n$-cluster tilting module.

\subsection{The results of this paper} 

\subsubsection{A motivating example}\label{subsubsec:a motivating example} Our aim is to combine the results of both \cite{DI} and \cite{VAS2} to present a new method of constructing $n$-cluster tilting subcategories of abelian categories. Let $Q$ be a quiver, let $\cR\subseteq \K Q$ be an admissible ideal and let $A=\K Q/\cR$. Assume moreover that $Q$ is of the form
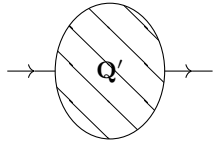
\[\begin{tikzpicture}[scale=0.9, transform shape]

\draw[->] (8.5,0) -- (8.9,0);
\draw[-] (8.9,0) -- (9.2,0);
\draw[->] (10.8,0) -- (11.2,0);
\draw[-] (11.2,0) -- (11.5,0);

\draw[pattern=northwest, hatch distance=15pt, hatch thickness = 0.3pt] (10,0) ellipse (0.8cm and 1cm);
\node (Z) at (10,0) {$\mathbf{Q'}$};
\end{tikzpicture}\]
where $Q'$ is not empty and  $\begin{tikzpicture}[scale=0.8, transform shape, decoration={markings, mark=at position 0.6 with {\arrow{>}}}] 
    \node (A) at (0,0) {}; 
    \node (B) at (1.4,0) {}; 
    \draw[postaction={decorate}] (A) -- (B);
\end{tikzpicture}$ is a linearly oriented quiver of Dynkin type $A$. Assume also that no path in $\begin{tikzpicture}[scale=0.8, transform shape, decoration={markings, mark=at position 0.6 with {\arrow{>}}}] 
    \node (A) at (0,0) {}; 
    \node (B) at (1.4,0) {}; 
    \draw[postaction={decorate}] (A) -- (B);
\end{tikzpicture}$ belongs to $\cR$. For an integer $z\in\ZZ$, let $Q[z]$ be a copy of $Q$ with vertices and arrows labelled as follows.
\begin{align*}
\left(Q[z]\right)_0 &=\{i[z] \mid i\in Q_0\},\\
\left(Q[z]\right)_1 &= \{\alpha[z]:i[z]\to j[z] \mid \alpha\in Q_1 \text{ with $\alpha:i\to j$}\}. 
\end{align*}
Denote by $\cR[z]$ the ideal of $\K Q[z]$ corresponding to the ideal $\cR$ of $\K Q$. We set $A[z]\coloneqq \K Q[z]/\cR[z]$. Clearly $A[z]$ is canonically isomorphic to $A$.

In this case, the gluing $A[1]\glue A[0]$ is by definition equal to $\K Q_1/\cR_1$ where $Q_1$ is the quiver
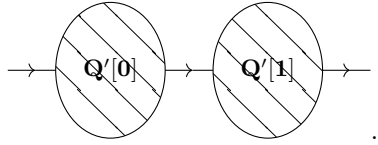
\[\begin{tikzpicture}[scale=0.9, transform shape]

\draw[->] (5.3,0) -- (5.7,0);
\draw[-] (5.7,0) -- (6,0);

\draw[pattern=northwest, hatch distance=15pt, hatch thickness = 0.3pt] (6.8,0) ellipse (0.8cm and 1cm);
\node (Z02) at (6.8,0) {$\mathbf{Q'[0]}$};

\draw[->] (7.6,0) -- (8,0);
\draw[-] (8,0) -- (8.3,0);

\draw[pattern=northwest, hatch distance=15pt, hatch thickness = 0.3pt] (9.1,0) ellipse (0.8cm and 1cm);
\node (Z01) at (9.1,0) {$\mathbf{Q'[1]}$};

\draw[->] (9.9,0) -- (10.3,0);
\draw[-] (10.3,0) -- (10.6,0);
\end{tikzpicture}.\]
and $\cR_1$ is generated by all elements in $\K Q[0]$ that lie in $\cR[0]$, all elements in $\K Q[1]$ that lie in $\cR[1]$, as well as all paths from $Q'[0]$ to $Q'[1]$. By repeating this procedure infinitely many times, we may then consider the infinite quiver $\Qks$ given by
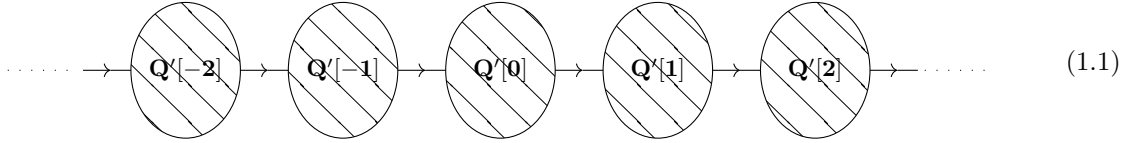
\begin{equation}\label{picture:infinite quiver}
\begin{tikzpicture}[baseline={(current bounding box.center)}, scale=0.9, transform shape]
\draw[loosely dotted] (4.2,0) -- (5.2,0);

\draw[->] (5.3,0) -- (5.7,0);
\draw[-] (5.7,0) -- (6,0);

\draw[pattern=northwest, hatch distance=15pt, hatch thickness = 0.3pt] (6.8,0) ellipse (0.8cm and 1cm);
\node (Z02) at (6.8,0) {$\mathbf{Q'[-2]}$};

\draw[->] (7.6,0) -- (8,0);
\draw[-] (8,0) -- (8.3,0);

\draw[pattern=northwest, hatch distance=15pt, hatch thickness = 0.3pt] (9.1,0) ellipse (0.8cm and 1cm);
\node (Z01) at (9.1,0) {$\mathbf{Q'[-1]}$};

\draw[->] (9.9,0) -- (10.3,0);
\draw[-] (10.3,0) -- (10.6,0);

\draw[pattern=northwest, hatch distance=15pt, hatch thickness = 0.3pt] (11.4,0) ellipse (0.8cm and 1cm);
\node (Z) at (11.4,0) {$\mathbf{Q'[0]}$};

\draw[->] (12.2,0) -- (12.6,0);
\draw[-] (12.6,0) -- (12.9,0);

\draw[pattern=northwest, hatch distance=15pt, hatch thickness = 0.3pt] (13.7,0) ellipse (0.8cm and 1cm);
\node (Z1) at (13.7,0) {$\mathbf{Q'[1]}$};

\draw[->] (14.5,0) -- (14.9,0);
\draw[-] (14.9,0) -- (15.2,0);

\draw[pattern=northwest, hatch distance=15pt, hatch thickness = 0.3pt] (16,0) ellipse (0.8cm and 1cm);
\node (Z2) at (16,0) {$\mathbf{Q'[2]}$};

\draw[->] (16.8,0) -- (17.2,0);
\draw[-] (17.2,0) -- (17.5,0);

\draw[loosely dotted] (17.6,0) -- (18.6,0);
\end{tikzpicture}
\end{equation}
Since this is now an infinite quiver, we cannot speak of a fi\-nite-di\-men\-sion\-al path algebra over $\Qks$. However, we can consider instead the path category $\K \Qks$ and the two-sided ideal $\Rks\subseteq \K\Qks$ generated by all elements in $\K Q[z]$ that lie in $\cR[z]$ and all paths $Q'[z]\to Q'[z+1]$ for $z\in \ZZ$. In some sense this path category can be thought of as the infinite gluing $\cdots\glue A \glue A \glue A \glue \cdots$. Moreover, if $\m A$ has an $n$-fractured subcategory $\cM$ satisfying certain symmetry conditions, then $\cM$ induces a subcategory $\Mks$ of $\m (\K \Qks /\Rks)$ which can be thought of as an analogue of an $n$-fractured subcategory. Under certain conditions, the subcategory $\Mks$ is an actual $n$-cluster tilting subcategory. The study of this situation is the main motivation for this paper.

\subsubsection{The general situation} As a first step, we generalize the above situation. Let $G$ be a \emph{directed tree}, that is $G$ is a directed graph where the underlying undirected graph is acyclic. We also assume that every vertex in $G$ is the source and target of finitely many arrows in $G$. We denote the set of vertices of $G$ by $V_G$ and the set of arrows of $G$ by $E_G$. We also denote by $\dirI_G$ be the directed set of finite connected subgraphs of $G$.

We start by decorating each vertex $v\in V_G$ by a rep\-re\-sen\-ta\-tion-di\-rect\-ed bound quiver algebra $\La_v=\K Q_v/\cR_v$ in such a way that if $e:u\to v$ is an arrow of $G$, then the algebra $\La_{\langle u,v\rangle}\coloneqq \La_u\glue \La_v=\K Q_{\langle u,v\rangle}/\cR_{\langle u,v\rangle}$ is well-defined. We also introduce assumptions which have the effect that for every $H\in\dirI_G$, we may perform the gluings induced by the arrows of $H$ in any order to obtain a rep\-re\-sen\-ta\-tion-di\-rect\-ed algebra $\La_H=\K Q_H/\cR_H$. This is the data of a \emph{gluing system} $(\La_v)_{v\in V_G}$; for the complete definition see Definition \ref{def:gluing system}. Such a system comes equipped with canonical algebra epimorphisms $F_{HK}:\La_K\twoheadrightarrow \La_H$ for all finite connected subgraphs $H\subseteq K$ of $G$. Furthermore, the category $\m\La_{H}$ can be described completely via the categories $\m\La_v$ for $v\in V_H$. To apply our methods to a potential limit $\La_{G}$ of this system, we would like to be able to describe $\m\La_{G}$ via the categories $\m\La_H$ for $H\in\dirI_G$.

One way to do this is the following. Consider the (potentially infinite) quiver $\Qks$ obtained after performing the (potentially infinitely many) gluings induced by the arrows of $G$. We may consider its path category $\K \Qks$. We show that the union $\Rks=\bigcup_{H\in\dirI_G}\cR_H\subseteq \K \Qks$ is a two-sided ideal of $\K \Qks$. By the definition of the ideal $\Rks$ it is not too difficult to see that $\m(\proj (\K\Qks/\Rks))$ can be described via the categories $\m\La_H$ for $H\in\dirI_G$. In particular, if $G$ is finite, we have $\m\La_G \equivalent \m(\proj(\K \Qks/\Rks))$. In this way we may consider $\proj(\K \Qks/\Rks)$ as a suitable candidate for the limit of the system $(\La_v)_{v\in V_G}$; we will soon make this more precise. 

Assume now that the module category $\m\La_v$ of each algebra $\La_v$ in our gluing system admits an $n$-fractured subcategory $\cM_v\subseteq \m\La_v$, and that again these $n$-fractured subcategories are compatible with the gluing. This is the data of an \emph{$n$-fractured system}, see Definition \ref{def:n-fractured system}. By gluing the $n$-fractured subcategories of $\m\La_v$ for $v\in V_H$, the category $\m\La_{H}$ obtains an $n$-fractured subcategory $\cM_{H}$. Furthermore, by construction, the restriction of scalars functor $F_{HK\ast}:\m\La_H\to \m\La_K$ satisfies $F_{HK\ast}(\cM_H) \subseteq \cM_K$ for $H,K\in\dirI_G$ with $H\subseteq K$. This data gives rise to a subcategory $\Mks\subseteq \m (\K\Qks/\Rks)$. In general $\Mks$ is not an $n$-cluster tilting subcategory. To remedy this we introduce the notion of \emph{complete} $n$-fractured systems, see Definition \ref{def:n-fractured system}. We have the following theorem, which is a special case of Theorem \ref{thrm:complete n-fractured gives n-cluster tilting}. 

\begin{theorem}\label{thrm:complete n-fractured Mks is n-ct}
Let $\cL=\left(\La_v\right)_{v\in V_G}$ be a gluing system and $\left(\cM_v\right)_{v\in V_G}$ be a complete $n$-fractured system of $\cL$. Then $\Mks\subseteq \m(\K\Qks/\Rks)$ is an $n$-cluster tilting subcategory.
\end{theorem}

In the above we have considered the category $\m(\K\Qks/\Rks)$ which, in general, is not the module category of a fi\-nite-di\-men\-sion\-al algebra. To deal with this situation we first prove our results in an abstract categorical setting where we consider modules over categories instead of modules over algebras. 

Starting with a gluing system $(\La_v)_{v\in V_G}$ we obtain an algebra $\La_H$ for every $H\in\dirI_G$ and an algebra epimorphism $F_{HK}:\La_K\to \La_H$ for every $H,K\in\dirI_G$ with $H\subseteq K$. Furthermore, these epimorphisms satisfy $F_{HK}\comp F_{KL}=F_{HL}$ for every $H,K,L\in\dirI_G$ with $H\subseteq K \subseteq L$. In particular, the pair $(\La_H,F_{HK})_{H,K\in\dirI_G}$ is an inverse system in the category of fi\-nite-di\-men\-sion\-al algebras. Since $\m\La_H \equivalent \m(\proj\La_H)$ for every $H\in \dirI_G$, we may consider the category $\proj\La_H$ as a categorical replacement for the algebra $\La_H$. Similarly, we may consider the extension of scalars functor $F_{HK}^{\ast}:\proj\La_K\to \proj\La_H$ as a replacement of the algebra epimorphism $F_{HK}:\La_K \to\La_H$ for $H,K\in\dirI_G$ with $H\subseteq K$. Note that we do not have a strict equality $F_{HK}^{\ast}\comp F_{KL}^{\ast}=F_{HL}^{\ast}$. However, in Section \ref{subsubsec:gluing systems of representation-directed algebras} we define natural isomorphisms $\theta_{HKL}:F_{HK}\comp F_{KL}\To F_{HL}$. Then the triple $\left(\proj\La_H,F_{HK},\theta_{HKL}\right)_{H,K,L\in\dirI_G}$ can be thought of as an inverse system over $\dirI_G$, up to natural isomorphism.

This motivates us to consider the general situation where we study an inverse system $\left(\cC_i, F_{ij},\theta_{ijk}\right)$ of Krull--Schmidt categories over a directed set $\dirI$, see Definition \ref{def:inverse system}. This situation is more technical than usual inverse systems because of the introduction of natural isomorphisms $\theta_{ijk}:F_{ij}\comp F_{jk}\To F_{ik}$ satisfying certain compatibility conditions. To avoid set theoretic issues, we consider this inverse system in the category of small categories $\Cat$. It turn out that the categorical inverse limit of such a system is too big for our purposes. However, we find a subcategory $\Cks$ of the inverse limit, suitable for our purposes, called the \emph{firm source}, see Definition \ref{def:firm objects and Cks}. The firm source $\Cks$ comes equipped with canonical functors $\Pks_i:\Cks\to \cC_i$ and natural isomorphisms $\Tks_{ij}:F_{ij}\comp \Pks_j \To \Pks_i$. In particular we have the following theorem.

\begin{theorem}\label{thrm:Cks is Krull-Schmidt}
Let $\left(\cC_i,F_{ij},\theta_{ijk}\right)$ be a $\Cat$-inverse system over a directed set $\dirI$. If all $\cC_i$ are Krull--Schmidt categories, then the firm source $\Cks$ is a Krull--Schmidt category. 
\end{theorem}

Theorem \ref{thrm:Cks is Krull-Schmidt} follows by Corollary \ref{cor:properties of Cks}. In fact we further prove that we can describe indecomposable objects in $\Cks$ using indecomposable objects of $\cC_i$. In particular, we show in Corollary \ref{cor:equivalence of quiver and projectives} that the firm source of the system $\left(\proj\La_H,F_{HK},\theta_{HKL}\right)_{H,K,L\in\dirI_G}$ is equivalent to $\proj(\K \Qks/\Rks)$. It is in this sense that we may formally consider $\proj(\K \Qks/\Rks)$ as the limit of the system $\left(\La_v\right)_{v\in V_G}$.

Next, we want to generalize the notion of $n$-fractured subcategories in the abstract setting. In a sense, an $n$-fractured subcategory can also be thought as a subcategory which is close to being $n$-cluster tilting but it fails to have the correct extensions with a few indecomposable modules. In Definition \ref{def:asymptotically weakly n-cluster tilting} we mimic this property of $n$-fractured subcategories to define a generalization of complete $n$-fractured systems in the abstract case of abelian categories; we call such a system an \emph{asymptotically weakly $n$-cluster tilting system}. We obtain the following general result.

\begin{theorem}\label{thrm:inverse system gives n-ct}
Let $(\Cks,\Pks_i,\Tks_{ij})$ be the firm source of the $\Cat$-inverse system $\left(\cC_i,F_{ij},\theta_{ijk}\right)$ over a directed set $\dirI$ and let $\left(\cM_i\right)$ be an asymptotically weakly $n$-cluster tilting system of $\left(\m\cC_i,F_{ij\ast},\theta_{ijk\ast}\right)$. If $\m\Cks$ is locally bounded and $\cM_i$ is functorially finite for every $i\in\dirI$, then $\Mks \coloneqq \add\{ \Pks_{i\ast}\left(\cM_i\right) \mid i\in\dirI\}$ is an $n$-cluster tilting subcategory of $\m\Cks$.
\end{theorem}

Theorem \ref{thrm:inverse system gives n-ct} is a special case of Corollary \ref{cor:n-ct if locally bounded}. In particular, for a gluing system, the firm source $\proj(\K \Qks/\Rks)$ is always locally bounded and each $\cM_H$ is functorially finite since $\La_H$ is rep\-re\-sen\-ta\-tion-fi\-nite for every $H\in\dirI_G$. Hence Theorem \ref{thrm:complete n-fractured Mks is n-ct} is a corollary of Theorem \ref{thrm:inverse system gives n-ct}.

\subsubsection{A motivating example revisited} Although starting from a complete $n$-fractured system we obtain an $n$-cluster tilting subcategory in $\m(\K \Qks/\Rks)$, this is not, in general, an $n$-cluster tilting subcategory of the module category of a fi\-nite-di\-men\-sion\-al algebra. To remedy this, we use the results of \cite{DI} we apply an orbit construction to $\m(\K\Qks/\Rks)$. First assume that $G$ is the graph
\[\cdots \xrightarrow{\alpha_{-3}} -2 \xrightarrow{\alpha_{-2}} -1 \xrightarrow{\alpha_{-1}} 0 \xrightarrow{\hspace*{3.3pt}\alpha_{0}\hspace*{3.3pt}} 1 \xrightarrow{\hspace*{3.3pt}\alpha_{1}\hspace*{3.3pt}} 2 \xrightarrow{\hspace*{3.3pt}\alpha_{2}\hspace*{3.3pt}} \cdots. \]
Then $(A[z])_{z\in V_G}$ is a gluing system of $G$ where $A[z]$ is as in Section \ref{subsubsec:a motivating example}. Moreover, in this case the infinite quiver $\Qks$ is as in (\ref{picture:infinite quiver}). There is an admissible $\ZZ$-action on $\K\Qks/\Rks$ defined by letting $k\in\ZZ$ map a path in $k$-steps to the right in $\Qks$. This $\ZZ$-action on $\K\Qks/\Rks$ induces a $\ZZ$-action on $\proj(\K \Qks/\Rks)$, which satisfies the conditions of \cite[Theorem 2.13]{DI} and so it gives rise to an $n$-cluster tilting subcategory $\tilM$ of $\m\big(\proj(\K \Qks/\Rks)/\ZZ\big)$. We then show the following theorem.

\begin{theorem}\label{thrm:n-cluster tilting for quotient}
There is an equivalence of categories $\m\big(\proj(\K \Qks/\Rks)/\ZZ\big)\equivalent \m(\K \tilQ/\tilR)$ where $\K \tilQ/\tilR$ is a fi\-nite-di\-men\-sion\-al bound quiver algebra. In particular, $\m(\K \tilQ/\tilR)$ admits an $n$-cluster tilting subcategory $\tilM$. 
\end{theorem}

Theorem \ref{thrm:n-cluster tilting for quotient} follows by Corollary \ref{cor:La infinity n-ct}. In this way we obtain an $n$-cluster tilting module of a fi\-nite-di\-men\-sion\-al algebra. Moreover, this new algebra is not, in general, rep\-re\-sen\-ta\-tion-di\-rect\-ed, although it is still rep\-re\-sen\-ta\-tion-fi\-nite. 

In fact our results are quite more general and we apply them to some more classes of gluing systems, giving us many more examples of $n$-cluster tilting modules. In all cases the input should be a class of rep\-re\-sen\-ta\-tion-di\-rect\-ed algebras whose module categories admit $n$-fractured subcategories. Many such algebras where produced in \cite{VAS} and \cite{VAS2} inside the class of Nakayama algebra. In Proposition \ref{prop:starlike algebras} we classify the radical square zero path algebras where the underlying graph of the quiver is starlike, whose module categories admit an $n$-cluster tilting subcategory. These algebras give, in themselves, more examples of $n$-cluster tilting subcategories. By gluing them finitely many times appropriately we get a large class of examples of fi\-nite-di\-men\-sion\-al algebras, whose module categories admit an $n$-cluster tilting subcategory. In particular we have the following theorem.

\begin{theorem}\label{thrm:any sinks and sources}
For all $s,t\in\ZZ_{\geq 0}$ there exists a bound quiver algebra $\La=\K Q/\cR$ such that $Q$ has $s$ sources and $t$ sinks and such that $\m\La$ admits an $n$-cluster tilting subcategory.
\end{theorem}

The proof of Theorem \ref{thrm:any sinks and sources} is constructive; for examples see Example \ref{ex:4 sources 3 sinks} and Example \ref{ex:k-simultaneous gluing}. Most (but not all) of our examples are given by radical square zero algebras. This is not because of some obstacle for other classes of algebras, it is rather that radical square zero algebras are easier to compute with and the class of radical square zero algebras is closed under the operations of gluing. 

This paper is organized as follows. In Section \ref{sec:preliminaries} we establish notation and recall some basic facts about modules over categories. In Section \ref{sec:admissible targets weakly n-ct} we define asymptotically weak $n$-cluster tilting systems and show that, under reasonable conditions, they give rise to $n$-cluster tilting subcategories. Although the nature of these results is quite technical, the proofs are straightforward. We include however most of the details to avoid ambiguities. In Section \ref{sec:asymptotically weakly n-cluster tilting systems from algebras} we show how a complete $n$-fractured system of rep\-re\-sen\-ta\-tion-di\-rect\-ed algebras gives rise to an asymptotically weak $n$-cluster tilting system. In Section \ref{sec:applications} we start by considering finitely many gluings and then proceed to the infinite case. We then recall and apply the main result of \cite{DI} to obtain fi\-nite-di\-men\-sion\-al algebras with $n$-cluster tilting modules. In Section \ref{sec:concrete examples} we apply the above methods to specific examples of bound quiver algebras. Many examples are computed explicitly and we obtain new families of fi\-nite-di\-men\-sion\-al algebras whose module categories $n$-cluster tilting modules. We also include an index of definitions and notation used in this paper.

\section{Preliminaries}\label{sec:preliminaries}

\subsection{Background and notation}\label{subsec:background and notation}
Throughout this paper, $n$ denotes a positive integer and $\K$\index[symbols]{K@$\K$} denotes a field.

We start by setting up notation and recalling some background on categories and functor categories. We denote by $\M\K$ the category of $\K$-vector spaces with morphisms given by linear maps. In this paper by a ($\K$-)algebra $\La$ we mean a basic fi\-nite-di\-men\-sion\-al unital associative algebra over $\K$ and by $\La$-module we mean a right $\La$-module. We denote by $\M \La$ the category of $\La$-modules and by $\m \La$ the full subcategory of finitely generated $\La$-modules. We denote by $\proj\La$\index[symbols]{proj L, inj L@$\proj\La$, $\inj\La$} and $\inj\La$\index[symbols]{proj L, inj L@$\proj\La$, $\inj\La$} the full subcategories of $\m\La$ consisting of projective and injective modules, respectively.

Let $\cC$ be a category. We denote by $\obj(\cC)$ the collection of objects in $\cC$ and for an object $x$ in $\cC$ we simply write $x\in \cC$ instead of $x\in\obj\left(\cC\right)$. For $x,y\in\cC$ we write $\cC\left(x,y\right)$ for the set of morphisms $f:x\to y$. We say that $\cC$ is a \emph{small}\index[definitions]{small category} category if $\obj(\cC)$ is a set and we say that $\cC$ is \emph{skeletally small}\index[definitions]{skeletally small category} if the class of isomorphism classes of objects in $\cC$ is a set. For functors $F,G:\cC\to \cD$ between categories and a natural transformation $\eta$ from $F$ to $G$ we sometimes write $\eta:F\To G$\index[symbols]{e ta@$\eta:F\To G$}. If $\cC$ is isomorphic respectively equivalent to a category $\cD$, we write $\cC\isom \cD$\index[symbols]{C isomorphic D@$\cC\isom \cD$} respectively $\cC\equivalent \cD$\index[symbols]{C equivalent D@$\cC\equivalent \cD$}.

Let $\cM$ be a full subcategory of a category $\cC$ and let $x\in \cC$. A \emph{right $\cM$-approximation}\index[definitions]{right approximation} of $x$ is a morphism $f:m\to x$ with $m\in\cM$ such that all morphisms $f':m'\to x$ with $m'\in \cM$ factor through $f$. We say that $\cM$ is \emph{contravariantly finite}\index[definitions]{contravariantly finite subcategory} in $\cC$ if every $x\in \cC$ has a right $\cM$-approximation. The notions \emph{left $\cM$-approximation}\index[definitions]{left approximation} and \emph{covariantly finite}\index[definitions]{covariantly finite subcategory} are defined dually. We say that $\cM$ is \emph{functorially finite}\index[definitions]{functorially finite subcategory} in $\cC$ if $\cM$ is both contravariantly finite and covariantly finite.

We say that a category $\cC$ is \emph{$\K$-linear}\index[definitions]{k-linear category@$\K$-linear category} (respectively \emph{preadditive}\index[definitions]{preadditive category}) if for every $x,y\in\cC$ the set of morphisms $\cC(x,y)$ is a $\K$-vector space (respectively abelian group)  and composition of morphisms is $\K$-bilinear (respectively bilinear). We say that a functor $F:\cC\to \cD$ between $\K$-linear (respectively preadditive) categories is \emph{$\K$-linear}\index[definitions]{k-linear functor@$\K$-linear functor} (respectively \emph{additive}\index[definitions]{additive functor}) if the induced map on objects $F:\cC(x,y)\to \cD(F(x),F(y))$ is a $\K$-module homomorphism (respectively group homomorphism) for every $x,y\in\cC$. All functors between $\K$-linear (respectively preadditive) categories are assumed to be $\K$-linear (respectively additive). We say that $\cC$ is \emph{Hom-finite} if $\cC$ is a $\K$-linear category and the dimension of $\cC(x,y)$ is finite for every $x,y\in\cC$. We say that $\cC$ is an \emph{additive category}\index[definitions]{additive category} if $\cC$ is preadditive and admits all finite products. If $\cC$ is an additive category and $x,y\in\cC$, then the \emph{direct sum}\index[definitions]{direct sum} of $x$ and $y$ denoted by $x\oplus y$ is the product of $x$ and $y$.

Let $\cM$ be a collection of objects in an additive category $\cC$. We denote by $\add(\cM)$\index[symbols]{add(M)@$\add(\cM)$} the \emph{additive closure}\index[definitions]{additive closure of a subcategory} of $\cM$, that is the smallest full subcategory of $\cC$ containing $\cM$ and being closed under direct sums, direct summands and isomorphisms. If $\cM$ is a subcategory of $\cC$, we set $\add(\cM)\coloneqq \add(\obj(\cM))$.

Let $\cC$ be a skeletally small additive category. A nonzero object $x\in \cC$ is called \emph{indecomposable}\index[definitions]{indecomposable object} if $x\isom a\oplus b$ implies $a\isom 0$ or $b\isom 0$. We denote by $\ind\cC$\index[symbols]{ind C@$\ind\cC$} a chosen set of representatives of isomorphism classes of indecomposable objects in $\cC$. We say that $\cC$ is a \emph{Krull--Schmidt category}\index[definitions]{Krull--Schmidt category} if each object in $\cC$ is isomorphic to a finite direct sum of indecomposable objects in $\cC$ having local endomorphism rings.

There is an equivalent way to characterize a Krull--Schmidt category. We say that a category $\cC$ has \emph{split idempotents}\index[definitions]{split idempotents} if for every idempotent $e\in\cC(x,x)$ there exists an object $y\in\cC$ and morphisms $f\in\cC(x,y)$ and $g\in\cC(y,x)$ such that $g\comp f=e$ and $f\comp g=\Id_y$. We say that a ring $R$ is \emph{semiperfect}\index[definitions]{semiperfect ring} if there exists a decomposition $1_R=\sum_{i=1}^k e_i$ of identity where $e_i$ are idempotents such that $e_ie_j=0$ for $i\neq j$ and $e_i Re_i$ is a local ring. Then a skeletally small additive category $\cC$ is a Krull--Schmidt category if and only if it has split idempotents and $\cC(x,x)$ is a semiperfect ring for all $x\in\cC$.

Let $\cC$ be a skeletally small $\K$-linear category. A \emph{(right) $\cC$-module}\index[definitions]{right module} is a $\K$-linear functor $M:\cC^{\op}\to \M \K$. We denote by $\M\cC$\index[symbols]{Mod C, mod C@$\M\cC$, $\m\cC$} the category of all $\cC$-modules with morphisms given by natural transformations. For $M,N\in \M\cC$ we write $\Hom_{\cC}\left(M,N\right)$ instead of $\M\cC\left(M,N\right)$. If $\cC$ is moreover a Krull--Schmidt category, then the \emph{support of a module}\index[definitions]{support of a module} $M\in\M\cC$ is defined to be
\[\supp\left(M\right) = \left\{ x\in \ind\cC \mid M\left(x\right) \neq 0 \right\} \subseteq \ind\cC.\]\index[symbols]{supp M@$\supp(M)$}

A $\cC$-module $M:\cC^{\op} \to \M\K$ is called \emph{finitely presented}\index[definitions]{finitely presented module} if there exists an exact sequence
\[ \cC\left( - , y\right) \rightarrow \cC\left(-,x\right) \rightarrow M \rightarrow 0\]
in $\M\cC$. We denote by $\m\cC$\index[symbols]{Mod C, mod C@$\M\cC$, $\m\cC$} the full subcategory of $\M\cC$ consisting of finitely presented modules. The full subcategory of $\m\cC$ consisting of all projective objects is denoted by $\proj\cC$ and the full subcategory of $\m\cC$ consisting of all injective objects is denoted $\inj\cC$. By Yoneda's Lemma we have the fully faithful \emph{Yoneda functor}\index[definitions]{Yoneda functor} $h_{\cC}:\cC \to \m\cC$\index[symbols]{hC@$h_{\cC}$} given by $x\mapsto \cC(-,x)$ and $f\in\cC(x,y)\mapsto f\circ-: \cC(-,x)\to \cC(-,y)$ which restricts to an equivalence between $\cC$ and $\proj\cC$. In particular we have that $\m\cC$ has enough projective objects. With this setup, for an algebra $\La$ we can identify $\m\La$ with $\m(\proj\La)$. 

Let $\cC$ be a skeletally small $\K$-linear Hom-finite Krull--Schmidt category. It is easy to see that $\m\cC$ is closed under cokernels in $\M\cC$ and by the Horseshoe Lemma it follows that $\m\cC$ is closed under extensions in $\M\cC$. However, it is not true in general that $\m\cC$ is closed under kernels. Recall from \cite{AR} that $\cC$ is called a \emph{dualizing $\K$-variety}\index[definitions]{dualizing $\K$-variety} if the functors $\D:\M\cC\to\M\left(\cC^{\op}\right)$ and $\D:\M\left(\cC^{\op}\right)\to\M\cC$ given by $\D=\Hom_{\K}(-,\K)$\index[symbols]{D HomK K@$\D=\Hom_{\K}(-,\K)$} induce dualities $\D:\m\cC\to\m\left(\cC^{\op}\right)$ and $\D:\m\left(\cC^{\op}\right)\to\m\cC$. In this case $\m\cC$ is closed under kernels in $\M\cC$ since $\m\left(\cC^{\op}\right)$ is closed under cokernels in $\M\left(\cC^{\op}\right)$. Hence if $\cC$ is a dualizing $\K$-variety, then $\m\cC$ is closed under kernels, cokernels and extensions in $\M\cC$ and so it is an abelian subcategory of $\M\cC$. It follows that in this case  $\m\cC$ has enough injective objects and that complete sets of representatives of indecomposable projective and injective objects in $\m\cC$ are then given by
    \[\ind(\proj\cC) = \left\{ {\cC}\left(-,x\right) \in \m\cC \mid x\in \ind\cC\right\} \text{ and } \ind(\inj\cC) = \left\{ D{\cC}\left(x,-\right) \in \m\cC \mid x\in \ind\cC\right\}.\] 
For more details on dualizing $\K$-varieties, we refer to \cite[Section 2]{AR}.

Let $\cC$ be a a skeletally small $\K$-linear Krull--Schmidt category. We say that $\cC$ is \emph{locally bounded}\index[definitions]{locally bounded category} if for all $x\in\ind\cC$ we have
\[\sum_{y\in\ind\cC} \left(\dim_{\K}\cC\left(x,y\right)+\dim_{\K}\cC\left(y,x\right)\right) < \infty.\]
Clearly if $\cC$ is locally bounded, then $\cC$ is Hom-finite. In this case, the objects in $\m\cC$ are precisely the \emph{finite-length ones}, that is modules for which a \emph{composition series} exists, i.e. a finite filtration such that the quotient of any two consecutive terms is simple. Therefore $\supp(M)$ is a finite set for all $M\in\m\cC$ and the functor $\D: \M\cC \longleftrightarrow \M\left(\cC^{\op}\right)$ restricts to a duality $\D: \m\cC \longleftrightarrow \m\left(\cC^{\op}\right)$ satisfying $\D\circ \D \isom \Id_{\m\cC}$. Hence a locally bounded category is always a dualizing $\K$-variety.

Next we recall some background on the representation theory of fi\-nite-di\-men\-sion\-al algebras. A \emph{quiver}\index[definitions]{quiver} $Q=(Q_0,Q_1,s,t)$\index[symbols]{Q (Q1 Q2 s t)@$Q=(Q_0,Q_1,s,t)$} is a quadruple consisting of a set $Q_0$ of \emph{vertices}, a set $Q_1$ of \emph{arrows} and two maps $s,t:Q_1\to Q_0$ called \emph{source map} and \emph{target map}. We say that $Q$ is \emph{finite} if both $Q_0$ and $Q_1$ are finite sets. A \emph{subquiver} of a quiver $Q$ is a quiver $Q'$ with $Q'_0\subseteq Q_0$, $Q'_1\subseteq Q_1$ and such that $\restr{s}{Q'_0}=s'$ and $\restr{t}{Q'_1}=t'$; it is called \emph{full} if every arrow $\alpha\in Q_1$ with $s(\alpha),t(\alpha)\in Q'_0$ is in $Q'_1$. For a vertex $i\in Q_0$, the \emph{outgoing degree}\index[definitions]{outgoing degree of a vertex} of $i$, denoted by $\delta^{+}(i)$, is the number of outgoing arrows starting at $i$ and the \emph{ingoing degree}\index[definitions]{ingoing degree of a vertex} of $i$, denoted by $\delta^{-}(i)$, is the number of ingoing arrows ending at $i$. In other words we have $\delta^{+}(i) = \abs{\{s^{-1}(i)\}}$\index[symbols]{delta@$\delta^{+}(i)$, $\delta^{-}(i)$, $\delta(i)$}, and $\delta^{-}(i)=\abs{\{t^{-1}(i)\}}$\index[symbols]{delta@$\delta^{+}(i)$, $\delta^{-}(i)$, $\delta(i)$}. The \emph{degree}\index[definitions]{degree of a vertex} of $i$ is then the number $\delta (i)\coloneqq \delta^{+}(i)+\delta^{-}(i)$\index[symbols]{delta@$\delta^{+}(i)$, $\delta^{-}(i)$, $\delta(i)$}. If $\delta(i)<\infty$ for every vertex $i\in Q_0$ then we say that $Q$ is \emph{locally finite}\index[definitions]{locally finite quiver}. All quivers considered in this paper will be locally finite. For a quiver $Q$ and $l\geq 1$, a \emph{path of length $l$} in $Q$ is a sequence of arrows $p=\alpha_1\cdots\alpha_l$, such that $t(\alpha_k)=s(\alpha_{k+1})$; we define $s(p)=s(a_1)$ and $t(p)=t(\alpha_l)$. We denote a path $p$ from a vertex $i$ to a vertex $j$ in $Q$ by $i\overset{p}{\leadsto}j$\index[symbols]{i q@$i\overset{p}{\leadsto}j$}. We also assign a trivial path $\epsilon_i$ of length $0$ to each vertex $i\in Q_i$. 

Throughout, we denote by $\overrightarrow{A}_h$\index[symbols]{A h@$\overrightarrow{A}_h$} the quiver $1\overset{\alpha_1}{\longrightarrow} 2 \overset{\alpha_2}{\longrightarrow} 3 \longrightarrow\cdots\longrightarrow h-1\overset{\alpha_{h-1}}{\longrightarrow} h$.

Using quivers one can define \emph{bound quiver algebras} and their \emph{representations}\index[definitions]{representation of a bound quiver algebra}. For details on bound quiver algebras, their representations and their connection to the representation theory of fi\-nite-di\-men\-sion\-al $\K$-algebras, as well as the notation we use, we refer to \cite[Chapters II-III]{ASS}.

Let $Q$ be a quiver. We denote by $R_Q$\index[symbols]{RQ@$R_Q$} the ideal of $\K Q$ generated by $Q_1$ and call it the \emph{arrow ideal}\index[definitions]{arrow ideal} of $Q$. Contrary to the usual notation, we denote admissible ideals of the path algebra $\K Q$ by $\cR$\index[symbols]{KQ R@$\K Q/\cR$}. Throughout all ideals of quiver path algebras are assumed to be admissible. We identify $\K Q/\cR$-modules with representations of the bound quiver $(Q,\cR)$. For a representation $M$ of $\K Q/\cR$ we denote the \emph{support}\index[definitions]{support of a representation} of $M$ by $\supp(M)\coloneqq \{i\in Q_0\mid M_i\neq 0\}$\index[symbols]{supp M@$\supp(M)$}. 

Let $\La$ be a $\K$-algebra. Given a subcategory $\cA$ of $\m \La$, we denote by $\Sub(\cA)$\index[symbols]{Sub A@$\Sub(\cA)$} the subcategory of $\m A$ containing all submodules of modules in $\cA$ and by $\Fac(\cA)$\index[symbols]{Fac A@$\Fac(\cA)$} the subcategory of $\m A$ containing all factor modules of modules in $\cA$. For $M\in \m\La$ we denote by $\om (M)$\index[symbols]{omega M@$\om(M)$, $\om^{-}(M)$} the \emph{syzygy}\index[definitions]{syzygy of a module} of $M$, that is the kernel of $P\twoheadrightarrow M$, where $P$ is the projective cover of $M$ and by $\om^- (M)$\index[symbols]{omega M@$\om(M)$, $\om^{-}(M)$} the \emph{cosyzygy}\index[definitions]{cosyzygy of a module} of $M$, that is the cokernel of $M\hookrightarrow I$ where $I$ is the injective hull of $M$. Note that $\om (M)$ and $\om^- (M)$ are unique up to isomorphism. The \emph{projective dimension}\index[definitions]{projective dimension of a module} of $M$, denoted by $\pdim(M)$\index[symbols]{proj dim M@$\pdim(M)$}, is defined to be
\[\pdim(M)=\begin{cases} k, &\mbox{\text{ if there exists $k\in \ZZ$ with $\om^k(M)$ nonzero and $\om^i(M)=0$ for $i>k$,}} \\ \infty, &\mbox{ otherwise.}\end{cases}\]
The \emph{injective dimension}\index[definitions]{injective dimension of a module} $\idim(M)$\index[symbols]{inj dim M@$\idim(M)$} of $M$ is defined dually. The \emph{global dimension}\index[definitions]{global dimension of an algebra} $\gldim(\La)$\index[symbols]{gl dim Lambda@$\gldim(\La)$} of $\La$ is defined to be 
\[\gldim(\La) = \sup\{\pdim(M) \mid M\in \m\La\}.\]
We denote by $\tau$\index[symbols]{T b@$\tau$, $\tau^-$, $\tn$, $\tno$} and $\tau^-$\index[symbols]{T b@$\tau$, $\tau^-$, $\tn$, $\tno$} the \emph{Aus\-lan\-der--Rei\-ten translations}\index[definitions]{Aus\-lan\-der--Rei\-ten translations}, see \cite[Definition IV.2.3]{ASS}. Following \cite{IYA1}, we denote by $\tn$\index[symbols]{T b@$\tau$, $\tau^-$, $\tn$, $\tno$} and $\tno$\index[symbols]{T b@$\tau$, $\tau^-$, $\tn$, $\tno$} the \emph{$n$-Aus\-lan\-der--Rei\-ten translations}\index[definitions]{n-Aus\-lan\-der--Rei\-ten translations@$n$-Aus\-lan\-der--Rei\-ten translations} defined by $\tn (M) = \tau \om^{n-1} (M)$ and $\tno (M) = \tau^- \om^{-(n-1)}(M)$. If $X\subseteq\La$, we denote by $\langle X \rangle$ the two-sided ideal of $\La$ generated by $X$. We denote the Aus\-lan\-der--Rei\-ten quiver of $\La$ by $\Gamma(\La)$\index[symbols]{Gamma Lambda@$\Gamma(\La)$}. When drawing $\Gamma(\La)$ in examples, we label the vertex $[M]$\index[symbols]{M a@$[M]$} for an indecomposable $M\in\m\La$ using the composition series of $M$. For more details on classical Aus\-lan\-der--Rei\-ten theory we refer to \cite[Chapter IV]{ASS}.

A \emph{directed graph}\index[definitions]{directed graph} is a quiver with no multiple arrows between two vertices. Recall that by our conventions all quivers are locally finite and so all directed graphs are locally finite too. If $G$ is a directed graph, we write $V_G$\index[symbols]{VG@$V_G$} instead of $G_0$, we write $E_G$\index[symbols]{EG@$E_G$} instead of $G_1$, we say ``subgraph'' instead of ``subquiver'' and we say ``walk'' instead of ``path''. If $V\subseteq V_G$ is a collection of vertices in $G$, then the \emph{subgraph induced in $G$ by $V$}\index[definitions]{subgraph induced by vertices}, denoted by $\langle V \rangle$\index[symbols]{V@$\langle V \rangle$, $V\subseteq V_G$}, is the full subgraph of $G$ with vertex set $V$. We say that $G$ is a \emph{directed tree}\index[definitions]{directed tree} if $G$ is connected and the underlying undirected graph is acyclic. 

\subsection{Induced functors between module categories}
We briefly recall some facts about induced functors between module categories; for more details we refer to \cite{KRA}. For the rest of this section, let $F:\cC \to \cD$ be a functor between skeletally small $\K$-linear Krull--Schmidt categories. Then $F$ induces an exact functor $F_{\ast}:\M\cD \to \M\cC$\index[symbols]{F Q@$F_{\ast}$, $F^{"!}$, $F^{\ast}$} given by $F_{\ast}(M) = M\circ F$. This functor has a right adjoint $F^{!}:\M\cC\to\M\cD$\index[symbols]{F Q@$F_{\ast}$, $F^{"!}$, $F^{\ast}$}, given by 
\begin{align*} 
F^{!}(M)(y) &= \Hom_{\cC}\left(\cD\left(F(-),y\right),M\right),
\end{align*}
for every $M\in \M\cC$ and $y\in \cD$. Moreover $F_{\ast}$ has a left adjoint as well. To describe the left adjoint, first recall that there exists a unique, up to natural isomorphism, \emph{tensor product} bifunctor $-\otimes_{\cC}-:\M\cC\times \M\left(\cC^{\op}\right)\to\M\K$ characterized by
\begin{enumerate}[label=(\roman*)]
    \item $M\otimes_{\cC} \cC(x,-) \isom M(x)$ and $\cC(-,x)\otimes_{\cC} N \isom N(x)$, and
    \item $M\otimes_{\cC}-$ and $-\otimes_{\cC} N$ have right adjoints,
\end{enumerate}
for every $M\in \M\cC$, $N\in\M\left(\cC^{\op}\right)$ and $x\in \cC$. Then the left adjoint $F^{\ast}:\M\cC\to \M\cD$\index[symbols]{F Q@$F_{\ast}$, $F^{"!}$, $F^{\ast}$} is given by
\begin{align*} 
F^{\ast}(M)(y) &= M\otimes_{\cC}\cD\left(y,F(-)\right),
\end{align*}
for every $M\in \M\cC$ and $y\in \cD$. It follows that $F^{\ast}$ is right exact and preserves projectives and $F^{!}$ is left exact and preserves injectives. In particular, we have that $F^{\ast}\circ h_{\cC} = h_{\cD}\circ F$ and $F^{\ast}$ restricts to a functor $\m\cC\to \m\cD$. However, the functors $F_{\ast}$ and $F^{!}$ do not restrict to functors between $\m\cC$ and $\m\cD$ in general. In this paper we are interested in the cases where these functors do indeed restrict to functors between $\m\cC$ and $\m\cD$. 

\begin{definition}
Let $F:\cC \to \cD$ be a functor between skeletally small $\K$-linear Krull--Schmidt categories. We say that the functor $F$ is \emph{coherent preserving}\index[definitions]{coherent preserving functor} if the functors $F_{\ast}$ and $F^{!}$ restrict to functors between $\m\cC$ and $\m\cD$.
\end{definition}

The following proposition collects some basic properties of the above functors.

\begin{proposition}\label{prop:properties of dense and full}
Let $F:\cC\to \cD$ be a functor between skeletally small $\K$-linear Krull--Schmidt categories.
\begin{enumerate}[label=(\alph*)]
    \item If $F$ is dense, then $\left.F^{\ast}\right|_{\proj\cC}:\proj\cC\to \proj\cD$ is dense and $F_{\ast}:\M\cD\to \M\cC$ is faithful.
    \item If $F$ is full, then $\left.F^{\ast}\right|_{\proj\cC}:\proj\cC\to \proj\cD$ is full.
    \item If $F$ is dense and full, then $F^{\ast}:\m\cC\to \m\cD$ is dense and $F_{\ast}:\M\cD\to \M\cC$ is fully faithful.
    \item If $F$ is coherent preserving, then the following conditions are equivalent.
    \begin{enumerate}[label=(\roman*)]
        \item $F_{\ast}:\m\cD\to \m\cC$ is fully faithful.
        \item the counit of the adjunction $\left(F^{\ast},F_{\ast}\right)$ is an isomorphism $F^{\ast}F_{\ast}(M)\overset{\sim}{\longto} M$ for all $M\in\m\cD$.
        \item the unit of the adjunction $\left(F_{\ast},F^{!}\right)$ is an isomorphism $M\overset{\sim}{\longto}F^{!}F_{\ast}(M)$ for all $M\in\m\cD$.
    \end{enumerate}
\end{enumerate}
\end{proposition}

\begin{proof}
\begin{enumerate}[label=(\alph*)]
    \item Let us first show that $F^{\ast}$ is dense. Let $P\in \proj\cD$. Then $P\isom \cD(-,y)$ for some $y\in \cD$. Since $F$ is dense, there exists an $x\in \cC$ such that $F(x)\isom y$. Then $\cC(-,x)\in \proj\cC$ and
    \[F^{\ast}\left(\cC(-,x)\right) = F^{\ast} \circ h_{\cC}(x) = h_{\cD}\circ F(x) = \cD(-,F(x)) \isom \cD(-,y) \isom P,\]
    which shows that $F^{\ast}$ is dense.
    
    To show that $F_{\ast}$ is faithful, let $\eta,\xi\in \Hom_{\cD}(M,N)$ for some $M,N\in\M\cD$ be such that $F_{\ast}(\eta)=F_{\ast}(\xi)$. We need to show that $\eta_y=\xi_y$ for all $y\in\cD$. Since $F$ is dense, there exists an $x\in\cC$ such that $F(x)\isom y$. Let $s:F(x)\to y$ be an isomorphism. Then by naturality of $\eta$ and $\xi$, and since $\eta F=F_{\ast}(\eta)=F_{\ast}(\xi)=\xi F$ we have
    \[\eta_y\comp M(s)=N(s)\comp \eta_{F(x)}=N(s)\comp \xi_{F(x)} = \xi_y\comp M(s),\]
    from which it follows that $\eta_y=\xi_y$ since $M(s)$ is an isomorphism.
    
    \item Let $P,Q\in \proj\cC$. Then we have that $P\isom \cC(-,x)$ and $Q\isom \cC(-,x')$ for some $x,x'\in\cC$ and so it is enough to show that the induced map
    \[F^{\ast}: \Hom_{\cC}\left(\cC(-,x),\cC(-,x')\right) \longto \Hom_{\cD}\left(\cD(-,F(x)),\cD(-,F(x'))\right)\]
    is surjective. By the Yoneda embedding, every $\eta\in \Hom_{\cD}\left(\cD\left(-,F(x)\right),\cD\left(-,F(x')\right)\right)$ is of the form $\eta=g\circ-$ for some $g\in \cD(F(x),F(x'))$. Since $F$ is full, it follows that $g=F(f)$ for some $f\in \cC(x,x')$. Then $f\circ-\in \Hom_{\cC}\left(\cC(-,x),\cC(-,x')\right)$ and
    \[F^{\ast}(f\circ -) = F^{\ast}\circ h_{\cC}(f) = h_{\cD}\circ F(f) = F(f)\circ-=g\circ-,\]
    which shows that $F^{\ast}$ is full.
    
    \item By (a) and (b) we have that the functor $\left.F^{\ast}\right|_{\proj\cC}:\proj\cC\to \proj\cD$ is dense and full. Let $M\in \m\cD$. Then there exists an exact sequence
    \[\cD(-,y_2)\overset{g}{\longto} \cD(-,y_1)\longto M\longto 0\]
    in $\M\cD$. In particular, we have $M\isom \coker (g)$. Since $\left.F^{\ast}\right|_{\proj\cC}$ is dense and full, there exist $x_1,x_2\in \cC$ and $f\in\cC(x_2,x_1)$ such that $F^{\ast}\left(\cC(-,x_1)\right)\isom \cD(-,y_1)$, $F^{\ast}\left(\cC(-,x_2)\right)\isom \cD(-,y_2)$ and $F^{\ast}(f)\isom g$. Consider the exact sequence 
    \[\cC(-,x_2) \overset{f}{\longto} \cC(-,x_1) \longto \coker (f) \longto 0\]
    in $\M\cC$. By applying the right exact functor $F^{\ast}$ we get the right exact sequence
    \[\cD\left(-,F(x_2)\right) \overset{F^{\ast}(f)}{\longto}\cD(\left(-,F(x_1)\right) \longto F^{\ast}(\coker (f))\longto 0\]
    where 
    \[F^{\ast}(\coker (f() \isom \coker (F^{\ast}(f)) \isom \coker (g) \isom M,\]
    which shows that $F^{\ast}:\m\cC\to\m\cD$ is dense.
    
    By (a) we have that the functor $F_{\ast}$ is faithful. To show that the functor $F_{\ast}$ is full, let $\eta\in\Hom_{\cC}\left(M\comp F, N\comp F\right)$ for some $M,N\in\M\cD$. We need to find $\overline{\eta}\in \Hom_{\cD}(M,N)$ such that $F_{\ast}\left(\overline{\eta}\right)=\eta$. Let $y\in\cD$. Since $F$ is dense, there exists an $x\in\cC$ such that $F(x)\isom y$ and an isomorphism $s:F(x)\to y$. Define $\overline{\eta}_y = N(s) \comp \eta_x \comp M(s)^{-1}$. It is straightforward to show that $\overline{\eta}_y$ is independent of the choice of $x$ and $s$ and that $\overline{\eta}\in\Hom_{\cD}(M,N)$.
    
    \item Since $F$ is coherent preserving, we have the adjoint pairs $\left(F^{\ast},F_{\ast}\right)$ and $\left(F_{\ast},F^{!}\right)$ of functors between $\m\cC$ and $\m\cD$. The equivalence of (i) and (ii) follows from \cite[Theorem IV.3.1]{CWM} and the equivalence of (i) and (iii) from the dual statement. \qedhere
\end{enumerate}
\end{proof}

In the case of dualizing $\K$-varieties we also have the following proposition.

\begin{proposition}\label{prop:injective sent to correct injective}
Let $\cC$ and $\cD$ be two dualizing $\K$-varieties.
\begin{enumerate}[label=(\alph*)]
    \item For every $M\in\m\cC$ we have an isomorphism $\Hom_{\cC}\left(M, \D\cC(x,-)\right) \isom \D M(x)$, which is natural in $x$.
    \item Let $F:\cC\to \cD$ be a functor. Then $F^!\left(\D\cC(x,-)\right) \isom \D\cD\left(F(x),-\right)$.
    \item For every $M\in\m\cC$ we have that $F^{!}(M)$ is finitely presented.
\end{enumerate}
\end{proposition}

\begin{proof}
\begin{enumerate}[label=(\alph*)]
    \item Since $\D:\m\cC\to \m\left(\cC^{\text{op}}\right)$ is fully faithful, we have a bijection 
    \[\D_{M,N}:\Hom_{\cC}(M,N)\overset{\sim}{\longto} \Hom_{\cC^{\text{op}}}(\D N,\D M)\]
    for every $M$, $N\in\m\cC$. In particular, for $N=\D\cC(x,-)$ we have $\D N\isom \cC(x,-)$ for every $x\in \cC$. Then
    \[\Hom_{\cC^{\text{op}}}\left(\cC(x,-),\D M\right) = \Hom_{\cC^{\text{op}}}\left(\cC^{\text{op}}(-,x),\D M\right) \isom \D M(x).\]
    Since the functor $\D$ induces a natural isomorphism between the bifunctors $\Hom_{\cC}(-,-)$ and $\Hom_{\cC^{\text{op}}}(\D(-),\D(-))$, the above bijection is natural in $M$ and so we conclude that we have an isomorphism
    \[\Hom_{\cC}\left(M, \D\cC(x,-)\right) \isom \D M(x)\]
    natural in $x$.
    \item For every $y\in\cD$ we have 
    \[F^!\left(\D\cC(x,-)\right)(y) = \Hom_{\cC}\left(\cD\left(F(-),y\right),\D\cC(x,-)\right) \isom \D\cD\left(F(x),y\right)\]
    by (a). Hence the functors $F^!\left(\D\cC(x,-)\right)$ and $\D\cD\left(F(x),-\right)$ are isomorphic, as required. 
    \item Since $\cC$ is a dualizing $\K$-variety, the $\cC$-module $M$ has an injective resolution in $\m\cC$. Let
    \begin{equation}\label{eq:inj resol of M}
    0\longto M \longto \D\cC\left(x_1,-\right)\longto\D\cC\left(x_2,-\right)
    \end{equation}
    be the beginning of an injective resolution of $M$. By (b) we have that applying the functor $F^{!}$ to (\ref{eq:inj resol of M}) we get an injective resolution  
    \[0\longto F^!(M) \longto \D\cD\left(F\left(x_1\right),-\right)\longto  \D\cD\left(F\left(x_2\right),-\right),\]
    from which it follows that $F^{!}(M)\in\m\cD$.\qedhere
\end{enumerate}
\end{proof}

\section{Admissible targets of asymptotically weakly \texorpdfstring{$n$}{n}-cluster tilting systems}\label{sec:admissible targets weakly n-ct}

We begin with some motivation for the notions of this section. Recall from \cite[Definition 2.5]{HJV} that $(A,\cM_A)$ is an \emph{$n$-homological pair} if $A$ is a fi\-nite-di\-men\-sion\-al $\K$-algebra and $\cM_A\subseteq \m A$ is an $n$-cluster tilting subcategory and that a \emph{morphism of $n$-homological pairs} $f:(A,\cM_A)\to (B,\cM_B)$ is an algebra morphism $f:A\to B$ such that $f_{\ast}(\cM_B)\subseteq \cM_A$. 

Motivated from the notion of $n$-fractured subcategories for rep\-re\-sen\-ta\-tion-di\-rect\-ed algebras in \cite{VAS2}, let us consider the following generalization of an $n$-cluster tilting subcategory of $\m A$. Let $\cM'_A\subseteq \m A$ be a subcategory, not necessarily abelian. Let also $\cM_A\subseteq \cM'_A$ be such that $\cM_A$ is functorially finite in $\cM'_A$ and such that 
\begin{align*}
    \cM_A &= \{ x \in \cM'_A \mid \Ext^{i}_{A}\left(x, \cM_A\right)=0 \text{ for $1\leq i \leq n-1$}\} \\
     &= \{ x \in \cM'_A \mid \Ext^{i}_{A}\left(\cM_A, x\right)=0 \text{ for $1\leq i \leq n-1$}\}.
\end{align*}
Then we call $(A,\cM'_A,\cM_A)$ an \emph{$n$-fractured triple}. For an explicit case that is of interest to us we refer to Proposition \ref{prop:n-fractured with exts}. Note that for $\cM'_A=\m A$ we obtain that $\cM_A$ is an $n$-cluster tilting subcategory. A \emph{morphism of $n$-fractured triples} $f:(A,\cM'_A,\cM_A)\to (B,\cM'_B,\cM_B)$ is then a morphism $f:A\to B$ of algebras such that $f_{\ast}\left(\cM'_B\right)\subseteq \cM'_A$ and $f_{\ast}\left(\cM_B\right)\subseteq \cM_A$.

Consider now an infinite sequence of morphisms of $n$-fractured triples $f_i:(A_{i+1},\cM'_{i+1},\cM_{i+1}) \to (A_{i},\cM'_{i},\cM_{i})$ for $i\in \ZZ_{\geq 0}$ and assume that $f_i:A_{i+1}\to A_i$ is an algebra epimorphism. In particular, we have an inclusion of module categories $f_{i\ast}(\m A_i) \subseteq \m A_{i+1}$ which restricts to inclusions $f_{i\ast}(\cM'_i) \subseteq \cM'_{i+1}$ and $f_{i\ast}(\cM_i) \subseteq \cM_{i+1}$. Informally, we may think of all three sequences of subcategories $\left(\m A_i\right)$, $\left(\cM'_i\right)$ and $\left(\cM_i\right)$ as getting bigger as $i\to \infty$. Let us denote 
\[\mathbf{A} \coloneqq \lim_{i\to \infty}\left(\m A_i\right),\;\; \cM'\coloneqq \lim_{i\to \infty}\left(\cM'_i\right),\;\; \cM\coloneqq \lim_{i\to \infty}\left(\m M_i\right).\]
Still informally, we get a triple $\left(\mathbf{A}, \cM', \cM\right)$ with $\cM\subseteq \cM'\subseteq \mathbf{A}$ and we may ask whether this is an $n$-fractured triple in some appropriate sense. More importantly, we are interested in the situation $\cM'=\mathbf{A}$, since in this case we may get that $\cM$ is an $n$-cluster tilting subcategory.

It turns out that the informal situation described above is not true in general. The limit $\mathbf{A}$ is, in general, too big. However, in some cases we may replace $\mathbf{A}$ by an appropriate smaller subcategory $\cA\subseteq \mathbf{A}$ satisfying $\cM\subseteq \cM'\subseteq \cA$ which we call an \emph{admissible target}. Given that the infinite sequence $f_i:(A_{i+1},\cM'_{i+1},\cM_{i+1}) \to (A_{i},\cM'_{i},\cM_{i})$ is growing fast enough, the above situation gives an $n$-fractured triple (in an appropriate sense) $(\cA,\cM',\cM)$ with $\cA=\cM'$. In particular $\cM$ is an $n$-cluster tilting subcategory of $\cA$.

The above definitions comes from the world of algebras and modules over algebras, but there is no reason one cannot consider skeletally small $\K$-linear Krull--Schmidt categories and modules over skeletally small $\K$-linear Krull--Schmidt categories instead. Indeed the theory can be developed more easily in that case since limits of categories have a description better suited to our computations. To return to the world of algebras we use the equivalence $\m A \equivalent \m (\proj A)$. 

Our first aim in this section is to construct a direct system of abelian categories, called \emph{asymptotically weakly $n$-cluster tilting}, modelling the above situation. Under certain assumptions we can find an $n$-cluster tilting subcategory in an admissible target of such an asymptotically weakly $n$-cluster tilting system. Notice that a limit of the direct system is not an admissible target in general. In subsection \ref{subsection:admissible targets} we define admissible targets, in subsection \ref{subsection:construction of admissible targets} we provide a general method of constructing admissible targets and in subsection \ref{subsection:asymptotically weakly n-cluster tilting systems} we define asymptotically weakly $n$-cluster tilting systems and provide the main result for this section.

\subsection{Admissible targets}\label{subsection:admissible targets} To avoid set theoretic issues, we consider limits over systems of small categories and functors between them. We denote by $\Cat$\index[symbols]{Cat@$\Cat$} the category of small categories with morphisms given by functors. Using the usual strict definition of inverse limits in the category $\Cat$ does not produce the intended results. Instead we require that the compatibility conditions for the connecting functors hold up to a natural isomorphism.

\begin{definition}
A \emph{directed set}\index[definitions]{directed set} $(\dirI, \leq)$\index[symbols]{(I, <)@$(\dirI, \leq)$} is a set $\dirI$ with a preorder $\leq$ such that every pair of elements in $\dirI$ has an upper bound. 
\end{definition}

For the rest of this section we fix a directed set $(\dirI, \leq)$.

\begin{definition}\label{def:inverse system}
\begin{enumerate}[label=(\alph*)]
\item An \emph{inverse system of categories over $\dirI$}\index[definitions]{inverse system of categories over $\dirI$} is a triple $\left(\cC_i,F_{ij},\theta_{ijk}\right)_{\dirI}$\index[symbols]{(Ci, Fij, thetaijk)@$\left(\cC_i,F_{ij},\theta_{ijk}\right)$} where
\begin{enumerate}[label=(\roman*)]
    \item $\cC_i$ is a category for every $i\in\dirI$,
    \item $F_{ij}:\cC_j\to \cC_i$ is a functor for every $i,j\in\dirI$ with $i<j$, and
    \item $\theta_{ijk}:F_{ij}\comp F_{jk} \To F_{ik}$ is a natural isomorphism for every $i,j,k\in\dirI$ with $i< j < k$,
\end{enumerate}
such that the diagram
\[\begin{tikzcd}
F_{ij}\comp F_{jk}\comp F_{kl} \arrow[rr, Rightarrow, "\theta_{ijk} F_{kl}"] \arrow[d, Rightarrow, "F_{ij} \theta_{jkl}"] 
& & F_{ik}\comp F_{kl} \arrow[d, Rightarrow, "\theta_{ikl}"] \\
F_{ij}\comp F_{jl} \arrow[rr, Rightarrow, "\theta_{ijl}"] & & F_{il}
\end{tikzcd}\]
commutes for every $i,j,k,l\in\dirI$, with $i<j<k<l$, i.e. 
\begin{equation}\label{eq:commutativity in inverse system}
    \theta_{ijl}\comp \left(F_{ij} \theta_{jkl}\right) = \theta_{ikl}\comp \left(\theta_{ijk} F_{kl}\right).
\end{equation}
A \emph{$\Cat$-inverse system over $\dirI$}\index[definitions]{Cat-inverse system over I@$\Cat$-inverse system over $\dirI$} is an inverse system of categories $\left(\cC_i,F_{ij},\theta_{ijk}\right)_{\dirI}$ over $\dirI$ where all $\cC_i$ are small categories.    

\item A \emph{source}\index[definitions]{source of an inverse system of categories over $\dirI$} of the inverse system of categories $\left(\cC_i,F_{ij},\theta_{ijk}\right)_{\dirI}$ over $\dirI$ is a triple $\left(\cC, \Phi_i, \Theta_{ij}\right)_{\dirI}$ where
\begin{enumerate}[label=(\roman*)]
    \item $\cC$ is a category,
    \item $\Phi_{i}:\cC\to \cC_i$ is a functor for every $i\in\dirI$, and
    \item $\Theta_{ij}:F_{ij}\comp \Phi_{j} \To \Phi_{i}$ is a natural isomorphism for every $i,j\in\dirI$ with $i<j$,
\end{enumerate}
such that the diagram
\[\begin{tikzcd}
F_{ij}\comp F_{jk}\comp \Phi_{k} \arrow[rr, Rightarrow, "\theta_{ijk} \Phi_{k}"] \arrow[d, Rightarrow, "F_{ij} \Theta_{jk}"] 
& & F_{ik}\comp \Phi_{k} \arrow[d, Rightarrow, "\Theta_{ik}"] \\
F_{ij}\comp \Phi_{j} \arrow[rr, Rightarrow, "\Theta_{ij}"] & & \Phi_{i}
\end{tikzcd}\]
commutes for every $i,j,k\in\dirI$, with $i<j<k$, i.e. 
\begin{equation}\label{eq:commutativity of source}
    \Theta_{ij}\comp \left(F_{ij} \Theta_{jk}\right) = \Theta_{ik}\comp \left(\theta_{ijk} \Phi_{k}\right).
\end{equation}
A \emph{$\Cat$-source}\index[definitions]{Cat-source of an inverse system of categories over I@$\Cat$-source of an inverse system of categories over $\dirI$} is a source $(\cC,\Phi_i,\Theta_{ij})$ where $\cC$ is a small category.

\item A \emph{morphism of sources}\index[definitions]{morphism of sources of an inverse system of categories over $\dirI$} $G:\left(\cC, \Phi_i, \Theta_{ij}\right)_{\dirI}\to \left(\cD, \Psi_i, \Xi_{ij}\right)_{\dirI}$ of the inverse system $\left(\cC_i,F_{ij},\theta_{ijk}\right)_{\dirI}$ of categories over $\dirI$ is a functor $G:\cC\to \cD$ such that $\Phi_i \comp G= \Psi_i$ for all $i\in \dirI$ and $\Theta_{ij}G=\Xi_{ij}$ for all $i,j\in \dirI$ with $i<j$. 

\item The \emph{category of $\Cat$-sources}\index[definitions]{category of $\Cat$-sources of an inverse system of categories over $\dirI$} over the inverse system of categories $\left(\cC_i,F_{ij},\theta_{ijk}\right)_{\dirI}$ over $\dirI$, denoted by $\Sources\left(\cC_i,F_{ij},\theta_{ijk}\right)_{\dirI}$, is the category where objects are $\Cat$-sources of $\left(\cC_i,F_{ij},\theta_{ijk}\right)_{\dirI}$, morphisms are morphisms of sources and composition is defined as composition of functors.

\item An \emph{inverse limit}\index[definitions]{inverse limit of an inverse system of categories over $\dirI$} of the inverse system of categories $\left(\cC_i,F_{ij},\theta_{ijk}\right)_{\dirI}$ over $\dirI$, denoted by \linebreak $\varprojlim_{\dirI}\left(\cC_i,F_{ij},\theta_{ijk}\right)$, is a terminal object in $\Sources\left(\cC_i,F_{ij},\theta_{ijk}\right)_{\dirI}$. In other words, it is a $\Cat$-source $\left(\cC, \Phi_i, \Theta_{ij}\right)_{\dirI}$, such that for every $\Cat$-source $\left(\cD,\Psi_{i},\Xi_{ij}\right)_{\dirI}$, there exists a unique functor $G:\cD\to \cC$ satisfying $\Phi_i \comp G= \Psi_i$ for all $i\in \dirI$ and $\Theta_{ij}G=\Xi_{ij}$ for all $i,j\in \dirI$ with $i<j$.
\end{enumerate}
The notions of \emph{direct system of categories over $\dirI$}\index[definitions]{direct system of categories over $\dirI$}, \emph{$\Cat$-direct system over $\dirI$}\index[definitions]{Cat-direct system over I@$\Cat$-direct system over $\dirI$}, \emph{target}\index[definitions]{target of a direct system of categories over $\dirI$}, \emph{$\Cat$-target}\index[definitions]{Cat-target of a direct system of categories over I@$\Cat$-target of a direct system of categories over $\dirI$}, \emph{morphism of targets}\index[definitions]{morphism of targets of a direct system of categories over $\dirI$}, \emph{category of $\Cat$-targets}\index[definitions]{category of $\Cat$-targets of a direct system of categories over $\dirI$} and \emph{direct limit}\index[definitions]{direct limit of a direct system of categories over $\dirI$} are defined dually.
\end{definition}

When the directed set $\dirI$ is clear from context, we do not write $\dirI$ as a subscript and we simply say ``inverse system (of categories)'' instead of ``inverse system (of categories) over $\dirI$''. Note that in many references it is common to use ``inverse system'' for ``inverse system over $\NN$'' or ``inverse system over $\ZZ$''. In this paper we do not make any such assumption on $\dirI$. Since an inverse limit is a terminal object in a category, if an inverse limit exists it is unique up to a unique isomorphism. Hence we use the definite article ``the'' when referring to an inverse limit that exists. Moreover, in the following we simply say ``let $\varprojlim\left(\cC_i,F_{ij},\theta_{ijk}\right) = \left(\cC,\Phi_i,\Theta_{ij}\right)$ be an inverse limit'' instead of ``let $\left(\cC,\Phi_i,\Theta_{ij}\right)$ be an inverse system of categories with inverse limit $\varprojlim\left(\cC_i,F_{ij},\theta_{ijk}\right)$''. We use similar conventions for direct systems.

\begin{remark}
\begin{enumerate}[label=(\alph*)]
    \item Notice that a morphism of sources $G:(\cC,\Phi,\Theta_{ij})\to (\cD,\Psi_i,\Xi_{ij})$ is a functor $G:\cC\to \cD$ that satisfies some additional properties. In particular, if $\cC$ and $\cD$ are small categories, then functors from $\cC$ to $\cD$ form a set. Hence the category of $\Cat$-sources is well defined.
    
    \item We remark that condition (iii) of Definition \ref{def:inverse system}(a) is not used throughout this paper. However, notice that if $\left(\cC_i,F_{ij},\theta_{ijk}\right)$ is an inverse system of categories, and if $\left(\cC,\Phi_i,\Theta_{ij}\right)$ is a source of this inverse system, then condition (iii) of Definition \ref{def:inverse system}(b) implies that condition (iii) of Definition \ref{def:inverse system}(a) holds for every $x_l$ in the image of $\Phi_l$. Hence, if we want to find a source that is as big as possible, then condition (iii) of Definition \ref{def:inverse system}(a) is necessary.

    As an example, assume that we have an inverse system $\left(\cC_i,F_{ij},\theta_{ijk}\right)$ of categories where condition (iii) of Definition \ref{def:inverse system}(a) does not necessarily hold. Assume moreover that $\dirI$ has a maximal element $v$ and all functors $F_{ij}$ are surjective on objects and on morphisms. In this setting it makes sense that we would like the triple $\left(\cC_v, F_{iv}, \theta_{ijv}\right)$ to be a source of the inverse system. But this is true if and only if condition (iii) of Definition \ref{def:inverse system}(a) holds. 

    \item We also remark that the only notions of direct systems that are used in this paper are those of a direct system of categories over $\dirI$ and target of a direct system of categories over $\dirI$. 
\end{enumerate}
\end{remark}

Inverse limits of $\Cat$-inverse systems always exist. To show this, let us first give a definition.

\begin{definition}\label{def:concrete inverse limit category}
The \emph{concrete inverse limit}\index[definitions]{concrete inverse limit of a $\Cat$-inverse system} of a $\Cat$-inverse system $\left(\cC_i,F_{ij},\theta_{ijk}\right)$ is a triple $\left(\cC,\Phi_i,\Theta_{ij}\right)$\index[symbols]{(C, Phii, Thetaij)@$\left(\cC, \Phi_i, \Theta_{ij}\right)$} where
\begin{enumerate}
    \item[(i)] $\cC$ is the category given by the following data.
    \begin{enumerate}
        \item[$\bullet$] Objects in $\cC$ are pairs of sequences $\left(\left\{x_i\right\}_{i\in\dirI},\left\{f_{ij}\right\}_{i<j\in\dirI}\right)$, which we simply write as  $\left(x_i,f_{ij}\right)$, where $x_i\in\cC_i$ and $f_{ij}:F_{ij}(x_j) \to x_i$ is an isomorphism for all $i,j\in\dirI$ with $i<j$, such that the diagram
        \begin{equation}
            \begin{tikzcd}
                F_{ij}\comp F_{jk}\left(x_{k}\right) \arrow[r, "F_{ij}\left(f_{jk}\right)"] \arrow[d, "\left(\theta_{ijk}\right)_{x_k}"] & F_{ij}\left(x_{j}\right) \arrow[d, "f_{ij}"] \\
                F_{ik}\left(x_{k}\right) \arrow[r, "f_{ik}"] &  x_{i}
            \end{tikzcd}
        \end{equation}
        commutes for all $i,j,k\in\dirI$ with $i<j<k$, i.e.
        \begin{equation}\label{eq:commutativity of objects in inverse limit category}
            f_{ij}\comp F_{ij}\left(f_{jk}\right) = f_{ik}\comp \left(\theta_{ijk}\right)_{x_k}.
        \end{equation}
        
        \item[$\bullet$] Morphisms from $\left(x_i,f_{ij}\right)$ to $\left(y_i,g_{ij}\right)$ are sequences $\left\{s_i\right\}_{i\in\dirI}$, which we simply write as $\left(s_i\right)$, such that $s_i\in\cC_i\left(x_i,y_i\right)$ and such that the diagram
        \begin{equation}
            \begin{tikzcd}
                F_{ij}\left(x_{j}\right) \arrow[r, "F_{ij}\left(s_{j}\right)"] \arrow[d, "f_{ij}"] & F_{ij}\left(y_{j}\right) \arrow[d, "g_{ij}"] \\
                x_{i} \arrow[r, "s_{i}"] &  y_{i}
            \end{tikzcd}
        \end{equation}
        commutes for all $i,j\in\dirI$ with $i<j$, i.e.
        \begin{equation}\label{eq:commutativity of morphisms in inverse limit category}
            s_i\comp f_{ij} = g_{ij}\comp F_{ij}(s_j)
        \end{equation}

        \item[$\bullet$] Composition of morphisms is defined componentwise, that is if $(s_i)\in \cC\left(\left(x_i,f_{ij}\right),\left(y_i,g_{ij}\right)\right)$ and $(r_i)\in \cC\left(\left(y_i,g_{ij}\right), \left(z_i,h_{ij}\right)\right)$, then $\left(r_i\right)\comp\left(s_i\right)=\left(r_i\comp s_i\right)\in \cC\left(\left(x_i,f_{ij}\right),\left(z_i,h_{ij}\right)\right)$. 
    \end{enumerate}

    \item[(ii)] $\Phi_i:\cC\to \cC_i$ is the functor defined on objects by $\Phi_i\left(\left(x_i,f_{ij}\right)\right)=x_i$ and on morphisms by $\Phi_i\left(\left(s_i\right)\right)=s_i$, for all $i\in\dirI$. 

    \item[(iii)] $\Theta_{ij}:F_{ij}\comp \Phi_j \To \Phi_i$ is the natural isomorphism defined by $\left(\Theta_{ij}\right)_{\left(x_i,f_{ij}\right)} = f_{ij}$, for all $i,j\in\dirI$ with $i<j$. 
\end{enumerate}
\end{definition}

As the following proposition shows, the concrete inverse limit is the inverse limit.

\begin{proposition}\label{prop:inverse limit category}
Let $\left(\cC,\Phi_i,\Theta_{ij}\right)$ be the concrete inverse limit of a $\Cat$-inverse system $\left(\cC_i,F_{ij},\theta_{ijk}\right)$. Then $\varprojlim\left(\cC_i,F_{ij},\theta_{ijk}\right)=\left(\cC,\Phi_i,\Theta_{ij}\right)$.
\end{proposition}

\begin{proof}
It is straightforward to check that $\cC$ is a category, that $\Phi_i:\cC\to \cC_i$ are functors and that $\Theta_{ij}:F_{ij}\comp \Phi_j\To \Phi_i$ are natural isomorphisms. To show that $\left(\cC,\Phi_i,\Theta_{ij}\right)$ is a source of $\left(\cC_i,F_{ij},\theta_{ijk}\right)$, we have to check that (\ref{eq:commutativity of source}) is satisfied. For every $\left(x_i,f_{ij}\right)\in\cC$ we have
\begin{align*}
    \left(\Theta_{ij}\comp \left(F_{ij} \Theta_{jk}\right)\right)_{\left(x_i,f_{ij}\right)} &= \left(\Theta_{ij}\right)_{\left(x_i,f_{ij}\right)}\comp \left(F_{ij} \Theta_{jk}\right)_{\left(x_i,f_{ij}\right)} && \\
    &=f_{ij} \comp F_{ij}\left(\left(\Theta_{jk}\right)_{\left(x_i,f_{ij}\right)}\right) && \\
    &=f_{ij} \comp F_{ij}\left(f_{jk}\right) && \\
    &=f_{ik} \comp \left(\theta_{ijk}\right)_{x_k} &&\text{(by (\ref{eq:commutativity of objects in inverse limit category}))} \\
    &=f_{ik} \comp \left(\theta_{ijk}\right)_{\Phi_k\left(x_i,f_{ij}\right)} && \\
    &=\left(\Theta_{ik}\right)_{\left(x_i,f_{ij}\right)} \comp \left(\theta_{ijk} \Phi_{k}\right)_{\left(x_i,f_{ij}\right)} && \\
    &=\left(\Theta_{ik}\comp \left(\theta_{ijk} \Phi_{k}\right)\right)_{\left(x_i,f_{ij}\right)}, &&
\end{align*}
and so $\Theta_{ij}\comp \left(F_{ij} \Theta_{jk}\right)=\Theta_{ik}\comp \left(\theta_{ijk} \Phi_{k}\right)$, as required. That $\cC$ is a small category follows since 
\[\obj(\cC) \subseteq \left(\prod_{i\in\dirI}\obj(\cC_i)\right) \times \left(\prod_{i<j\in\dirI}\cC_i\left(-,-\right)\right).\]
Hence $\left(\cC,\Phi_i,\Theta_{ij}\right)$ is a $\Cat$-source of $\left(\cC_i,F_{ij},\theta_{ijk}\right)$.

It remains to show the universal property of $\left(\cC,\Phi_i,\Theta_{ij}\right)$. Let $\left(\cD,\Psi_{i},\Xi_{ij}\right)$ be another $\Cat$-source of $\left(\cC_i,F_{ij},\theta_{ijk}\right)$. Define a mapping $G:\cD\to \cC$ by $G(x)=\left(\Psi_i(x), \left(\Xi_{ij}\right)_x\right)$ for all $x\in \cD$ and $G(f)=\left(\Psi_i(f)\right)$ for all $f\in \cD(x,y)$. It is straightforward to show that $G$ is indeed a functor. Then, for all $i\in\dirI$ and for all $x\in \cD$, we have
\[\Phi_i\comp G(x) = \Phi_i\left( \Psi_i(x), \left(\Xi_{ij}\right)_x\right) = \Psi_i(x),\]
while for all $f\in\cD(x,y)$ we have
\[\Phi_i\comp G(f) = \Phi_i\left(\Psi_i(f)\right) = \Psi_i(f),\]
which shows that $\Phi_i\comp G = \Psi_i$. Moreover, for all $i,j\in\dirI$ with $i<j$ and all $x\in\cD$ we have
\[\left(\Theta_{ij}G\right)_x = \left(\Theta_{ij}\right)_{G(x)} = \left(\Theta_{ij}\right)_{\left(\Psi_i(x),\left(\Xi_{ij}\right)_x\right)}=\left(\Xi_{ij}\right)_x,\]
which shows that $\Theta_{ij}G=\Xi_{ij}$. Hence $G$ is a morphism of sources.

It remains to show that $G$ is unique. Let $G':\cD\to \cC$ be a functor satisfying $\Phi_i\comp G'=\Psi_i$ for all $i\in\dirI$ and $\Theta_{ij}G'=\Xi_{ij}$ for all $i,j\in\dirI$ with $i<j$. Let $x\in\cD$ and write $G'(x) = \left(y_i, g_{ij}\right)\in\cC$. Then
\[\Psi_i(x)=\Phi_i\comp G' (x) = \Phi_i\left(y_i, g_{ij}\right) = y_i,\]
while
\[\left(\Xi_{ij}\right)_x = \left(\Theta_{ij}G'\right)_x = \left(\Theta_{ij}\right)_{G'(x)} = \left(\Theta_{ij}\right)_{\left(y_i,g_{ij}\right)} = g_{ij},\]
from which it follows that $G'(x) = \left(\Psi_i(x),\left(\Xi_{ij}\right)_{x}\right) = G(x)$. Now let $f\in\cD(x,y)$ and write $G'(f)=\left(f_i)\right)$. Then
\[\Psi_i(f) = \Phi_i\comp G'(f) = \Phi_i\left((f_i)\right) = f_i,\]
and so $G'(f)=\left(\Psi_i(f)\right)=G(f)$. Hence $G=G'$ and so $G$ is unique.
\end{proof}

As explained, we are not interested in computing direct limits. Instead, the following class of direct systems and targets is suitable for our purposes.

\begin{definition}\label{def:admissible target}
Let $\left(\cA_i,G_{ij},\zeta_{ijk}\right)$\index[symbols]{(Ai, Gij, zetajik)@$\left(\cA_i,G_{ij},\zeta_{ijk}\right)$} be a direct system of categories where all $\cA_i$ are abelian and all $G_{ij}$ are fully faithful exact functors. An \emph{admissible target}\index[definitions]{admissible target of a direct system} of the direct system $\left(\cA_i,G_{ij},\zeta_{ijk}\right)$ is a target $\left(\cA,\Psi_i,Z_{ij}\right)$\index[symbols]{(A, Psii, Zij)@$\left(\cA,\Psi_i,Z_{ij}\right)$} together with adjunctions $\left(L_i,\Psi_i\right)$ and $\left(\Psi_i,R_i\right)$ such that the following conditions hold.
\begin{enumerate}[label=(\roman*)]
    \item The category $\cA$ is abelian and the functors $\Psi_i$ are fully faithful.
    \item For all $x\in\cA$ there exists a $t\in\dirI$ such that for every $u>t$ we have that the unit of the adjunction $(L_u,\Psi_u)$ and the counit of the adjunction $\left(\Psi_u,R_u\right)$ are isomorphisms when evaluated at $x$. In particular, we have $\Psi_u L_u(x)\isom x$ and $\Psi_u R_u(x)\isom x$. 
    \item For all $x\in\cA$ and $k>0$ there exists a $t\in\dirI$ such that for every $u>t$ and $z\in \cA_u$ we have $\Ext_{\cA}^r(x,\Psi_u(z)) \isom \Ext_{\cA_u}^r\left(L_u(x),z\right)$ for all $0<r<k$.
    \item For all $y\in\cA$ and $k>0$ there exists a $t\in\dirI$ such that for every $u>t$ and $z\in \cA_u$ we have $\Ext_{\cA}^r(\Psi_u(z),y) \isom \Ext_{\cA_u}^r\left(z,R_u(y)\right)$ for all $0<r<k$.
\end{enumerate}
\end{definition}

\begin{remark}\label{rem:homological embedding}
Note that if $\left(\cA,\Psi_i,Z_{ij}\right)$ is an admissible target, then the functors $\Psi_i$ are exact since they admit both a left and a right adjoint.
Moreover, since $\Psi_i$ is fully faithful, we have that if $z\in \cA_i$, then $L_i\Psi_i(z)\isom z$ and $R_i\Psi_i(z)\isom z$. 

Conditions (iii) and (iv) of Definition \ref{def:admissible target} are close to the notion of \emph{$k$-homological embedding} in the sense of \cite[Definition 3.6]{PSA}. More precisely, assume that for all $x\in\cA$ and $k>0$ there exists $t\in\dirI$ such that for every $u>t$ the functor $L_u$ is a $(k+1)$-homological embedding. Then condition (iii) of Definition \ref{def:admissible target} holds and similarly for condition (iv).
\end{remark} 

In the next section we develop a general way of constructing admissible targets.

\subsection{Construction of admissible targets}\label{subsection:construction of admissible targets}

To construct an admissible target of a direct system of categories we start from a $\Cat$-inverse system. Let $\left(\cC_i,F_{ij},\theta_{ijk}\right)$ be a $\Cat$-inverse system where each $\cC_i$ is a small $\K$-linear Krull--Schmidt category with concrete inverse limit $\left(\cC,\Phi_i,\Theta_{ij}\right)$. We can then consider the direct system of the module categories $\m\cC_i$ with $F_{ij\ast}$ as connecting functors. Our aim is to find an admissible target of this direct system of categories.

Since an admissible target is an abelian category, we may construct it by considering a category of modules over an appropriate dualizing $\K$-variety. It turns out that in general the concrete source $\cC$ is not even a Krull--Schmidt category. However, there is always a subcategory $\Cks\subseteq \cC$ which is Krull--Schmidt and moreover such that $(\Cks,\Pks_i,\Tks_{ij})$ is a $\Cat$-source of $(\cC_i,F_{ij},\theta_{ijk})$, where $\Pks_i$ and $\Tks_{ij}$ are the restrictions of $\Phi_i$ and $\Theta_{ij}$ on $\Cks$. We call $(\Cks,\Pks_i,\Tks_{ij})$ the \emph{firm source}. We show that the firm source satisfies many nice homological properties. In particular, we show that if $\Cks$ is a dualizing $\K$-variety, then $(\m\Cks,\Pks_{i\ast},\Tks_{ij\ast})$ is an admissible target of $(\m\cC_i,F_{ij\ast},\theta_{ijk\ast})$.

It is not difficult to see that at least some properties of the categories $\cC_i$ and the functors $F_{ij}$ are inherited by the category $\cC$ and the functors $\Phi_i$ respectively, as the following results show.

\begin{proposition}\label{prop:additive limit} Let $\left(\cC,\Phi_i,\Theta_{ij}\right)$ be the concrete inverse limit of the $\Cat$-inverse system $\left(\cC_i,F_{ij},\theta_{ijk}\right)$. If all $\cC_i$ are $\K$-linear additive categories and all $F_{ij}$ are $\K$-linear functors, then $\cC$ is a $\K$-linear additive category and all $\Phi_i$ are $\K$-linear functors.
\end{proposition}

\begin{proof}
Let $x,y\in\cC$ where $x=\left(x_i,f_{ij}\right)$ and $y=\left(y_i,g_{ij}\right)$. Let $s,r\in \cC(x,y)$ where $s=\left(s_i\right)$ and $r=\left(r_i\right)$ and let $a\in \K$. We define addition and $\K$-multiplication in $\cC(x,y)$ by
\[s+r\coloneqq \left(s_i+r_i\right) \text{ and } a\cdot s\coloneqq (a s_i). \]
It is straightforward to check that this definition turns $\cC$ into a $\K$-linear category and that the functors $\Phi_i$ are $\K$-linear. In particular $\cC$ is preadditive. To show that $\cC$ is additive, we need to show that every pair of objects in $\cC$ has a product and $\cC$ has a zero object. The pair $\left(0_i,0_{ij}\right)$ where $0_i\in \cC_i$ is a zero object and $0_{ij}:F_{ij}(0_j)\to 0_i$ is the unique isomorphism can be easily shown to be a zero object in $\cC$. To show that every pair of objects in $\cC$ has a product, let $x,y\in\cC$ where $x=\left(x_i,f_{ij}\right)$ and $y=\left(y_i,g_{ij}\right)$. Since $F_{ij}:\cC_j\to \cC_i$ is an additive functor, we have that $\left(F_{ij}\left(x_j\oplus y_j\right), F_{ij}\left(\pi_{x_j}\right), F_{ij}\left(\pi_{y_j}\right)\right)$ is a product of $F_{ij}\left(x_j\right)$ and $F_{ij}\left(y_j\right)$. Hence there exists a unique morphism $f_{ij}\oplus g_{ij}:F_{ij}\left(x_j\oplus y_j\right)\to x_i\oplus y_i$ such that both squares of the diagram
\begin{equation*}
    \begin{tikzcd}
        F_{ij}\left(x_j\right) \arrow[d, "f_{ij}"] && F_{ij}\left(x_j\oplus y_j\right) \arrow[ll, swap, "F_{ij}\left(\pi_{x_j}\right)"] \arrow[rr, "F_{ij}\left(\pi_{y_j}\right)"] \arrow[d, "f_{ij}\oplus g_{ij}"] && F_{ij}\left(y_j\right) \arrow[d, "g_{ij}"] \\
        x_i && x_i\oplus y_i \arrow[ll, swap, "\pi_{x_i}"] \arrow[rr, "\pi_{y_i}"] && y_i
    \end{tikzcd}
\end{equation*}
commute. It is a long but straightforward process to show first that $x\oplus y\coloneqq\left(x_i\oplus y_i, f_{ij}\oplus g_{ij}\right)$ is an object in $\cC$, then that $\pi_x\coloneqq\left(\pi_{x_i}\right)$ is a morphism in $\cC\left(\left(x_i\oplus y_i, f_{ij}\oplus g_{ij}\right), (x_i,f_{ij})\right)$ and $\pi_y\coloneqq\left(\pi_{y_i}\right)$ is a morphism in $\cC\left(\left(x_i\oplus y_i, f_{ij}\oplus g_{ij}\right), (y_i,g_{ij})\right)$ and finally show that $\left(x\oplus y, \pi_x, \pi_y\right)$ is a product of $x$ and $y$, the details of which are left to the reader.
\end{proof}

\begin{lemma}\label{lem:C has split idempotents}
Let $\left(\cC,\Phi_i,\Theta_{ij}\right)$ be the concrete inverse limit of the $\Cat$-inverse system $\left(\cC_i,F_{ij},\theta_{ijk}\right)$. Assume that all $\cC_i$ have split idempotents, and let $x=\left(x_i,f_{ij}\right)\in\cC$ and $e=\left(e_i\right)\in\End_{\cC}\left(x\right)$ be an idempotent. Then the following statements hold.
\begin{enumerate}[label=(\alph*)]
    \item For every $i\in\dirI$ we have that $e_i:x_i\to x_i$ is an idempotent.
    
    \item For every $i\in\dirI$, let $s_i:x_i\to y_i$ and $t_i:y_i\to x_i$ be such that $t_i\comp s_i=e_i$ and $s_i\comp t_i=\Id_{y_i}$. Then $y=\left(y_i,s_i\comp f_{ij}\comp F_{ij}\left(t_j\right)\right)$ is an object in $\cC$.
    
    \item Set $s\coloneqq \left(s_i\right)$ and $t\coloneqq \left(t_i\right)$. Then $s:x\to y$ and $t:y\to x$ are morphisms in $\cC$ such that $t\comp s=e$ and $s\comp t=\Id_y$. In particular, the category $\cC$ has split idempotents.
\end{enumerate}
\end{lemma}

\begin{proof}
\begin{enumerate}[label=(\alph*)]
    \item We have $\left(e_i\right)=\left(e_i\right)\comp\left(e_i\right) = \left(e_i\comp e_i\right)$ which implies that $e_i=e_i\comp e_i$ for all $i\in\dirI$ and so $e_i\in\End_{\cC_i}\left(x_i\right)$ is an idempotent.
    
    \item Let $i,j\in \dirI$ with $i<j$. Set $g_{ij}\coloneqq  s_i \comp f_{ij}\comp F_{ij}\left(t_j\right)$. To show that $\left(y_i,g_{ij}\right)$ is an object in $\cC$ we need to show that $g_{ij}:F_{ij}\left(y_j\right)\to y_i$ is an isomorphism and that (\ref{eq:commutativity of objects in inverse limit category}) is satisfied for $g_{ij}$. To show that $g_{ij}$ is an isomorphism, one can easily verify that $g_{ij}^{-1}=F_{ij}\left(s_j\right)\comp f_{ij}^{-1} \comp t_i$ using (\ref{eq:commutativity of morphisms in inverse limit category}). To show that (\ref{eq:commutativity of objects in inverse limit category}) is satisfied, we compute for every $i<j<k$ in $\dirI$ that
    \begin{align*}
        g_{ij}\comp F_{ij}\left(g_{jk}\right) &= s_i\comp f_{ij}\comp F_{ij}\left(t_j\right)\comp F_{ij}\left(s_j\comp f_{jk}\comp F_{jk}\left(t_k\right)\right) && \\
        &=s_i\comp f_{ij}\comp F_{ij}\left(t_j\comp s_j\right) \comp F_{ij}\left(f_{jk}\right)\comp \left(F_{ij}\comp F_{jk}\right)\left(t_k\right) &&\\
        &=s_i\comp f_{ij}\comp F_{ij}\left(e_j\right) \comp F_{ij}\left(f_{jk}\right)\comp \left(F_{ij}\comp F_{jk}\right)\left(t_k\right) &&\\
        &=s_i\comp e_i\comp f_{ij} \comp F_{ij}\left(f_{jk}\right)\comp \left(F_{ij}\comp F_{jk}\right)\left(t_k\right) &&\text{(by (\ref{eq:commutativity of morphisms in inverse limit category}))}\\
        &=s_i\comp f_{ij} \comp F_{ij}\left(f_{jk}\right)\comp \left(F_{ij}\comp F_{jk}\right)\left(t_k\right) &&\\
        &=s_i\comp f_{ik} \comp \left(\theta_{ijk}\right)_{x_k} \comp \left(F_{ij}\comp F_{jk}\right)\left(t_k\right) &&\text{(by (\ref{eq:commutativity of objects in inverse limit category}))}\\
        &=s_i\comp f_{ik} \comp F_{ik}\left(t_k\right)\comp \left(\theta_{ijk}\right)_{y_k} &&\text{(naturality of $\theta_{ijk}$)}\\
        &=g_{ik} \comp \left(\theta_{ijk}\right)_{y_k}.
    \end{align*}
    Hence indeed $y\coloneqq\left(y_i,g_{ij}\right)\in\cC$.
    
    \item To show that $s$ is a morphism we need to show that (\ref{eq:commutativity of morphisms in inverse limit category}) is satisfied. For every $i<j$ in $\dirI$ we have
    \[
    g_{ij}\comp F_{ij}\left(s_j\right) 
    = s_i\comp f_{ij}\comp F_{ij}\left(t_j\right) \comp F_{ij}\left(s_j\right) 
    =s_i\comp f_{ij}\comp F_{ij}\left(e_j\right)  
    =s_i\comp e_i \comp f_{ij} 
    =s_i\comp f_{ij},
    \]
    by using (\ref{eq:commutativity of morphisms in inverse limit category}) for $e:x\to x$. Similarly, to show that $t$ is a morphism, for every $i<j$ in $\dirI$ we have
    \begin{align*}
    t_i\comp g_{ij} 
    &= t_i \comp s_i \comp f_{ij}\comp F_{ij}\left(t_j\right)
    =e_i \comp f_{ij} \comp F_{ij}\left(t_j\right)\\ 
    &=f_{ij}\comp F_{ij}\left(e_j\right) \comp F_{ij}\left(t_j\right)
    =f_{ij}\comp F_{ij}\left(e_j\comp t_j\right) \\
    &=f_{ij}\comp F_{ij}\left(t_j\right),
    \end{align*}
    again by using (\ref{eq:commutativity of morphisms in inverse limit category}) for $e:x\to x$. Hence $s$ and $t$ are morphisms in $\cC$. Finally, we have
    \begin{align*}
        &t\comp s = \left(t_i\right) \comp \left(s_i\right) = \left(t_i\comp  s_i\right) = \left(e_i\right) = e, \text{ and} \\
        &s\comp t = \left(s_i\right) \comp \left(t_i\right) = \left(s_i\comp t_i\right) = \left(\Id_{y_i}\right) = \Id_{y},
    \end{align*}
    which proves that $\cC$ has split idempotents. \qedhere
\end{enumerate}
\end{proof}

In general, it is not true that objects in $\cC$ have semiperfect endomorphism rings even if objects in $\cC_i$ have semiperfect endomorphism rings. Hence the property of being a Krull--Schmidt category is not inherited by $\cC$. Note that Lemma \ref{lem:C has split idempotents} is a first step towards constructing a Krull--Schmidt subcategory of $\cC$. The following technical lemma is the next step.

\begin{lemma}\label{lem:connecting functors eventually bijective implies limit functors eventually bijective}
Let $\left(\cC,\Phi_i,\Theta_{ij}\right)$ be the concrete inverse limit of the $\Cat$-inverse system $\left(\cC_i,F_{ij},\theta_{ijk}\right)$. Let $x,y\in \cC$ and $t\in \dirI$. Assume that the induced map \[F_{tu}:\cC_u\left(\Phi_u(x),\Phi_u(y)\right)\to \cC_t\left(F_{tu}\comp \Phi_u(x),F_{tu}\comp\Phi_{u}(y)\right)\]
is bijective for all $u>t$. Then the following statements hold.
\begin{enumerate}[label=(\alph*)]
    \item The induced map 
    \[F_{uv}:\cC_{v}\left(\Phi_v(x),\Phi_v(y)\right)\to \cC_u\left(F_{uv}\comp\Phi_v\left(x\right),F_{uv}\comp\Phi_v\left(y\right)\right)\]
    is bijective for all $v>u>t$.
    \item The induced map 
    \[\Phi_u:\cC\left(x,y\right) \to \cC_u\left(\Phi_u(x),\Phi_{u}(y)\right)\]
    is bijective for all $u>t$.
\end{enumerate}
\end{lemma}

\begin{proof}
\begin{enumerate}[label=(\alph*)]
    \item Let $v,u\in\dirI$ with $v>u>t$. Notice first that since 
    \[F_{uv}\comp \Phi_v(x) = F_{uv}(x_v)\isom x_u = \Phi_u(x) \text{ and } F_{uv}\comp \Phi_v(y)=F_{uv}(y_v)\isom y_u = \Phi_u(y),\] 
    and since the induced map 
    \[F_{tu}:\cC_u\left(\Phi_u(x),\Phi_u(y)\right)\to \cC_t\left(F_{tu}\comp \Phi_u(x),F_{tu}\comp\Phi_{u}(y)\right)\]
    is bijective by assumption, it follows that the induced map
    \[F_{tu}:\cC_u\left(F_{uv}\comp\Phi_v(x),F_{uv}\comp\Phi_v(y)\right)\to \cC_t\left(F_{tu}\comp F_{uv}\comp\Phi_v(x),F_{tu}\comp F_{uv}\comp\Phi_{v}(y)\right)\]
    is also bijective. Moreover, since $\theta_{tuv}:F_{tu}\comp F_{uv} \To F_{tv}$ is a natural isomorphism, we get a bijective map of sets
    \[G: \cC_t\left(F_{tu}\comp F_{uv}\comp\Phi_v(x),F_{tu}\comp F_{uv}\comp\Phi_v(y)\right) \to \cC_t\left(F_{tv}\comp\Phi_v(x), F_{tv}\comp\Phi_v(y)\right)\]
    defined by
    \[G(f)= \left(\theta_{tuv}\right)_{y_v}\comp f \comp \left[\left(\theta_{tuv}\right)_{x_v}\right]^{-1}.\]
    By naturality of $\theta_{tuv}$, it follows that the diagram of maps
    \begin{equation*}
        \begin{tikzcd}
            \cC_v\left(\Phi_v(x),\Phi_v(y)\right) \arrow[r, "F_{uv}"] \arrow[d, "F_{tv}"] & \cC_u\left(F_{uv}\comp\Phi_v\left(x\right), F_{uv}\comp \Phi_v\left(y\right)\right) \arrow[d, "F_{tu}"] \\
            \cC_t\left(F_{tv}\comp\Phi_v(x), F_{tv}\comp\Phi_v(y)\right) & \cC_t\left(F_{tu}\comp F_{uv}\comp\Phi_v(x),F_{tu}\comp F_{uv}\comp\Phi_{v}(y)\right) \arrow[l, swap, "G"]
        \end{tikzcd}
    \end{equation*}
    commutes. Since the maps $F_{tv}$, $F_{tu}$ and $G$ are bijective, it follows that $F_{uv}$ is also bijective.

    \item Let us write $x=\left(x_i,f_{ij}\right)$ and $y=\left(y_i,g_{ij}\right)$ and let $u\in\dirI$ with $u> t$. Let us first show that the induced map $\Phi_u:\cC\left(x,y\right) \to \cC_u\left(\Phi_u(x),\Phi_{u}(y)\right)$ is injective. Let $s,r\in\cC\left(x,y\right)$ with $s=\left(s_i\right)$ and $r=\left(r_i\right)$ and assume that $\Phi_u(s)=\Phi_u(r)$. We show that $s=r$. Since we have $\Phi_u(s)=\Phi_u(r)$, it follows that $s_u=r_u$. Then, for $v>u$, we have from (\ref{eq:commutativity of morphisms in inverse limit category}) that
    \[g_{uv} \comp F_{uv}(s_v) = s_u\comp f_{uv} = r_u\comp f_{uv}= g_{uv}\comp F_{uv}(r_v).\]
    Since $g_{uv}$ is an isomorphism, it follows that $F_{uv}(s_v)=F_{uv}(r_v)$. Since $s_v,r_v\in \cC_v\left(x_v,y_v\right) = \cC_v\left(\Phi_v(x),\Phi_v(y)\right)$ and $F_{uv}:\cC_{v}\left(\Phi_v(x),\Phi_v(y)\right)\to \cC_u\left(F_{uv}\comp\Phi_v\left(x\right),F_{uv}\comp\Phi_v\left(y\right)\right)$ is bijective by (a), we conclude that $s_v=r_v$.

    Next, let $i\in\dirI$ and let $v\in\dirI$ be such that $v>i$ and $v>u$. Then, using $s_v=r_v$ and equation (\ref{eq:commutativity of morphisms in inverse limit category}), we have 
    \[ s_i = g_{iv}\comp F_{iv}(s_v) \comp f_{uv}^{-1} = g_{iv}\comp F_{iv}(r_v)\comp f_{uv}^{-1} = r_i,\]
    and so $s=(s_i)=(r_i)=r$, which proves injectivity.
    
    Let us now show surjectivity. Let $s_u\in\cC_u\left(\Phi_u(x),\Phi_u(y)\right)$. We construct a morphism $s\in\cC(x,y)$ such that $\Phi_u(s)=s_u$. For $v>u$, consider the morphism
    \[g_{uv}^{-1}\comp s_u \comp f_{uv}\in \cC_u\left(F_{uv}\left(x_v\right), F_{uv}\left(y_v\right)\right).\]
    By (a), there exists a unique morphism $f\in\cC_v\left(x_v,y_v\right)$ such that
    \begin{equation}\label{eq:definition of sv}
        F_{uv}(f) = g_{uv}^{-1}\comp s_u \comp f_{uv}.
    \end{equation}
    We set $s_v\coloneqq f$. Let $w>v$ in $\dirI$. A straightforward computation using (\ref{eq:commutativity of objects in inverse limit category}), (\ref{eq:definition of sv}) and naturality of $\theta_{uvw}$ shows that the diagram
    \begin{equation}\label{eq:commutativity for w>v}
        \begin{tikzcd}
            F_{vw}\left(x_{w}\right) \arrow[r, "F_{vw}\left(s_{w}\right)"] \arrow[d, "f_{vw}"] & F_{vw}\left(y_{w}\right) \arrow[d, "g_{vw}"] \\
            x_{v} \arrow[r, "s_{v}"] &  y_{v}
        \end{tikzcd}
    \end{equation}
    commutes. Next, let $w,v,i\in\dirI$ be such that $w>v>i$. Another straightforward computation using commutativity of (\ref{eq:commutativity for w>v}), (\ref{eq:commutativity of objects in inverse limit category}) and naturality of $\theta_{ivw}$ shows that
    \begin{equation*}
        g_{iv}\comp F_{iv}\left(s_v\right) \comp f_{iv}^{-1} = g_{iw}\comp F_{iw}\left(s_w\right)\comp f_{iw}^{-1}.
    \end{equation*}
    We set $s_i \coloneqq g_{iv}\comp F_{iv}\left(s_v\right) \comp f_{iv}^{-1}$, and it readily follows that $s\coloneqq (s_i)$ is a morphism in $\cC(x,y)$ satisfying $\Phi_u(s)=s_u$, which proves surjectivity of $\Phi_u$. \qedhere
\end{enumerate}
\end{proof}

\begin{definition}\label{def:firm objects and Cks}
Let $\left(\cC,\Phi_i,\Theta_{ij}\right)$ be the concrete inverse limit of the $\Cat$-inverse system $\left(\cC_i,F_{ij},\theta_{ijk}\right)$. We say that an object $x\in\cC$ is \emph{firm}\index[definitions]{firm object} if there exists a $t=t(x)\in\dirI$ such that for all $u>t$ and all $a,b\in\cC_u$ the induced maps
\begin{align*}
    & F_{tu}:\cC_u\left(\Phi_u(x),b\right) \to \cC_t\left(F_{tu}\comp\Phi_u(x), F_{tu}(b)\right), \text{ and} \\
    & F_{tu}:\cC_u\left(a,\Phi_u(x)\right) \to \cC_t\left(F_{tu}(a),F_{tu}\comp\Phi_u(x)\right)
\end{align*} 
are isomorphisms. We define $\Cks$ to be the full subcategory of $\cC$ containing all firm objects in $\cC$, we denote $\Pks_i\coloneqq \restr{\Phi_i}{\Cks}$ and $\Tks_{ij}\coloneqq \restr{\Theta_{ij}}{\Cks}$ and we call $(\Cks,\Pks_i,\Tks_{ij})$\index[symbols]{(C, Phii, Thetaij)@$(\Cks,\Pks_i,\Tks_{ij})$} the \emph{firm source}\index[definitions]{firm source} of the inverse system $\left(\cC_i,F_{ij},\theta_{ijk}\right)$.
\end{definition}

Notice that while the firm source is clearly a $\Cat$-source of the $\Cat$-inverse system, in general it is not an inverse limit. The following proposition collects some basic properties of $\Cks$.

\begin{proposition}\label{prop:properties of Cks}
Let $\left(\cC,\Phi_i,\Theta_{ij}\right)$ be the concrete inverse limit of the $\Cat$-inverse system $\left(\cC_i,F_{ij},\theta_{ijk}\right)$.
\begin{enumerate}[label=(\alph*)]
    \item Let $e\in\End_{\cC}(x)$ be an idempotent and $s\in\cC(x,y)$ and $r\in\cC(y,x)$ be such that $r\comp s =e$ and $s\comp r=\Id_y$. Then $x\in\Cks$ implies $y\in\Cks$.
    
    \item Let $x\in\Cks$ and let $t=t(x)\in\dirI$ be as in Definition \ref{def:firm objects and Cks}. Then the induced map $\Pks_u:\End_{\Cks}(x)\to \End_{\cC_u}(\Pks_u(x))$ is bijective for all $u>t$.
    
    \item If all $\cC_i$ are $\K$-linear additive categories and all $F_{ij}$ are $\K$-linear functors, then $\Cks$ is a $\K$-linear additive category and all $\Pks_i$ are $\K$-linear functors.
\end{enumerate}
\end{proposition}

\begin{proof}
\begin{enumerate}[label=(\alph*)]
\item Let $x\in\Cks$ and $t=t(x)\in\dirI$ be as in Definition \ref{def:firm objects and Cks}. Let $u>t$. It is enough to show that the induced maps
    \begin{align*}
        & F_{tu}:\cC_u\left(\Phi_u(y),b\right) \to \cC_t\left(F_{tu}\comp\Phi_u(y), F_{tu}(b)\right), \text{ and} \\
        & F_{tu}:\cC_u\left(a,\Phi_u(y)\right) \to \cC_t\left(F_{tu}(a),F_{tu}\comp\Phi_t(y)\right)
    \end{align*} 
    are bijective for all $a,b\in\cC_u$. Let us only show that the first map is bijective; the other map can be similarly shown to be bijective. 
    
    To show injectivity, let $f,g:\Phi_u(y)\to b$ be such that $F_{tu}(f)=F_{tu}(g)$. Then
    \[F_{tu}\left(f\comp s_u\right) = F_{tu}(f)\comp F_{tu}\left(s_u\right) = F_{tu}(g)\comp F_{tu}\left(s_u\right) = F_{tu}\left(g\comp s_u\right),\]
    and since the map $F_{tu}:\cC_u\left(\Phi_u(x),b\right) \to \cC_t\left(F_{tu}\comp\Phi_u(x), F_{tu}(b)\right)$ is bijective by assumption, it follows that $f\comp s_u = g\comp s_u$. Hence
    \[f=f\comp \Id_{\Phi_u(y)} = f\comp s_u\comp r_u = g\comp s_u\comp r_u = g\comp \Id_{\Phi_u(y)} = g,\]
    which proves injectivity.
    
    Next, we show surjectivity. Let $h\in \cC_t\left(F_{tu}\comp \Phi_u(y),F_{tu}(b)\right)$. Then we have $h\comp F_{tu}\left(s_u\right)\in \cC_t\left(F_{tu}\comp \Phi_u(x), F_{tu}(b)\right)$. Since the map $F_{tu}:\cC_u\left(\Phi_u(x),b\right) \to \cC_t\left(F_{tu}\comp\Phi_u(x), F_{tu}(b)\right)$ is bijective by assumption, it follows that there exists a unique $p\in\cC_u\left(\Phi_u(x),b\right)$ such that $F_{tu}(p)=h\comp F_{tu}\left(s_u\right)$. Then $p\comp r_u\in \cC_u\left(\Phi_u(y),b\right)$ and
    \begin{align*}
        F_{tu}\left(p\comp r_u\right) &= F_{tu}(p)\comp F_{tu}\left(r_u\right) = h\comp F_{tu}\left(s_u\right)\comp F_{tu}\left(r_u\right) = h\comp F_{tu}\left(s_u\comp r_u\right)\\
        &= h \comp F_{tu}\left(\Id_{\Phi_u(y)}\right) = h\comp \Id_{F_{tu}\comp \Phi_u(y)} = h,
    \end{align*}
    which proves surjectivity.
    
    \item Follows immediately by Lemma \ref{lem:connecting functors eventually bijective implies limit functors eventually bijective}(b).
    
    \item By Proposition \ref{prop:additive limit}, we have that $\cC$ is a $\K$-linear additive category and all $\Phi_i$ are $\K$-linear functors. Since $\Cks$ is a full subcategory of $\cC$ and $\Pks_i=\restr{\Phi_i}{\Cks}$ it follows that $\Cks$ is a $\K$-linear category and all $\Pks_i$ are $\K$-linear functors. To show that $\Cks$ is additive, let $x,y\in\Cks$ and let $t(x),t(y)\in\dirI$ be as in Definition \ref{def:firm objects and Cks}. Let $t\in\dirI$ be such that $t>t(x)$ and $t>t(y)$. Let $u>t$. It is enough to show that the induced maps
    \begin{align*}
        & F_{tu}:\cC_u\left(a,\Phi_u(x\oplus y)\right) \to \cC_t\left(F_{tu}(a),F_{tu}\comp\Phi_t(x\oplus y)\right), \text{ and} \\
        & F\coloneqq F_{tu}:\cC_u\left(\Phi_u(x\oplus y),b\right) \to \cC_t\left(F_{tu}\comp\Phi_u(x\oplus y), F_{tu}(b)\right),
    \end{align*} 
    are bijective for all $a,b\in\cC_u$. Let us only show that $F$  is bijective; the other map can be similarly shown to be bijective. Since $F$ and $\Phi_u$ are additive functors, we may write $F$ as
    \[F=\left(\begin{matrix} F_{tu}^x & 0 \\ 0 & F_{tu}^y\end{matrix}\right),\] where
    \begin{align*}
        & F_{tu}^x:\cC_u\left(\Phi_u(x),b\right) \to \cC_t\left(F_{tu}\comp\Phi_u(x), F_{tu}(b)\right), \text{ and} \\
        & F_{tu}^y:\cC_u\left(\Phi_u(y),b\right) \to \cC_t\left(F_{tu}\comp\Phi_u(y), F_{tu}(b)\right)
    \end{align*} 
    are the corresponding induced maps for $x$ and $y$. Since $x,y\in\Cks$ and $u>t>t(x)$ and $u>t>t(y)$, we have that $F_{tu}^x$ and $F_{tu}^y$ are bijective. It follows that $F$ is bijective, as required. Hence $x\oplus y\in\Cks$ and so $\Cks$ is an additive category. Since all $\Phi_i$ are additive functors, it immediately follows that all $\Pks_i$ are additive functors.\qedhere
\end{enumerate}
\end{proof}

In particular, we have the following corollary.

\begin{corollary}\label{cor:properties of Cks}
Let $\left(\cC,\Phi_i,\Theta_{ij}\right)$ be the concrete inverse limit of the $\Cat$-inverse system $\left(\cC_i,F_{ij},\theta_{ijk}\right)$. If all $\cC_i$ are Krull--Schmidt categories, then $\Cks$ is a Krull--Schmidt category. In particular, for every $x\in\Cks$, if $t=t(x)\in\dirI$ is as in Definition \ref{def:firm objects and Cks}, then for every $u>t$ we have that $x$ is indecomposable if and only if $\Pks_u(x)$ is indecomposable.
\end{corollary}

\begin{proof}
By Proposition \ref{prop:properties of Cks}(c) it follows that $\Cks$ is additive. Since all $\cC_i$ have split idempotents, the category $\cC$ has split idempotents by Lemma \ref{lem:C has split idempotents}(c). Hence $\Cks$ has split idempotents by Proposition \ref{prop:properties of Cks}(a). Moreover, the induced map $\Pks_u:\End_{\Cks}(x)\to \End_{\cC_u}(\Pks_u(x))$ is additive by Proposition \ref{prop:properties of Cks}(c), is a ring morphism since $\Pks_u$ is a functor and is bijective by Proposition \ref{prop:properties of Cks}(b). Hence it is a ring isomorphism and in particular $\End_{\Cks}(x)\isom \End_{\cC_u}(\Pks_u(x))$ is a semiperfect ring for every $x\in \Cks$. Hence $\Cks$ is a Krull--Schmidt category.
    
In particular, since $\Cks$ is Krull--Schmidt, we have that $x$ is indecomposable if and only if $\End_{\Cks}(x)$ is a local ring. This statement is true if and only if $\End_{\cC_u}(\Pks_u(x))$ is a local ring, which is true if and only if $\Pks_u(x)$ is indecomposable, since $\cC_u$ is a Krull--Schmidt category.
\end{proof}

For the rest of this section we fix a $\Cat$-inverse system $\left(\cC_i,F_{ij},\theta_{ijk}\right)$, where all $\cC_i$ are small $\K$-linear Krull--Schmidt categories and all functors $F_{ij}$ are $\K$-linear, and we denote its firm source by $(\Cks,\Pks_i,\Tks_{ij})$. Such a system induces functors
\[\begin{tikzpicture}
\node (A) at (0,0) {$\M\cC_j$};
\node (B) at (3,0) {$\M\cC_i$\nospacepunct{,}};

\draw[->] (B) -- node[above] {$F_{ij\ast}$} (A); 
\draw[->] (A) to [out=30,in=150]  node[auto]{$F_{ij}^{\ast}$} (B);
\draw[->] (A) to [out=-30,in=-150]  node[below]{$F_{ij}^{!}$} (B);
\end{tikzpicture}
\]
for all $i<j$, where $\left(F_{ij}^{\ast},F_{ij\ast}\right)$ and $\left(F_{ij\ast},F_{ij}^{!}\right)$ are adjoint pairs. If moreover the functors $F_{ij}$ are coherent preserving, the above diagram restricts to
\[\begin{tikzpicture}
\node (A) at (0,0) {$\m\cC_j$};
\node (B) at (3,0) {$\m\cC_i$\nospacepunct{,}};

\draw[->] (B) -- node[above] {$F_{ij\ast}$} (A); 
\draw[->] (A) to [out=30,in=150]  node[auto]{$F_{ij}^{\ast}$} (B);
\draw[->] (A) to [out=-30,in=-150]  node[below]{$F_{ij}^{!}$} (B);
\end{tikzpicture}
\]
for all $i<j$ and again $\left(F_{ij}^{\ast},F_{ij\ast}\right)$ and $\left(F_{ij\ast},F_{ij}^{!}\right)$ are adjoint pairs. Moreover, for every $i<j<k$ we have $F_{ik\ast} \isom \left(F_{ij}\circ F_{jk}\right)_{\ast} =  F_{jk\ast} \circ F_{ij\ast}$, that is we have a natural isomorphism $ \theta_{ijk\ast}: F_{jk\ast} \circ F_{ij\ast} \To F_{ik\ast}$ defined via $\left( \theta_{ijk\ast}\right)_M = M\theta_{ijk}$ for all $M\in\m\cC_i$. It is easy to see that the triple $\left(\m\cC_i, F_{ij\ast}, \theta_{ijk\ast}\right)$\index[symbols]{(mod Ci, Fij, thetaijk)@$\left(\m\cC_i, F_{ij\ast}, \theta_{ijk\ast}\right)$} is a direct system of categories.

Now notice that $\Cks$ is a small $\K$-linear Krull--Schmidt category by Proposition \ref{prop:properties of Cks}(c) and Corollary \ref{cor:properties of Cks}(c). Hence we also have functors 
\[\begin{tikzpicture}
\node (A) at (0,0) {$\M\Cks$};
\node (B) at (3,0) {$\M\cC_i$\nospacepunct{,}};

\draw[->] (B) -- node[above] {$ \Pks_{i\ast}$} (A); 
\draw[->] (A) to [out=30,in=150]  node[auto]{$\Pks_i^{\ast}$} (B);
\draw[->] (A) to [out=-30,in=-150]  node[below]{$\Pks_{i}^{!}$} (B);
\end{tikzpicture}
\]
for all $i\in \dirI$, where $(\Pks_{i}^{\ast}, \Pks_{i\ast})$ and $(\Pks_{i\ast},\Pks_{i}^{!})$ are adjoint pairs. If moreover the functors $\Pks_{i}$ are coherent preserving, the above diagram restricts to
\[\begin{tikzpicture}
\node (A) at (0,0) {$\m\Cks$};
\node (B) at (3,0) {$\m\cC_i$\nospacepunct{,}};

\draw[->] (B) -- node[above] {$\Pks_{i\ast}$} (A); 
\draw[->] (A) to [out=30,in=150]  node[auto]{$\Pks_{i}^{\ast}$} (B);
\draw[->] (A) to [out=-30,in=-150]  node[below]{$\Pks_{i}^{!}$} (B);
\end{tikzpicture}
\]
for all $i\in \dirI$ and again $(\Pks_{i}^{\ast}, \Pks_{i\ast})$ and $(\Pks_{i\ast},\Pks_{i}^{!})$ are adjoint pairs. Moreover, for every $i<j$ we have $\Pks_{i\ast} \isom (F_{ij}\comp \Pks_j)_{\ast} =  \Pks_{j\ast}\comp F_{ij\ast}$, that is we have a natural isomorphism 
$ \Tks_{ij\ast}: \Pks_{j\ast}\comp F_{ij\ast} \To  \Pks_{i\ast}$ defined via $( \Tks_{ij\ast})_M = M\Tks_{ij}$ for all $M\in\m\cC_i$. It is easy to see that the triple $(\m\Cks,\Pks_{i\ast},\Tks_{ij\ast})$\index[symbols]{(mod C, Phii, Thetaij)@$(\m\Cks,\Pks_{i\ast},\Tks_{ij\ast})$} is a target of the direct system of categories $\left(\m\cC_i, F_{ij\ast}, \theta_{ijk\ast}\right)$. In this case, objects of $\m\Cks$ can be described using objects of $\m\cC_i$ in the sense of the following lemma.

\begin{lemma}\label{lem:KS modules are preserved} Let $i\in\dirI$ and $M\in\m\Cks$.
\begin{enumerate}[label=(\alph*)]
    \item Let $\eta^i$ be the unit of the adjunction $(\Pks^{\ast}_i, \Pks_{i\ast})$ where $\Pks^{\ast}_i:\M\Cks\to \M\cC_i$ and $ \Pks_{i\ast}:\M\cC_i \to \M\Cks$. Then there exists a $t\in\dirI$ such that for all $u>t$ we have that $\eta^u_M$ is an isomorphism. In particular, we have $\Pks_{u\ast}\Pks_{u}^{\ast}\left(M\right)\isom M$.
    
    \item Assume that $\Cks$ is a dualizing $\K$-variety. Let $\epsilon^i$ be the counit of the adjunction $(\Pks_{i\ast}, \Pks_{i}^{!})$ where $\Pks_{i\ast}:\M\cC_i\to \M\Cks$ and $ \Pks_{i}^{!}:\M\Cks \to \M\cC_i$. Then there exists a $t\in\dirI$ such that for all $u>t$ we have that $\epsilon^u_M$ is an isomorphism. In particular, we have $\Pks_{u\ast}\Pks_{u}^{!}\left(M\right)\isom M$.  
\end{enumerate}
\end{lemma}

\begin{proof}
\begin{enumerate}[label=(\alph*)]
    \item We first prove the result for $M=\Cks(-,y)\in\m\Cks$. For all $i\in\dirI$ we have
    \[\Pks_{i}^{\ast}(\Cks(-,y)) = \Pks_{i}^{\ast}\circ h_{\Cks}(y) = h_{\cC_i}\circ \Pks_i(y) = \cC_i(-,\Pks_i(y)),\]
    and so
    \[ \Pks_{i\ast}\Pks_{i}^\ast(\Cks(-,y)) =  \Pks_{i\ast}(\cC_i(-,\Pks_i(y))) = \cC_i(\Pks_i(-),\Pks_i(y)).\]
    Then the component of the unit $\eta^i$ of the adjunction $(\Pks^{\ast}_i, \Pks_{i\ast})$ at $\Cks(-,y)$ is given by the map 
    \[\left(\eta^i_{\Cks(-,y)}\right)_x:\Cks(x,y)\to \cC_i(\Pks_i(x),\Pks_i(y))\]
    with $\left(\eta^i_{\Cks(-,y)}\right)_x=\Pks_i^{x,y}$. Since $y\in\Cks$, for $u>t$ we have that the induced map $\Pks_u^{x,y}:\Cks(x,y)\to \cC_u(\Pks_u(x),\Pks_u(y))$ is bijective by Lemma \ref{lem:connecting functors eventually bijective implies limit functors eventually bijective}(b). It follows that for $u>t$ the map $\left(\eta^u_{\Cks(-,y)}\right)_x$ is bijective, and hence $\eta^u_{\Cks(-,y)}$ is an isomorphism.
    
    Now let $M\in\m\Cks$ be arbitrary. Then there exist $y_1,y_2\in \Cks$ such that there exists an exact sequence
    \[ \Cks(-,y_2)\to \Cks(-,y_1) \to M \to 0\]
    in $\M\Cks$. Let $t\in\dirI$ be such that $t>t(y_1)$ and $t>t(y_2)$. Then for every $u>t$ we have that $\eta^u_{\Cks(-,y_1)}$ and $\eta^u_{\Cks(-,y_2)}$ are isomorphisms. Since for every $i\in\dirI$ the functor $ \Pks_{i\ast}$ is exact and the functor $\Pks_{i}^{\ast}$ is right exact, for $u>t$ we get a commutative diagram
    \[
    \begin{tikzpicture}
        \node (B) at (2,2) {$\Cks(-,y_2)$};
        \node (C) at (6,2) {$\Cks(-,y_1)$};
        \node (D) at (9.5,2) {$M$};
        \node (E) at (11.5,2) {$0$};
        \node (F) at (13,2) {$0$};

        \node (B1) at (2,0) {$\Pks_{u\ast}\Pks_{u}^{\ast}\left(\Cks(-,y_2)\right)$};
        \node (C1) at (6,0) {$\Pks_{u\ast}\Pks_{u}^{\ast}\left(\Cks(-,y_1)\right)$};
        \node (D1) at (9.5,0) {$\Pks_{u\ast}\Pks_{u}^{\ast}\left(M\right)$};
        \node (E1) at (11.5,0) {$0$};
        \node (F1) at (13,0) {$0$\nospacepunct{,}};

        \draw[->] (B) to (C);
        \draw[->] (C) to (D);
        \draw[->] (D) to (E);
        \draw[->] (E) to (F);

        \draw[->] (B1) to (C1);
        \draw[->] (C1) to (D1);
        \draw[->] (D1) to (E1);
        \draw[->] (E1) to (F1);

        \draw[->] (B) to node[right] {$\eta^u_{\Cks(-,y_2)}$} (B1);
        \draw[->] (C) to node[right] {$\eta^u_{\Cks(-,y_1)}$} (C1);
        \draw[->] (D) to node[right] {$\eta^u_M$} (D1);
        \draw[->] (E) to (E1);
        \draw[->] (F) to (F1);
    \end{tikzpicture}
    \]
where the two leftmost downwards arrows are isomorphisms and the rows are exact. The result follows by the Five Lemma.

\item If $\Cks$ is a dualizing $\K$-variety, then every $M\in\m\Cks$ admits an injective presentation. Since $\Pks_{i}^{!}(D\Cks(x,-)) \isom D\cC_i(\Pks_i(x),-)$ by Proposition \ref{prop:injective sent to correct injective}(b), the result follows by applying the dual arguments to (a). \qedhere
\end{enumerate}
\end{proof}

Lemma \ref{lem:KS modules are preserved} states that a given finitely presented $\Cks$-module is preserved, up to isomorphism, under the functors $ \Pks_{i\ast}\Pks_i^{\ast}$ and $\Pks_{i\ast}\Pks_i^{!}$ for a sufficiently high $i$. We can use this result to compute extensions between modules in $\m\Cks$ as in the following proposition.

\begin{proposition}\label{prop:KS extensions}
Assume that all functors $\Pks_i$ are dense. Then the following statements hold.
\begin{enumerate}[label=(\alph*)]
    \item For all $M\in\m\Cks$ and $k>0$ there exists a $t\in\dirI$ such that for all $u>t$ and $L\in\m\cC_u$ we have $\Ext_{\Cks}^r(M,\Pks_{u\ast}(L)) \isom \Ext_{\cC_u}^r(\Pks_{u}^{\ast}(M), L)$ for all $0<r<k$.
    
    \item Assume that $\Cks$ is a dualizing $\K$-variety. Then for all $N\in\m\Cks$ and $k>0$ there exists a $t\in\dirI$ such that for all $u>t$ and $L\in\m\cC_u$ we have $\Ext_{\Cks}^r\left(\Pks_{u\ast}(L),N\right) \isom \Ext_{\cC_u}^r\left(L, \Pks_{u}^{!}(N)\right)$ for all $0<r<k$. 
\end{enumerate}
\end{proposition}

\begin{proof}
\begin{enumerate}[label=(\alph*)]
    \item Let $M\in \m\Cks$ and let
    \begin{equation}\label{eq:proj res}
        P_k \to \cdots \to P_1 \to P_0 \to M \to 0
    \end{equation}
    be the start of a projective resolution of $M$ in $\m\Cks$. We first claim that there exists $t\in\dirI$ such that for all $u>t$ and $0\leq j\leq k$ we have $\Pks_{u\ast}\Pks_{u}^{\ast}\left(P_j\right)\isom P_j$ and such that
    \begin{equation}\label{eq:proj res induced}
        \Pks_{u}^{\ast}(P_k) \to \cdots \to \Pks_{u}^{\ast}(P_1) \to \Pks_{u}^{\ast}(P_0) \to \Pks_{u}^{\ast}(M) \to 0
    \end{equation} 
    is the start of a projective resolution of $\Pks_{u}^{\ast}(M)$ in $\m\cC_u$. To see this, notice that the Yoneda embedding $h_{\Cks}:\Cks\to \m\Cks$ identifies $\Cks$ with $\proj\Cks$. Hence there exists a complex $y_k\to \cdots \to y_1\to y_0$ in $\Cks$ such that
    \begin{equation}\label{eq:proj res with Cs}
        \Cks(-,y_k) \to \cdots \to \Cks(-,y_1)\to \cC(-,y_0)
    \end{equation}
    is isomorphic to $P_k\to\cdots \to P_1\to P_0$ as it appears in (\ref{eq:proj res}). By Lemma \ref{lem:KS modules are preserved}(a), for each $0\leq j \leq k$ there exists a $t_j=t(y_j)\in\dirI$ such that for all $u>t_j$ we have $ \Pks_{u\ast}\Pks_{u}^{\ast}(\Cks(-,y_j))\isom \Cks(-,y_j)$. We choose $t$ such that $t>t_j$ for all $0\leq j \leq k$. It remains to show that (\ref{eq:proj res induced}) is a projective resolution of $\Pks_{u}^{\ast}(M)$ in $\m\cC_u$ for $u>t$. Notice that since $u>t(y_j)$, by Lemma \ref{lem:connecting functors eventually bijective implies limit functors eventually bijective} we have that the induced map $\Pks_{u}^{z,y_j}:\Cks(z,y_j) \to \cC_u(\Pks_u(z),\Pks_u(y_j))$ is bijective for all $z\in\Cks$. Moreover, for all $0\leq j \leq k$ we have 
    \[\Pks_{u}^{\ast}(\Cks(-,y_j)) = \Pks_{u}^{\ast}\comp h_{\Cks}(y_j) = h_{\cC_u}\comp \Pks_u(y_j) = \cC_u(-,\Pks_u(y_j)).\]
    Hence by applying $\Pks_{u}^{\ast}$ to (\ref{eq:proj res}) we get a sequence isomorphic to
    \[ \cC_u(-,\Pks_u(y_k))\to \cdots \to \cC_u(-,\Pks_u(y_1))\to \cC_u(-,\Pks_u(y_0)) \to \Pks_{u}^{\ast}(M)\to 0.\]
    This sequence is exact at $\cC_u(-,\Pks_u(y_0))$ and $\Pks_{u}^{\ast}(M)$ since $\Pks_{u}^{\ast}$ is right exact. To show exactness at the other positions, let $x_u\in \cC_u$. We need to show that the sequence
    \[\cC_u(x_u, \Pks_u(y_{j+1})) \to \cC_u(x_u, \Pks_u(y_j)) \to \cC_u(x_u, \Pks_u(y_{j-1}))\]
    is exact for $j\in \{1,\dots,k-1\}$. Since $\Pks_u$ is dense, there exists an $x\in \Cks$ such that $\Pks_u(x)\isom x_u$. But then, since $u>t$, we have the commutative diagram
    \[\begin{tikzpicture}
        \node (A1) at (2.75,2) {$\Cks(x,y_{j+1})$};
        \node (A2) at (6.75,2) {$\Cks(x,y_j)$};
        \node (A3) at (10.75,2) {$\Cks(x,y_{j-1})$};  
    
        \node (B1) at (2.75,1) {$\cC_u(\Pks_u(x),\Pks_u(y_{j+1}))$};
        \node (B2) at (6.75,1) {$\cC_u(\Pks_u(x),\Pks_u(y_j))$};
        \node (B3) at (10.75,1) {$\cC_u(\Pks_u(x),\Pks_u(y_{j-1}))$\nospacepunct{,}};   
    
        \draw[->] (A1) -- (A2);
        \draw[->] (A2) -- (A3);
    
        \draw[->] (B1) -- (B2);
        \draw[->] (B2) -- (B3);
        
        \draw[draw=none] (A1) to node[rotate=270] {$\isom$} (B1);
        \draw[draw=none] (A2) to node[rotate=270] {$\isom$} (B2);
        \draw[draw=none] (A3) to node[rotate=270] {$\isom$} (B3);
        \end{tikzpicture}
    \]
    where the top row is exact by exactness of (\ref{eq:proj res with Cs}). Since $\Pks_u(x)\isom x_u$, it follows that (\ref{eq:proj res induced}) is the start of a projective resolution of $\Pks_{u}^{\ast }(M)$ and the claim is proved.

    Now let $L\in\m\cC_u$. By applying $\Hom_{\Cks}(-,\Pks_{u\ast}(L))$ to (\ref{eq:proj res}) and using the fact that $\Pks_{u}^{\ast}$ is left adjoint to $ \Pks_{u\ast}$ we have the following isomorphism of complexes
    \[\begin{tikzpicture}
        \node (A1) at (0,2) {$0$};
        \node (A2) at (2.75,2) {$\Hom_{\Cks}(P_0,\Pks_{u\ast}(L))$};
        \node (A3) at (6.75,2) {$\Hom_{\Cks}(P_1,\Pks_{u\ast}(L))$};
        \node (A4a) at (9,2) {$ $};
        \node (A4b) at (9.5,2) {$ $};
        \node (A5) at (11.75,2) {$\Hom_{\Cks}(P_{k},\Pks_{u\ast}(L))$};
        
        \node (B1) at (0,1) {$0$};
        \node (B2) at (2.75,1) {$\Hom_{\cC_u}(\Pks_{u}^{\ast}(P_0),L)$};
        \node (B3) at (6.75,1) {$\Hom_{\cC_u}(\Pks_{u}^{\ast}(P_1),L)$};
        \node (B4a) at (9,1) {$ $};
        \node (B4b) at (9.5,1) {$ $};
        \node (B5) at (11.75,1) {$\Hom_{\cC_u}(\Pks_{u}^{\ast}(P_{m}),L)$\nospacepunct{.}};
    
        \draw[->] (A1) -- (A2);
        \draw[->] (A2) -- (A3);
        \draw[->] (A3) -- (A4a);
        \draw[dotted] (A4a) -- (A4b);
        \draw[->] (A4b) -- (A5);
    
        \draw[->] (B1) -- (B2);
        \draw[->] (B2) -- (B3);
        \draw[->] (B3) -- (B4a);
        \draw[dotted] (B4a) -- (B4b);
        \draw[->] (B4b) -- (B5);
        
        \draw[draw=none] (A1) to node[rotate=270] {$=$} (B1);
        \draw[draw=none] (A2) to node[rotate=270] {$\isom$} (B2);
        \draw[draw=none] (A3) to node[rotate=270] {$\isom$} (B3);
        \draw[draw=none] (A5) to node[rotate=270] {$\isom$} (B5);
    \end{tikzpicture}
    \]
    Since $\Ext^r_{\Cks}(M,\Pks_{u\ast}(L))$ is the cohomology of the top row at position $r$ and $\Ext^r_{\cC_u}(\Pks_{u}^{\ast}(M),L)$ is the cohomology of the bottom row at position $r$, the result follows.
    
    \item Proved using dual arguments and Lemma \ref{lem:KS modules are preserved}(b). \qedhere
\end{enumerate}
\end{proof}

As a corollary of the above statements, we have the following.

\begin{corollary}\label{cor:admissible target}
Let $\left(\cC_i,F_{ij},\theta_{ijk}\right)$ be a $\Cat$-inverse system, where all $\cC_i$ are dualizing $\K$-varieties and all $F_{ij}$ are coherent preserving $\K$-linear functors. Assume that $\Cks$ is a dualizing $\K$-variety and all functors $\Pks_i$ are dense, full and coherent preserving. Then $(\m\Cks,\Pks_{i\ast},\Tks_{ij})$ together with the adjunctions $(\Pks_i^{\ast},\Pks_{i\ast})$ and $(\Pks_{i\ast},\Pks_i^{!})$ is an admissible target of $(\m\cC_i,F_{ij\ast},\theta_{ijk\ast})$.
\end{corollary}

\begin{proof}
Since all $\cC_i$ are dualizing $\K$-varieties, it follows that $\m\cC_i$ is an abelian category for all $i\in\dirI$. The functors $F_{ij\ast}$ are always exact and since $\Phi_i \isom F_{ij}\comp \Phi_j$, it follows that all functors $F_{ij}$ are dense and full. Hence by Proposition \ref{prop:properties of dense and full}(c) the functors $F_{ij\ast}:\m\cC_i\to \m\cC_j$ are fully faithful. 

By the same argument, the functors $\Pks_{i\ast}$ are also fully faithful, so condition (i) of Definition \ref{def:admissible target} is satisfied. Moreover condition (ii) is satisfied by Lemma \ref{lem:KS modules are preserved} and conditions (iii) and (iv) are satisfied by Proposition \ref{prop:KS extensions}.
\end{proof}
   
\subsection{Asymptotically weakly \texorpdfstring{$n$}{n}-cluster tilting systems}\label{subsection:asymptotically weakly n-cluster tilting systems}

Let $\cA$ be an abelian category. For $x,y\in \cA$ we set 
\begin{equation*}
    \Ext_{\cA}^{k\sim l}(x,y) \coloneqq \bigoplus_{i=k}^l\Ext^i_{\cA}(x,y).
\end{equation*}
\index[symbols]{ExtA(x,y)@$\Ext_{\cA}^{k\sim l}(x,y)$}Clearly $\Ext_{\cA}^{k\sim l}(x,y)=0$ if and only if $\Ext_{\cA}^i(x,y)=0$ for all $k\leq i \leq l$. We begin by recalling the following definition.

\begin{definition}[\cite{IYA2}]
Let $\cA$ be an abelian category. A full subcategory $\cM$ of $\cA$ is called \emph{weakly $n$-cluster tilting}\index[definitions]{weakly $n$-cluster tilting subcategory} if 
\begin{align*}
    \cM &= \{ x \in \cA \mid \Ext^{1\sim n-1}_{\cA}\left(x, \cM\right)=0\} \\
    & = \{ x \in \cA \mid \Ext^{1\sim n-1}_{\cA}\left(\cM, x\right)=0\}. 
\end{align*}
If moreover $\cM$ is functorially finite, then $\cM$ is called \emph{$n$-cluster tilting}\index[definitions]{n-cluster tilting subcategory@$n$-cluster tilting subcategory}.
\end{definition}

The following definition is the main object for this section.

\begin{definition}\label{def:asymptotically weakly n-cluster tilting}
Let $\left(\cA_i,G_{ij},\zeta_{ijk}\right)$ be a direct system of categories where all $\cA_i$ are abelian and all $G_{ij}$ are exact fully faithful functors. For every $i\in \dirI$, let $\cM_i$ be a full subcategory of $\cA_i$. We call $\left(\cM_i\right)$\index[symbols]{(Mi)@$\left(\cM_i\right)$} an \emph{asymptotically weakly $n$-cluster tilting system}\index[definitions]{asymptotically weakly $n$-cluster tilting system} if the following conditions hold.
\begin{enumerate}[label=(\roman*)]
    \item $\Ext^{1\sim n-1}_{\cA_i}(\cM_i,\cM_i)=0$.
    \item For every $j>i$ we have $G_{ij}\left({\cM_i}\right)\subseteq \cM_j$.
    \item If $x\in \cA_{i}$ and $\Ext^{1\sim n-1}_{\cA_j}\left( G_{ij}(x),\cM_j\right)=0$ for all $j>i$, then $x\in\cM_i$.
    \item If $x\in \cA_{i}$ and $\Ext^{1\sim n-1}_{\cA_j}\left(\cM_j, G_{ij}(x)\right)=0$ for all $j>i$, then $x\in\cM_i$.
\end{enumerate}
\end{definition}

An asymptotically weakly $n$-cluster tilting system of a direct system of categories $\left(\cA_i,G_{ij},\zeta_{ijk}\right)$ gives rise to a weakly $n$-cluster tilting subcategory of any admissible target of $\left(\cA_i,G_{ij},\zeta_{ijk}\right)$ as the following theorem shows.

\begin{theorem}\label{thrm:from asymptotic to n-ct}
Let $\left(\cM_i\right)$ be an asymptotically weakly $n$-cluster tilting system of a direct system of categories $\left(\cA_i,G_{ij},\zeta_{ijk}\right)$ and let $\left(\cA,\Psi_{i},Z_{ij}\right)$ together with adjunctions $\left(L_i,\Psi_i\right)$ and $\left(\Psi_i,R_i\right)$ be an admissible target of $\left(\cA_i,G_{ij},\zeta_{ijk}\right)$. Set
\[\cM \coloneqq \add\left\{ \Psi_{i}\left(\cM_i\right) \mid i\in\dirI\right\}.\]
\index[symbols]{M b@$\cM$}Then $\cM$ is a weakly $n$-cluster tilting subcategory of $\cA$. In particular, if $\cM$ is functorially finite in $\cA$, then $\cM$ is an $n$-cluster tilting subcategory of $\cA$.
\end{theorem}

\begin{proof}
First we claim that for all $m_i\in\cM_i$ and $m_j\in \cM_j$ we have
\begin{equation}\label{eq:M is rigid}
\Ext^{1\sim n-1}_{\cA}\left( \Psi_{i}(m_i), \Psi_{j}(m_j)\right)=0.
\end{equation}
By condition (ii) in Definition \ref{def:admissible target} applied to $\Psi_i(m_i)$, there exists a $t\in\dirI$ such that for all $u>t$ we have
\begin{equation}\label{eq:t for preservence}
    \Psi_{u}L_{u}\Psi_{i}(m_i) \isom  \Psi_{i}(m_i).
\end{equation}
If moreover we let $u>i$ and $u>j$, then we also have
\begin{equation}\label{eq:j for system}
    \Psi_{i}(m_i)\isom \Psi_{u}G_{iu}(m_i) \text{ and } \Psi_{j}(m_j) \isom \Psi_{u}G_{ju}(m_j),
\end{equation}
since $\left(\cA,\Psi_{i},Z_{ij}\right)$ is a target of $\left(\cA_i,G_{ij},\zeta_{ijk}\right)$.
In particular, by (\ref{eq:t for preservence}) and (\ref{eq:j for system}) we have
\[\Psi_{u}L_{u}\Psi_{i}(m_i) \isom \Psi_{i}(m_i)\isom \Psi_{u}G_{iu}(m_i),\]
which implies that $L_{u}\Psi_{i}(m_i)\isom G_{iu}(m_i)$, since $\Psi_{u}$ is fully faithful and so it reflects isomorphisms. Hence for $u>i$ we have shown that $L_{u}\Psi_{i}(m_i)\isom G_{iu}(m_i)$. Now let $t'$ be as in condition (iii) of Definition \ref{def:admissible target} applied to $\Psi_i(m_i)$ and let $u$ be greater than $t$, $t'$, $i$ and $j$. Then we have 
\begin{align*}
    \Ext^r_{\cA}\left(\Psi_{i}(m_i),  \Psi_{j}(m_j)\right) &\isom \Ext_{\cA_{u}}^r \left(\Psi_{i}(m_i), \Psi_{u} G_{ju}(m_j)\right) \\
    &\isom \Ext_{\cA_{u}}^r\left( L_u\Psi_i(m_i), G_{ju}(m_j)\right) \\
    &\isom \Ext_{\cA_{u}}^r\left(G_{iu}(m_i), G_{ju}(m_j)\right) \\
    &=0,
\end{align*}
where the last equality follows by conditions (i) and (ii) of Definition \ref{def:asymptotically weakly n-cluster tilting}. This proves (\ref{eq:M is rigid}). But then, by (\ref{eq:M is rigid}), it follows that
\begin{align*}
    \cM &\subseteq \{ x \in \cA \mid \Ext^{1\sim n-1}_{\cA}\left(x, \cM\right)=0\}, \text{ and }\\
    \cM &\subseteq \{ x \in \cA \mid \Ext^{1\sim n-1}_{\cA}\left( \cM,x\right)=0\}.
\end{align*}
Hence it remains to show the reverse inclusions.

We first show the inclusion $\{ x \in \cA\mid \Ext^{1\sim n-1}_{\cA}\left(x, \cM\right)=0\}\subseteq \cM$. Let $x\in \cA$ be such that $\Ext^{1\sim n-1}_{\cA}\left(x, \cM\right)=0$. We need to show that $x\in \cM$. Let $t\in\dirI$ be such that condition (ii) and condition (iii) for $k=n$ of Definition \ref{def:admissible target} applied to $x$ are both satisfied. Fix $i>t$ and pick $j>i$. We claim that
\begin{equation}\label{eq:claim for supseteq}
    \Ext_{\cA_j}^{1\sim n-1}\left(G_{ij}L_{i}(x), \cM_j\right)=0.
\end{equation}
To show this claim, notice first that since $j>i>t$, we have
\begin{align*}
    \Psi_{j} G_{ij}L_i(x) &\isom \Psi_{i}L_i(x) \isom x \isom\Psi_{j}L_j(x),
\end{align*}
and so $G_{ij}L_i(x) \isom L_j(x)$ since $\Psi_{j}$ reflects isomorphisms. Hence to prove (\ref{eq:claim for supseteq}), it is enough to show that for all $r\in\{1,\dots,n-1\}$ we have 
\begin{equation}\label{eq:ext with Phi_j}
    \Ext_{\cA_j}^r\left(L_j(x),\cM_j\right)=0.
\end{equation}
Let $m_j\in \cM_j$. For all $r\in\{1,\dots,n-1\}$, we have by condition (iii) of Definition \ref{def:admissible target} that
\begin{align*}
    \Ext^r_{\cA_j}\left(L_j(x), m_j\right) \isom \Ext^r_{\cA}\left(x, \Psi_{j}(m_j)\right)= 0,
\end{align*} 
where the last equality holds since $\Ext^{1\sim n-1}_{\cA}\left(x, \cM\right)=0$. Hence we have shown (\ref{eq:ext with Phi_j}) and so (\ref{eq:claim for supseteq}) holds. Since $j>i$ was arbitrary, it follows that (\ref{eq:claim for supseteq}) holds for every $j>i$ and so by condition (iii) in Definition \ref{def:asymptotically weakly n-cluster tilting}, we have that $L_i(x)\in \cM_i$. But then $x\isom \Psi_{i}L_i(x)\in \cM$, as required. 

Finally, the remaining inclusion $\{ x \in \cA\mid \Ext^{1\sim n-1}_{\cA}\left(\cM,x\right)=0\}\subseteq \cM$ is shown using dual arguments as well as condition (iv) of Definition \ref{def:admissible target} and condition (iv) of Definition \ref{def:asymptotically weakly n-cluster tilting}. 
\end{proof}

For the rest of this section we fix an asymptotically weakly $n$-cluster tilting system $\left(\cM_i\right)$ of a direct system of categories $\left(\cA_i,G_{ij},\zeta_{ijk}\right)$ and an admissible target $\left(\cA,\Psi_{i},Z_{ij}\right)$ of $\left(\cA_i,G_{ij},\zeta_{ijk}\right)$ together with adjunctions $\left(L_i,\Psi_i\right)$ and $\left(\Psi_i,R_i\right)$. We further set $\cM \coloneqq \add\left\{ \Psi_{i}\left(\cM_i\right) \mid i\in\dirI\right\}$. We continue this section by giving a sufficient condition for the subcategory $\cM$ to be functorially finite.

\begin{proposition}\label{prop:from n-fractured to functorially finite}
\begin{enumerate}[label=(\alph*)]
    \item Assume that for all $x\in\cA$ there exists a $t\in\dirI$ such that for all $u>t$ the following condition holds.
    \begin{itemize}
        \item If $f_u:m_u\to L_u(x)$ is a right $\cM_u$-approximation, then $G_{uv}(f_u):G_{uv}(m_u)\to G_{uv}L_u(x)$ is a right $\cM_v$-approximation of $G_{uv}L_u(x)$ for every $v>u$.
    \end{itemize}
    Let $x\in\cA$. Let $u\in\cI$ be such that $u>t$ and $u>t(x)$, where $t(x)$ is as in condition (ii) of Definition \ref{def:admissible target} applied to $x$.
    If $f_u:m_u\to L_u(x)$ is a right $\cM_u$-approximation, then $\Psi_u(f_u):\Psi_u(m_u)\to \Psi_u L_u(x)$ is a right $\cM$-approximation of $\Psi_u L_u(x)$. If moreover $\cM_i$ is contravariantly finite for every $i\in \dirI$, then $\cM$ is contravariantly finite.
    
    \item Assume that for all $x\in\cA$ there exists a $t\in\dirI$ such that for all $u>t$ the following condition holds.
    \begin{itemize}
        \item If $f_u:R_u(x)\to m_u$ is a left $\cM_u$-approximation, then $G_{uv}(f_u):G_{uv}R_u(x)\to G_{uv}(m_u)$ is a left $\cM_v$-approximation of $G_{uv}R_u(x)$ for every $v>u$.
    \end{itemize}
    Let $x\in\cA$. Let $u\in\cI$ be such that $u>t$ and $u>t(x)$, where $t(x)$ is as in condition (ii) of Definition \ref{def:admissible target} applied to $x$.
    If $f_u:R_u(x)\to m_u$ is a left $\cM_u$-approximation, then $\Psi_u(f_u):\Psi_u R_u(x)\to \Psi_u(m_u)$ is a left $\cM$-approximation of $\Psi_u R_u(x)$. If moreover $\cM_i$ is covariantly finite for every $i\in\dirI$, then $\cM$ is covariantly finite.
    
    \item Assume that the conditions of (a) and (b) hold. Then $\cM$ is functorially finite.
\end{enumerate}
\end{proposition}

\begin{proof}
(c) follows immediately from (a) and (b). We only prove (a); the proof of (b) is similar. Denote by $\eta^i:\Id_{\cA}\To \Psi_i L_i$ the unit of the adjunction $\left(L_i,\Psi_i\right)$. Let $x\in \cA$ and let $u\in\dirI$ be such that $u>t$ and $u>t(x)$. Then $L_u(x)\in\cA_u$ and $\eta^u_x:x\to \Psi_u L_u(x)$ is an isomorphism. Let $f_u:m_u\to L_u(x)$ be a right $\cM_u$-approximation of $L_u(x)$. We need to show that $\Psi_u(f_u):\Psi_u(m_u)\to \Psi_u L_u(x)$ is a right $\cM$-approximation of $\Psi_u L_u(x)$. 

To prove this, let $f':m'\to \Psi_u L_u(x)$ be a morphism in $\cA$ with $m'\in \cM$. We need to show that $f'$ factors through $\Psi_u(f_u)$. Without loss of generality, we may assume $m'\isom \psi_i(m_i)$ for some $m_i\in\cM_i$. Let $t(m')\in\dirI$ be as in condition (ii) of Definition \ref{def:admissible target} applied to $m'$ and let $v\in\dirI$ be such that $v$ is bigger than $t(m')$, $u$ and $i$. In particular, since $v>u>t(x)$, we have an isomorphism
\[\left(\left(Z_{uv}\right)_{L_u(x)}\right)^{-1}\comp \eta_x^u \comp \left(\eta^v_x\right)^{-1}\comp \Psi_v L_v\left(\left(\eta_x^u\right)^{-1}\right): \Psi_v L_v \Psi_u L_u(x) \to \Psi_v G_{uv} L_u(x).\]
Since $\Psi_v$ is full and faithful, we set 
\[s\coloneqq \Psi_v^{-1}\left(\left(\left(Z_{uv}\right)_{L_u(x)}\right)^{-1}\comp \eta_x^u \comp \left(\eta^v_x\right)^{-1}\comp \Psi_v L_v\left(\left(\eta_x^u\right)^{-1}\right)\right):L_v \Psi_u L_u(x)\to G_{uv}L_u(x).\]

Next, we claim that $L_v(m')\in\cM_v$. To see this, we first have that $\eta^v_{m'}:m'\to\Psi_{v}L_v(m')$ is an isomorphism since $v>t(m')$. Moreover, since $v>i$, we have $\Psi_{i}\isom \Psi_{v}\comp G_{iv}$ and so
\begin{align*}
    \Psi_v L_v(m')\isom m' \isom \Psi_{i}\left(m_{i}\right)\isom \Psi_{v}\comp G_{i v}\left(m_{i}\right).
\end{align*}
Since $\Psi_v$ is fully faithful, it reflects isomorphisms and so $L_v(m')\isom G_{i v}\left(m_{i}\right)$. But $m_{i}\in\cM_{i}$ and so $L_v(m')\isom G_{i v}\left(m_{i}\right)\in\cM_{v}$ by condition (ii) of Definition \ref{def:asymptotically weakly n-cluster tilting}, as claimed.

Next, since $v>u>t$ we have that $G_{uv}(f_u):G_{uv}(m_u)\to G_{uv}L_u(x)$ is a right $\cM_v$-approximation of $G_{uv}L_u(x)$. Since $L_v(m')\in\cM_v$ and we have a morphism $s\comp L_v(f'): L_v(m')\to G_{uv}L_u(x)$, there exists a morphism $g_v:L_v(m')\to G_{uv}(m_u)$ such that 
\begin{equation}\label{eq:def of xi}
    G_{uv}(f_u) \comp g_v = s\comp L_v(f').
\end{equation} 
We set $g':m'\to \Psi_u(f_u)$ to be the morphism $g'\coloneqq \left(Z_{uv}\right)_{m_u}\comp \Psi_v(g_v) \comp \eta^v_{m'}$. We claim that $\Psi_u(f_u)\comp g'=f'$. Indeed we compute
\begin{align*}
    \lefteqn{\Psi_u(f_u)\comp g'}\quad \\
    &= \Psi_u(f_u)\comp \left(Z_{uv}\right)_{m_u}\comp \Psi_v(g_v) \comp \eta^v_{m'} &&\\
    &= \left(Z_{uv}\right)_{L_u(x)}\comp \Psi_v G_{uv}(f_u)\comp \Psi_v(g_v) \comp \eta^v_{m'} &&\text{\small(naturality of $Z_{uv}$)} \\
    &=\left(Z_{uv}\right)_{L_u(x)}\comp \Psi_v\left(G_{uv}(f_u)\comp g_v\right) \comp \eta_{m'}^v &&\\
    &=\left(Z_{uv}\right)_{L_u(x)}\comp \Psi_v\left(s\comp L_v(f')\right) \comp \eta_{m'}^v &&\text{\small(\ref{eq:def of xi})} \\
    &=\left(Z_{uv}\right)_{L_u(x)}\comp \left(\left(Z_{uv}\right)_{L_u(x)}\right)^{-1}\comp \eta_x^u \comp \left(\eta^v_x\right)^{-1}\comp \Psi_v L_v\left(\left(\eta_x^u\right)^{-1}\right)\comp \Psi_v L_v(f')\comp \eta^v_{m'} &&\text{\small(definition of $s$)} \\
    &=\eta_x^u \comp \left(\eta^v_x\right)^{-1}\comp \Psi_v L_v\left(\left(\eta_x^u\right)^{-1}\right)\comp \Psi_v L_v(f')\comp \eta^v_{m'} && \\
    &=\eta_x^u \comp \left(\eta^v_x\right)^{-1}\comp \Psi_v L_v\left(\left(\eta_x^u\right)^{-1}\right)\comp \eta^v_{\Psi_u L_u(x)}\comp f' &&\text{\small(naturality of $\eta^v$)} \\
    &= \eta_x^u \comp \left(\eta^v_x\right)^{-1} \comp\eta_x^v \comp \left(\eta_x^u\right)^{-1}\comp f' &&\text{\small(naturality of $\eta^v$)} \\
    &= f',
\end{align*}
as required. Therefore $\Psi_u(f_u):\Psi_u(m_u)\to \Psi_u L_u(x)$ is a right $\cM$-approximation.
 
If moreover $\cM_i$ is contravariantly finite for every $i\in\dirI$, then we show that $\cM$ is contravariantly finite. Let $x\in \cA$ and let $u\in\dirI$ be such that $u>t$ and $u>t(x)$. Since $\cM_u$ is contravariantly finite, there exists a right $\cM_u$-approximation $f_u:m_u\to L_u(x)$ of $L_u(x)$. It follows that $\Psi_u(f_u):\Psi_u(m)\to \Psi_u L_u(x)$ is a right $\cM$-approximation of $\Phi_u L_u(x)$. Since $u>t(x)$, it follows that $\Psi_u L_u(x) \isom x$. Hence by composing $\Psi_u(f_u)$ with an isomorphism $\Psi_u L_u(x) \overset{\sim}{\to} x$ we obtain a right $\cM$-approximation of $x$. Since $x$ was arbitrary, it follows that $\cM$ is contravariantly finite.
\end{proof}

We finish this section with the following case which is of special interest to us.

\begin{corollary}\label{cor:n-ct if locally bounded}
If $\cA$ is a locally bounded category and $\cM_i$ is functorially finite for every $i\in\dirI$, then $\cM$ is an $n$-cluster tilting subcategory of $\cA$.
\end{corollary}

\begin{proof}
By Theorem \ref{thrm:from asymptotic to n-ct} it is enough to show that $\cM$ is functorially finite. By Proposition \ref{prop:from n-fractured to functorially finite}(c) it is enough to show that conditions (a) and (b) of Proposition \ref{prop:from n-fractured to functorially finite} hold. Let us show that condition (a) holds; condition (b) can be shown similarly to hold.

Since $\cA$ is Krull--Schmidt it is enough to show that condition (a) of Proposition \ref{prop:from n-fractured to functorially finite} holds for indecomposable objects in (a). Let $x\in\ind\cA$. Since $\cA$ is locally bounded, the set $\cA^x\coloneqq\{y\in\ind\cA \mid \cA(x,y)\neq 0 \text{ or } \cA(y,x)\neq 0\}$ is finite. Hence by Definition \ref{def:admissible target}(ii) and (iii) there exists a $t$ such that for every $u>t$ we have $\Psi_u L_u(y)\isom y$ for every $y\in\cA^x$ and $\Ext^r_{\cA}\left(x,\Psi_u(z)\right)\isom \Ext^r_{\cA_u}\left(L_u(x),z\right)$ for every $z\in \cA_u$ and for all $0<r<n$. Note in particular that $x\in\cA^x$. 

Now let $u>t$ and assume that $f_u:m_u\to L_u(x)$ is a right $\cM_u$-approximation. Let $v>u$. It is enough to show that $G_{uv}(f_u): G_{uv}(m_u)\to G_{uv}L_u(x)$ is a right $\cM_v$-approximation of $G_{uv}L_u(x)$. Let $f'_v:m'_v\to G_{uv}L_u(x)$ be a morphism with $m_v'\in \cM_v$. We want to show that $f'_v$ factors through $G_{uv}(f_u)$. Without loss of generality, we may assume that $m'_v$ is indecomposable. Since $\Psi_v:\cA_v\to \cA$ is fully faithful, it follows that $\Psi_v(f'_v):\Psi_v(m'_v)\to \Psi_v G_{uv} L_u(x)$ is nonzero. Since $\Psi_v G_{uv} L_u(x)\isom \Psi_u L_u(x) \isom x$, we have that $\Psi_v(m'_v)\in \cA^x$. We claim that $L_u \Psi_v(m'_v)\in \cM_u$.

To show this, let $w>u$, let $0<k<n$ and let $m_{w}\in \cM_{w}$. Since $w>u>t$ and since $\Psi_v(m'_v)\in \cA^x$, we have
\[\Psi_{w}G_{uw}L_u \Psi_v(m'_v) \isom \Psi_u L_u \Psi_v(m'_v) \isom \Psi_v(m'_v) \isom \Psi_{w}L_{w}\Psi_v(m'_v),\]
which implies, since $\Psi_{w}$ is fully faithful, that
\begin{equation}\label{eq:since psiv' is fully faitfhul}
    G_{uw}L_u\Psi_v(m'_v) \isom L_{w}\Psi_v(m'_v).
\end{equation}
Then we have
\begin{align*}
    \Ext_{\cA_{w}}^k\left(G_{uw}L_u \Psi_v(m'_v), m_{w}\right) &\isom \Ext^k_{\cA_{w}}\left(L_{w}\Psi_v(m'_v), m_{w}\right) &&\text{(by (\ref{eq:since psiv' is fully faitfhul}))}\\
    &\isom \Ext^k_{\cA}\left(\Psi_v(m'_v),\Psi_{w}(m_{w})\right) &&\text{(since $w>t$ and $\Psi_v(m'_v)\in \cA_x$)} \\
    &\isom 0.
\end{align*}
Hence $\Ext^{1\sim n-1}_{\cA_{w}}\left(G_{uw}\left(L_u\Psi_v(m'_v)\right), \cM_{w}\right)=0$ for every $w>u$ and so by Definition \ref{def:asymptotically weakly n-cluster tilting}(iii) it follows that $L_u\Psi_v(m'_v)\in \cM_u$. By applying the functor $L_u\Psi_v$ to the diagram
\[
\begin{tikzcd}
    G_{uv}(m_u) \arrow[rr, "G_{uv}\left(f_u\right)"]  && G_{uv}L_u(x) \\
    & m'_v \arrow[ur, "f'_v"] &  
\end{tikzcd}
\]
in $\cA_v$, we get the diagram
\begin{equation}\label{eq:big diagram}
\begin{tikzcd}
    L_u\Psi_vG_{uv}(m_u) \arrow[rr, "L_u\Psi_vG_{uv}\left(f_u\right)"] && L_u\Psi_vG_{uv}L_u(x) \\
    & L_u\Psi_v(m'_v) \arrow[ur, "L_u\Psi_v(f'_v)"] & \\  
    L_u\Psi_u(m_u) \arrow[uu, "\isom"] &&  L_u\Psi_u L_u(x) \arrow[uu, "\isom"] \\
    m_u \arrow[u, "\isom"] \arrow[rr, "f_u"] && L_u(x) \arrow[u, "\isom"]  
\end{tikzcd}
\end{equation}
in $\cA_u$, where all isomorphisms are natural and the outer square commutes. It follows that there exists a morphism $g:L_u\Psi_v(m'_v)\to L_u\Psi_vG_{uv}(m_u)$ such that $L_u\Psi_v G_{uv}(f_u) \comp g = L_u \Psi_v(f'_v)$. Moreover, since $\Psi_v G_{uv}L_u\Psi_v(m'_v)\isom \Psi_u L_u \Psi_v(m'_v)\isom \Psi_v(m'_v)$, and since $\Psi_v$ reflects isomorphisms, it follows that $G_{uv}L_u\Psi_v(m'_v)\isom m'_v$. Applying the functor $G_{uv}$ to (\ref{eq:big diagram}), we get a diagram
\[
\begin{tikzcd}
    G_{uv}L_u\Psi_vG_{uv}(m_u) \arrow[rr, "G_{uv}L_u\Psi_vG_{uv}\left(f_u\right)"] && G_{uv}L_u\Psi_vG_{uv}L_u(x) \\
    & G_{uv}L_u\Psi_v(m'_v) \arrow[ur, "G_{uv}L_u\Psi_v(f'_v)"] \arrow[ul, "G_{uv}(g)"] & \\
    & m'_v \arrow[u, "\isom"] \arrow[rd, "f'_v"]& \\
    G_{uv}(m_u) \ar[uuu,"\isom"] \arrow[rr, "G_{uv}(f_u)"] && G_{uv}L_u(x)\nospacepunct{,} \ar[uuu,"\isom"]\\
\end{tikzcd}
\]
where the top triangle, the right square and the outer square all commute. It follows that $f'_v$ factors through $G_{uv}(f_u)$ as required, which completes the proof.
\end{proof}

\section{Asymptotically weakly \texorpdfstring{$n$}{n}-cluster tilting systems from rep\-re\-sen\-ta\-tion-di\-rect\-ed algebras}\label{sec:asymptotically weakly n-cluster tilting systems from algebras}

In this section we present a way of constructing asymptotically weakly $n$-cluster tilting systems that satisfy the requirements of Theorem \ref{thrm:from asymptotic to n-ct}. This construction is based in the notion of gluing of rep\-re\-sen\-ta\-tion-di\-rect\-ed algebras that was introduced in \cite{VAS2}. To this end we recall the basic idea of that construction; for more details and proofs we refer to \cite{VAS2}.

\subsection{Gluing systems of rep\-re\-sen\-ta\-tion-di\-rect\-ed algebras and admissible targets}

We construct a direct system of categories with an admissible target based on the theory developed on Section \ref{subsection:construction of admissible targets} and the notion of gluing of rep\-re\-sen\-ta\-tion-di\-rect\-ed algebras.

\subsubsection{Gluing of rep\-re\-sen\-ta\-tion-di\-rect\-ed algebras}\label{subsubsec:gluing of representation-directed algebras} In this section, all algebras are assumed to be rep\-re\-sen\-ta\-tion-di\-rect\-ed, unless stated otherwise. To simplify the exposition, we assume that all algebras are given by quivers with relations. The first important notion is that of an abutment.

\begin{definition}\cite[Proposition 2.6]{VAS2}\label{def:abutments}
Let $A=\K Q_A/\cR_A$ be a rep\-re\-sen\-ta\-tion-di\-rect\-ed algebra given by a quiver with relations such that $Q_A$ is of the form
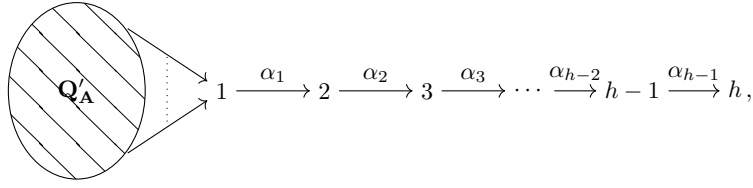
\[\begin{tikzpicture}[scale=0.9, transform shape]
\node(A) at (0,0) {$1$};
\node(B) at (1.5,0) {$2$};
\node(C) at (3,0) {$3$};
\node(D) at (4.5,0) {$\cdots$};
\node(E) at (6,0) {$h-1$};
\node(F) at (7.5,0) {$h$\nospacepunct{,}};
\node(G) at (-1.5,1) {$ $};
\node(H) at (-1.5,-1) {$ $};

\draw[->] (A) -- node[above] {$\alpha_1$} (B);
\draw[->] (B) -- node[above] {$\alpha_2$} (C);
\draw[->] (C) -- node[above] {$\alpha_3$} (D);
\draw[->] (D) -- node[above] {$\alpha_{h-2}$} (E);
\draw[->] (E) -- node[above] {$\alpha_{h-1}$} (F);

\draw[->] (G) -- (A);
\draw[->] (H) -- (A);

\draw[dotted] (-0.8, 0.5) -- (-0.8, -0.5);

\draw[pattern=northwest, hatch distance=15pt, hatch thickness = 0.3pt] (-2.12,0) ellipse (1cm and 1.3cm);
\node (Y) at (-2.12,0) {$\mathbf{Q'_A}$};

\end{tikzpicture}\]
and no path of the form $\alpha_i\cdots\alpha_{j}$ with $1\leq i \leq j \leq h-1$ is in $\cR_A$. In this case, we say that the indecomposable projective module $P=P(1)$ corresponding to the vertex $1$ is a \emph{left abutment}\index[definitions]{left abutment} and we call $\height(P)=h$\index[symbols]{height(P), height(I)@$\height(P)$, $\height(I)$} the \emph{height of $P$}\index[definitions]{height of a left abutment}. We say that a collection $\{P_1,\dots,P_k\}$ of left abutments is \emph{independent}\index[definitions]{independent collection of left abutments} if $\supp(P_i)\cap \supp(P_j)=\varnothing$ for $i\neq j$. The notions of a \emph{right abutment}\index[definitions]{right abutment}, the \emph{height}\index[definitions]{height of a right abutment} of a right abutment and \emph{independent}\index[definitions]{independent collection of right abutments} collection of right abutments are defined dually.  
\end{definition}

\begin{remark} Let $A=\K Q_A/\cR_A$ be a rep\-re\-sen\-ta\-tion-di\-rect\-ed algebra. In particular $Q_A$ is acyclic and so without loss of generality we may assume that $A$ is as in Definition \ref{def:abutments}. Then we have the following observations.
\begin{enumerate}[label=(\alph*)]
    \item The indecomposable projective $A$-module corresponding to the vertex $k$ is a left abutment of height $h-k+1$ for all $k\in\{1,\dots,h\}$. In particular, left abutments of height $1$ are exactly the simple projective $A$-modules.
    
    \item We set
    \begin{align*}
        \cP_A=\cP\coloneqq\add(A),\quad \ABP_A=\ABP\coloneqq\add\left\{ P\in \cP \mid \text{$P$ is a left abutment of $A$}\right\}, \\ 
        \cI_A=\cI\coloneqq\add(D(A)),\quad \IAB_A=\IAB\coloneqq\add\left\{I\in \cI \mid \text{$I$ is a right abutment of $A$}\right\}.
    \end{align*}\index[symbols]{P b@$\cP$}\index[symbols]{I aa@$\cI$} \index[symbols]{P c@$\ABP$}\index[symbols]{I ab@$\IAB$}
    Define the preorder $\leq$ on $\ABP$ by
    \[\text{$P\leq W$\index[symbols]{P f@$P\leq W$} if and only if $\supp(P) \subseteq \supp(W)$}\]
    for left abutments $P,W\in\ABP$. We say that a left abutment $W\in\ABP$ is \emph{maximal}\index[definitions]{maximal left abutment} if $W\leq W'$ for some $W'\in\ABP$ implies $W\isom W'$. Similarly we define $\leq$ and \emph{maximal}\index[definitions]{maximal right abutment} elements on $\IAB$\index[symbols]{I p@$I\leq J$}. We set
    \begin{align*}
        \MABP_A =\MABP\coloneqq\add\left\{ P\in \ABP \mid \text{$P$ is maximal}\right\}, \quad \sMABPind = \MABP\cap \ind(A),\\
        \MIAB_A =\MIAB\coloneqq\add\left\{ I\in \IAB \mid \text{$I$ is maximal}\right\}, \quad \sMABIind = \MIAB \cap \ind(A).
    \end{align*}\index[symbols]{P e@$\MABP$}\index[symbols]{I mab@$\MIAB$}
    In particular, we have that if $P,W\in\ABP$ and $P\leq W$, then $\{P,W\}$ is not an independent collection of left abutments. It follows that if $A$ admits exactly $m$ maximal left (respectively right) abutments up to isomorphism, then any independent collection of left (respectively right) abutments has at most $m$ elements.

    \item \cite[Definition 2.11]{VAS2} There is a canonical $\K$-algebra epimorphism $\pi_P:A\to \K \overrightarrow{A}_h$\index[symbols]{P g@$\pi_P$, $\pi_I$} given by identifying the full subquiver of $Q_{A}$ containing the vertices $\{1,\dots,h\}$ with the quiver $\overrightarrow{A}_h$. In other words, if $\epsilon_1,\dots,\epsilon_h$ are the primitive idempotents of $A$ corresponding to the vertices $1,\dots,h$, and if $\epsilon=\sum_{i=1}^h\epsilon_i$, then $\pi_P$ is the quotient map $\pi_P:A\to A/\langle 1-\epsilon \rangle$. We call $\pi_P$ the \emph{footing at $P$}\index[definitions]{footing of a left abutment} and similarly we define the \emph{footing}\index[definitions]{footing of a right abutment} $\pi_I$\index[symbols]{P g@$\pi_P$, $\pi_I$} for a right abutment $I$.
\end{enumerate}
\end{remark}

Next we recall the notion of gluing along abutments. Although gluing can be defined more generally, when the algebras are given by quivers with relations the following definition is enough.

\begin{definition}\cite[Lemma 2.21]{VAS2}\label{def:gluing quivers}
Let $A=\K Q_A/\cR_A$ be a rep\-re\-sen\-ta\-tion-di\-rect\-ed algebra given by a quiver with relations of the form
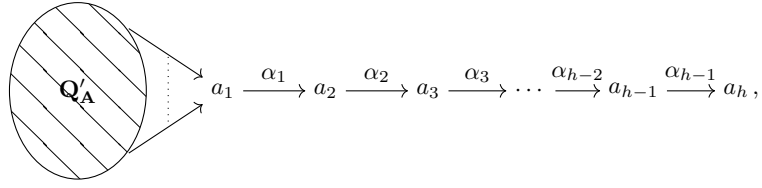
\[\begin{tikzpicture}[scale=0.9, transform shape]
\node(A) at (0,0) {$a_1$};
\node(B) at (1.5,0) {$a_2$};
\node(C) at (3,0) {$a_3$};
\node(D) at (4.5,0) {$\cdots$};
\node(E) at (6,0) {$a_{h-1}$};
\node(F) at (7.5,0) {$a_h$\nospacepunct{,}};
\node(G) at (-1.5,1) {$ $};
\node(H) at (-1.5,-1) {$ $};

\draw[->] (A) -- node[above] {$\alpha_1$} (B);
\draw[->] (B) -- node[above] {$\alpha_2$} (C);
\draw[->] (C) -- node[above] {$\alpha_3$} (D);
\draw[->] (D) -- node[above] {$\alpha_{h-2}$} (E);
\draw[->] (E) -- node[above] {$\alpha_{h-1}$} (F);

\draw[->] (G) -- (A);
\draw[->] (H) -- (A);

\draw[dotted] (-0.8, 0.5) -- (-0.8, -0.5);

\draw[pattern=northwest, hatch distance=15pt, hatch thickness = 0.3pt] (-2.12,0) ellipse (1cm and 1.3cm);
\node (Y) at (-2.12,0) {$\mathbf{Q'_A}$};
\end{tikzpicture}\]
where no path of the form $\alpha_i\cdots\alpha_{j}$ with $1\leq i \leq j \leq h-1$ is in $\cR_A$, and $B=\K Q_B/\cR_B$ be a rep\-re\-sen\-ta\-tion-di\-rect\-ed algebra given by a quiver with relations of the form
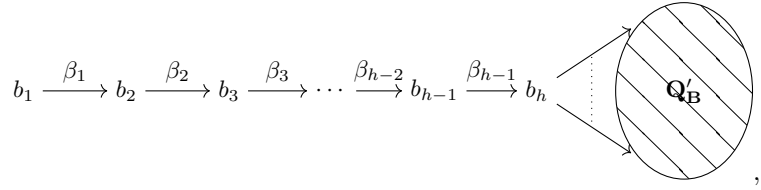
\[\begin{tikzpicture}[scale=0.9, transform shape]
\node(A) at (0,0) {$b_1$};
\node(B) at (1.5,0) {$b_2$};
\node(C) at (3,0) {$b_3$};
\node(D) at (4.5,0) {$\cdots$};
\node(E) at (6,0) {$b_{h-1}$};
\node(F) at (7.5,0) {$b_h$};
\node(G) at (9,1) {$ $};
\node(H) at (9,-1) {$ $};

\draw[->] (A) -- node[above] {$\beta_1$} (B);
\draw[->] (B) -- node[above] {$\beta_2$} (C);
\draw[->] (C) -- node[above] {$\beta_3$} (D);
\draw[->] (D) -- node[above] {$\beta_{h-2}$} (E);
\draw[->] (E) -- node[above] {$\beta_{h-1}$} (F);

\draw[->] (F) -- (G);
\draw[->] (F) -- (H);

\draw[dotted] (8.3, 0.5) -- (8.3, -0.5);

\draw[pattern=northwest, hatch distance=15pt, hatch thickness = 0.3pt] (9.65,0) ellipse (1cm and 1.3cm);
\node (Z) at (9.65,0) {$\mathbf{Q'_B}$};
\end{tikzpicture},\]
where no path of the form $\beta_i\cdots\beta_{j}$ with $1\leq i \leq j \leq h-1$ is in $\cR_B$. Then $P=P_A(1)$ is a left abutment of height $h$ and $I=I_B(h)$ is a right abutment of height $h$. The \emph{gluing of $A$ and $B$ along $P$ and $I$}\index[definitions]{gluing of algebras}, denoted by $\La\coloneqq B \glue[P][I] A$,\index[symbols]{B glue A@$B \glue[P][I] A$} is defined to be the algebra $\La=\K Q_{\La}/\cR_{\La}$ given by a quiver with relations where $Q_\La$ is the quiver 
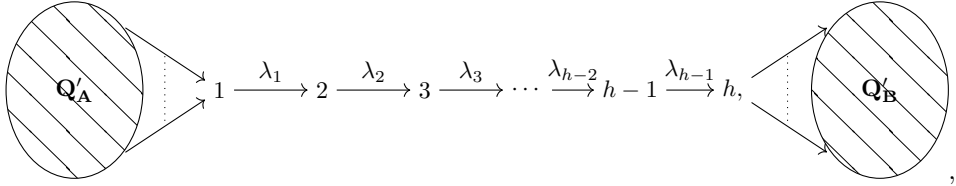
\[\begin{tikzpicture}[scale=0.9, transform shape]
\node(A) at (0,0) {$1$};
\node(B) at (1.5,0) {$2$};
\node(C) at (3,0) {$3$};
\node(D) at (4.5,0) {$\cdots$};
\node(E) at (6,0) {$h-1$};
\node(F) at (7.5,0) {$h$,};
\node(G) at (-1.5,1) {$ $};
\node(H) at (-1.5,-1) {$ $};
\node(G2) at (9,1) {$ $};
\node(H2) at (9,-1) {$ $};

\draw[->] (A) -- node[above] {$\lambda_1$} (B);
\draw[->] (B) -- node[above] {$\lambda_2$} (C);
\draw[->] (C) -- node[above] {$\lambda_3$} (D);
\draw[->] (D) -- node[above] {$\lambda_{h-2}$} (E);
\draw[->] (E) -- node[above] {$\lambda_{h-1}$} (F);

\draw[->] (G) -- (A);
\draw[->] (H) -- (A);
\draw[->] (F) -- (G2);
\draw[->] (F) -- (H2);

\draw[dotted] (-0.8, 0.5) -- (-0.8, -0.5);
\draw[dotted] (8.3, 0.5) -- (8.3, -0.5);

\draw[pattern=northwest, hatch distance=15pt, hatch thickness = 0.3pt] (9.65,0) ellipse (1cm and 1.3cm);
\draw[pattern=northwest, hatch distance=15pt, hatch thickness = 0.3pt] (-2.12,0) ellipse (1cm and 1.3cm);

\node (Y) at (-2.12,0) {$\mathbf{Q'_A}$};
\node (Z) at (9.65,0) {$\mathbf{Q'_B}$};
\end{tikzpicture},\]
and $\cR_\La$ is generated by all elements in $\cR_A$ and $\cR_B$ as well as all paths starting from $\mathbf{Q'_A}$ and ending in $\mathbf{Q'_B}$, under the identifications $a_i=b_i=i$ and $\alpha_i=\beta_i=\lambda_i$.
\end{definition}

More precisely, we can express the quiver $Q_{\La}$ using the quivers $Q_A$ and $Q_B$ as follows:
\begin{align*}
    \left(Q_{\La}\right)_0 &= \left(\left(Q_B\right)_0 \coprod \left(Q_A\right)_0 \right)/ \sim^{0}, \text{ and}\\
    \left(Q_{\La}\right)_1 &= \left(\left(Q_B\right)_1 \coprod \left(Q_A\right)_1\right) / \sim^{1},
\end{align*}
where $\sim^{0}$\index[symbols]{((0@$\sim^{0}$, $\sim^{1}$} and $\sim^{1}$\index[symbols]{((0@$\sim^{0}$, $\sim^{1}$} are the equivalence relations generated by 
\begin{align*}
    a_i \sim^{0} b_i &\text{ for $1\leq i \leq h$, and} \\
    \alpha_i \sim^{1} \beta_i &\text{ for $1\leq i \leq h-1$,}
\end{align*}
or equivalently
\begin{align}
    \label{eq:S0} 
    a\sim^{0} b &\text{ for all $a\in\supp(P)$, $b\in\supp(I)$ with $\height(P_A(a))+\height(I_B(b)) = h + 1$, and} \\
    \label{eq:S1}
    \alpha\sim^{1}\beta &\text{ for all $\alpha\in \left(Q_A\right)_1$, $\beta\in\left(Q_B\right)_1$ with $s(\alpha)\sim^{0}s(\beta)$.}
\end{align}
In particular, we have that $Q_A$ and $Q_B$ are full subquivers of $Q_{\La}$. Using this description, we set $\epsilon_A\coloneqq \sum_{i\in (Q_A)_0}\epsilon_i$ for the sum of all primitive idempotents of $\La$ corresponding to the vertices of $Q_A$, and similarly we set $\epsilon_B\coloneqq \sum_{i\in (Q_B)_0}\epsilon_i$. With this notation, it follows that
\begin{align}\label{eq:ideal of gluing}
\cR_{\La} = \langle \cR_{A}+\cR_{B}+\left(1_{\La}-\epsilon_B\right)\K Q_{\La} \left(1_{\La}-\epsilon_A\right)\rangle.
\end{align}

In the following remark we collect the basic properties of gluing that we need.

\begin{remark}\label{rem:gluing remarks} Throughout this remark let $\La=B \glue[P][I] A$.
\begin{enumerate}[label=(\alph*)]
    \item \cite[Definition 2.19]{VAS2} The algebra $\La$ is the pullback of the diagram $A\overset{\pi_P}{\longrightarrow} \K \overrightarrow{A}_h \overset{\pi_I}{\longleftarrow} B$. That is we have a pullback diagram
    \begin{equation}\label{eq:pullback gluing}
        \begin{tikzcd}
            \La \arrow[d, swap, "\prescript{\La}{A}{\pi}"] \arrow[r, "\prescript{\La}{B}{\pi}"] & B \arrow[d, "\pi_{I}"] \\
            A \arrow[r, swap, "\pi_{P}"] & \K \overrightarrow{A}_h\nospacepunct{,}
        \end{tikzcd}
    \end{equation}
    where all arrows are epimorphisms and $\prescript{\La}{A}{\pi}$ can be seen as identifying the full subquiver of $Q_{\La}$ corresponding to $Q_A$ with $Q_A$ while $\prescript{\La}{B}{\pi}$ can be seen as identifying the full subquiver of $Q_{\La}$ corresponding to $Q_B$ with $Q_B$. In other words $\prescript{\La}{A}{\pi}$ is the quotient map $\prescript{\La}{A}{\pi}:\La\to \La/\langle 1-\epsilon_A\rangle=A$\index[symbols]{pi Lambda to A, pi Lambda to B@$\prescript{\La}{A}{\pi}:\La\to A$, $\prescript{\La}{B}{\pi}:\La\to B$} and similarly for $\prescript{\La}{B}{\pi}$\index[symbols]{pi Lambda to A, pi Lambda to B@$\prescript{\La}{A}{\pi}:\La\to A$, $\prescript{\La}{B}{\pi}:\La\to B$}. 
    
    \item By the pullback diagram (\ref{eq:pullback gluing}), we have the following functors defined by $\prescript{\La}{A}{\pi}$ and $\prescript{\La}{B}{\pi}$ on the corresponding module categories.
    \[\begin{tikzpicture}[baseline={(current bounding box.center)}, scale=0.9, transform shape]
    \node (A) at (0,0) {$\m\La$};
    \node (B) at (3,0) {$\m A$};

    \draw[->] (B) -- node[above] {$\prescript{\La}{A}{\pi}_{\ast}$} (A); 
    \draw[->] (A) to [out=30,in=150]  node[auto]{$\prescript{\La}{A}{\pi}^{\ast}$} (B);
    \draw[->] (A) to [out=-30,in=-150]  node[below]{$\prescript{\La}{A}{\pi}^{!}$} (B);
    \end{tikzpicture}\;\; \text{ and } \;\; 
    \begin{tikzpicture}[baseline={(current bounding box.center)}, scale=0.9, transform shape]
    \node (A) at (0,0) {$\m\La$};
    \node (B) at (3,0) {$\m B$\nospacepunct{.}};
    
    \draw[->] (B) -- node[above] {$\prescript{\La}{B}{\pi}_{\ast}$} (A); 
    \draw[->] (A) to [out=30,in=150]  node[auto]{$\prescript{\La}{B}{\pi}^{\ast}$} (B);
    \draw[->] (A) to [out=-30,in=-150]  node[below]{$\prescript{\La}{B}{\pi}^{!}$} (B);
    \end{tikzpicture}\]
    In particular, it is easy to see that if $M$ is a representation of $A$, then $\prescript{\La}{A}{\pi}_{\ast}(M)$ is the representation of $\La$ given by extending $M$ to $Q_{\La}$ by putting $0$ on all vertices and arrows of $Q_{\La}$ which lie outside $Q_A$. Similarly for representations of $B$ viewed as representations of $\La$ through $\prescript{\La}{B}{\pi}_{\ast}$. With this fact and Definition \ref{def:gluing quivers} we can compute the image of indecomposable projective $A$-modules and $B$-modules under the functors $\prescript{\La}{A}{\pi}_{\ast}$ and $\prescript{\La}{B}{\pi}_{\ast}$. We see that $\prescript{\La}{B}{\pi}_{\ast}(\epsilon_i B)\isom \epsilon_i \La$ for all $i\in Q_B$ and $\prescript{\La}{A}{\pi}_{\ast}(\epsilon_i A)\isom \epsilon_i \La$ for all $i\in Q_A\setminus\{1,\dots,h\}$; a dual statement holds for indecomposable injective modules. Moreover, we also have 
    \[\prescript{\La}{A}{\pi}^{\ast}\left(\epsilon_i\La\right) = \epsilon_i\La\otimes_{\La} \La/\langle 1 - \epsilon_A \rangle  \isom \epsilon_i \La/ \langle 1- \epsilon_A \rangle =\prescript{\La}{A}{\pi}(\epsilon_i\La)= \begin{cases} \epsilon_i A, &\mbox{if $i\in \left(Q_A\right)_0$,} \\ 0, &\mbox{otherwise,}\end{cases}\]
    and similarly
    \[\prescript{\La}{B}{\pi}^{\ast}\left(\epsilon_i\La\right) \isom \begin{cases} \epsilon_i B, &\mbox{if $i\in \left(Q_B\right)_0$,} \\ 0, &\mbox{otherwise.}\end{cases}\] 
    
    \item \cite[Corollary 2.44]{VAS2} Using (b) and Definition \ref{def:abutments} we can easily find the (maximal) left abutments of $\La$: they are, up to isomorphism, either of the form $\prescript{\La}{A}{\pi}_{\ast}(W_A)$ for a (maximal) left abutment $W_A$ of $A$ or of the form $\prescript{\La}{B}{\pi}_{\ast}(W_B)$ for a (maximal) left abutment $W_B$ of $B$. More precisely, every (maximal) left abutment of $B$ gives rise to a unique (maximal) left abutment of $\La$ of the same height and every (maximal) left abutment of $A$, except for those that are supported on a subset of $\{1,\dots,h\}$, gives rise to a unique (maximal) left abutment of $\La$ of the same height. In particular, if $W_A$ is a (maximal) left abutment of $A$ which is independent from $P$, then $\prescript{\La}{A}{\pi}_{\ast}(W_A)$ is a (maximal) left abutment of $\La$. A dual result holds for right abutments.
    
    \item \cite[Example 2.20]{VAS2} Let $A$ and $B$ be as in Definition \ref{def:gluing quivers}. Then $Q'_A=\varnothing$ if and only if $A=\K \overrightarrow{A}_h$ and in this case it readily follows that $B\glue[P][I] \K \overrightarrow{A}_h=B$; similarly $Q'_B=\varnothing$ if and only if $B=\K \overrightarrow{A}_h$ and then $\K \overrightarrow{A}_h\glue[P][I] A=A$. We call a gluing of either of these forms a \emph{trivial gluing}\index[definitions]{trivial gluing}.
    
    \item \cite[Lemma 2.32]{VAS2} Let $i,j\in \left(Q_{\La}\right)_0$. Assume that $\epsilon_j\La\epsilon_i\neq 0$. Then it follows that there is a nonzero path from $j$ to $i$ in $\La$. By the shape of $Q_{\La}$, it follows that we cannot have $j\in \left(Q_B'\right)_0$ and $i\in\left(Q_A'\right)_0$. Moreover, by the relations $R_{\La}$ it follows that we cannot have $j\in\left(Q_A'\right)_0$ and $i\in\left(Q_B'\right)_0$. Hence we either have $i,j\in \left(Q_A\right)_0$ and $\epsilon_j \La\epsilon_i \isom \epsilon_j A \epsilon_i$ or $i,j\in\left(Q_B\right)_0$ and $\epsilon_j \La \epsilon_i \isom \epsilon_j B \epsilon_i$.
\end{enumerate}
\end{remark}

The representation theory of $\La=B\glue[P][I] A$ can be understood through the representation theory of $A$ and $B$. To see how, we first give a representation theoretic characterization of abutments.

\begin{proposition}\cite[Proposition 2.13]{VAS2}\label{prop:abutments in AR-quiver}
Let $A$ be a rep\-re\-sen\-ta\-tion-di\-rect\-ed algebra and let $P=P_1$ be an indecomposable projective $A$-module. Then $P$ is a left abutment of height $h$ if and only if there exist indecomposable projective modules $P_2,\dots,P_h$ such that
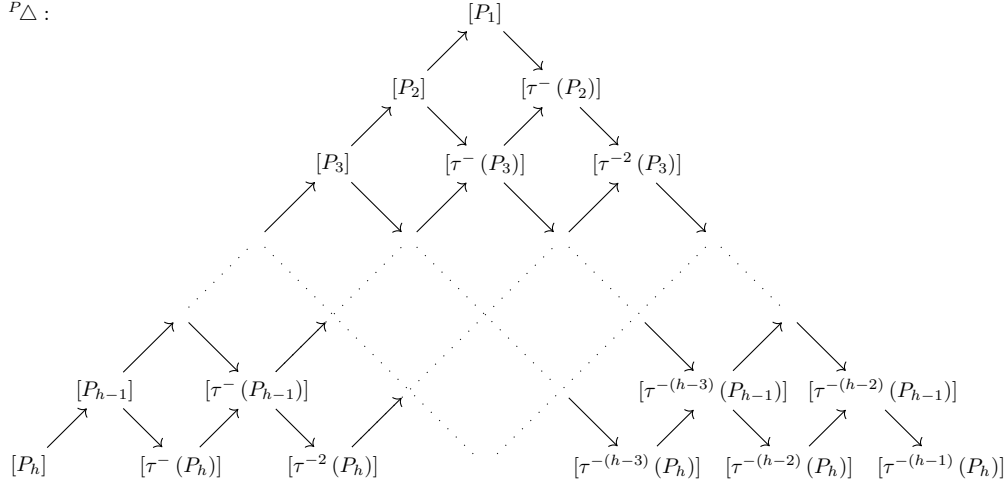
\[\begin{tikzpicture}[scale=0.8, transform shape]
\node (00) at (0,7.5) {$\PD:$};

\node (11) at (0,0) {$[P_h]$};
\node (21) at (2.5,0) {$[\tau^-\left(P_h\right)]$};
\node (31) at (5,0) {$[\tau^{-2}\left(P_h\right)]$};
\node (41) at (7.5,0) { };
\node (51) at (10,0) {$[\tau^{-(h-3)}\left(P_h\right)]$};
\node (61) at (12.5,0) {$[\tau^{-(h-2)}\left(P_h\right)]$};
\node (71) at (15,0) {$[\tau^{-(h-1)}\left(P_h\right)]$};

\node (12) at (1.25,1.25) {$[P_{h-1}]$};
\node (22) at (3.75,1.25) {$[\tau^-\left(P_{h-1}\right)]$};
\node (32) at (6.25,1.25) { };
\node (42) at (8.75,1.25) { };
\node (52) at (11.25,1.25) {$[\tau^{-(h-3)}\left(P_{h-1}\right)]$};
\node (62) at (13.75,1.25) {\;\;\;\;\;\;$[\tau^{-(h-2)}\left(P_{h-1}\right)]$};

\node (13) at (2.5,2.5) { };
\node (23) at (5,2.5) { };
\node (33) at (7.5,2.5) { };
\node (43) at (10,2.5) { };
\node (53) at (12.5,2.5) { };

\node (14) at (3.75,3.75) { };
\node (24) at (6.25,3.75) { };
\node (34) at (8.75,3.75) { };
\node (44) at (11.25,3.75) { };

\node (15) at (5,5) {$[P_3]$};
\node (25) at (7.5,5) {$[\tau^-\left(P_3\right)]$};
\node (35) at (10,5) {$[\tau^{-2}\left(P_3\right)]$};

\node (16) at (6.25,6.25) {$[P_2]$};
\node (26) at (8.75,6.25) {$[\tau^-\left(P_2\right)]$};

\node (17) at (7.5,7.5) {$[P_1]$};

\draw[->] (11) -- (12);
\draw[->] (12) -- (21);
\draw[->] (21) -- (22);
\draw[->] (22) -- (31);

\draw[->] (51) -- (52);
\draw[->] (52) -- (61);
\draw[->] (61) -- (62);
\draw[->] (62) -- (71);

\draw[->] (12) -- (13);
\draw[->] (13) -- (22);
\draw[->] (22) -- (23);

\draw[->] (14) -- (15);
\draw[->] (15) -- (16);
\draw[->] (16) -- (17);

\draw[->] (15) -- (24);
\draw[->] (24) -- (25);
\draw[->] (25) -- (34);
\draw[->] (34) -- (35);
\draw[->] (35) -- (44);

\draw[->] (16) -- (25);
\draw[->] (17) -- (26);
\draw[->] (26) -- (35);
\draw[->] (25) -- (26);

\draw[->] (31) -- (32);
\draw[->] (42) -- (51);
\draw[->] (43) -- (52);
\draw[->] (53) -- (62);

\draw[->] (52) -- (53);

\draw[loosely dotted] (44) -- (41);
\draw[loosely dotted] (14) -- (41);

\draw[loosely dotted] (34) -- (43);
\draw[loosely dotted] (44) -- (53);

\draw[loosely dotted] (24) -- (42);
\draw[loosely dotted] (13) -- (14);
\draw[loosely dotted] (22) -- (24);
\draw[loosely dotted] (31) -- (34);
\end{tikzpicture}\]
is a full subquiver of $\Gamma(A)$, there are no other arrows in $\Gamma(A)$ going into $\PD$\index[symbols]{P a@$\PD$} and, moreover, all northeast arrows are monomorphisms, all southeast arrows are epimorphisms and all modules in the same row have the same dimension. In particular, $\tau^{-i}\left(P_h\right)$ is the simple top of $P_{h-i}$ for $1\leq i \leq h-1$. We call $\PD$ \emph{the foundation of $P$}\index[definitions]{foundation of a left abutment}.
\end{proposition}

A dual statement to Proposition \ref{prop:abutments in AR-quiver} holds for right abutments. If $I$ is a right abutment, we denote by $\DI$\index[symbols]{I  D@$\DI$} the \emph{foundation of $I$}\index[definitions]{foundation of a right abutment}. Notice that we can identify the quivers $\PD$ and $\DI$ with the Aus\-lan\-der--Rei\-ten quiver $\Gamma(K\overrightarrow{A}_h)$. If $P$ is a left abutment of $A$ and $I$ is a right abutment of $A$, we also set
\[\cF_P \coloneqq \add\left(\{ X \in \m A \mid \text{$X$ indecomposable and $[X]\in \PD$}\}\right), \text{ and } \]\index[symbols]{F P@$\cF_P$}
\[\cG_I \coloneqq \add\left(\{ X \in \m A \mid \text{$X$ indecomposable and $[X]\in \DI$}\}\}\right).\phantom{ and }\]\index[symbols]{GaaI@$\cG_I$}
It readily follows that a collection of left abutments $\{P_1,\dots,P_k\}$ is independent if and only if $\cF_{P_i}\cap \cF_{P_j}=\{0\}$ for $i\neq j$ and a collection of right abutments $\{I_1,\dots,I_k\}$ is independent if and only if $\cG_{I_i}\cap \cG_{I_j}=\{0\}$ for $i\neq j$. 

Using Proposition \ref{prop:abutments in AR-quiver}  we can easily describe the Aus\-lan\-der--Rei\-ten quiver of the gluing $\La$ of two algebras $A$ and $B$.

\begin{proposition}\cite[Corollary 2.41]{VAS2}\label{prop:AR glued}
Let $\La=B\glue[P][I]A$ be the gluing of two rep\-re\-sen\-ta\-tion-di\-rect\-ed algebras $A$ and $B$ along the abutments $P$ and $I$. For the Aus\-lan\-der--Rei\-ten quiver of $\La$ we have, as quivers, $\Gamma(\La)=\Gamma(B)\coprod_{\triangle} \Gamma(A)$, where the righthand side denotes the amalgamated sum under the identification $\triangle=\PD=\DI$. Moreover, in this identification, the vertex $[M]$ in $\Gamma(A)$ corresponds to the vertex $[\prescript{\La}{A}{\pi}_{\ast}(M)]$ in $\Gamma(\La)$ and the vertex $[N]$ in $\Gamma(B)$ corresponds to the vertex $[\prescript{\La}{B}{\pi}_{\ast}(N)]$ in $\Gamma(\La)$.
\end{proposition}

Hence to draw the Aus\-lan\-der--Rei\-ten quiver of $\La$, one draws the Aus\-lan\-der--Rei\-ten quiver of $B$ where the vertex $[N]$ is replaced by $[\prescript{\La}{B}{\pi}_{\ast}(N)]$, then one extends this quiver at $\DI=\PD$ by drawing the Aus\-lan\-der--Rei\-ten quiver of $A$ where the vertex $[M]$ is replaced by $[\prescript{\La}{A}{\pi}_{\ast}(M)]$. It follows immediately that $\La$ is also a rep\-re\-sen\-ta\-tion-di\-rect\-ed algebra.

We illustrate the above notions by an example.

\begin{example}
Let $B$ be given by the quiver with relations
\[\begin{tikzpicture}[scale=0.9, transform shape]
\node (1) at (0,1) {$1$};
\node (2) at (1,1) {$2$};
\node (3) at (2,1) {$3$};
\node (4) at (3,1) {$4$};
\node (5) at (4,0.5) {$5$};
\node (6) at (5,0.5) {$6$};
\node (7) at (6,0.5) {$7$\nospacepunct{.}};
\node (1') at (1,0) {$1'$};
\node (2') at (2,0) {$2'$};
\node (3') at (3,0) {$3'$};

\draw[->] (1) to (2);
\draw[->] (2) to (3);
\draw[->] (3) to (4);
\draw[->] (4) to (5);
\draw[->] (5) to (6);
\draw[->] (6) to (7);
\draw[->] (1') to (2');
\draw[->] (2') to (3');
\draw[->] (3') to (5);

\draw[dotted] (1.5,1) to [out=60,in=120] (2.5,1);
\draw[dotted] (2.5,1) to [out=-60,in=190] (3.5,0.75);
\draw[dotted] (2.5,0) to [out=60,in=170] (3.5,0.25);
\draw[dotted] (3.5,0.75) to [out=30,in=120] (5.5,0.5);
\draw[dotted] (3.5,0.25) to [out=190,in=-120] (4.5,0.5);
\end{tikzpicture}\]
Then the left abutments are $P_B(5)\geq P_B(6) \geq P_B(7)$ with heights $3$, $2$ and $1$ respectively. The right abutments are $I_B(3)\geq I_B(2) \geq I_B(1)$ with heights $3$, $2$ and $1$ respectively and $I_B(3')\geq I_B(2')\geq I_B(1')$ with heights $3$, $2$ and $1$ respectively. The only maximal left abutment is $P_B(5)$ and the maximal right abutments are $I_B(3)$ and $I_B(3')$. Hence there is no independent collection of left abutments with more than one element, while the collections $\{I_B(3),I_B(3')\}$ and $\{I_B(1),I_B(2')\}$ are examples of independent collections of right abutments. Looking at the Aus\-lan\-der--Rei\-ten  quiver $\Gamma(B)$ of $B$, we can make the same observations via Proposition \ref{prop:abutments in AR-quiver}:
\[\begin{tikzpicture}[scale=0.9, transform shape]
\tikzstyle{nct}=[shape= rectangle, minimum width=6pt, minimum height=7.5, draw, inner sep=0pt]
\tikzstyle{nct2}=[circle, minimum width=6pt, draw, inner sep=0pt]
\tikzstyle{nct3}=[circle, minimum width=6pt, draw=white, inner sep=0pt, scale=0.9]

\node[nct3] (A) at (0,0) {$\qthree{}[7][]$};
\node[nct3] (B) at (0.7,0.7) {$\qthree{6}[7]$};
\node[nct3] (C) at (1.4,1.4) {$\qthree{5}[6][7]$};
\node[nct3] (D) at (1.4,0) {$\qthree{}[6][]$};
\node[nct3] (E) at (2.1,0.7) {$\qthree{5}[6]$};
\node[nct3] (F) at (2.8,1.4) {$\qthree{4}[5][6]$};
\node[nct3] (G) at (2.8,0) {$\qthree{}[5][]$};
\node[nct3] (H) at (3.5,0.7) {$\qthree{4}[5]$};
\node[nct3] (I) at (3.5,-0.7) {$\qthree{3'}[5]$};
\node[nct3] (J) at (4.2,0) {$\begin{smallmatrix} 4 && 3' \\ & 5 &
\end{smallmatrix}$};
\node[nct3] (K) at (4.9,0.7) {$\qthree{}[3'][]$};
\node[nct3] (L) at (4.9,-0.7) {$\qthree{}[4][]$};
\node[nct3] (M) at (5.6,1.4) {$\qthree{2'}[3']$};
\node[nct3] (N) at (5.6,-1.4) {$\qthree{3}[4]$};
\node[nct3] (O) at (6.3,2.1) {$\qthree{1'}[2'][3']$};
\node[nct3] (P) at (6.3,0.7) {$\qthree{}[2'][]$};
\node[nct3] (Q) at (6.3,-0.7) {$\qthree{}[3][]$};
\node[nct3] (R) at (7,1.4) {$\qthree{1'}[2']$};
\node[nct3] (S) at (7,-1.4) {$\qthree{2}[3]$};
\node[nct3] (T) at (7.7,0.7) {$\qthree{}[1'][]$};
\node[nct3] (U) at (7.7,-0.7) {$\qthree{}[2][]$};
\node[nct3] (V) at (7.7,-2.1) {$\qthree{1}[2][3]$};
\node[nct3] (W) at (8.4,-1.4) {$\qthree{1}[2]$};
\node[nct3] (X) at (9.1,-0.7) {$\qthree{}[1][]$\nospacepunct{.}};

\draw[->] (A) to (B);
\draw[->] (B) to (C);
\draw[->] (B) to (D);
\draw[->] (C) to (E);
\draw[->] (D) to (E);
\draw[->] (E) to (F);
\draw[->] (E) to (G);
\draw[->] (F) to (H);
\draw[->] (G) to (H);
\draw[->] (G) to (I);
\draw[->] (H) to (J);
\draw[->] (I) to (J);
\draw[->] (J) to (K);
\draw[->] (J) to (L);
\draw[->] (K) to (M);
\draw[->] (L) to (N);
\draw[->] (M) to (O);
\draw[->] (M) to (P);
\draw[->] (N) to (Q);
\draw[->] (O) to (R);
\draw[->] (P) to (R);
\draw[->] (Q) to (S);
\draw[->] (R) to (T);
\draw[->] (S) to (U);
\draw[->] (S) to (V);
\draw[->] (U) to (W);
\draw[->] (V) to (W);
\draw[->] (W) to (X);

\draw[loosely dotted] (A.east) -- (D);
\draw[loosely dotted] (B.east) -- (E);
\draw[loosely dotted] (D.east) -- (G);
\draw[loosely dotted] (E.east) -- (H);
\draw[loosely dotted] (G.east) -- (J);
\draw[loosely dotted] (H.east) -- (K);
\draw[loosely dotted] (I.east) -- (L);
\draw[loosely dotted] (K.east) -- (P);
\draw[loosely dotted] (M.east) -- (R);
\draw[loosely dotted] (L.east) -- (Q);
\draw[loosely dotted] (Q.east) -- (U);
\draw[loosely dotted] (P.east) -- (T);
\draw[loosely dotted] (S.east) -- (W);
\draw[loosely dotted] (U.east) -- (X);
\end{tikzpicture}\]

Let $A$ be given by the quiver with relations 
\[\begin{tikzpicture}[scale=0.9, transform shape]
\node (0) at (1,0) {$0$};
\node (1) at (2,0) {$1$};
\node (2) at (3,0) {$2$};
\node (3) at (4,0) {$3$\nospacepunct{.}};

\draw[->] (0) to (1);
\draw[->] (1) to (2);
\draw[->] (2) to (3);

\draw[dotted] (1.5,0) to [out=60,in=120] (3.5,0);
\end{tikzpicture}\]
The Aus\-lan\-der--Rei\-ten quiver $\Gamma(A)$ of $A$ is
\[\begin{tikzpicture}[scale=0.9, transform shape]
\tikzstyle{nct3}=[circle, minimum width=6pt, draw=white, inner sep=0pt, scale=0.9]

\node[nct3] (A) at (0,0) {$\qthree{}[3][]$};
\node[nct3] (B) at (0.7,0.7) {$\qthree{2}[3]$};
\node[nct3] (C) at (1.4,1.4) {$\qthree{1}[2][3]$};
\node[nct3] (D) at (1.4,0) {$\qthree{}[2][]$};
\node[nct3] (E) at (2.1,0.7) {$\qthree{1}[2]$};
\node[nct3] (F) at (2.8,1.4) {$\qthree{0}[1][2]$};
\node[nct3] (G) at (2.8,0) {$\qthree{}[1][]$};
\node[nct3] (H) at (3.5,0.7) {$\qthree{0}[1]$};
\node[nct3] (J) at (4.2,0) {$\qthree{}[0][]$\nospacepunct{.}};

\draw[->] (A) to (B);
\draw[->] (B) to (C);
\draw[->] (B) to (D);
\draw[->] (C) to (E);
\draw[->] (D) to (E);
\draw[->] (E) to (F);
\draw[->] (E) to (G);
\draw[->] (F) to (H);
\draw[->] (G) to (H);
\draw[->] (H) to (J);

\draw[loosely dotted] (A.east) -- (D);
\draw[loosely dotted] (B.east) -- (E);
\draw[loosely dotted] (D.east) -- (G);
\draw[loosely dotted] (E.east) -- (H);
\draw[loosely dotted] (G.east) -- (J);
\end{tikzpicture}\]

Then $P=P_A(1)$ is a left abutment of height $3$. By setting $I=I_B(3)$ we have that the gluing $\La = B \glue[P][I] A$ is defined and $\La$ is given by the quiver with relations
\[\begin{tikzpicture}[scale=0.9, transform shape]
\node (1) at (0,1) {$1$};
\node (2) at (1,1) {$2$};
\node (3) at (2,1) {$3$};
\node (4) at (3,1) {$4$};
\node (5) at (4,0.5) {$5$};
\node (6) at (5,0.5) {$6$};
\node (7) at (6,0.5) {$7$\nospacepunct{.}};
\node (1') at (1,0) {$1'$};
\node (2') at (2,0) {$2'$};
\node (3') at (3,0) {$3'$};

\node (0) at (-1,1) {$0$};
\draw[->] (0) to (1);

\draw[->] (1) to (2);
\draw[->] (2) to (3);
\draw[->] (3) to (4);
\draw[->] (4) to (5);
\draw[->] (5) to (6);
\draw[->] (6) to (7);
\draw[->] (1') to (2');
\draw[->] (2') to (3');
\draw[->] (3') to (5);

\draw[dotted] (-0.5,1) to [out=60,in=120] (1.5,1);
\draw[dotted] (1.5,1) to [out=60,in=120] (2.5,1);
\draw[dotted] (2.5,1) to [out=-60,in=190] (3.5,0.75);
\draw[dotted] (2.5,0) to [out=60,in=170] (3.5,0.25);
\draw[dotted] (3.5,0.75) to [out=30,in=120] (5.5,0.5);
\draw[dotted] (3.5,0.25) to [out=190,in=-120] (4.5,0.5);
\end{tikzpicture}\]

We see that the left abutments of $B$ remain left abutments of $\La$: $\prescript{\La}{B}{\pi}_{\ast}(P_B(5))\isom P_{\La}(5)$, $\prescript{\La}{B}{\pi}_{\ast}(P_B(6))\isom P_{\La}(6)$ and $\prescript{\La}{B}{\pi}_{\ast}(P_B(7))\isom P_{\La}(7)$ are all left abutments. Left abutments of $A$, however, correspond to the vertices $\{1,2,3\}$ which are glued and so they stop being left abutments: for example $P_A(2)$ is a left abutment of $A$ but $\prescript{\La}{A}{\pi}_{\ast}(P_A(2))$ is not a left abutment of $\La$. 

On the other hand, right abutments of $A$ remain right abutments of $\La$: $\prescript{\La}{A}{\pi}_{\ast}(I_A(0))\isom I_{\La}(0)$, $\prescript{\La}{A}{\pi}_{\ast}(I_A(1))\isom I_{\La}(1)$ and $\prescript{\La}{A}{\pi}_{\ast}(I_A(2))\isom I_{\La}(2)$ are all right abutments. Similarly, the right abutments of $B$ which do not have a support intersecting $\{1,2,3\}$ remain right abutments of $\La$: $\prescript{\La}{B}{\pi}_{\ast}(I_B(1'))\isom I_{\La}(1')$, $\prescript{\La}{B}{\pi}_{\ast}(I_B(2'))\isom I_{\La}(2')$ and $\prescript{\La}{B}{\pi}_{\ast}(I_B(3'))\isom I_{\La}(3')$ are all right abutments, but the right abutments of $B$ with support a subset of $\{1,2,3\}$ are not right abutments of $\La$. 

By Proposition \ref{prop:AR glued}, the Aus\-lan\-der--Rei\-ten quiver $\Gamma(\La)$ of $\La$ is
\[\begin{tikzpicture}[scale=0.9, transform shape]
\tikzstyle{nct3}=[circle, minimum width=6pt, draw=white, inner sep=0pt, scale=0.9]
\node[nct3] (A) at (0,0) {$\qthree{}[7][]$};
\node[nct3] (B) at (0.7,0.7) {$\qthree{6}[7]$};
\node[nct3] (C) at (1.4,1.4) {$\qthree{5}[6][7]$};
\node[nct3] (D) at (1.4,0) {$\qthree{}[6][]$};
\node[nct3] (E) at (2.1,0.7) {$\qthree{5}[6]$};
\node[nct3] (F) at (2.8,1.4) {$\qthree{4}[5][6]$};
\node[nct3] (G) at (2.8,0) {$\qthree{}[5][]$};
\node[nct3] (H) at (3.5,0.7) {$\qthree{4}[5]$};
\node[nct3] (I) at (3.5,-0.7) {$\qthree{3'}[5]$};
\node[nct3] (J) at (4.2,0) {$\begin{smallmatrix} 4 && 3' \\ & 5 &
\end{smallmatrix}$};
\node[nct3] (K) at (4.9,0.7) {$\qthree{}[3'][]$};
\node[nct3] (L) at (4.9,-0.7) {$\qthree{}[4][]$};
\node[nct3] (M) at (5.6,1.4) {$\qthree{2'}[3']$};
\node[nct3] (N) at (5.6,-1.4) {$\qthree{3}[4]$};
\node[nct3] (O) at (6.3,2.1) {$\qthree{1'}[2'][3']$};
\node[nct3] (P) at (6.3,0.7) {$\qthree{}[2'][]$};
\node[nct3] (Q) at (6.3,-0.7) {$\qthree{}[3][]$};
\node[nct3] (R) at (7,1.4) {$\qthree{1'}[2']$};
\node[nct3] (S) at (7,-1.4) {$\qthree{2}[3]$};
\node[nct3] (T) at (7.7,0.7) {$\qthree{}[1'][]$};
\node[nct3] (U) at (7.7,-0.7) {$\qthree{}[2][]$};
\node[nct3] (V) at (7.7,-2.1) {$\qthree{1}[2][3]$};
\node[nct3] (W) at (8.4,-1.4) {$\qthree{1}[2]$};
\node[nct3] (X) at (9.1,-0.7) {$\qthree{}[1][]$};
\node[nct3] (Y) at (9.1,-2.1) {$\qthree{0}[1][2]$};
\node[nct3] (Z) at (9.8,-1.4) {$\qthree{0}[1]$};
\node[nct3] (AA) at (10.4,-0.7) {$\qthree{}[0][]$\nospacepunct{,}};

\draw[->] (A) to (B);
\draw[->] (B) to (C);
\draw[->] (B) to (D);
\draw[->] (C) to (E);
\draw[->] (D) to (E);
\draw[->] (E) to (F);
\draw[->] (E) to (G);
\draw[->] (F) to (H);
\draw[->] (G) to (H);
\draw[->] (G) to (I);
\draw[->] (H) to (J);
\draw[->] (I) to (J);
\draw[->] (J) to (K);
\draw[->] (J) to (L);
\draw[->] (K) to (M);
\draw[->] (L) to (N);
\draw[->] (M) to (O);
\draw[->] (M) to (P);
\draw[->] (N) to (Q);
\draw[->] (O) to (R);
\draw[->] (P) to (R);
\draw[->] (Q) to (S);
\draw[->] (R) to (T);
\draw[->] (S) to (U);
\draw[->] (S) to (V);
\draw[->] (U) to (W);
\draw[->] (V) to (W);
\draw[->] (W) to (X);
\draw[->] (X) to (Z);
\draw[->] (Z) to (AA);
\draw[->] (W) to (Y);
\draw[->] (Y) to (Z);

\draw[loosely dotted] (A.east) -- (D);
\draw[loosely dotted] (B.east) -- (E);
\draw[loosely dotted] (D.east) -- (G);
\draw[loosely dotted] (E.east) -- (H);
\draw[loosely dotted] (G.east) -- (J);
\draw[loosely dotted] (H.east) -- (K);
\draw[loosely dotted] (I.east) -- (L);
\draw[loosely dotted] (K.east) -- (P);
\draw[loosely dotted] (M.east) -- (R);
\draw[loosely dotted] (L.east) -- (Q);
\draw[loosely dotted] (Q.east) -- (U);
\draw[loosely dotted] (P.east) -- (T);
\draw[loosely dotted] (S.east) -- (W);
\draw[loosely dotted] (U.east) -- (X);
\draw[loosely dotted] (W.east) -- (Z);
\draw[loosely dotted] (X.east) -- (AA);
\end{tikzpicture}\]
where we can verify again which abutments of $A$ and $B$ give rise to abutments of $\La$.
\end{example}

\subsubsection{Gluing systems of rep\-re\-sen\-ta\-tion-di\-rect\-ed algebras}\label{subsubsec:gluing systems of representation-directed algebras}
In this section we construct a direct system of categories with an admissible target arising from a collection of rep\-re\-sen\-ta\-tion-di\-rect\-ed algebras which we iteratively glue. Let us first introduce some notation. Let $G$ be a directed tree, that is a locally finite connected directed graph where the underlying undirected graph is acyclic. If $H$ is a connected subgraph of $G$, then clearly $H$ is also a directed tree. We define $\dirI_G$\index[symbols]{I a a@$\dirI_G$} to be the set of all finite connected full subgraphs of $G$. Clearly $\subseteq$ is a partial order on $\dirI_G$ and if $H_1,H_2\in\dirI_G$ then, by connectedness of $G$, there exists a finite connected subgraph $H$ of $G$ such that $H_1\subseteq H$ and $H_2\subseteq H$. In particular, $\left(\dirI_G,\subseteq\right)$ is a directed set. 

We have the following definition. 

\begin{definition}\label{def:gluing system}
A \emph{gluing system}\index[definitions]{gluing system} on a directed tree $G$ is a triple $(\La_{v},P_{e},I_{e})_{v\in V_G,e\in E_G}$\index[symbols]{(Lav, Pe, Ie)@$(\La_{v},P_{e},I_{e})_{v\in V_G,e\in E_G}$}, where
\begin{itemize}
    \item for each vertex $v\in V_G$, we have that $\La_v=\K Q_v/\cR_v$ is a rep\-re\-sen\-ta\-tion-di\-rect\-ed algebra, and
    \item for each arrow $e:u\to v\in E_G$, we have that $P_e$ is a left abutment of $\La_v$ with footing $\pi_{P_e}$ and $I_e$ is a right abutment of $\La_u$ with footing $\pi_{I_e}$, 
\end{itemize}
    such that
\begin{enumerate}[label=(\roman*)]
    \item the collection $P_{t^{-1}(v)}$ is independent for every $v\in V_G$ and the collection $I_{s^{-1}(v)}$ is independent for every $v\in V_G$, 
    \item for each arrow $e:u\to v$ of $G$ we have that the abutments $P_{e}$ and $I_{e}$ have the same height $h_{e}$. In particular, the algebra $\La_{\langle u,v\rangle}\coloneqq \La_u\glue[P_{e}][I_{e}] \La_v=\K Q_{\langle u,v \rangle}/\cR_{\langle u,v\rangle}$\index[symbols]{Lambda u, v@$\La_{\langle u,v\rangle}$}\index[symbols]{KQuv Ruv@$\K Q_{\langle u,v \rangle}/\cR_{\langle u,v\rangle}$} is defined and we have a pullback diagram
    \begin{equation}\label{eq:pullback gluing system}
        \begin{tikzcd}
            \La_{\langle u,v \rangle} \arrow[d, swap, "\prescript{\langle u,v\rangle}{\langle v\rangle}{\pi}"] \arrow[r, "\prescript{\langle u,v\rangle}{\langle u\rangle}{\pi}"] & \La_u \arrow[d, "\pi_{I_e}"] \\
            \La_v \arrow[r, swap, "\pi_{P_e}"] & \K \overrightarrow{A}_h\nospacepunct{,}
        \end{tikzcd}
    \end{equation}
    and
    \item there exists no infinite path of the form $v_{0} \to v_{1} \to v_{2}\to \cdots$ or $\cdots \to v_{-2} \to v_{-1} \to v_{0}$ in $G$ such that all corresponding gluings are trivial.
\end{enumerate}
\end{definition}

Condition (iii) of Definition \ref{def:gluing system} is included for technical reasons; see Lemma \ref{lem:finitely many vertices}.

Let $(\La_{v},P_{e},I_{e})_{v\in V_G,e\in E_G}$ be a gluing system on a directed tree $G$. Let $e:u\to v\in E_G$ be an arrow. Then this arrow corresponds to a gluing of $\La_u$ and $\La_v$ over some abutments $P_e$ and $I_e$, both of the same height $h_e$. Following (\ref{eq:S0}) and (\ref{eq:S1}), we set
\begin{align*}
    a\sim^{0}_e b &\text{ for all $a\in\supp(P_e)$, $b\in\supp(I_e)$ with $\height(P_{\La_v}(a))+\height(I_{\La_u}(b)) = h_e + 1$, and} \\
    \alpha\sim^{1}_e\beta &\text{ for all $\alpha\in \left(Q_{\La_v}\right)_1$, $\beta\in\left(Q_{\La_u}\right)_1$ with $s(\alpha)\sim^{0}_e s(\beta)$.}
\end{align*}\index[symbols]{((e@$\sim^{0}_e$, $\sim^{1}_e$}

Conditions (i) and (ii) of a gluing system allow us to glue over two arrows in $G$ in any order. More precisely let $u,v,w\in V_G$ be three vertices such that $\langle u,v,w\rangle$ is connected. In particular, since $G$ is a directed tree, we have that $\langle u,v,w\rangle$ is one of the following graphs
\begin{enumerate}[label=(C\arabic*)]
    \item $u\overset{e_1}{\longto} v \overset{e_2}{\longto} w$, \label{case:C1} 
    \item $u\overset{e_1}{\longfrom} v \overset{e_2}{\longto} w$, \label{case:C2}
    \item $u\overset{e_1}{\longto}v\overset{e_2}{\longfrom} w$. \label{case:C3}
\end{enumerate}
Assume that we are in the case \ref{case:C1}. Consider the diagrams
\begin{align*}
    &\begin{tikzcd}[ampersand replacement=\&]
        \La_{\langle u,v\rangle} \glue \La_{w}\arrow[dd, swap, "\prescript{\langle u,v,w\rangle}{\langle w\rangle}{\pi}"] \arrow[r, "\prescript{\langle u,v,w\rangle}{\langle u,v\rangle}{\pi}"] \arrow[ddr, phantom, "(2)"] 
        \& \La_{\langle u,v\rangle} \arrow[d, swap, "\prescript{\langle u,v\rangle}{\langle v\rangle}{\pi}"] \arrow[r, "\prescript{\langle u,v\rangle}{\langle u\rangle}{\pi}"] \arrow[dr, phantom, "(1)"]
        \&\La_u \arrow[d, "\pi_{I_{e_1}}"]\\
        \& \La_v \arrow[d, "\pi_{I_{e_2}}"] \arrow[r, swap, "\pi_{P_{e_1}}"] 
        \& \K A_{h_{e_1}}\nospacepunct{,} \\
        \La_w \arrow[r, swap, "\pi_{P_{e_2}}"] 
        \& \K A_{h_{e_2}} \&
    \end{tikzcd}
    &\begin{tikzcd}[ampersand replacement=\&]
        \La_{u}\glue\La_{\langle v,w\rangle} \arrow[d, swap, "\prescript{\langle u,v,w\rangle}{\langle v,w\rangle}{\pi}"] \arrow[rr, "\prescript{\langle u,v,w\rangle}{\langle u\rangle}{\pi}"] \arrow[drr, phantom, "(4)"] 
        \&\&\La_u \arrow[d, "\pi_{I_{e_1}}"]\\
        \La_{\langle v,w\rangle} \arrow[d, swap, "\prescript{\langle v,w\rangle}{\langle w\rangle}{\pi}"] \arrow[r, "\prescript{\langle v,w\rangle}{\langle v\rangle}{\pi}"] \arrow[dr, phantom, "(3)"]
        \& \La_v \arrow[d, "\pi_{I_{e_2}}"] \arrow[r, swap, "\pi_{P_{e_1}}"] 
        \& \K A_{h_{e_1}}\nospacepunct{,} \\
        \La_w \arrow[r, swap, "\pi_{P_{e_2}}"] 
        \& \K A_{h_{e_2}} \&
    \end{tikzcd}
\end{align*}
where all squares (1),(2),(3) and (4) are pullbacks. In particular, the algebra $\La_{\langle u,v\rangle}\glue \La_{w}$ is the gluing of $\La_{w }$ and $\La_{\langle u,v,\rangle}$ along $P_{e_2}$ and $\prescript{\langle u,v\rangle}{u}{\pi}_{\ast}\left(I_{e_2}\right)$ while the algebra $\La_{u}\glue \La_{\langle v,w\rangle}$ is the gluing of $\La_{\langle v,w\rangle}$ and $\La_u$ along $\prescript{\langle v,w\rangle}{\langle w\rangle}{\pi}_{\ast}(P_{e_1})$ and $I_{e_1}$. Notice that the fact that $\prescript{\langle u,v\rangle}{\langle u\rangle}{\pi}_{\ast}(I_{e_2})$ is a right abutment of $\La_{\langle u,v\rangle}$ with footing $\pi_{I_{e_2}}\comp \prescript{\langle u,v\rangle}{\langle v\rangle}{\pi}$ and that $\prescript{\langle v,w\rangle}{\langle w\rangle}{\pi}_{\ast}(P_{e_1})$ is a left abutment of $\La_{\langle v,w\rangle}$ with footing $\pi_{P_{e_1}}\comp \prescript{\langle v,w\rangle}{\langle v\rangle}{\pi}$ follows from Remark \ref{rem:gluing remarks}(c). A straightforward computation shows that $\La_{\langle u,v\rangle} \glue \La_{w}$ and $ \La_{u}\glue\La_{\langle v,w\rangle}$ can both be thought of as the limit of the diagram
\begin{align*}
    &\begin{tikzcd}[ampersand replacement=\&]
        \&\&\La_u \arrow[d, "\pi_{I_{e_1}}"]\\
        \& \La_v \arrow[d, "\pi_{I_{e_2}}"] \arrow[r, swap, "\pi_{P_{e_1}}"] 
        \& \K A_{h_{e_1}}\nospacepunct{,} \\
        \La_w \arrow[r, swap, "\pi_{P_{e_2}}"] 
        \& \K A_{h_{e_2}} \&
    \end{tikzcd}
\end{align*}
with the obvious choice of morphisms. In particular, if $\iota:\La_{\langle u,v\rangle} \glue \La_{w} \to  \La_{u}\glue\La_{\langle v,w\rangle}$ is the unique isomorphism induced by the universal property of the limit, then we have
\begin{align}\label{eq:connecting morphisms 1}
    \prescript{\langle u,v,w \rangle}{\langle u \rangle}{\pi}\comp \iota &= \prescript{\langle u,v \rangle}{\langle u \rangle}{\pi}\comp \prescript{\langle u,v,w \rangle}{\langle u,v \rangle}{\pi}, \\
    \label{eq:connecting morphisms 2}
    \prescript{\langle v,w \rangle}{\langle v \rangle}{\pi}\comp \prescript{\langle u,v,w \rangle}{\langle v,w \rangle}{\pi} \comp \iota &= \prescript{\langle u,v \rangle}{\langle v \rangle}{\pi}\comp \prescript{\langle u,v,w \rangle}{\langle u,v \rangle}{\pi}, \\
    \label{eq:connecting morphisms 3}
    \prescript{\langle v,w \rangle}{\langle w \rangle}{\pi}\comp \prescript{\langle u,v,w \rangle}{\langle v,w \rangle}{\pi} \comp \iota &= \prescript{\langle u,v,w \rangle}{\langle w \rangle}{\pi}.
\end{align}
We denote this common limit by $\La_{\langle u,v,w\rangle}$ and we denote the connecting morphisms by
\begin{align}
    \prescript{\langle u,v,w \rangle}{\langle u \rangle}{\pi}:\La_{\langle u,v,w \rangle}\to \La_u, \\
    \prescript{\langle u,v,w \rangle}{\langle v \rangle}{\pi}:\La_{\langle u,v,w \rangle}\to \La_v, \\
    \prescript{\langle u,v,w \rangle}{\langle w \rangle}{\pi}:\La_{\langle u,v,w \rangle}\to \La_w,
\end{align}
where each of these is defined by (\ref{eq:connecting morphisms 1}), (\ref{eq:connecting morphisms 2}) and (\ref{eq:connecting morphisms 3}) respectively. In particular, we have formulas
\begin{align}\label{eq:connecting formulas 1}
    \prescript{\langle u,v,w \rangle}{\langle u \rangle}{\pi} &= \prescript{\langle u,v \rangle}{\langle u \rangle}{\pi}\comp \prescript{\langle u,v,w \rangle}{\langle u,v \rangle}{\pi}, \\
    \label{eq:connecting formulas 2}
    \prescript{\langle u,v,w \rangle}{\langle v \rangle}{\pi} &= \prescript{\langle u,v \rangle}{\langle v \rangle}{\pi}\comp \prescript{\langle u,v,w \rangle}{\langle u,v \rangle}{\pi} =
    \prescript{\langle v,w \rangle}{\langle v \rangle}{\pi}\comp \prescript{\langle u,v,w \rangle}{\langle v,w \rangle}{\pi}, \\
    \label{eq:connecting formulas 3}
    \prescript{\langle u,v,w \rangle}{\langle w \rangle}{\pi} &= \prescript{\langle v,w \rangle}{\langle w \rangle}{\pi}\comp \prescript{\langle u,v,w \rangle}{\langle v,w \rangle}{\pi},
\end{align}

In this case we can compute the quiver $Q_{\langle u,v,w\rangle}$ of $\La_{\langle u,v,w\rangle}$ in two ways, depending on whether we glue first at $e_1$ and then at $e_2$ or vice versa. It is easy to see that the quiver is the same independently of the order of the gluings. In particular for $k\in\{0,1\}$, we have that 
\begin{align}\label{eq:quiver of gluing in two edges}
    \left(Q_{\langle u,v,w\rangle}\right)_k = \left( \left(Q_u\right)_k \coprod \left(Q_v\right)_k \coprod \left(Q_w\right)_k \right)/\sim^{k}_{\langle u,v,w\rangle},
\end{align}
where $\sim^k_{\langle u,v,w\rangle}$ is the equivalence relation generated by $\sim^{k}_{e_1}$ and $\sim^{k}_{e_2}$. In particular, we have that $Q_u$, $Q_v$ and $Q_w$ are full subquivers of $Q_{\langle u,v,w\rangle}$. Using this description, we set 
\[\epsilon_u = \sum_{i\in \left(Q_u\right)_0}\epsilon_i, \;\; \epsilon_v = \sum_{i\in \left(Q_v\right)_0}\epsilon_i,\;\; \epsilon_w = \sum_{i\in \left(Q_w\right)_0}\epsilon_i,\]
and it follows by applying (\ref{eq:ideal of gluing}) twice that
\begin{equation}\label{eq:ideal of gluing in two edges C1}
    \begin{split}
        \cR_{\langle u,v,w \rangle} &= \langle \cR_u+\cR_v+\cR_w + \left(1_{\langle u,v,w \rangle}-\epsilon_u-\epsilon_w\right)\K Q_{\langle u,v,w \rangle} \left(1_{\langle u,v,w \rangle}-\epsilon_v-\epsilon_w\right) \\
        &+ \left(1_{\langle u,v,w \rangle}-\epsilon_u-\epsilon_v\right)\K Q_{\langle u,v,w \rangle} \left(1_{\langle u,v,w \rangle}-\epsilon_u-\epsilon_w\right)\rangle.
    \end{split}
\end{equation}

A similar calculation in the cases \ref{case:C2} respectively \ref{case:C3} shows that we can compute the successive gluings induced by $e_1$ and $e_2$ in any order as the limits of the diagrams
\begin{align*}
    &\begin{tikzcd}[ampersand replacement=\&]
        \La_{u} \arrow[r, "\pi_{P_{e_1}}"]  
        \& \K A_{h_{e_1}} \\
        \& \La_v \arrow[d, "\pi_{I_{e_2}}"] \arrow[u, swap, "\pi_{I_{e_1}}"] \\
        \La_w \arrow[r, swap, "\pi_{P_{e_2}}"] 
        \& \K A_{h_{e_2}}\nospacepunct{,} \&
    \end{tikzcd}
    &\begin{tikzcd}[ampersand replacement=\&]
        \La_w \arrow[d, swap, "\pi_{I_{e_2}}"] 
        \&\&\La_u \arrow[d, "\pi_{I_{e_1}}"]\\
        \K A_{h_{e_2}} 
        \& \La_v \arrow[l, "\pi_{P_{e_2}}"] \arrow[r, swap, "\pi_{P_{e_1}}"] 
        \& \K A_{h_{e_1}}\nospacepunct{,} \\
    \end{tikzcd}
\end{align*}
respectively. Notice that in this case it is important that the abutments $I_{e_1}$ and $I_{e_2}$ respectively $P_{e_1}$ and $P_{e_2}$ are independent, since this ensures that they remain abutments when viewed as modules over the glued algebra; see also Remark \ref{rem:gluing remarks}(c). By the above discussion we see that we can perform the gluings induced by two arrows in a gluing system in any order. Then $\La_{\langle u,v,w \rangle}$ and the connecting morphisms are defined similarly and satisfy the formulas (\ref{eq:connecting formulas 1}), (\ref{eq:connecting formulas 2}) and (\ref{eq:connecting formulas 3}) as well as the quiver of $\La_{\langle u,v,w \rangle}$ is given by (\ref{eq:quiver of gluing in two edges}) and the ideal $\cR_{\langle u,v,w\rangle}$ is given by
\begin{equation*}\label{eq:ideal of gluing in two edges C2}
    \begin{split}
        \cR_{\langle u,v,w \rangle} &= \langle \cR_u+\cR_v+\cR_w + \left(1_{\langle u,v,w \rangle}-\epsilon_v-\epsilon_w\right)\K Q_{\langle u,v,w \rangle} \left(1_{\langle u,v,w \rangle}-\epsilon_u-\epsilon_w\right) \\
        &+ \left(1_{\langle u,v,w \rangle}-\epsilon_u-\epsilon_v\right)\K Q_{\langle u,v,w \rangle} \left(1_{\langle u,v,w \rangle}-\epsilon_u-\epsilon_w\right)\rangle,
    \end{split}
\end{equation*}
in case \ref{case:C2} and by
\begin{equation*}\label{eq:ideal of gluing in two edges C3}
    \begin{split}
        \cR_{\langle u,v,w \rangle} &= \langle \cR_u+\cR_v+\cR_w + \left(1_{\langle u,v,w \rangle}-\epsilon_u-\epsilon_w\right)\K Q_{\langle u,v,w \rangle} \left(1_{\langle u,v,w \rangle}-\epsilon_v-\epsilon_w\right) \\
        &+ \left(1_{\langle u,v,w \rangle}-\epsilon_u-\epsilon_w\right)\K Q_{\langle u,v,w \rangle} \left(1_{\langle u,v,w \rangle}-\epsilon_u-\epsilon_v\right)\rangle.
    \end{split}
\end{equation*}
in case \ref{case:C3}. In particular, we can write the ideal $\cR_{\langle u,v,w\rangle}$ in all three cases \ref{case:C1}, \ref{case:C2} and \ref{case:C3} by the formula
\begin{equation}\label{eq:ideal of gluing in two edges}
    \cR_{\langle u,v,w \rangle} = \left\langle \cR_u+\cR_v+\cR_w+ \sum_{e\in\{e_1,e_2\}}\left(1_{\langle u,v,w \rangle}-p_e\right)\K Q_{\langle u,v,w \rangle} \left(1_{\langle u,v,w \rangle}-q_e\right)\right\rangle,
\end{equation}
where
\[p_e=\sum_{\substack{x\in\{u,v,w\} \\ x\neq t(e)}}\epsilon_x\text{ and } q_e=\sum_{\substack{x\in\{u,v,w\} \\ x\neq s(e)}}\epsilon_x.\]
Notice also that by Proposition \ref{prop:AR glued} the gluing of two rep\-re\-sen\-ta\-tion-di\-rect\-ed algebras is again a rep\-re\-sen\-ta\-tion-di\-rect\-ed  algebra. Using these facts, if $H$ is a finite connected full subgraph of $G$, then we can generalize Definition \ref{def:gluing quivers} by a straightforward induction on the number of vertices of $H$.

\begin{definition}\label{def:gluing over subgraph}
Let $(\La_{v},P_{e},I_{e})_{v\in V_G,e\in E_G}$ be a gluing system on a directed tree $G$ and let $H\in\dirI_G$ be a finite connected full subgraph of $G$. The \emph{gluing algebra over $H$}\index[definitions]{gluing over a finite connected subgraph} is defined to be the algebra $\La_H=\K Q_H/\cR_H$\index[symbols]{Lambda H@$\La_H$}\index[symbols]{KQH RH@$\K Q_H/\cR_H$} obtained by performing the gluings induced by the edges $E_H$ of $H$ in any order. For each $v\in V_H$ we define the idempotent $\epsilon_v=\sum_{i\in \left(Q_v\right)_0}\epsilon_i\in \La_H$. Then $\La_H/\langle 1- \epsilon_v \rangle=\La_v$ and we denote the projection morphism $\La_H\to \La_H/\langle 1- \epsilon_v \rangle$ by $\prescript{H}{v}{\pi}$\index[symbols]{pi H to v@$\prescript{H}{v}{\pi}:\La_H\to\La_v$}.
\end{definition}

Let $H\in \dirI_G$. By repeatedly applying formula (\ref{eq:quiver of gluing in two edges}) we can describe the quiver $Q_H$. In particular, for $k\in\{0,1\}$ we have 
\begin{equation}\label{eq:quiver of gluing subgraph}
\left(Q_H\right)_k = \left(\coprod_{v\in V_H}\left(Q_v\right)_0 \right)/\sim^k_H,
\end{equation}
where $\sim^k_H$\index[symbols]{((H@$\sim^{0}_H$, $\sim^{1}_H$} is the equivalence relation generated by the equivalence relations $\sim^k_e$ for all $e\in E_H$. In particular, for every $v\in V_H$ we have that $Q_v$ is a full subquiver of $Q_H$. It follows by repeatedly applying (\ref{eq:ideal of gluing in two edges}) that
\begin{equation}\label{eq:ideal of gluing subgraph}
    \cR_H = \left\langle\sum_{v\in V_H}\cR_v + \sum_{e\in E_H} \left(1_{H}-p_e\right)\K Q_{H} \left(1_{H}-q_e\right)\right\rangle,
\end{equation}
where
\[p_e=\sum_{\substack{v\in V_H \\ v\neq t(e)}}\epsilon_x\text{, and }q_e=\sum_{\substack{x\in V_H \\ x\neq s(e)}}\epsilon_x.\]

\begin{remark}\label{rem:gluing over subgraph remarks} 
Throughout let $(\La_{v},P_{e},I_{e})_{v\in V_G,e\in E_G}$ be a gluing system on a directed tree $G$ and $H\in\dirI_G$.
\begin{enumerate}[label=(\alph*)]
    \item The algebra $\La_H$ can be seen as the limit of a diagram in the following way. First we define the index category $J=J(H)$ as follows. For every vertex $v\in V_H$, we have one corresponding object $v\in J$ and for every arrow $e\in E_H$, we have one corresponding object $h_e\in J$ and two morphisms $g_e:s(e)\to h_e$ and $f_e:t(e)\to h_e$. Notice that there is no need to define composition in $J$ since the only compositions appearing in $J$ involve identity arrows. Next, we define a functor $F:J\to \kAlg$, the category of fi\-nite-di\-men\-sion\-al unital associative $\K$-algebras, by $F(v)=\La_v$ and $F(h_e)=\K A_{h_e}$ on objects of $J$ and $F(f_e)=\pi_{P_e}$ and $F(g_e)=\pi_{I_e}$ on non-identity morphisms of $J$. Then $\La_H$ is the limit of the diagram $F:J\to \kAlg$. 
    
    \item If $H,K\in\dirI_G$ and $H\subseteq K$, then clearly $\left(\La_K,\prescript{K}{v}{\pi}\right)$ is a cone to the diagram $F$ from (a). Then there exists a unique morphism $\iota:\La_K\to \La_H$ such that $\prescript{K}{v}{\pi}=\prescript{H}{v}{\pi}\comp \iota$. We can compute $\iota$ in the following way. First, notice that $Q_H$ is a full subquiver of $Q_K$. Then for $\epsilon_H=\sum_{i\in \left(Q_H\right)_0}\epsilon_i\in \La_K$ we have $\La_H=\La_K/\langle 1-\epsilon_K\rangle$. We set $\prescript{K}{H}{\pi}:\La_K\to \La_H$\index[symbols]{pi K to H@$\prescript{K}{H}{\pi}:\La_K\to \La_H$} to be the quotient map $\La_K\to\La_K/\langle 1-\epsilon_K\rangle$. Then by using the formulas (\ref{eq:connecting formulas 1}), (\ref{eq:connecting formulas 2}) and (\ref{eq:connecting formulas 3}) inductively, for every $v\in V_H$ we have
    \begin{equation}
        \prescript{K}{v}{\pi}=\prescript{H}{v}{\pi}\comp\prescript{K}{H}{\pi},
    \end{equation}
    showing that $\iota=\prescript{K}{H}{\pi}$. Clearly, if $H,K,L\in\dirI_G$ with $H\subseteq K \subseteq L$, then we have
    \begin{equation}\label{eq:connecting formula for subgraphs}
        \prescript{L}{H}{\pi}=\prescript{K}{H}{\pi}\comp\prescript{L}{K}{\pi}.
    \end{equation}
    
    \item If $H,K\in\dirI_G$ and $H\subseteq K$, then the algebra epimorphism $\prescript{K}{H}{\pi}:\La_K\to \La_H$ defines functors
    \[\begin{tikzpicture}[scale=0.9, transform shape]]
    \node (A) at (0,0) {$\m\La_K$};
    \node (B) at (3,0) {$\m \La_H$\nospacepunct{,}};

    \draw[->] (B) -- node[above] {$\prescript{K}{H}{\pi}_{\ast}$} (A); 
    \draw[->] (A) to [out=30,in=150]  node[auto]{$\prescript{K}{H}{\pi}^{\ast}$} (B);
    \draw[->] (A) to [out=-30,in=-150]  node[below]{$\prescript{K}{H}{\pi}^{!}$} (B);
    \end{tikzpicture}\]
    which satisfy similar properties to the functors in Remark \ref{rem:gluing remarks}(b). In particular we have that if $M$ is a representation of $\La_H$, then $\prescript{K}{H}{\pi}_{\ast}(M)$ is the representation of $\La_K$ given by extending $M$ to $Q_K$ by putting $0$ on all vertices and arrows of $Q_K$ which do not lie in $Q_H$. Moreover, we also have
    \[\prescript{K}{H}{\pi}^{\ast}\left(\epsilon_i\La_K\right) \isom \begin{cases} \epsilon_i\La_H, &\mbox{if $i\in \left(Q_H\right)_0$,} \\ 0, &\mbox{otherwise.} \end{cases}\]
    
    \item Let $i,j\in \left(Q_{H}\right)_0$ and assume that $\epsilon_j\La_H\epsilon_i\neq 0$. By inductively applying the argument of Remark \ref{rem:gluing remarks}(e), it follows that there exists a vertex $v\in V_H$ such that $i,j\in\left(Q_v\right)_0$ and $\epsilon_j\La_H \epsilon_i \isom \epsilon_j\La_v\epsilon_i$.
\end{enumerate}
\end{remark}

By Remark \ref{rem:gluing over subgraph remarks}(b) it follows that $\left(\La_H,\prescript{K}{H}{\pi}\right)$ is a usual inverse system over $\dirI_G$ in $\kAlg$. We want to use this inverse system to define a $\Cat$-inverse system over $\dirI_G$. First, for $H\in \dirI_G$ we define a small category $\cC_H$. Consider the category $\proj \La_{H}$. This is a skeletally small category. To see this let $\cC_{H}\subseteq \proj \La_{H}$ be the full subcategory with direct summands of $\La_H^n$ for some $n\in\NN$ as objects, that is 
\[\obj(\cC_{H}) = \left\{ M \subseteq \La_{H}^n \mid \text{$n\in \NN$ and there exists $N\subseteq\La^n_H$ with $M\oplus N = \La_{H}^n$} \right\}.\]
Then $\cC_H$ is a small subcategory equivalent to $\proj \La_{H}$. We denote by $\iota_H:\cC_H\to \proj \La_H$ the inclusion functor and by $q_H:\proj\La_H\to\cC_H$ a quasi-inverse to $\iota_H$. Notice that a complete set of representatives of indecomposable objects in $\cC_{H}$ is given by 
\[\ind(\cC_{H}) = \{\epsilon_i\La_H \subseteq \La_H \mid i\in\left(Q_H\right)_0\}.\]
Next, for $H\subseteq K$ in $\dirI_G$ we set
\[F_{HK} \coloneqq q_H\comp \prescript{K}{H}{\pi}^{\ast}\comp \iota_K:\cC_K\to\cC_H.\]
Notice that $F_{HK}$ is an additive functor since it is the composition of additive functors. If $i\in\left(Q_H\right)_0$, we have canonical isomorphisms
\begin{align*}
F_{HK}\left(\epsilon_i\La_H\right)&=q_H\comp\prescript{K}{H}{\pi}^{\ast}\comp\iota_K\left(\epsilon_i\La_K\right) 
= q_H\comp\prescript{K}{H}{\pi}^{\ast}\left(\epsilon_i\La_K\right)
=q_H\left(\epsilon_i\La_K\otimes_{\La_K}\La_H\right)\\
&\isom q_H\left(\epsilon_i\La_H\right)
=q_H\iota_H\left(\epsilon_i\La_H\right)
\isom \epsilon_i\La_H.
\end{align*}
In this case we set $\delta_{HK}^i:F_{HK}\left(\epsilon_i\La_K\right)\to \epsilon_i\La_H$\index[symbols]{deltaHKi@$\delta_{HK}^i$} to be the composition of these isomorphisms. If $i\in\left(Q_K\right)_{0}\setminus\left(Q_H\right)_{0}$, then we have $F_{HK}\left(\epsilon_i\La_K\right)=0$.

Finally, for $H\subseteq K \subseteq L$ we define a natural isomorphism $\theta_{HKL}:F_{HK}\comp F_{KL}\To F_{HL}$ via the sequence of natural isomorphisms
\begin{align*}
    F_{HK}\comp F_{KL} &= q_H \comp \prescript{K}{H}{\pi}^{\ast}\comp \iota_K \comp q_K \comp \prescript{L}{K}{\pi}^{\ast}\comp\iota_L && \\
    &\isom q_H\comp \prescript{K}{H}{\pi}^{\ast}\comp \Id_{\proj\La_K} \comp\prescript{L}{K}{\pi}^{\ast}\comp\iota_L && \\
    &= q_H\comp \prescript{K}{H}{\pi}^{\ast}\comp\prescript{L}{K}{\pi}^{\ast}\comp\iota_L && \\
    &\isom q_H \comp \left(\prescript{K}{H}{\pi}\comp \prescript{L}{K}{\pi}\right)^{\ast}\comp\iota_L && \\
    &=q_H\comp \prescript{L}{H}{\pi}^{\ast}\comp\iota_L &&\text{by (\ref{eq:connecting formula for subgraphs})}\\
    &=F_{HL}.
\end{align*}

\begin{proposition}\label{lem:the system is an inverse system}
The triple $\left(\cC_H, F_{HK}, \theta_{HKL}\right)$\index[symbols]{(CH, FHK, thetaHKL)@$\left(\cC_H, F_{HK}, \theta_{HKL}\right)$} is a $\Cat$-inverse system over $\dirI_G$. Moreover, the firm source $\Cks$ of the $\Cat$-inverse system is a Krull--Schmidt category.
\end{proposition}

\begin{proof}
That the triple is a $\Cat$-inverse system over $\dirI_G$ follows since it is easy to see that (\ref{eq:commutativity in inverse system}) holds using (\ref{eq:connecting formula for subgraphs}). That $\Cks$ is a Krull--Schmidt category follows by Corollary \ref{cor:properties of Cks}(c) since $\cC_H$ is a Krull--Schmidt category for each $H$ and the functors $F_{HK}$ are additive.
\end{proof}

For the rest of this section, we fix a gluing system $(\La_{v},P_{e},I_{e})_{v\in V_G,e\in E_G}$ on a directed tree $G$. Our next aim is to study the firm source $(\Cks,\Pks_H,\Tks_{HK})$\index[symbols]{(C, PhiH, ThetaHK)@$(\Cks,\Pks_H,\Tks_{HK})$} of the induced $\Cat$-inverse system $(\cC_H, F_{HK}, \theta_{HKL})$ over $\dirI_G$. Before we proceed, notice that if $H,K\in \dirI_G$ with $H\subseteq K$, then $Q_H$ is a full subquiver of $Q_K$ by (\ref{eq:quiver of gluing subgraph}), hence $\K Q_H$ is a subspace of $\K Q_K$ and $\cR_H\subseteq \cR_K$ by (\ref{eq:ideal of gluing subgraph}). We set $\Qks\coloneqq\bigcup_{H\in\dirI_G}Q_H$\index[symbols]{KQ R@$\K\Qks/\Rks$}. By the definition of a gluing and gluing system it is clear that every vertex in $\Qks$ is the source and target of at most finitely many arrows and so $\Qks$ is a quiver. Notice in particular that if $i\in \Qks$, then $i\in \left(Q_H\right)_0$ for some $H\in \dirI_G$ and by (\ref{eq:quiver of gluing subgraph}) it follows that $i\in \left(Q_v\right)_0$ for some $v\in V_H$. For a vertex $i\in \Qks$ we set $V(i)\coloneqq \{v\in V_G \mid i\in (Q_v)_0\}$\index[symbols]{V(i)@$V(i)$}.

We begin with a technical lemma which makes use of condition (iii) in Definition \ref{def:gluing system}.

\begin{lemma}\label{lem:finitely many vertices} Let $i\in \Qks$. Then $V(i)$ is a finite set and $\langle V(i)\rangle$ is a connected graph.
\end{lemma}

\begin{proof}
Let us first show that $\langle V(i)\rangle$ is connected. We first prove a more general claim. Let $v_0,\dots,v_{k+1}$ be a sequence of vertices in $G$ connected by arrows $e_0,\dots,e_{k}$ as in $v_0\xleftrightarrow{e_0}v_1\xleftrightarrow{e_1}v_2\xleftrightarrow{e_2}\cdots \xleftrightarrow{e_{k-1}}v_k\xleftrightarrow{e_k}v_{k+1}$ where $\xleftrightarrow{\phantom{e_0}}$ denotes an arrow in either direction. Then we claim that if $i\in \left(Q_{v_0}\right)$ and $i\in \left(Q_{v_{k+1}}\right)_0$, it follows that $i\in \left(Q_{v_k}\right)_0$. To see this, notice that $\La_{\langle v_0,\dots,v_k,v_{k+1}\rangle} = \La_{\langle v_0,\dots,v_{k}\rangle} \glue \La_{v_{k+1}}$ or $\La_{\langle v_0,\dots,v_k,v_{k+1}\rangle} =  \La_{v_{k+1}}\glue \La_{\langle v_0,\dots,v_{k}\rangle}$, depending on the orientation of $e_k$. Since $i\in \left(Q_{v_0}\right)_0$, it follows that $i\in\left(Q_{\langle v_0,\dots,v_k\rangle}\right)_0$. Since the gluing is done over an abutment of $\La_{v_k}$ and $i\in \left(Q_{v_{k+1}}\right)_0$, it follows that $i\in\left(Q_{v_k}\right)_0$ as claimed.

Now, let $u,v\in V(i)$ and let $H\in \dirI_G$ be such that $u,v\in V_H$. Then there exists a walk $u\xleftrightarrow{e_0}w_1\xleftrightarrow{e_1}\dots\xleftrightarrow{e_{k-1}} w_k\xleftrightarrow{x_k}v$ in $H$. Since $i\in \left(Q_u\right)_0$ and $i\in \left(Q_v\right)_0$, it follows by our claim that $i\in \left(Q_{w_k}\right)_0$. In particular $w_k\in V(i)$. Continuing inductively, we have $w_1,\dots,w_k\in V(i)$ and so there is a walk between $u$ and $v$ in $\langle V(i)\rangle$. Since $u$ and $v$ were arbitrary, it follows that $\langle V(i)\rangle$ is connected.

Let us now show that $V(i)$ is finite. Since $\langle V(i) \rangle$ is connected, and since $G$ is locally finite, it is enough to show that there is no infinite sequence of vertices $\{v_k\}_{k\geq 0}$ in $G$ connected by arrows $\{e_k\}_{k\geq 0}$  as in 
\begin{equation}\label{eq:infinite chain cannot exist}
    v_0\xleftrightarrow{e_0}v_1\xleftrightarrow{e_1}v_2\xleftrightarrow{e_2}\cdots \xleftrightarrow{e_{k-1}}v_k\xleftrightarrow{e_k}v_{k+1}\xleftrightarrow{e_{k+1}}\cdots
\end{equation}
and such that $i\in \left(Q_{v_k}\right)_0$ for every $k\geq 0$. If there is a subgraph $v_{m-1}\rightarrow v_m \leftarrow v_{m+1}$ in (\ref{eq:infinite chain cannot exist}), then $\left(Q_{v_{m-1}}\right)\cap \left(Q_{v_{m+1}}\right)=\varnothing$ by Definition \ref{def:gluing system}(i) and hence we may assume that every arrow in (\ref{eq:infinite chain cannot exist}) is oriented towards the right. 

Since $i\in \left(Q_{v_{k-1}}\right)_0\cap \left(Q_{v_{k}}\right)_0$ we get $i\in\supp(P_{e_{k-1}})$ and $i\in\supp(I_{e_{k}})$ for all $k\geq 1$. Since $P_{e_{k-1}}$ is a left abutment of $\La_{v_k}$ and $I_{e_{k}}$ is a right abutment of $\La_{v_k}$ it follows that for every $k\geq 1$, the algebra $\La_{v_k}$ is isomorphic to $\K \overrightarrow{A}_{h_k}/\cR_k$ for some $h_k\geq 1$. Define
\[x_k\coloneqq \abs{\{ u\in \overrightarrow{A}_{h_k} \mid \text{ there exists a path from $u$ to $i$ in $\K\overrightarrow{A}_{h_k}$}\}}.\]
Clearly $x_k\geq 1$ for all $k\geq 1$. By Definition \ref{def:gluing quivers}, since $i\in \left(Q_{v_{k-1}}\right)_0\cap \left(Q_{v_{k}}\right)_0$, we have that $x_k\leq x_{k-1}$ with equality if and only if the gluing is trivial. Since $x_k\geq 1$, this implies that infinitely many of the gluings in (\ref{eq:infinite chain cannot exist}) are trivial, which contradicts Definition \ref{def:gluing system}(iii).
\end{proof}

Another technical lemma in a similar spirit is the following.

\begin{lemma}\label{lem:technical connecting map is bijective}
Let $H,K\in \dirI_G$ with $H\subseteq K$. Let $i\in \left(Q_{H}\right)_0$ and $j\in \left(Q_{K}\right)_0$. Let $F_{HK}$ be the induced map \[F_{HK}:\cC_K(\epsilon_i\La_K,\epsilon_j\La_K)\to \cC_H(F_{HK}\left(\epsilon_i\La_K\right),F_{HK}\left(\epsilon_j\La_K\right)).\]
\begin{enumerate}[label=(\alph*)]
    \item If $j\not\in (Q_H)_0$, then $\cC_H(F_{HK}\left(\epsilon_i\La_K\right),F_{HK}\left(\epsilon_j\La_K\right))=0$.
    \item If $j\in (Q_H)_0$, then $F_{HK}$ is bijective.
\end{enumerate}
In particular, the induced map $F_{HK}$ is always surjective.
\end{lemma}

\begin{proof}
Since $F_{HK}=q_H\comp\prescript{K}{H}{\pi}^{\ast}\comp\iota_K$ and $q_H$ and $\iota_K$ are fully faithful functors, it is enough to study the map 
\begin{equation}\label{eq:prescript map is surjective}
\prescript{K}{H}{\pi}^{\ast}:\cC_K\left(\epsilon_i\La_K,\epsilon_j\La_K\right)\to \cC_H\left(\epsilon_i\La_K\otimes_{\La_K}\La_H,\epsilon_j\La_K\otimes_{\La_K}\La_H\right).
\end{equation}
\begin{enumerate}[label=(\alph*)]
    \item If $j\not\in \left(Q_H\right)_0$, then $F_{HK}\left(\epsilon_j\La_K\right)=\epsilon_j\La_K\otimes_{\La_K}\La_H\isom \epsilon_j\La_H=0$.
    
    \item Assume that $j\in\left(Q_H\right)_0$. Since $H\subseteq K$, we have that $Q_H\subseteq Q_K$. We first claim that every path $p$ in $Q_K$ from $j$ to $i$ lies in $Q_H$. Assume to a contradiction that there exists a path $p$ in $Q_K$ from $j$ to $i$ such that $p$ does not lie in $Q_H$. Since $Q_H$ is a full subquiver of $Q_K$, there exists a vertex $k\in\left(Q_K\right)_0$ and a path $j\leadsto k \leadsto i$ in $Q_K$ such that $k\not\in\left(Q_H\right)_0$. Let $v_j\in V_H$ be a vertex such that $j\in (Q_{v_j})_0$ and $v_k\in V_K$ be such that $k\in (Q_{v_k})_0$. Since $k\not\in (Q_H)_0$, we have that $v_k\not\in V_H$. Since $K$ is connected, there exists a walk $w$ in $K$ from $v_j$ to $v_k$. Since $K$ is a directed tree, this walk $w$ is unique. Moreover, since there exists a path $j\leadsto k$ in $K$ the walk $w$ is oriented as $v_k\rightarrow \dots\rightarrow v_j$ by the definition of gluing. Similarly, if $v_i\in V_H$ is a vertex such that $i\in(Q_{v_i})_0$, there exists a walk $w'$ in $K$ of the form $v_i\rightarrow \dots \rightarrow v_j$. Since $H$ is connected, there exists a (potentially empty) walk $w''$ between $v_i$ and $v_j$. But then by concatenating the walks $w$, $w'$ and $w''$ it follows that $K$ is not a directed tree, which is a contradiction.

    By the above claim it follows that we have $\epsilon_j\K Q_K\epsilon_i=\epsilon_j\K Q_H\epsilon_i$. Using (\ref{eq:ideal of gluing subgraph}) one can also check that $\epsilon_j\cR_H\epsilon_i=\epsilon_j\cR_K\epsilon_i$. Next, employing the identifications
    \begin{align*}
        \cC_K\left(\epsilon_i\La_K,\epsilon_j\La_K\right)&\isom \epsilon_j\La_K\epsilon_i,\text{ and }\\ \cC_H\left(\epsilon_i\La_K\otimes_{\La_K}\La_H,\epsilon_j\La_K\otimes_{\La_K}\La_H\right)&\isom \cC_H\left(\epsilon_i\La_H,\epsilon_j\La_H\right)\isom \epsilon_j\La_H\epsilon_i,
    \end{align*}
    we find that the map $\prescript{K}{H}{\pi}^{\ast}$ is given by mapping a path in $Q_K$ from $j$ to $i$ to the corresponding path in $Q_H$. Since 
    \[\epsilon_j\La_K\epsilon_i = \epsilon_j(\K Q_K/\cR_K)\epsilon_i = \epsilon_j (\K Q_H /\cR_H) \epsilon_i = \epsilon_j \La_H\epsilon_i,\]
    the result follows.\qedhere
\end{enumerate}
\end{proof}

Now we are ready to describe the indecomposable objects of $\Cks$ as well as the morphisms between them. 

\begin{proposition}\label{prop:indecomposables in firm source}
\begin{enumerate}[label=(\alph*)]
    \item Let $i\in \Qks_0$. Then the pair $P(i)=\left(P(i)_H,p_{HK}\right)$\index[symbols]{(P(i), pHK)@$\left(P(i)_H,p_{HK}\right)$} where
    \begin{align*}
        & P(i)_H = {\begin{cases} \epsilon_i\La_H, &\mbox{if $i\in \left(Q_H\right)_0$,} \\ 0, &\mbox{otherwise,}\end{cases}} 
        && p_{HK} = {\begin{cases} \delta_{HK}^i, &\mbox{if $i\in \left(Q_H\right)_0$,} \\ 0, &\mbox{otherwise,}\end{cases}} 
    \end{align*}
    defines an indecomposable object in $\Cks$. 
    
    \item We have $P(i)\isom P(j)$ if and only if $i=j$.
    
    \item If $x\in\Cks$ is indecomposable, then $x\isom P(i)$ for some $i\in\Qks_0$.
    
    \item Let $i,j\in \Qks_0$. Then the induced map
    \[\Pks_H:\Cks(P(i),P(j)) \to \cC_H\left(\epsilon_i\La_H,\epsilon_j\La_H\right)\]
    is bijective for every $H$ such that $i,j\in \left(Q_H\right)_0$.
    
    \item $\Cks$ is locally bounded.
    
    \item $\m\Cks$ is locally bounded.
\end{enumerate}
\end{proposition}

\begin{proof}
\begin{enumerate}[label=(\alph*)]
    \item Let us show first that $P(i)$ is an indecomposable object in $\Cks$. First let us show that it is an object in the concrete inverse limit $\cC$, that is let us check that $P(i)$ satisfies Definition \ref{def:concrete inverse limit category}(i). Let $H,K\in\dirI_G$ with $H\subseteq K$. If $i\not\in \left(Q_H\right)_0$, then it follows by the definition of $F_{HK}$ that $F_{HK}\left(\epsilon_i\La_K\right)=0$ and so $p_{HK}=0$ is indeed an isomorphism $F_{HK}(P(i)_K)\to P(i)_H$. If $i\in \left(Q_H\right)_0$, then $i\in\left(Q_K\right)_0$ and so $p_{HK}=\delta_{HK}^i:F_{HK}\left(\epsilon_i\La_K\right)\to \epsilon_i\La_H$ is an isomorphism by definition. It remains to show that (\ref{eq:commutativity of objects in inverse limit category}) is satisfied. This follows by a straightforward computation using the fact that both the maps $\delta_{HK}^i$ and the natural isomorphism $\theta_{HKL}$ were constructed by making canonical choices.

    Next let us show that $P(i)\in \Cks$, that is that $P(i)$ is firm. By Lemma \ref{lem:finitely many vertices} it follows that $\langle V(i) \rangle\in\dirI_G$. We set $T\coloneqq \langle V(i)\rangle$. We show that $T\in\dirI_G$ satisfies the condition of Definition \ref{def:firm objects and Cks}. Since $T=\langle V(i) \rangle$, we have that $i\in \left(Q_T\right)_0$. Now let $H\in \dirI_G$ be such that $T\subseteq H$. We need to show that for all $a,b\in\cC_H$, the induced maps
    \[F_{TH}:\cC_H\left(\Phi_H(P(i)),b\right)\to \cC_T\left(F_{TH}\comp \Phi_H(P(i)),F_{TH}(b)\right),\text{ and }\]
    \[F_{TH}:\cC_H\left(a,\Phi_H(P(i))\right)\to \cC_T\left(F_{TH}(a),F_{TH}\comp \Phi_H(P(i))\right)\text{\phantom{, and }}\]
    are bijective. Let us show that the first map is bijective; the other can be similarly shown to be bijective. Without loss of generality, assume that $b$ is indecomposable. Then $b=\epsilon_j\La_H$ for some $j\in \left(Q_H\right)_0$ and $\Phi_H(P(i))=\epsilon_i\La_H$ by definition of $P(i)$. Hence $\cC_H(\Phi_H(P(i)),b)=\cC_H(\epsilon_i\La_H,\epsilon_j\La_H)$. By Lemma \ref{lem:technical connecting map is bijective}, the map $F_{TH}$ is surjective. Moreover, by the same lemma, it is enough to show that if $\cC_H(\epsilon_i\La_H,\epsilon_j\La_H)\isom\epsilon_j\La_H\epsilon_i\neq 0$, then $j\in\left(Q_T\right)_0$. Assume $\epsilon_j\La_H\epsilon_i\neq 0$. Then it follows by Remark \ref{rem:gluing over subgraph remarks}(d) that there exists a vertex $v\in V_H$ such that $i,j\in \left(Q_v\right)_0$. But then $v\in V(i)$ and so $i,j\in \left(Q_T\right)_0$, as required. This shows that $P(i)$ is firm and hence $P(i)\in\Cks$. 
    
    Finally, we need to show that $P(i)$ is indecomposable. But this follows immediately by Corollary \ref{cor:properties of Cks}(c) since for $H\supset T$ we have $\Phi_H(P(i))=\epsilon_i\La_H$ which is indecomposable. 
    
    \item Clearly $i=j$ implies that $P(i)\isom P(j)$. On the other hand, if $P(i)\isom P(j)$, then $\epsilon_i\La_H \isom \epsilon_j\La_H$ for some $H\in\dirI_G$ with $H>T(i)$ and $H>T(j)$. But then it follows that $i=j$ since $\epsilon_i$ respectively $\epsilon_j$ are the idempotents corresponding to the vertices $i$ respectively $j$ of $Q_H$. 

    \item Now, assume that $x=\left(x_H,f_{HK}\right)\in \Cks$ is indecomposable. We show that $x\isom P(i)$ for some $i\in\Qks_0$. Since $x$ is firm, there exists $T=T(x)\in \dirI_G$ as in Definition \ref{def:firm objects and Cks}. Let $T'>T$. By Corollary \ref{cor:properties of Cks}(c) we have that $\Pks_{T'}(x)$ is indecomposable and so $\Pks_{T'}(x)\isom\epsilon_i\La_{T'}$ for some $i\in\left(Q_{T'}\right)_0$. We claim that $x\isom P(i)$. Let $H\in\dirI_G$. By taking $K\in\dirI_G$ large enough, it is not difficult to show that
    \[\Pks_H(x) \isom F_{HK}\comp \Pks_K(x) =F_{HK}(\epsilon_i\La_K)\isom \begin{cases}\epsilon_i\La_H, &\mbox{if $i\in\left(Q_H\right)_0$,} \\ 0, &\mbox{otherwise.} \end{cases}\]
    Similarly we pick $s_{T'}:\epsilon_i\La_{T'}\to\epsilon_i\La_{T'}$ to be any isomorphism and we extend it to an isomorphism $s:x\to P(i)$ which shows that $x\isom P(i)$.
    
    \item Since $P(i)$ is indecomposable, there exists a $T=T(P(i))$ as in Definition \ref{def:firm objects and Cks}. Let $H\in\dirI_G$ be such that $i,j\in \left(Q_H\right)_0$ and let $K\in\dirI_G$ be such that $K\supset H$ and $K\supset T$. Then we have the isomorphism $\Pks_H \isom F_{HK}\comp\Pks_K$ and the induced map
    \[\Pks_K:\Cks(P(i),P(j)) \to \cC_K\left(\epsilon_i\La_K,\epsilon_j\La_K\right)\]
    is bijective by Lemma \ref{lem:connecting functors eventually bijective implies limit functors eventually bijective}(b). Moreover, the map 
    \[F_{HK}:\cC_K\left(\epsilon_i\La_K,\epsilon_j\La_K\right)\to\cC_H(F_{HK}\left(\epsilon_i\La_K\right),F_{HK}\left(\epsilon_j\La_K\right))\isom \cC_H\left(\epsilon_i\La_H,\epsilon_j\La_H\right)\]
    is bijective by Lemma \ref{lem:technical connecting map is bijective}(b). It follows that $\Pks_H$ is bijective, as required.

    \item Since $\{P(i)\}_{i\in \Qks_0}$ is a complete set of representatives of indecomposable objects in $\Cks$, it is enough to show that for $i\in \Qks_0$ there are only finitely many $j\in \Qks_0$ such that $\Cks\left(P(i),P(j)\right)\neq 0$ or $\Cks\left(P(j),P(i)\right)\neq 0$. Assume that $\Cks\left(P(i),P(j)\right)\neq 0$, and let $H\in\dirI_G$ be such that $i,j\in\left(Q_H\right)_0$. Then by (b) we have that
    \[0\neq \Cks\left(P(i),P(j)\right)\isom \cC_H\left(\epsilon_i\La_H,\epsilon_j\La_H\right)=\Hom_{\La_H}\left(\epsilon_i\La_H,\epsilon_j\La_H\right)\isom \epsilon_j\La_H \epsilon_i,\]
    and so by Remark \ref{rem:gluing over subgraph remarks}(d) there exists a $v_j\in V_H$ such that $i,j\in \left(Q_{v_j}\right)_0$ and $\epsilon_j\La_H\epsilon_i\isom \epsilon_j\La_v\epsilon_i$. In particular, we have $v_j\in V(i)$. Hence $\Cks(P(i),P(j))\neq 0$ implies that $j\in \bigcup_{v\in V(i)} \left(Q_v\right)_0$. Since $V(i)$ is a finite set by Lemma \ref{lem:finitely many vertices} and $\left(Q_v\right)_0$ is finite by assumption, it follows that $\{j\in \Qks_0 \mid \Cks(P(i),P(j))\neq 0\}$ is a finite set. Similarly we show that $\{j\in \Qks_0 \mid \Cks(P(j),P(i))\neq 0\}$ is also finite.
    
    \item Since the map 
    \[\Hom_{\Cks}\left(\Cks(-,P(i)),\Cks(-,P(j)\right) \to \Cks\left(P(i),P(j)\right)\] 
    defined by $\eta\mapsto \eta_x(\Id_{P(i)})$ is an isomorphism and since $\Cks$ is locally bounded, it follows that $\m\Cks$ is also locally bounded.\qedhere
\end{enumerate}
\end{proof}

Next we show that we can think of $\Cks$ as the (potentially infinite) quiver with relations obtained by performing the (potentially infinitely many) gluings induced by the edges $E_G$ of $G$. Recall that if $Q$ is a quiver, then the \emph{path category}\index[definitions]{path category of a quiver} $\K Q$ of $Q$ is the $\K$-linear category with $\obj\left(\K Q\right)=Q_0$, where the morphism space $\K Q(i,j)$ for $i,j\in \K Q$ is the $\K$-vector space spanned by all paths from $i$ to $j$ in $Q$, under the assumption that there is a trivial identity path if $i=j$ and where composition is given by concatenation. 

In particular, since $\cR_H\subseteq \cR_K$ for $H\subseteq K$, it is not difficult to see that $\Rks=\bigcup_{H\in \dirI_G} \cR_H$\index[symbols]{KQ R@$\K\Qks/\Rks$} is a two-sided ideal of the path category $\K \Qks$. Hence the $\K$-linear category $\K \Qks / \Rks$ is well-defined.

\begin{corollary}\label{cor:equivalence of quiver and projectives}
There is an equivalence of categories $\proj(\K \Qks / \Rks) \equivalent \Cks$.
\end{corollary}

\begin{proof}
Throughout we use the identification $\m(\K \Qks/\Rks)\equivalent \rep(\Qks,\Rks)$. We define a functor $F:\m\Cks\to \m(\K \Qks/\Rks)$. For a $\Cks$-module $M\in\m\Cks$, we define a representation $F(M)\in\rep(\Qks,\Rks)$ as follows. Let $i\in \Qks_0$. We define $(F(M))_i \coloneqq M(P(i))$. Let $\alpha:j\to i$ be an arrow in $\Qks$ and let $H\in\dirI_G$ be such that $i,j\in (Q_H)_0$. Then $\Pks_{H}^{-1}(\alpha)\in \Cks(P(i),P(j))$ is well-defined by Proposition \ref{prop:indecomposables in firm source}(d). Moreover, if $H'\in\dirI_G$ is another graph such that $i,j\in (Q_{H'})_0$, then by taking $K\in\dirI_G$ with $H\subseteq K$ and $H'\subseteq K$ and using the fact that $\Cks$ is a firm source it easily follows that $\Pks_{H'}^{-1}(\alpha)=\Pks_{H}^{-1}(\alpha)$ and so $\Pks^{-1}_H(\alpha)$ is independent of $H$. We define $(F(M))_{\alpha} \coloneqq M(\Pks^{-1}_H(\alpha))$. Next, let $\eta:M\to N$ be a morphism in $\Hom_{\Cks}(M,N)$. We define $(F(\eta))_{i}:(F(M))_i\to (F(N))_i$ by $(F(\eta))_i\coloneqq\eta_{P(i)}$. It is a straightforward verification to see that $F$ is an equivalence of categories that restricts to an equivalence $\proj(\Cks)\equivalent \proj(\K\Qks/\Rks)$. Since we have an equivalence $\proj(\Cks)\equivalent \Cks$, the result follows. 
\end{proof}

Our next aim is to prove that $(\m\Cks,\Pks_{H\ast},\Tks_{HK\ast})$\index[symbols]{(mod C, PhiH, ThetaHK)@$(\m\Cks,\Pks_{H\ast},\Tks_{HK\ast})$} together with the adjunctions $(\Pks_H^{\ast},\Pks_{H\ast})$ and $(\Pks_{H\ast},\Pks_{H}^{!})$ is an admissible target of $\left(\m\cC_H,F_{HK\ast},\theta_{HKL\ast}\right)$\index[symbols]{(mod CH, FHK, thetaHKL)@$\left(\m\cC_H,F_{HK\ast},\theta_{HKL\ast}\right)$}. To show this, we need to first show that the functors $\Pks_H$ are coherent preserving for all $H\in\dirI_G$. For this we need to study the functors $\Pks_{H\ast}$ and $\Pks_{H}^{!}$. That the functor $\Pks^!_H$ preserves finitely presented modules follows from Proposition \ref{prop:injective sent to correct injective}(c). For the functor $\Pks_{H\ast}$ we have to work a bit more. First we have the following two technical lemmas dealing with indecomposable projective modules.

\begin{lemma}\label{lem:properties of restriction of scalars}
Let $i\in \Qks_0$ and assume that $\langle V(i) \rangle \subseteq H \subseteq K$ for some $H,K\in \dirI_G$. Then the following statements hold.
\begin{enumerate}[label=(\alph*)]
    \item If $j\in \left(Q_K\right)_0$, then $\epsilon_i\La_K\epsilon_j = \begin{cases} \epsilon_i\La_H\epsilon_j, &\mbox{if $j\in \left(Q_H\right)_0$,} \\ 0, &\mbox{otherwise.} \end{cases}$
    
    \item $F_{HK\ast}\left(\epsilon_i\La_H\right) = \epsilon_i\La_K$.
    
    \item $\Pks_{H\ast}(\epsilon_i\La_H) = \Cks\left(-,P(i)\right)$.
\end{enumerate}
\end{lemma}

\begin{proof}
\begin{enumerate}[label=(\alph*)]
    \item Assume that $\epsilon_i\La_K\epsilon_j\neq 0$. Then by Remark \ref{rem:gluing over subgraph remarks}(d) there exists a vertex $v\in V_G$ such that $i,j\in \left(Q_v\right)_0$. In particular, $v\in V(i) \subseteq H$ and so $j\in \left(Q_H\right)_0$. The result follows by Proposition \ref{prop:indecomposables in firm source}(d).
    
    \item Follows immediately by \cite[Lemma 2.4]{ASS} and Remark \ref{rem:gluing over subgraph remarks}(c).
    
    \item Notice that we can consider $\epsilon_i\La_H=M_i$ as a module in $\m\cC_H$ via the equivalence $\m\cC_H \equivalent\m(\proj\La_H) \equivalent\m\La_H$. In particular, we have $M_i\left(\epsilon_j\La_H\right) = \epsilon_i\La_H\epsilon_j$.
    Then let us evaluate both sides of (c) on an indecomposable object $P(j)\in \Cks$. By picking $K$ big enough so that $j\in \left(Q_K\right)_0$ and $H\subseteq K$ we have by Proposition \ref{prop:indecomposables in firm source}(d) that $\Cks(P(j),P(i))\isom \epsilon_i\La_K\epsilon_j$. Then we compute
    \[\Pks_{H\ast}(\epsilon_i\La_H)\left(P(j)\right) = M_i\left(\Pks_H(P(j)\right)= \begin{cases} \epsilon_i\La_H\epsilon_j, &\mbox{if $j\in\left(Q_H\right)_0$,} \\ 0, &\mbox{otherwise,}\end{cases}\overset{(a)}{=}\epsilon_i\La_K\epsilon_j\]
    and we see that both sides of (c) agree on objects. Similarly it is easy to see that they agree on morphisms.\qedhere
\end{enumerate}
\end{proof}

\begin{lemma}\label{lem:proj covers}
Let $H\in \dirI_G$ and $M\in \m\cC_H$. Let $\bigoplus_{i\in I}\left(\epsilon_i\La_H\right)^{s_i}$ be a projective cover of $M$ for some $I\subseteq \left(Q_H\right)_0$. If $K\in\dirI_G$ is such that $H\subseteq K$ and $\langle V(I)\rangle \subseteq K$, then $\bigoplus_{i\in I}\left(\epsilon_i\La_K\right)^{s_i}$ is a projective cover of $F_{HK\ast}(M)$.
\end{lemma}

\begin{proof}
Let $\rad(M)$ be the radical of $M$. Then $M/\rad(M) \isom \bigoplus_{i\in I}(S_H(i))^{s_i}$ where $S_H(i)$ is the simple representation of $\La_H$ corresponding to the vertex $i\in\left(Q_H\right)_0$. By \cite[Lemma III.2.2]{ASS} and a direct computation it follows that $\rad\left(F_{HK\ast}(M)\right)\isom F_{HK\ast}\left(\rad(M)\right)$. Since moreover $F_{HK\ast}$ is exact and additive, we have
\begin{align*} 
F_{HK\ast}(M)/\rad\left(F_{HK\ast}(M)\right) &\isom F_{HK\ast}(M)/F_{HK\ast}\left(\rad(M)\right) \\
&\isom F_{HK\ast}\left(M/\rad(M)\right) \\
&\isom F_{HK\ast}\left(\bigoplus_{i\in I} (S_H(i))^{s_i}\right) \\
&\isom \bigoplus_{i\in I}\left(F_{HK\ast}(S_H(i))\right)^{s_i} \\
&\isom \bigoplus_{i\in I}(S_K(i))^{s_i},
\end{align*}
from which it follows that $P_K(M)=\bigoplus_{i\in I}\left(\epsilon_i\La_K\right)^{s_i}$ is a projective cover of $F_{HK\ast}(M)$.
\end{proof}

With this we are ready to prove the main result about the functor $\Pks_{H\ast}$.

\begin{lemma}\label{lem:restriction is finitely presented}
If $M\in\m\cC_H$, then $\Pks_{H\ast}(M)$ is finitely presented.
\end{lemma}

\begin{proof}
Let $P(M)=\bigoplus_{i\in I}\left(\epsilon_i\La_H\right)^{s_i}$ be a projective cover of $M$ for some $I\subseteq \left(Q_H\right)_0$. Let $K\in\dirI_G$ be such that $H\subseteq K$ and $\langle V(I)\rangle \subseteq K$. Then by Lemma \ref{lem:proj covers} we have that $F_{HK\ast}(P(M))\isom \bigoplus_{i\in I}\left(\epsilon_i\La_H\right)^{s_i}$ is a projective cover of $F_{HK\ast}(M)$. Let $p:F_{HK\ast}(P(M))\to F_{HK\ast}(M)$ be a minimal epimorphism and let $N=\ker(p)$. Let $P(N)=\bigoplus_{j\in J}\left(\epsilon_j\La_K\right)^{t_j}$ be a projective cover of $N$ for some $J\subseteq \left(Q_K\right)_0$. Let $L\in\dirI_G$ be such that $K\subseteq L$ and $\langle V(J)\rangle \subseteq L$. Then by Lemma \ref{lem:proj covers} we have that $F_{KL\ast}(P(N))\isom\bigoplus_{j\in J}\left(\epsilon_j\La_L\right)^{t_j}$ is a projective cover of $F_{KL\ast}(N)$ and that $F_{KL\ast}\left(F_{HL\ast}(P(M)\right) \isom \bigoplus_{i\in I}\left(\epsilon_i \La_L\right)^{s_i}$ is a projective cover of $F_{KL\ast}\comp F_{HK\ast}(M)$. It follows that
\begin{equation}\label{eq:proj resol in C_L}
\bigoplus_{j\in J}\left(\epsilon_j\La_L\right)^{t_j} \longto \bigoplus_{i\in I}\left(\epsilon_i\La_L\right)^{s_i} \longto F_{KL\ast}\comp F_{HK\ast}(M) \longto 0
\end{equation}
is the start of a projective resolution of $F_{KL\ast}\comp F_{HK\ast}(M)$. Moreover, for every $k\in I\cup J$ we have $\langle V(k) \rangle \subseteq L$ and hence $\Pks_{L\ast}\left(\epsilon_k\La_L\right) = \Cks\left(-,P(k)\right)$ by Lemma \ref{lem:properties of restriction of scalars}(c). Hence by applying the exact functor $\Pks_{L\ast}$ to (\ref{eq:proj resol in C_L}) we get the exact sequence
\[\bigoplus_{j\in J}(\Cks(-,P(j)))^{t_j}\longto \bigoplus_{i\in I}(\Cks(-,P(i))^{s_i}\longto \Pks_{L\ast}\comp F_{KL\ast}\comp F_{HK\ast}(M)\longto 0,\]
which shows that $\Pks_{L\ast}\comp F_{KL\ast}\comp F_{HK\ast}(M)$ is finitely presented. But $\Pks_{L\ast}\comp F_{KL\ast}\comp F_{HK\ast}(M)\isom \Pks_{L\ast}\comp F_{HL\ast}(M) \isom \Pks_{H\ast}(M)$, and so $\Pks_{H}(M)$ is finitely presented.
\end{proof}

We are now ready to prove that a gluing system gives rise to an admissible target.

\begin{proposition}
Let $(\La_{v},P_{e},I_{e})_{v\in V_G, e\in E_G}$ be a gluing system on a directed tree $G$ and let \linebreak $(\Cks,\Pks_H,\Tks_{HK})$ be the firm source of the induced $\Cat$-inverse system $(\cC_H, F_{HK}, \theta_{HKL})$ over $\dirI_G$. Then $(\m\Cks,\Pks_K,\Tks_{HK})$ together with the adjunctions $(\Pks_H^{\ast},\Pks_{H\ast})$ and $(\Pks_{H\ast},\Pks_{H}^{!})$ is an admissible target of $(\m\cC_H,F_{HK\ast},\theta_{HKL\ast})$.
\end{proposition}

\begin{proof}
It is enough to check that the conditions of Corollary \ref{cor:admissible target} are satisfied. By the equivalences
\[\m \cC_H \equivalent \m (\proj\La_{H}) \equivalent \m\La_{H}.\]
we have that $\cC_H$ is a dualizing $\K$-variety and by the commutativity of the diagrams
\[\begin{tikzcd}
    \m\La_H \arrow[r, "f_{HK\ast}"] \arrow[d, "\equivalent"] & \m\La_K \arrow[d, "\equivalent"] \\
    \m\cC_H \arrow[r, "F_{HK\ast}"] & \m\cC_K
\end{tikzcd}
    \text{ and }
\begin{tikzcd}
    \m\La_H \arrow[d, "\equivalent"] & \m\La_K \arrow[d, "\equivalent"] \arrow[l, swap, "f_{HK}^!"]\\
    \m\cC_H & \m\cC_K \arrow[l, swap, "F_{HK}^{!}"]
\end{tikzcd}\]
it follows that $F_{HK}$ is a coherent preserving $\K$-linear functor. Since $\Cks$ is locally bounded by Proposition \ref{prop:indecomposables in firm source}(e), it follows that $\Cks$ is a dualizing $\K$-variety. Let $H,K\in \dirI_G$ with $H\subseteq K$. Then $F_{HK}:\cC_L\to \cC_H$ is dense because if $\epsilon_i\La_H$ is an indecomposable object in $\cC_H$, then $i\in \left(Q_H\right)_0\subseteq \left(Q_K\right)_0$ and so $F_{HK}(\epsilon_i\La_K)=\epsilon_i\La_H$. Moreover, the functor $F_{HK}$ is full by Lemma \ref{lem:technical connecting map is bijective}. Finally the functor $\Pks_H$ is coherent preserving by Proposition \ref{prop:injective sent to correct injective}(c) and Lemma \ref{lem:restriction is finitely presented}.
\end{proof}

\subsection{\texorpdfstring{$n$}{n}-fractured systems of rep\-re\-sen\-ta\-tion-di\-rect\-ed algebras}

In this section we extend the definition of gluing systems in a way compatible with the notion of fracturings developed in \cite{VAS2} to construct an asymptotically weakly $n$-cluster tilting system. Using Theorem \ref{thrm:from asymptotic to n-ct} we show that this extension of the definition induces an $n$-cluster tilting subcategory in the admissible target determined by the gluing system.

\subsubsection{Gluing of \texorpdfstring{$n$}{n}-fractured subcategories of rep\-re\-sen\-ta\-tion-di\-rect\-ed algebras}

Throughout this section let $A$ and $B$ be rep\-re\-sen\-ta\-tion-di\-rect\-ed algebras.

We start by recalling the notions of fracturings and $n$-fractured subcategories introduced in \cite{VAS2}. Recall that if $P$ is a left abutment of $A$, then the quiver $\PD$ is the same as the Aus\-lan\-der--Rei\-ten quiver $\Gamma(\K \overrightarrow{A}_h)$ of $\K \overrightarrow{A}_h$. Hence we can formally view direct sums of modules appearing in $\PD$ as $\K \overrightarrow{A}_h$-modules. With this observation, we have the following definition.

\begin{definition}\cite[Definition 3.4]{VAS2}
Let $W$ be a maximal left abutment of $A$ of height $h$. Let $T^{(W)}$\index[symbols]{TW, TJ@$T^{(W)}$, $T^{(J)}$} be an $A$-module such that all its indecomposable summands are isomorphic to modules appearing in $\nsup{W}{\triangle}$. Then $T^{(W)}$ is called a \emph{fracture}\index[definitions]{fracture of a maximal left abutment} if it is a tilting module when viewed as a $\K \overrightarrow{A}_h$-module.
\end{definition}

Dually we define a \emph{fracture}\index[definitions]{fracture of a maximal right abutment} $T^{(J)}$\index[symbols]{TW, TJ@$T^{(W)}$, $T^{(J)}$} for a maximal right abutment $J$. Before we proceed, let us introduce one more piece of notation. Let $\cU$ and $\cV$ be subcategories of a Krull--Schmidt category $\cA$. We set $\cU_{\setminus \cV}$\index[symbols]{U minus V@$\cU_{\setminus \cV}$} to be the additive closure of all indecomposable objects $x\in \cU$ such that $x\not\in \cV$. 

Next, we need the notion of a fracturing. A left fracturing of $A$ is an $A$-module constructed in the following way. First, we assign a fracture $T^{(W)}$ to each maximal left abutment $W$ of $A$. Then the direct sum of all these fractures is a left fracturing of $A$. More formally we have the following definition.

\begin{definition}\cite[Definition 3.7]{VAS2}
A \emph{left fracturing}\index[definitions]{left fracturing of an algebra} $T^L$ of $A$ is an $A$-module of the form
\[T^L=\hspace{-0.5em}\bigoplus\limits_{W \in \cP_{\text{ind}}^{\text{mab}}}\hspace{-0.5em}T^{(W)}\]
where $T^{(W)}$ is a fracture of $W$. We set  $\cP^L \coloneqq \add \left\{\cP_{\setminus\sABP}, T^L\right\}$\index[symbols]{P d@$\cP^L$}. We define a \emph{right fracturing}\index[definitions]{right fracturing of an algebra} and the subcategory $\cI^R$\index[symbols]{I N@$\cI^R$} similarly. A \emph{fracturing}\index[definitions]{fracturing of an algebra} of $A$ is a pair $(T^L,T^R)$\index[symbols]{(TL, TR)@($T^L$, $T^R$)} where $T^L$ is a left fracturing and $T^R$ is a right fracturing.
\end{definition}

Let $T^L$ be a left fracturing of $A$. For a maximal left abutment $W$ of $A$ and a left abutment $P$ with $P\leq W$, we choose a submodule $T^{(P)}$ of $T^{(W)}$ such that $T^{(P)}\in\cF_P$\index[symbols]{TP, TI@$T^{(P)}$, $T^{(I)}$} and $T^{(W)}=T^{(P)}\oplus Q$ where no indecomposable direct summand of $Q$ is in $\cF_P$. In particular $T^{(P)}$ is isomorphic to the direct sum of all indecomposable summands of $T^{(W)}$ that appear in $\PD$. Notice that $T^{(P)}$ is unique up to isomorphism. Similarly we define $T^{(I)}$\index[symbols]{TP, TI@$T^{(P)}$, $T^{(I)}$} for a right abutment $I\leq J$. 

Now let $T^L_A$ and $T^L_B$ be left fracturings of $A$ and $B$ respectively. Let $P$ be a left abutment of $A$ and $I$ be a right abutment of $B$ such that $\height(P)=\height(I)$ and set $\La=B\glue[P][I] A$. By Remark \ref{rem:gluing remarks}(c) the maximal left abutments of $\La$ are either of the form $\prescript{\La}{B}{\pi}_{\ast}(W_B)$ for some maximal left abutment $W_B$ of $B$ or, if not, they are of the form
$\prescript{\La}{A}{\pi}_{\ast}(W_A)$ for some maximal left abutment $W_A$ of $A$. Moreover, in both cases, it follows by Proposition \ref{prop:AR glued} that $\prescript{\La}{B}{\pi}_{\ast}\left(T^{(B)}\right)$ is a fracture of $\prescript{\La}{B}{\pi}_{\ast}(W_B)$ and that $\prescript{\La}{A}{\pi}_{\ast} \left(T^{(A)}\right)$ is a fracture of $\prescript{\La}{A}{\pi}_{\ast}(W_A)$. Hence if $W_{\La}$ is a maximal left abutment, then we can define a fracture $T_{\ast}^{\left(W_{\La}\right)}$ of $\La$ by
\[T^{\left(W_{\La}\right)}_{\ast} \coloneqq \begin{cases}  \prescript{\La}{B}{\pi}_\ast\left(T^{(W_B)}\right), &\mbox{if $W_{\La}\isom \prescript{\La}{B}{\pi}_{\ast}(W_B)$,}\\  \prescript{\La}{A}{\pi}_\ast\left(T^{(W_A)}\right), &\mbox{otherwise.}\end{cases}\]\index[symbols]{TWLambda, TJLambda@$T^{\left(W_{\La}\right)}_{\ast}$, $T_{\ast}^{(J_{\La})}$}
Notice that $T^{\left(W_{\La}\right)}$ is well-defined by Remark \ref{rem:gluing remarks}(c). With this definition we define a left fracturing of $\La$ by setting
\begin{equation}\label{eq:glued fracturing definition} 
T_{\La}^L \coloneqq \hspace{-0.5em}\bigoplus\limits_{W_{\La} \in \cP_{\text{ind}}^{\text{mab}}}\hspace{-0.5em}T^{(W_{\La})}_{\ast}.
\end{equation}
Dually, given right fracturings $T^R_A$ and $T^R_B$ of $A$ and $B$ we can define a fracture $T_{\ast}^{(J_{\La})}$\index[symbols]{TWLambda, TJLambda@$T^{\left(W_{\La}\right)}_{\ast}$, $T_{\ast}^{(J_{\La})}$} for a maximal right abutment $J_{\La}$ of $\La$ and then define a right fracturing $T_{\La}^R$ of $\La$. Then, by the above considerations it follows that $(T_\La^L, T_\La^R)$ is a fracturing of $\La$, which we call the \emph{gluing of the fracturings $(T_A^L, T_A^R)$ and $(T_B^L, T_B^R)$ at $P$ and $I$}\index[definitions]{gluing of fracturings} and we denote it by $(T_\La^L, T_\La^R)=(T_B^L, T_B^R) \glue[P][I] (T_A^L, T_A^R)$.   

Before we give our next definition, we need the following generalization of \cite[Theorem 1]{VAS}.

\begin{proposition}\label{prop:n-fractured is self-orthogonal}
Let $\cM$ be a subcategory of $\m A$ and $(T^L,T^R)$ be a fracturing of $A$. Then the conditions (a) and (b) are equivalent. 
\begin{enumerate}[label=(\alph*)]
    \item
    \begin{enumerate}
        \item[(a1)] $\cP^L\subseteq \cM$.
        \item[(a2)] $\tn$ and $\tno$ induce mutually inverse bijections
        \[\begin{tikzpicture}
            \node (0) at (0,0) {$\ind\left(\cM_{\setminus\cP^L}\right)$};
            \node (1) at (3,0) {$\ind\left(\cM_{\setminus\cI^R}\right)$.};

            \draw[-latex] (0) to [bend left=20] node [above] {$\tn$} (1);
            \draw[-latex] (1) to [bend left=20] node [below] {$\tno$} (0);
        \end{tikzpicture}\]
        \item[(a3)] $\om^i (M)$ is indecomposable for all indecomposable $M\in \cM_{\setminus\cP^L}$ and $0<i<n$.
        \item[(a4)] $\om^{-i} (N)$ is indecomposable for all indecomposable $N\in \cM_{\setminus\cI^R}$ and $0<i<n$.
    \end{enumerate}
    \item 
    \begin{enumerate}
        \item[(b1)] $\cP^L \subseteq \cM$.
        \item[(b2)] Let $M\in \ind\left(\cM_{\setminus\cP^L}\right)$. Then $\tn(M)\in \cM$. Moreover, for every $X\in\m A$ we have $\Ext^{1\sim n-1}(M,X)\neq 0$ if and only if $\Ext^{1\sim n-1}_A(X,\tn(M))\neq 0$.
        \item[(b3)] Let $N\in \ind\left(\cM_{\setminus\cI^R}\right)$. Then $\tno(N)\in \cM$. Moreover, for every $Y\in\m A$ we have $\Ext^{1\sim n-1}_A(Y,N)\neq 0$ if and only if $\Ext^{1\sim n-1}_A(\tno(N),Y)\neq 0$.
    \end{enumerate}
\end{enumerate}
\end{proposition}

\begin{proof} 
This is a straightforward generalization of the equivalence of (a) and (b) in \cite[Theorem 1]{VAS} using the fact that $\Ext^{1\sim n-1}_A\left(T^{(W)},T^{(W)}\right)=0$ for every $W\in \sMABPind$.
\end{proof}

\begin{definition}\cite[Definition 3.10]{VAS2}\label{def:nctfract}
Assume that $(T^L, T^R)$ is a fracturing of $A$ and let $\cM$ be a full subcategory of $\m A$. Then $\cM$ is called a \emph{$(T^L,T^R,n)$-fractured subcategory}\index[definitions]{n-fractured subcategory@$n$-fractured subcategory} if any of the equivalent conditions of Proposition \ref{prop:n-fractured is self-orthogonal} hold.
\end{definition}

\begin{remark}\label{rem:generalization of n-ct}
This definition generalizes the notion of an $n$-cluster tilting subcategory for rep\-re\-sen\-ta\-tion-di\-rect\-ed algebras. In particular $\cM$ is an $n$-cluster tilting subcategory if and only if it is $(T^L,T^R,n)$-fractured and $T^L\isom \La$ and $T^R\isom D(\La)$ hold or equivalently if and only if $\cP^L=\cP$ and $\cI^R=\cI$. For more details on this, we refer to \cite[Theorem 1]{VAS} and \cite[Proposition 3.11]{VAS2}.
\end{remark}

We also have the following proposition.

\begin{proposition}\label{prop:n-fractured with exts}
Let $(T^L, T^R)$ be a fracturing of $A$ and let $\cM\subseteq \m A$ be a $(T^L,T^R,n)$-fractured subcategory. Then
\begin{align*}
        \cM &= \left\{ X \in \add\{\m A_{\setminus \Sub(T^L)},T^L\} \mid \Ext^{1\sim n-1}_{A}\left(X, \cM\right)=0\right\} \\
        & = \left\{ X \in \add\{\m A_{\setminus \Fac(T^R)}, T^R\} \mid \Ext^{1\sim n-1}_{A}\left(\cM, X\right)=0\right\}. 
    \end{align*}
\end{proposition}

\begin{proof}
This is a straightforward generalization of the implication (a) and (b) implies (c) in \cite[Theorem 1]{VAS}.
\end{proof}

Before giving the next result, let us introduce one more piece of notation. For a module $M$ we write $M\isom P\oplus\underline{M}$ where $\underline{M}$\index[symbols]{M e@$\underline{M}$, $\overline{M}$} has no projective indecomposable direct summand and $M\isom \overline{M}\oplus I$ where $\overline{M}$\index[symbols]{M e@$\underline{M}$, $\overline{M}$} has no injective indecomposable direct summand.

For technical reasons, we focus on the case $n\geq 2$ which imposes no restriction since $1$-cluster tilting subcategories are trivially classified.

We are interested in knowing when gluing two algebras whose module categories admit $n$-fractured subcategories gives an algebra whose module category also admits an $n$-fractured subcategory for the glued fracturings. The following theorem gives the answer to this question. 

\begin{theorem}\cite[Theorem 3.16]{VAS2}\label{thrm:fractsubcat}
Let $n\geq 2$. Let $A$ be a rep\-re\-sen\-ta\-tion-di\-rect\-ed algebra with a fracturing $(T^L_A, T^R_A)$ and let $\cM_A\subseteq \m A$ be a $(T^L_A, T^R_A,n)$-fractured subcategory. Let $W\in\m A$ be a maximal left abutment and let $P \in \cF_W$ be a left abutment of height $h$. Moreover, let $B$ be a rep\-re\-sen\-ta\-tion-di\-rect\-ed algebra with a fracturing $(T^L_B, T^R_B)$ and let $\cM_B\subseteq \m B$ be a $(T^L_B, T^R_B,n)$-fractured subcategory. Let $J\in\m B$ be a maximal right abutment and let $I \in \cG_J$ be a right abutment of height $h$. Set $\La\coloneqq B\glue[P][I] A$. If $\underline{T}_A^{(W)}\in \cF_P$ and $\overline{T}_B^{(J)}\in \cG_I$ and furthermore
\begin{equation} \label{eq:compatibility}
\prescript{\La}{A}{\pi}_{\ast}\left(T_A^{(P)}\right) \isom \prescript{\La}{B}{\pi}_{\ast}\left(T_B^{(I)}\right),
\end{equation}
then $(T_\La^L, T_\La^R)=(T_B^L, T_B^R) \glue[P][I] (T_A^L, T_A^R)$ is a fracturing of $\La$ and  $\cM_{\La}=\add\left\{\prescript{\La}{A}{\pi}_\ast(\cM_A), \prescript{\La}{B}{\pi}_\ast(\cM_B)\right\}\subseteq \m\La$ is a $(T_\La^L, T_\La^R,n)$-fractured subcategory.
\end{theorem}

The easiest case of applying Theorem \ref{thrm:fractsubcat} is when both $\cM_A$ and $\cM_B$ are actual $n$-cluster tilting subcategories and $P$ and $I$ have height $1$, that is they are simple projective and simple injective modules respectively. In this case, it follows that the assumptions of Theorem \ref{thrm:fractsubcat} are trivially satisfied. Furthermore, in this case, it is shown in \cite[Corollary 3.17]{VAS2} that $\cM_{\La}$ is an actual $n$-cluster tilting subcategory. For a generalization, see Proposition \ref{prop:gluing system of n-ct}.

Let $(\La_{v},P_{e},I_{e})_{v\in V_G,e\in E_G}$ be a gluing system. Our next step is to consider for each $v\in V_G$ an $n$-fractured subcategory $\cM_v\subseteq \m\La_v$ such that that for every $e:u\to v$ in $E_G$, the conditions of Theorem \ref{thrm:fractsubcat} are satisfied. Before give the precise definition, let us introduce one more piece of notation. Let $e:u\to v$ be an arrow in $G$. We denote by $W_e$ the unique (up to isomorphism) maximal left abutment of $\La_v$ such that $P_e\leq W_e$. Similarly we denote by $J_e$ the unique (up to isomorphism) maximal right abutment of $\La_u$ such that $I_e\leq J_e$. 

\begin{definition}\label{def:n-fractured system}
Let $n\geq 2$. Let $\cL=(\La_{v},P_{e},I_{e})_{v\in V_G,e\in E_G}$ be a gluing system on a directed tree $G$. For every $v\in V_G$, let $\left(T^L_v, T^R_v\right)$ be a fracturing of $\La_v$ and $\cM_v$ be a $(T^L_v,T^R_v,n)$-fractured subcategory of $\m\La_v$. We call $(\cM_v)_{v\in V_G}$\index[symbols]{(Mv)@$(\cM_v)_{v\in V_G}$} an \emph{$n$-fractured system}\index[definitions]{n-fractured system@$n$-fractured system} of $\cL$ if for every arrow $e:u\to v$ in $E_G$ we have that 
\begin{enumerate}[label=(\roman*)]
    \item $\underline{T}_A^{(W_{e})}\in\cF_{P_e}$ and $\overline{T}_B^{(J_{e})}\in \cG_{I_e}$, and
    \item $\prescript{\langle u,v\rangle}{\langle u\rangle}{\pi}_{\ast}\left(T_A^{(P_{e})}\right) \isom \prescript{\langle u,v\rangle}{\langle v \rangle}{\pi}_{\ast}\left(T_B^{(I_{e})}\right)$.
\end{enumerate}
In particular, we have that $(T^L_{\langle u,v\rangle}, T^{R}_{\langle u,v\rangle})\coloneqq\left(T^L_u, T^R_u\right)\glue \left(T^L_v, T^R_v\right)$ is a fracturing of $\La_{\langle u,v\rangle}$ and 
\[\cM_{\langle u,v\rangle}\coloneqq\add\left\{\prescript{\langle u,v\rangle}{\langle v\rangle}{\pi}_{\ast}\left(\cM_v\right), \prescript{\langle u,v \rangle}{\langle u \rangle}{\pi}_{\ast}\left(\cM_u\right)\right\}\subseteq \m\La_{\langle u,v\rangle}\]
is a $(T^L_{\langle u,v\rangle}, T^{R}_{\langle u,v\rangle}, n)$-fractured subcategory of $\m \La_{\langle u,v\rangle}$. If moreover for every $v\in V_G$ we have that
\begin{enumerate}
    \item[(iii)] for each maximal right abutment $J$ of $\La_u$ such that $T_u^{(J)}$ is not injective, there exists one arrow $e\in E_G$ with source $u$ such that $J_e\isom J$, and
    \item[(iv)] for each maximal left abutment $W$ of $\La_v$ such that $T_v^{(W)}$ is not projective, there exists one arrow $e\in E_G$ with target $v$ such that $W_e\isom W$, 
\end{enumerate}
then we say that the $n$-fractured system $(\cM_v)_{v\in V_G}$ is \emph{complete}\index[definitions]{complete $n$-fractured system}.
\end{definition}

Conditions (i) and (ii) of an $n$-fractured system allow us to glue $n$-fractured subcategories over two arrows in $G$ in any order. More precisely, let $u,v,w\in V_G$ be three vertices such that $\langle u,v,w\rangle$ is connected. A case by case analysis of the graphs \ref{case:C1}, \ref{case:C2} and \ref{case:C3} shows that the two different fracturings obtained depending on the order of the gluing coincide. More generally, if $H\in \dirI_G$ we denote by $\left(T_H^L,T_H^R\right)$ the fracturing of $\La_H$ obtained by repeatedly performing the gluings of fracturings induced by the edges $E_H$ of $H$ in any order. Similarly we denote by $\cM_H$ the $\left(T_H^L,T_H^R,n\right)$-fractured subcategory of $\m\La_H$ defined by $\cM_{H} = \add\left\{\prescript{H}{v}{\pi}_{\ast}\left(\cM_v\right) \mid v\in V_H\right\}$.

If $(\cM_v)_{v\in V_G}$ is an $n$-fractured system of a gluing system $\cL=(\La_{v},P_{e},I_{e})_{v\in V_G,e\in E_G}$, we set \[\Mks \coloneqq \add \left\{ \Pks_{v\ast}(\cM_v) \mid v\in V_G\right\} \subseteq \m\Cks.\]\index[symbols]{M c@$\Mks$}
Our aim is to show that if the $n$-fractured system $(\cM_v)_{v\in V_G}$ is complete, then $\Mks$ is an $n$-cluster tilting subcategory of $\m\Cks$. We start with the following proposition.

\begin{proposition}\label{prop:n-fractured is weakly n-ct}
Let $n\geq 2$, let $\cL=(\La_{v},P_{e},I_{e})_{v\in V_G,e\in E_G}$ be a gluing system on a directed tree $G$ and let $(\cM_v)_{v\in V_G}$ be an $n$-fractured system of $\cL$.
\begin{enumerate}[label=(\alph*)]
    \item For every $H\in\dirI_G$, we have $\Ext_{\La_H}^{1\sim n-1}\left(\cM_H,\cM_H\right)=0$.
    
    \item Let $H,K\in \dirI_G$ with $H\subseteq K$. Then $\prescript{K}{H}{\pi}_{\ast}(\cM_H)\subseteq \cM_K$.
\end{enumerate}
If moreover the system $(\cM_v)_{v\in V_G}$ is complete, then the following also hold.
\begin{enumerate}
    \item[(c)] If $M\in \m\La_H$ and $\Ext^{1\sim n-1}_{\La_K}\left(\prescript{K}{H}{\pi}_{\ast}(M),\cM_K\right)=0$ for all $K\in \dirI_G$ with $H\subseteq K$, then $M\in \cM_H$.
    
    \item[(d)] If $M\in \m\La_H$ and $\Ext^{1\sim n-1}_{\La_K}\left(\cM_K,\prescript{K}{H}{\pi}_{\ast}(M)\right)=0$ for all $K\in \dirI_G$ with $H\subseteq K$, then $M\in \cM_H$.
\end{enumerate}
\end{proposition}

\begin{proof}
\begin{enumerate}[label=(\alph*)]
    \item Follows from Proposition \ref{prop:n-fractured with exts} since $\cM_H$ is a $\left(T^L_H, T^R_H,n\right)$-fractured subcategory of $\m\La_H$.
    
    \item Using (\ref{eq:connecting formula for subgraphs}) we have
    \begin{align*}
        \prescript{K}{H}{\pi}_{\ast}(\cM_H) &= \prescript{K}{H}{\pi}_{\ast}\left(\add\left\{\prescript{H}{v}{\pi}_{\ast}(\cM_v) \mid v\in V_H\right\}\right) \\ 
        &=\add\left\{\prescript{K}{H}{\pi}_{\ast}\prescript{H}{v}{\pi}_{\ast}(\cM_v)\mid v\in V_H\right\} \\
        &=\add\left\{\prescript{K}{v}{\pi}_{\ast}(\cM_v)\mid v\in V_H\right\} \\
        &\subseteq \add\left\{\prescript{K}{v}{\pi}_{\ast}(\cM_v)\mid v\in V_K\right\} \\
        &=\cM_K.
    \end{align*}
    
    \item Let $M\in\m\La_H$ and assume that for all $K\in\dirI_G$ with $H\subseteq K$, we have \[\Ext^{1\sim n-1}_{\La_K}\left(\prescript{K}{H}{\pi}_{\ast}(M),\cM_K\right)=0.\] 
    We show that $M\in\cM_H$. By additivity of $\Ext$ and since $\cM_H$ is closed under direct sums and summands, we may assume that $M$ is indecomposable. Let $K\in\dirI_G$ with $H\subseteq K$. Since $\prescript{K}{H}{\pi}_{\ast}$ is exact, for all $r\in\{1,\dots,n-1\}$ we have
    \[\Ext_{\La_H}^{r}\left(M,\cM_H\right) \isom \Ext_{\La_K}^{r}\left(\prescript{K}{H}{\pi}_{\ast}(M),\prescript{K}{H}{\pi}_{\ast}(\cM_H)\right).\]
    Since $\prescript{K}{H}{\pi}_{\ast}(\cM_H)\subseteq \cM_K$ by (b), we conclude that $\Ext_{\La_H}^{1\sim n-1}\left(M,\cM_H\right) = 0$. Hence by Proposition \ref{prop:n-fractured with exts} it is enough to show that $M\in\add\left\{{\m \La_H}_{\setminus \Sub(T_H^L)}, T_H^L\right\}$. 
    
    Assume towards a contradiction that $M\in \Sub(T^L_H)$ but $M\not\in\add(T^L_H)$. 
    It follows that $M$ is isomorphic to an indecomposable submodule of a fracture $T^{(W_H)}$ of a maximal left abutment $W_H$ of $\La_H$. We claim that $\Ext^1_{\La_H}\left(T^{(W_H)},M\right)\neq 0$. 
    
    To show this set $h\coloneqq\height(W_H)$. We view $M$ as a $\K \overrightarrow{A}_{h}$- module and $T^{(W_H)}$ as a tilting $\K \overrightarrow{A}_{h}$-module. Under this identification, it is enough to show that $\Ext^1_{\K \overrightarrow{A}_{h}}\left(T^{(W_H)},M\right)\neq 0$. Assume to a contradiction that $\Ext^1_{\K \overrightarrow{A}_{h}}\left(T^{(W_H)},M\right)= 0$. Then we have that $M\in\Fac\left(T^{(W_H)}\right)$ by general tilting theory, see for example \cite[Theorem VI.2.5]{ASS}. Since we also have $M\in \Sub\left(T^{(W_H)}\right)$, we have nonzero maps $T_0\twoheadrightarrow M \hookrightarrow T_1$ where $T_0,T_1\in\add\left(T^{(W_H)}\right)$. Without loss of generality, we assume that $T_0$ and $T_1$ are indecomposable. Indeed, the Aus\-lan\-der--Rei\-ten quiver $\Gamma(\K\overrightarrow{A}_h)$ is described in Proposition \ref{prop:abutments in AR-quiver}. It follows by this description that if no indecomposable summand of $T_0$ has an epimorphism onto $M$ then all morphisms from $T_0$ to $M$ factor through $\rad(M)$; similarly for $T_1$. Then, since $M\not\in\add(T^L_H)$, we have $M\not\isom T_0$ and $M\not\isom T_1$. Again by the description of $\Gamma(\K \overrightarrow{A}_h)$ in Proposition \ref{prop:abutments in AR-quiver} it follows that there exists a nonzero map $\tau^-(T_0)\to T_1$ which does not factor through a projective $\K\overrightarrow{A}_h$-module. Hence by the Aus\-lan\-der--Rei\-ten duality \cite[Theorem IV.2.13]{ASS} we have
    \[\Ext^1_{\K \overrightarrow{A}_{h}}(T_1,T_0)\isom \D \underline{\Hom}_{\K\overrightarrow{A}_h}(\tau^{-}(T_0),T_1)\neq 0,\] 
    which contradicts the fact that $T^{(W_H)}$ is a tilting $\K \overrightarrow{A}_{h}$-module.
    
    Hence indeed $\Ext^1_{\La_H}(T^{(W_H)},M)\neq 0$. In particular there exists a nonprojective indecomposable summand $T$ of $\underline{T}^{(W_H)}$ such that $\Ext^1_{\La_H}(T,M)\neq 0$. Let $\epsilon_t\La_H$ be the projective cover of $T$ for some $t\in \left(Q_H\right)_0$ and set $K\coloneqq\langle V_H, V(t)\rangle$. We now claim that $\prescript{K}{H}{\pi}_{\ast}(T)\not\in \cP^L_K$. 
    
    We again show this claim by assuming instead towards a contradiction that $\prescript{K}{H}{\pi}_{\ast}(T)\in \cP^L_K$. Then since $\prescript{K}{H}{\pi}_{\ast}(T)$ is not projective, it follows that it appears in a fracture $T^{(W_K)}$ as a nonprojective indecomposable summand. Notice that by construction the maximal left abutment $W_K$ and the fracture $T^{(W_K)}$ are of the form $W_K\isom \prescript{K}{v}{\pi}_{\ast}(W_v)$ and $T^{(W_K)}\isom \prescript{K}{v}{\pi}_{\ast}\left(T^{(W_v)}\right)$ for some vertex $v\in V_K$ and some maximal left abutment $W_v$ of $\La_v$. 
    Since $T^{(W_v)}$ is not projective and since the system is complete, it follows by Definition \ref{def:n-fractured system}(iv) that there exists an arrow $e\in E_G$ with target $v$ such that $W_e\isom W_v$. Without loss of generality and by Definition \ref{def:gluing system}(iii) we can assume that the gluing induced by $e$ is not trivial. Moreover, by Definition \ref{def:n-fractured system}(i) we have that $\underline{T}^{(W_K)}\in\cF_{P_e}$. In particular we have $t\in\supp(T)\subseteq\supp(P_e)$ and so $s(e)\in V(t)$. But this contradicts the gluing being nontrivial since $V(t)\subseteq V_K$.
    
    Hence we have that $\prescript{K}{H}{\pi}_{\ast}(T)\not\in\cP^L_K$. Since $H\subseteq K$ we have
    \begin{equation}\label{eq:zero ext later}
    \Ext^{1\sim n-1}_{\La_K}\left(\prescript{K}{H}{\pi}_{\ast}(M),\cM_K\right)=0.
    \end{equation}
    Moreover, by exactness of $\prescript{K}{H}{\pi}_{\ast}$, we also have
    \[\Ext^{1}_{\La_K}\left(\prescript{K}{H}{\pi}_{\ast}(T),\prescript{K}{H}{\pi}_{\ast}(M)\right) \isom \Ext^1_{\La_H}(T,M) \neq 0.\]
    In particular, since $\prescript{K}{H}{\pi}_{\ast}(T)\in\ind\left(\left(\cM_K\right)_{\setminus \cP^L_K}\right)$, we have by Proposition \ref{prop:n-fractured is self-orthogonal}(b2) that
    \[\Ext^{1\sim n-1}_{\La_K}\left(\prescript{K}{H}{\pi}_{\ast}(M), \tn\prescript{K}{H}{\pi}_{\ast}(T)\right)\neq 0.\]
    But by Proposition \ref{prop:n-fractured is self-orthogonal}(a2) we have that $\tn\prescript{K}{H}{\pi}_{\ast}(T)\in \cM_K$, which contradicts (\ref{eq:zero ext later}). It follows that $M\in \add\left\{\m A_{\setminus \Sub(T^L)}, T^L\right\}$, as required.
    
    \item Similar to (c). \qedhere
\end{enumerate}
\end{proof}

The main result of this section is the following theorem.

\begin{theorem}\label{thrm:complete n-fractured gives n-cluster tilting}
Let $n\geq 2$, let $\cL=(\La_{v},P_{e},I_{e})_{v\in V_G,e\in E_G}$ be a gluing system on a directed tree $G$ and let $(\Cks,\Pks_H,\Tks_{HK})$ be the firm source of the induced $\Cat$-inverse system $\left(\cC_H, F_{HK}, \theta_{HKL}\right)$ over $\dirI_G$. Let $(\cM_v)_{v\in V_G}$ be a complete $n$-fractured system of $\cL$ and identify $\cM_v$ with its equivalent subcategory in $\m\cC_H \equivalent \m\La_H$. Then $\left(\cM_H\right)$ is an asymptotically weakly $n$-cluster tilting system of $\left(\m\cC_H, F_{HK\ast}, \theta_{HKL\ast}\right)$ and $\Mks =\add \left\{ \Pks_{v\ast}(\cM_v) \mid v\in V_G\right\} \subseteq \m\Cks$ is an $n$-cluster tilting subcategory. If moreover $G$ is finite, then $\cM_G$ is an $n$-cluster tilting subcategory of $\m\La_G$.
\end{theorem}

\begin{proof}
That $\left(\cM_H\right)$ is an asymptotically weakly $n$-cluster tilting system of  $\left(\m\cC_H, F_{HK\ast}, \theta_{HKL\ast}\right)$ follows immediately by Proposition \ref{prop:n-fractured is weakly n-ct}. To show that $\Mks$ is an $n$-cluster tilting subcategory it is enough to show that the conditions of Corollary \ref{cor:n-ct if locally bounded} are satisfied. We have that $\m\Cks$ is locally bounded by Proposition \ref{prop:indecomposables in firm source}(f). Finally, for every $H\in\dirI_G$ the subcategory $\cM_H\subseteq \m\La_H$ is functorially finite as it is a subcategory of the module category of a rep\-re\-sen\-ta\-tion-fi\-nite algebra.

If $G$ is finite, then notice that we have $\Qks=Q_G$ and $\Rks=\cR_G$. Then, using Corollary \ref{cor:equivalence of quiver and projectives} we have
    \begin{align*}
        \m\Cks &\equivalent \m(\proj(\K \Qks/\Rks)) =\m(\proj(\K Q_G/\cR_G) \equivalent \m(\K Q_G/\cR_G) =\m \La_G,
    \end{align*}
    and the above equivalence restricts to an equivalence $\Mks\equivalent \cM_G$. It follows that $\cM_G$ is an $n$-cluster tilting subcategory of $\m\La_G$.
\end{proof}

\section{Applications}\label{sec:applications}

In this section we apply the results of Section \ref{sec:asymptotically weakly n-cluster tilting systems from algebras} to produce abelian categories that admit $n$-cluster tilting subcategories. In many cases we show how we can obtain actual fi\-nite-di\-men\-sion\-al algebras such that their module category admits an $n$-cluster tilting subcategories.

\subsection{Complete \texorpdfstring{$n$}{n}-fractured systems from \texorpdfstring{$n$}{n}-cluster tilting subcategories}

In this short section we show how we can construct a complete $n$-fractured system starting from rep\-re\-sen\-ta\-tion-di\-rect\-ed algebras whose module categories admit $n$-cluster tilting subcategories.

For a bound quiver algebra $\La=\K Q/\cR$, we denote by $\sources_{\La}=\sources$\index[symbols]{sLambda@$\sources_{\La}$} the number of sources in the quiver $Q$ and by $\sinks_{\La}=\sinks$\index[symbols]{T a@$\sinks_{\La}$} the number of sinks in the quiver $Q$. 

Using Theorem \ref{thrm:complete n-fractured gives n-cluster tilting}, we have the following proposition which gives a way of constructing $n$-cluster tilting subcategories which generalizes \cite[Corollary 3.17]{VAS2}.

\begin{proposition}\label{prop:gluing system of n-ct}
Let $G$ be a directed tree and $n\geq 2$. Let $\cL=(\La_{v},P_{e},I_{e})_{v\in V_G,e\in E_G}$ be a gluing system on $G$ such that $P_e$ and $I_e$ are simple for every $e\in E_G$. Let also $\cM_v\subseteq \m\La_v$ be an $n$-cluster tilting subcategory. Then the following statements hold.
\begin{enumerate}[label=(\alph*)]
    \item The collection $\left(\cM_{v}\right)_{v\in V_G}$ is a complete $n$-fractured system of $\cL$.
    
    \item $\Mks$ is an $n$-cluster tilting subcategory of $\m\Cks$.
    
    \item If moreover $G$ is finite, then $\cM_G$ is an $n$-cluster tilting subcategory of $\m\La_G$.
\end{enumerate}
\end{proposition}

\begin{proof}
We only need to prove (a) since then (b) and (c) follow immediately by Theorem \ref{thrm:complete n-fractured gives n-cluster tilting}. We show that the conditions of Definition \ref{def:n-fractured system} are satisfied for $\left(\cM_v\right)_{v\in V_G}$. Since $\cM_v$ is an $n$-cluster tilting subcategory, it follows that $\cM_v$ is a $\left(\La_v, D(\La_v), n\right)$-fractured subcategory by Remark \ref{rem:generalization of n-ct}. In particular, every left fracture is either projective or injective and hence conditions (i)-(iv) are trivially satisfied.
\end{proof}

\begin{remark}\label{rem:gluing system needs sources and sinks}
We explain how to find a gluing system that satisfies the assumptions of Proposition \ref{prop:gluing system of n-ct}. Let $G$ be a directed tree and $n\geq 2$. For each vertex $v\in V_G$ let $\La_v=\K Q_v/\cR_v$ be a rep\-re\-sen\-ta\-tion-di\-rect\-ed algebra whose module category admits an $n$-cluster tilting subcategory $\cM_v$. Assume that $\La_v\not\isom \K\overrightarrow{A}_1$. Assume moreover that $\sinks_{\La_v}\geq \delta^{-}(v)$ and $\sources_{\La_v}\geq \delta^{+}(v)$. Notice that if $s\in\left(Q_v\right)_0$ is a sink, then the simple module $S(s)$ is a left abutment of $\La_v$ of height $1$. Moreover, since $\supp(S(s))=\{s\}$, different sinks of $Q_v$ correspond to independent left abutments of $\La_v$. Let $\{e_1,\dots,e_{\delta^{-}(v)}\}$ be the collection of arrows terming at $v$. Since $\sinks_{\La_v} \geq \delta^{-}(v)$, it follows there exists a set $\{s_1,\dots,s_{\delta^-(v)}\}$ of sinks in $Q_v$ with $s_i\neq s_j$ for $i\neq j$. Set $P_{e_i}\coloneqq S(s_i)$ for $1\leq i \leq \delta^-(v)$. By construction we have that $\{P_{e_i}\}_{i=1}^{\delta^{-}(v)}$ is an independent collection of simple left abutments of $\La_v$. Similarly if $\{e_1',\dots,e_{\delta^+(v)}'\}$ is the collection of arrows starting at $v$, we can construct an independent collection $\{I_{e_i}\}_{i=1}^{\delta^{+}(v)}$ of simple right abutments of $\La_v$.

We claim that the triple $\cL=(\La_{v},P_{e},I_{e})_{v\in V_G,e\in E_G}$ is a gluing system on $G$. We check that the conditions of Definition \ref{def:gluing system} are satisfied. Condition (i) is satisfied by construction. Condition (ii) is satisfied since each abutments is simple and hence of height $1$. Finally condition (iii) is satisfied since $\La_v\not\isom \K\overrightarrow{A}_1$, and hence no gluing is trivial. Clearly $\cL$ satisfies the assumptions of Proposition \ref{prop:gluing system of n-ct}. 
\end{remark}

Notice that to apply Remark \ref{rem:gluing system needs sources and sinks} for any directed tree $G$, we need as an input bound quiver algebras whose module categories admit $n$-cluster tilting subcategories and such that their quiver has an arbitrary number of sinks and sources. We construct such families of examples in section \ref{subsection:quivers with any number of sinks and sources}.

\subsection{Gluings and orbit categories}\label{sec:gluings and orbit categories}

Let $n\geq 2$. Let $(\La_{v},P_{e},I_{e})_{v\in V_G,e\in E_G}$ be a gluing system on a directed tree $G$ and let $(\cM_v)_{v\in V_G}$ be a complete $n$-fractured system of $\cL$. By Corollary \ref{cor:equivalence of quiver and projectives} we have the equivalence $\Cks \equivalent \proj(\K \Qks /\Rks)$ and by Theorem \ref{thrm:complete n-fractured gives n-cluster tilting} we have that $\Mks\subseteq\m\Cks$ is an $n$-cluster tilting subcategory. In general, the quiver $\Qks$ is infinite and hence $\Mks$ cannot be realized as the $n$-cluster tilting subcategory of the module category of a fi\-nite-di\-men\-sion\-al algebra. However, by taking our gluing system to have sufficient symmetry, we construct an orbit category $\Cks/\ZZ$ such that $\m(\Cks/\ZZ)\equivalent \m\tilLa$, where $\tilLa$ is a fi\-nite-di\-men\-sion\-al algebra which is not necessarily rep\-re\-sen\-ta\-tion-di\-rect\-ed. Furthermore, if the $n$-fractured system also has sufficient symmetry, we end up with an $n$-cluster tilting subcategory of $\m\tilLa$ as well. In this section we make these notions precise. We start by recalling some background on orbit categories from \cite{DI}.

Let $\cC$ be a skeletally small $\K$-linear category and let $\ccG$\index[symbols]{G@$\ccG$} be a group. A \emph{($\K$-linear) $\ccG$-action}\index[definitions]{k-linear G-action@$\K$-linear $\ccG$-action} on $\cC$ is a collection $\left\{F_g\in \Aut(\cC)\mid g\in \ccG\right\}$ of ($\K$-linear) automorphisms $F_g:\cC\to\cC$ such that $F_g\comp F_h=F_{gh}$ for all $g,h\in G$. Notice that a $\ccG$-action on $\cC$ induces a $\ccG$-action on $\m\cC$ via the collection $\left\{F_{g\ast}:\m\cC\to\m\cC \mid g\in \ccG\right\}$. For simplicity we write $g(x)$ instead of $F_g(x)$ for $x\in\cC$ and we write $g_{\ast}(M)$ instead of $F_{g\ast}(M)$ for $M\in\m\cC$. If $\cC$ is also a Krull--Schmidt category, a $\ccG$-action on $\cC$ is called \emph{admissible}\index[definitions]{admissible $\ccG$-action} if $g(x)\not\isom x$ for every $x\in\ind\cC$ and $g\in \ccG\setminus\{1\}$.  A subcategory $\cU$ of $\m\cC$ is called \emph{$\ccG$-equivariant}\index[definitions]{G-equivariant subcategory@$\ccG$-equivariant subcategory} if $g_{\ast}(\cU)=\cU$ for all $g\in \ccG$. 

Let $\cC$ be a locally bounded $\K$-linear Krull--Schmidt category and $\ccG$ a group acting admissibly on $\cC$. The \emph{orbit category}\index[definitions]{orbit category} $\cC/\ccG$ is the category given by the following data.
\begin{enumerate}
    \item[$\bullet$] The objects of $\cC/\ccG$\index[symbols]{C nG@$\cC/\ccG$} are the objects of $\cC$, that is $\obj(\cC/\ccG)\coloneqq \obj(\cC)$.
    \item[$\bullet$] Morphisms in $\cC/\ccG$ between $x\in\cC/\ccG$ and $y\in\cC/\ccG$ are collections of morphisms between $x\in\cC$ and $g(y)\in\cC$, that is $\cC/\ccG(x,y)\coloneqq \bigoplus_{g\in \ccG}\cC(x,g(y))$.
    \item[$\bullet$] Composition of morphisms for $a=(a_g)_{g\in \ccG}\in\cC/\ccG(x,y)$ and $b=(b_g)_{g\in \ccG}\in\cC/\ccG(y,z)$ is defined in the following way. First, for every $h\in \ccG$ we have a morphism $a_h:x\to h(y)$. Next for every $g\in \ccG$ we have a morphism $b_{h^{-1}g}:y\to h^{-1}g(z)$. Acting by $h$ on the morphism $b_{h^{-1}g}$ we get a morphism $h\left(b_{h^{-1}}\right):h(y)\to h(h^{-1}g)(z)=g(z)$. Then we set
    \[(ba)_g \coloneqq \sum_{h\in \ccG}h\left(b_{h^{-1}g}\right)a_h : x\to g(z),\]
    and so $ba\in \cC/\ccG(x,z)$.
\end{enumerate}
The natural functor $F:\cC\to \cC/\ccG$\index[symbols]{F C to G@$F:\cC\to \cC/\ccG$} induces the usual adjoint pair $\left(F^{\ast},F_{\ast}\right)$ with $F^{\ast}:\M (\cC/\ccG)\to \M \cC$ and $F_{\ast}:\M\cC\to \M(\cC/\ccG)$. As usual, the functor $F^{\ast}$ restricts to a functor between the categories of finitely presented modules and moreover, in this case, the functor $F^{\ast}$ is also exact. We warn the reader that we use the opposite notation of \cite{DI}, where the left adjoint is denoted by $F_{\ast}$ and the right adjoint by $F^{\ast}$.

We now recall one of the main results of \cite{DI}.

\begin{theorem}\cite[Corollary 2.15]{DI}\label{thrm:Darpo-Iyama} Let $\cC$ be a locally bounded $\K$-linear Krull--Schmidt category and $\ccG$ a free abelian group of finite rank acting admissibly on $\cC$. If $\cM\subseteq \m\cC$ is a locally bounded $\ccG$-equivariant $n$-cluster tilting subcategory, then $F^{\ast}(\cM)\subseteq \m(\cC/\ccG)$ is a locally bounded $n$-cluster tilting subcategory.
\end{theorem}

To apply Theorem \ref{thrm:Darpo-Iyama} we introduce the following notion.

\begin{definition}\label{def:self-gluable}
Let $n\geq 2$ and let $\La=\K Q/\cR$ be a rep\-re\-sen\-ta\-tion-di\-rect\-ed algebra with a fracturing $(T^L,T^R)$ and a $\left(T^L,T^R,n\right)$-fractured subcategory $\cM$. 
\begin{enumerate}[label=(\alph*)]
    \item A pair $(W,J)$\index[symbols]{(W, J)@$(W,J)$} of $\La$-modules is called a \emph{fractured pair (with respect to the fracturing $(T^L,T^R)$)}\index[definitions]{fractured pair} if the following conditions hold.
    \begin{enumerate}[label=(\roman*)]
        \item $W$ is a maximal left abutment of $\La$.
        \item If $W'$ is a maximal left abutment of $\La$ with $W\not\isom W'$, then $T^{(W')}$ is projective.
        \item $J$ is a maximal right abutment of $\La$.
        \item If $J'$ is a maximal right abutment of $\La$ with $J\not\isom J'$, then $T^{(J')}$ is injective. 
    \end{enumerate}
    
    \item Let $(W,J)$ be a fractured pair of $\La$. A pair $(P,I)$\index[symbols]{(P, I)@$(P,I)$} of $\La$-modules is called a \emph{compatible pair}\index[definitions]{compatible pair for a fractured pair} for the fractured pair $(W,J)$ if the following conditions hold.
    \begin{enumerate}[label=(\roman*)]
        \item $P$ is a left abutment of $\La$ such that $P\leq W$. 
        \item $\underline{T}^{(W)}\in \cF_{P}$.
        \item $I$ is a right abutment of $\La$ such that $I\leq J$
        \item $\overline{T}^{(J)}\in\cG_{I}$.
        \item We have $\height(P)=\height(I)$ and, furthermore, $\pi_{P}^{\ast}\left(T^{(P)}\right) \isom \pi_{I}^{!}\left(T^{(I)}\right)$.
    \end{enumerate}
    
    \item $\La$ is called \emph{$n$-self-gluable}\index[definitions]{n-self-gluable algebra@$n$-self-gluable algebra} if there exists a fractured pair $(W,J)$ and a compatible pair $(P,I)$ of $(W,J)$.
\end{enumerate}
\end{definition}

\begin{example}\label{ex:n-self-gluable easy}
Any rep\-re\-sen\-ta\-tion-di\-rect\-ed algebra such that its module category admits an $n$-cluster tilting subcategory is $n$-self-gluable. Indeed, let $\La$ be a rep\-re\-sen\-ta\-tion-di\-rect\-ed algebra such that $\m\La$ admits an $n$-cluster tilting subcategory $\cM$. In particular, the subcategory $\cM$ is a $\left(\La,D(\La),n\right)$-fractured subcategory. Let $W$ be a maximal left abutment of $\La$ and let $J$ be a maximal right abutment of $\La$. It then follows trivially that $(W,J)$ is a fractured pair. Moreover, let $P\leq W$ be a simple left abutment of $\La$ and $I\leq J$ be a simple right abutment of $\La$. Again it follows trivially that $(P,I)$ is a compatible pair. In particular, the algebra $\La$ is $n$-self-gluable.
\end{example}

Let $\La=\K Q/\cR$ be an $n$-self-gluable algebra. For each integer $z\in\ZZ$ let $Q[z]$ be a copy of $Q$ with vertices and arrows labelled as follows.
\begin{align*}
\left(Q[z]\right)_0 &=\{i[z] \mid i\in Q_0\},\\
\left(Q[z]\right)_1 &= \{\alpha[z]:i[z]\to j[z] \mid \alpha\in Q_1 \text{ with $\alpha:i\to j$}\}. 
\end{align*}
Denote by $\cR[z]$ the ideal of $\K Q[z]$ corresponding to the ideal $\cR$ of $\K Q$. We set $\La[z]\coloneqq \K Q[z]/\cR[z]$\index[symbols]{Lambda z@$\La[z]$}\index[symbols]{KQz Rz@$\K Q[z]/\cR[z]$}. For a module $M\in \m\La$ denote the corresponding module in $\m\La[z]$ by $M[z]$. In particular, for the $(T^L,T^R,n)$-fractured subcategory $\cM\subseteq \m\La$ we have that the corresponding subcategory $\cM[z]\subseteq \m\La[z]$\index[symbols]{M d@$\cM[z]$} is a $(T^L[z],T^R[z],n)$-fractured subcategory.

\begin{proposition}\label{prop:self-gluing system}
Let $\La=\K Q/\cR$ be an $n$-self gluable algebra and $(W,J)$ a fractured pair of $\La$ with a compatible pair $(P,I)$ and assume that $\La\not\isom \K \overrightarrow{A}_{h'}$ for any $h'$. Let $G=\overrightarrow{A}_{\infty}^{\infty}$\index[symbols]{A infinity@$\overrightarrow{A}_{\infty}^{\infty}$} be the graph
\[\cdots \xrightarrow{\alpha_{-3}} -2 \xrightarrow{\alpha_{-2}} -1 \xrightarrow{\alpha_{-1}} 0 \xrightarrow{\hspace*{3.3pt}\alpha_{0}\hspace*{3.3pt}} 1 \xrightarrow{\hspace*{3.3pt}\alpha_{1}\hspace*{3.3pt}} 2 \xrightarrow{\hspace*{3.3pt}\alpha_{2}\hspace*{3.3pt}} \cdots. \]
Then $V_G=\ZZ$ and for every $z\in\ZZ$ set $\La_z\coloneqq\La[z]$, $I_{\alpha_z}\coloneqq I[z]$, $P_{\alpha_z}\coloneqq P[z+1]$ and $\cM_z\coloneqq\cM[z]$. Then $\La^{\infty}_{\infty}\coloneqq \left(\La_{z},I_{\alpha_{z}},P_{\alpha_{z}}\right)_{z\in\ZZ}$\index[symbols]{Lambda aaa@$\La^{\infty}_{\infty}$} is a gluing system on $G$ and $(\cM_{z})_{z\in\ZZ}$ is a complete $n$-fractured system of $\La^{\infty}_{\infty}$.
\end{proposition}

\begin{proof}
To show that $\La^{\infty}_{\infty}$ is a gluing system, we show that the conditions of Definition \ref{def:gluing system} are satisfied. Condition (i) is satisfied trivially since every vertex of $v\in V_{G}$ has $\delta^{+}(v)=\delta^{-}(v)=1$. Condition (ii) is satisfied since by Definition \ref{def:self-gluable}(ii) we have
\[\height\left(I_{\alpha_z}\right)=\height(I[z]) = \height(I) = \height(P) = \height(P[z+1])=\height\left(P_{\alpha_z}\right).\]
Condition (iii) is satisfied since $\La\not\isom \K\overrightarrow{A}_{h'}$. 

To show that $(\cM_{z})_{z\in\ZZ}$ is a complete $n$-fractured system, we show that the conditions of Definition \ref{def:n-fractured system} are satisfied. Conditions (i) and (ii) follow immediately by Definition \ref{def:self-gluable}(ii). Moreover, by Definition \ref{def:self-gluable}(i), we have that $(W,J)$ is a fractured pair of $\La$. It follows that $(W[z],J[z])$ is a fractured pair of $\La[z]$ for every $z\in\ZZ$. In particular there is at most one maximal left abutment of $\La[z]$ with a nonprojective fracture, namely $W[z]$, and there is at most one maximal right abutment of $\La[z]$ with a noninjective fracture, namely $J[z]$. Since $W_{\alpha_z}\isom W[z+1]$ and $J_{\alpha_z}\isom J[z]$, conditions (iii) and (iv) are satisfied.
\end{proof}

For the rest of this section we fix an $n$-self gluable algebra $\La\not\isom\K\overrightarrow{A}_h$ and let $G$, $\La^{\infty}_{\infty}$ and $(\cM_{z})_{z\in\ZZ}$ be as in Proposition \ref{prop:self-gluing system}. We write $(\Cks, \Pks_{H}, \Tks_{HK})$ for the firm source of the induced $\Cat$-inverse system $(\cC_H,F_{HK},\theta_{HKL})$ over $\dirI_G$ and we have
\[\Mks =\add\left\{\Pks_{z\ast}(\cM[z])\mid z\in\ZZ\right\}\subseteq \m\Cks.\]
We have the following immediate result.

\begin{corollary}\label{cor:M infinity infinity is n-ct}
$\Mks$ is an $n$-cluster tilting subcategory of $\m\Cks$.
\end{corollary}

\begin{proof}
Follows immediately by Proposition \ref{prop:self-gluing system} and Theorem \ref{thrm:complete n-fractured gives n-cluster tilting}.
\end{proof}

Let us now study the quiver $\Qks$ and the ideal $\Rks$. The infinite quiver $\Qks$ is obtained by performing all the gluings induced by the arrows in $G$. Since $P$ and $I$ are abutments of $\La$ of height $h$, we have that $Q$ has the form
\begin{equation}\label{eq:Q with support of projective}
\begin{tikzpicture}[baseline={(current bounding box.center)},scale=0.9, transform shape]
\node(A) at (0,0) {$p_1$};
\node(B) at (1.5,0) {$p_2$};
\node(C) at (3,0) {$p_3$};
\node(D) at (4.5,0) {$\cdots$};
\node(E) at (6,0) {$p_{h-1}$};
\node(F) at (7.5,0) {$p_h$\nospacepunct{,}};
\node(G) at (-1.5,1) {$ $};
\node(H) at (-1.5,-1) {$ $};

\draw[->] (A) -- node[above] {$\alpha_1$} (B);
\draw[->] (B) -- node[above] {$\alpha_2$} (C);
\draw[->] (C) -- node[above] {$\alpha_3$} (D);
\draw[->] (D) -- node[above] {$\alpha_{h-2}$} (E);
\draw[->] (E) -- node[above] {$\alpha_{h-1}$} (F);

\draw[->] (G) -- (A);
\draw[->] (H) -- (A);

\draw[dotted] (-0.8, 0.5) -- (-0.8, -0.5);

\draw[pattern=northwest, hatch distance=15pt, hatch thickness = 0.3pt] (-2.12,0) ellipse (1cm and 1.3cm);
\node (Y) at (-2.12,0) {$\mathbf{Q_L'}$};
\end{tikzpicture}
\end{equation}
and the form
\begin{equation}\label{eq:Q with support of injective}
\begin{tikzpicture}[baseline={(current bounding box.center)},scale=0.9, transform shape]
\node(A) at (0,0) {$i_1$};
\node(B) at (1.5,0) {$i_2$};
\node(C) at (3,0) {$i_3$};
\node(D) at (4.5,0) {$\cdots$};
\node(E) at (6,0) {$i_{h-1}$};
\node(F) at (7.5,0) {$i_h$};
\node(G) at (9,1) {$ $};
\node(H) at (9,-1) {$ $};

\draw[->] (A) -- node[above] {$\beta_1$} (B);
\draw[->] (B) -- node[above] {$\beta_2$} (C);
\draw[->] (C) -- node[above] {$\beta_3$} (D);
\draw[->] (D) -- node[above] {$\beta_{h-2}$} (E);
\draw[->] (E) -- node[above] {$\beta_{h-1}$} (F);

\draw[->] (F) -- (G);
\draw[->] (F) -- (H);

\draw[dotted] (8.3, 0.5) -- (8.3, -0.5);

\draw[pattern=northwest, hatch distance=15pt, hatch thickness = 0.3pt] (9.65,0) ellipse (1cm and 1.3cm);
\node (Z) at (9.65,0) {$\mathbf{Q_R'}$};
\end{tikzpicture},
\end{equation}
where $P=P(p_1)$ and $I=I(i_h)$. It follows that $\supp(P)=\{p_1,\dots,p_h\}$ and $\supp(I)=\{i_1,\dots,i_h\}$.  In particular, the only case where $\supp(I)\cap \supp(P) \neq \varnothing$ is if $Q$ is $\overrightarrow{A}_{h'}$ for some $h'\geq h$. For simplicity let us assume that $Q\neq \overrightarrow{A}_{h'}$; the case $Q=\overrightarrow{A}_{h'}$ can be considered in the same way. Since $\supp(I)\cap \supp(P)=\varnothing$, we can draw the quiver $Q$ more abstractly as 
\[\begin{tikzpicture}[scale=0.9, transform shape]

\draw[->] (8.5,0) -- (8.9,0);
\draw[-] (8.9,0) -- (9.2,0);
\draw[->] (10.8,0) -- (11.2,0);
\draw[-] (11.2,0) -- (11.5,0);

\draw[pattern=northwest, hatch distance=15pt, hatch thickness = 0.3pt] (10,0) ellipse (0.8cm and 1cm);
\node (Z) at (10,0) {$\mathbf{Q'}$};
\end{tikzpicture}\]
where $Q'$ is not empty and the left $\begin{tikzpicture}[scale=0.8, transform shape, decoration={markings, mark=at position 0.6 with {\arrow{>}}}] 
    \node (A) at (0,0) {}; 
    \node (B) at (1.4,0) {}; 
    \draw[postaction={decorate}] (A) -- (B);
\end{tikzpicture}$ corresponds to the full subquiver of $Q$ with vertex set $\{i_1,\dots,i_h\}$ and the right $\begin{tikzpicture}[scale=0.8, transform shape, decoration={markings, mark=at position 0.6 with {\arrow{>}}}] 
    \node (A) at (0,0) {}; 
    \node (B) at (1.4,0) {}; 
    \draw[postaction={decorate}] (A) -- (B);
\end{tikzpicture}$ corresponds to the full subquiver of $Q$ with vertex set $\{p_1,\dots,p_h\}$. Then the quiver $\Qks$ has the form
\[\begin{tikzpicture}[scale=0.9, transform shape]
\draw[loosely dotted] (4.2,0) -- (5.2,0);

\draw[->] (5.3,0) -- (5.7,0);
\draw[-] (5.7,0) -- (6,0);

\draw[pattern=northwest, hatch distance=15pt, hatch thickness = 0.3pt] (6.8,0) ellipse (0.8cm and 1cm);
\node (Z02) at (6.8,0) {$\mathbf{Q'[-2]}$};

\draw[->] (7.6,0) -- (8,0);
\draw[-] (8,0) -- (8.3,0);

\draw[pattern=northwest, hatch distance=15pt, hatch thickness = 0.3pt] (9.1,0) ellipse (0.8cm and 1cm);
\node (Z01) at (9.1,0) {$\mathbf{Q'[-1]}$};

\draw[->] (9.9,0) -- (10.3,0);
\draw[-] (10.3,0) -- (10.6,0);

\draw[pattern=northwest, hatch distance=15pt, hatch thickness = 0.3pt] (11.4,0) ellipse (0.8cm and 1cm);
\node (Z) at (11.4,0) {$\mathbf{Q'[0]}$};

\draw[->] (12.2,0) -- (12.6,0);
\draw[-] (12.6,0) -- (12.9,0);

\draw[pattern=northwest, hatch distance=15pt, hatch thickness = 0.3pt] (13.7,0) ellipse (0.8cm and 1cm);
\node (Z1) at (13.7,0) {$\mathbf{Q'[1]}$};

\draw[->] (14.5,0) -- (14.9,0);
\draw[-] (14.9,0) -- (15.2,0);

\draw[pattern=northwest, hatch distance=15pt, hatch thickness = 0.3pt] (16,0) ellipse (0.8cm and 1cm);
\node (Z2) at (16,0) {$\mathbf{Q'[2]}$};

\draw[->] (16.8,0) -- (17.2,0);
\draw[-] (17.2,0) -- (17.5,0);

\draw[loosely dotted] (17.6,0) -- (18.6,0);
\end{tikzpicture}.\]

For two integers $a,b\in \ZZ$ with $a\leq b$ we denote by $[a,b]$ the integer interval $\{z\in \ZZ \mid a\leq z \leq b\}$. In particular, we have $\dirI_G=\{\langle [a,b]\rangle \mid a\leq b\}$. We set $Q[a,b]\coloneqq Q_{\langle [a,b]\rangle}$, $\cR[a,b]\coloneqq \cR_{\langle [a,b]\rangle}$ and $\La[a,b]\coloneqq \La_{\langle [a,b]\rangle}$. Notice that this notation agrees with our previous notation, since $\La[a,a]=\La[a]$. By definition the ideal $R[a,b]$ is generated by all elements in $\cR[z]$ for $z\in[a,b]$ as well as all paths from $Q'[z]$ to $Q'[z+1]$ for $z\in[a,b-1]$. The ideal $\Rks$ is then equal to $\Rks=\bigcup_{a\leq b}\cR[a,b]$. It follows that $\Rks$ is generated by all elements in $\cR[z]$ for $z\in \ZZ$ as well as all paths from $Q'[z]$ to $Q'[z+1]$ for all $z\in \ZZ$. In particular there is a canonical isomorphism $[z]:\La[a,b]\to \La[a+z,b+z]$ for all $z\in\ZZ$. 

Moreover, for $k\in\ZZ$ we have an isomorphism of categories $[k]:\K \Qks/\Rks \to \K \Qks/\Rks$ given by moving a path $k$-steps to the right in $\Qks$. In particular this isomorphism induces an isomorphism $[k]:\proj(\K \Qks/\Rks) \to \proj(\K \Qks/\Rks)$. 

\begin{lemma}\label{lem:admissible Z-action on C}
The collection $\{[k]\in \Aut(\proj(\K\Qks/\Rks))\}$ is an admissible $\ZZ$-action on $\proj(\K\Qks/\Rks)$. 
\end{lemma}

\begin{proof}
It is easy to see that the collection $\{[k]\in \Aut(\proj(\K\Qks/\Rks))\}$ is a $\ZZ$-action on $\proj(\K\Qks/\Rks)$. Moreover this action is admissible since
\[[k](P(i[z])) = P(i[z+k]) \isom P(i[z])\]
implies that $i[z+k]=i[z]$, which is true if and only if $k=0$ since $\La\not\isom \K \overrightarrow{A}_{h'}$. 
\end{proof}

Let $\{F_k\in \Aut(\Cks)\mid k\in \ZZ\}$ be the admissible $\ZZ$-action on $\Cks$ induced by $\{[k]\in \Aut(\proj(\K\Qks/\Rks)\mid k\in\ZZ\}$ that is compatible with the equivalence $\Cks\equivalent\proj(\K\Qks/\Rks)$ in Corollary \ref{cor:equivalence of quiver and projectives}.

\begin{lemma}\label{lem:Mks is Z-equivariant}
The subcategory $\Mks\subseteq \m\Cks$ is $\ZZ$-equivariant.
\end{lemma}

\begin{proof}
Viewing $\Cks$-modules as representations of $\Qks$ bound by $\Rks$, we have that a $\Cks$-module is of the form $M[z]$ for some $M\in\m\La$ and some $z\in\ZZ$. Let $M[z]\in\m\Cks$. It is easy to show that for every $k\in\ZZ$ we have $F_{k\ast}(\Pks_{z\ast}(M[z])) \isom \Pks_{(z-k)\ast}(M[z-k])$ by evaluating both sides. Since $\Mks = \add\left\{\Pks_{z\ast}(\cM[z])\mid z\in\ZZ\right\}$, it immediately follows that $\Mks$ is $\ZZ$-equivariant.
\end{proof}

\begin{corollary}\label{cor:Mks/Z is n-ct}
The subcategory $F^{\ast}(\Mks)\equivalent \Mks/\ZZ\subseteq \m(\Cks/\ZZ)$ is a locally bounded $n$-cluster tilting subcategory.
\end{corollary}

\begin{proof}
We show that the conditions of Theorem \ref{thrm:Darpo-Iyama} are satisfied for $\cC=\Cks$, $\cG=\ZZ$ and $\cM=\Mks$. The category $\Cks$ is a locally bounded $\K$-linear Krull--Schmidt category by Corollary \ref{cor:properties of Cks}. The free abelian group $\ZZ$ acts admissibly on $\Cks$ by Lemma \ref{lem:admissible Z-action on C}. The subcategory $\Mks\subseteq \m\Cks$ is locally bounded since $\m\Cks$ is locally bounded by Proposition \ref{prop:indecomposables in firm source}(f). The subcategory $\Mks$ is an $n$-cluster tilting subcategory of $\m\Cks$ by Corollary \ref{cor:M infinity infinity is n-ct}. The subcategory $\Mks$ is $\ZZ$-equivariant by Lemma \ref{lem:Mks is Z-equivariant}. Hence $F^{\ast}(\Mks)$ is a locally bounded $n$-cluster tilting subcategory. Finally, the equivalence $F^{\ast}(\Mks)\equivalent \Mks/\ZZ$ follows by \cite[Lemma 3.5(c)]{DI}.
\end{proof}

We can realize the category $\m(\Cks/\ZZ)$ as the module category of a fi\-nite-di\-men\-sion\-al algebra. Recall the form of $Q$ from (\ref{eq:Q with support of projective}) and (\ref{eq:Q with support of injective}). Consider the equivalence relation $\tilvert$ on $Q_0$ generated by $i_k\tilvert \;p_k$ for $1\leq k\leq h$ and the equivalence relation $\tilarr$ on $Q_1$ generated by $\alpha_k\tilarr \;\beta_k$ for $1\leq k\leq h-1$. Let $\tilQ$ be the quiver with $\tilQ_0=Q_0/\tilvert$ and $\tilQ_1=Q_1/\tilarr$, that is $\tilQ$ has the form
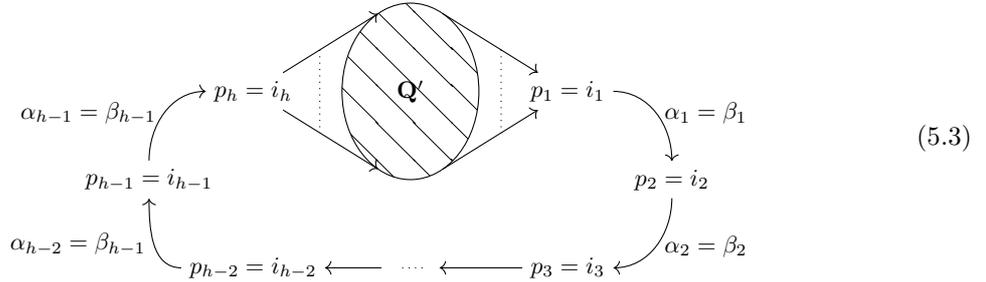
\begin{equation}\label{picture:self-gluing quiver}
\begin{tikzpicture}[baseline={(current bounding box.center)},scale=0.9, transform shape]
\draw[pattern=northwest, hatch distance=15pt, hatch thickness = 0.3pt] (0,0) ellipse (1 and 1.3);
\node (Q) at (0,0) {$\mathbf{Q'}$};

\node(1) at (2.3,0) {$p_1=i_1$};
\node(2) at (3.81986841536,-1.3) {$p_2=i_2$};
\node(3) at (2.3,-2.6) {$p_3=i_3$};
\node(dotsR) at (0.3,-2.6) {$ $};
\node(dotsL) at (-0.3,-2.6) {$ $};
\node(h-2) at (-2.3,-2.6) {$p_{h-2}=i_{h-2}$};
\node(h-1) at (-3.81986841536,-1.3) {$p_{h-1}=i_{h-1}$};
\node(h) at (-2.3,0) {$p_h=i_h$};

\draw[->] (0.47475302236,1.14415478389) -- (1);
\draw[->] (0.47475302236,-1.14415478389) -- (1);
\draw[dotted] (1.325, 0.611175075) -- (1.325, -0.611175075);
\draw[->] (1) to [out=0, in=90] node[draw=none, midway, right] {$\alpha_{1}=\beta_1$} (2);
\draw[->] (2) to [out=-90, in=0] node[draw=none, midway, right] {$\alpha_{2}=\beta_2$} (3);
\draw[->] (3) -- (dotsR);
\draw[dotted] (dotsR) -- (dotsL);
\draw[->] (dotsL) -- (h-2);
\draw[->] (h-2) to [out=180, in=-90] node[draw=none, midway, left] {$\alpha_{h-2}=\beta_{h-1}$} (h-1);
\draw[->] (h-1) to [out=90, in=180] node[draw=none, midway, left] {$\alpha_{h-1}=\beta_{h-1}$} (h);
\draw[->] (h) -- (-0.47475302236,1.14415478389);
\draw[->] (h) -- (-0.47475302236,-1.14415478389);
\draw[dotted] (-1.325, 0.611175075) -- (-1.325, -0.611175075);
\end{tikzpicture}
\end{equation}
In particular for a path in $Q$, we can consider the corresponding path in $\tilQ$. Let $\tilR$ be the ideal of $\K \tilQ$ generated by all elements in $\cR$, viewed as elements in $\K \tilQ$, together with all paths of the form $\gamma\alpha_1\dots\alpha_{h-1}\gamma'$ with $\gamma$ and $\gamma'$ arrows in $\tilQ$. In other words, the ideal $\tilR$ is generated by all elements in $\cR$ and all paths that start at $Q'$, go through the vertices $p_1,\dots,p_h$ and end up at $Q'$ again. Set $\tilLa=\K\tilQ/\tilR$\index[symbols]{Lambdaatilde@$\tilLa$}\index[symbols]{KQ R@$\K\tilQ/\tilR$}. 

\begin{proposition}\label{prop:finite-dimensional algebra from infinite quiver}
Let $\cD\coloneqq \K \Qks/\Rks$. There is an equivalence of categories $\proj(\cD/\ZZ)\equivalent \Cks/\ZZ$.
\end{proposition}

\begin{proof}
First observe that by Corollary \ref{cor:equivalence of quiver and projectives} we have an equivalence of categories $\proj(\cD)\equivalent \Cks$ which is compatible with the action of $\ZZ$ on both sides by construction. Hence it is enough to show that $\left(\proj\cD\right)/\ZZ \equivalent \proj\left(\cD/\ZZ\right)$. The Yoneda embedding $h_{\cD}:\cD\to \proj(\cD)$ induces a bijection $\obj(\cD) \equivalent \ind(\proj(\cD))$, which in turn induces a bijection $\obj(\cD/\ZZ)\equivalent \ind(\proj(\cD)/\ZZ)=\ind(\proj(\cD))/\ZZ$. On the other hand the Yoneda embedding $h_{\cD/\ZZ}$ induces a bijection $\obj(\cD/\ZZ)\equivalent\ind(\proj(\cD/\ZZ))$. It follows that there is a bijection $\ind(\proj(\cD)/\ZZ)\equivalent \ind(\proj(\cD/\ZZ))$ which induces an equivalence of categories $\left(\proj\cD\right)/\ZZ \equivalent \proj\left(\cD/\ZZ\right)$.
\end{proof}

\begin{corollary}\label{cor:La infinity n-ct}
There is an equivalence of categories $\m\tilLa\equivalent \m(\Cks/\ZZ)$. In particular $\m\tilLa$ admits an $n$-cluster tilting subcategory $\tilM$.
\end{corollary}

\begin{proof}
This equivalence can be obtained by Proposition \ref{prop:finite-dimensional algebra from infinite quiver} via
\[\m\tilLa \equivalent \m \left(\proj\left(\cD/\ZZ\right)\right) \equivalent \m\left(\left(\proj\cD\right)/\ZZ\right) \equivalent \m(\Cks/\ZZ).\]
Then $\m\tilLa$ admits an $n$-cluster tilting subcategory by Corollary \ref{cor:Mks/Z is n-ct}.
\end{proof}

\begin{corollary}\label{cor:La n-ct implies La infinity n-ct}
Let $\La$ be a rep\-re\-sen\-ta\-tion-di\-rect\-ed algebra whose module category admits an $n$-cluster tilting subcategory $\cM$. Let $W$ be any maximal left abutment of $\La$ and let $J$ be any maximal right abutment of $\La$. Let $P\leq W$ be a simple left abutment and $I\leq J$ be a simple right abutment. Then $\La$ is $n$-self-gluable and $\Mks/\ZZ$ is an $n$-cluster tilting subcategory of $\m\tilLa$.
\end{corollary}

\begin{proof}
Follows immediately by Example \ref{ex:n-self-gluable easy} and Corollary \ref{cor:La infinity n-ct}.
\end{proof}

\subsection{Simultaneous double gluings}\label{sec:double gluing}

In this section we describe a special case of a gluing coming from an orbit category. As a motivation, recall that our gluing operation takes us input two bound quiver algebras $A$ and $B$ where the quivers are of the form
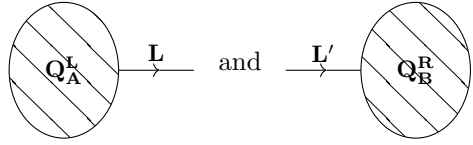
\[\begin{tikzpicture}[baseline={(current bounding box.center)}, scale=0.9, transform shape]

\draw[->] (10.8,0) -- (11.4,0);
\draw[-] (11.4,0) -- (11.9,0);

\draw[pattern=northwest, hatch distance=15pt, hatch thickness = 0.3pt] (10,0) ellipse (0.8cm and 1cm);
\node (Z) at (10,0) {$\mathbf{Q^L_A}$};

\node (X) at (11.35,0.25) {$\mathbf{L}$};
\end{tikzpicture}\;\; \text{ and } \;\;
\begin{tikzpicture}[baseline={(current bounding box.center)}, scale=0.9, transform shape]

\draw[->] (8.1,0) -- (8.7,0);
\draw[-] (8.7,0) -- (9.2,0);

\draw[pattern=northwest, hatch distance=15pt, hatch thickness = 0.3pt] (10,0) ellipse (0.8cm and 1cm);
\node (Z) at (10,0) {$\mathbf{Q^R_B}$};

\node (X) at (8.65,0.25) {$\mathbf{L'}$};
\end{tikzpicture}\]
where $\mathbf{L}=\mathbf{L'}=\overrightarrow{A}_h$ and there are no relations in the subquivers $\mathbf{L}$ and $\mathbf{L'}$ and returns a new bound given quiver algebra $\La=\K Q_{\La} /\cR_{\La}$ where the quiver is
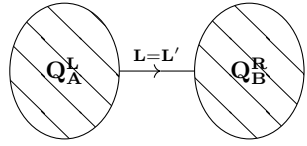
\[\begin{tikzpicture}[baseline={(current bounding box.center)}, scale=0.9, transform shape]

\draw[pattern=northwest, hatch distance=15pt, hatch thickness = 0.3pt] (7.3,0) ellipse (0.8cm and 1cm);
\node (Z) at (7.3,0) {$\mathbf{Q^L_A}$};

\draw[->] (8.1,0) -- (8.7,0);
\draw[-] (8.7,0) -- (9.2,0);

\draw[pattern=northwest, hatch distance=15pt, hatch thickness = 0.3pt] (10,0) ellipse (0.8cm and 1cm);
\node (ZZ) at (10,0) {$\mathbf{Q^R_B}$};
\node (X) at (8.65,0.25) {$\scriptstyle \mathbf{L}=\mathbf{L'}$};
\end{tikzpicture}\]
and $\cR_{\La}$ is generated by $\cR_A\cup \cR_B$ together with all paths from $\mathbf{Q^L_A}$ to $\mathbf{Q^R_B}$. In this section we generalize the above situation to the situation of taking as inputs two bound quiver algebras $A$ and $B$ where the quivers are of the form
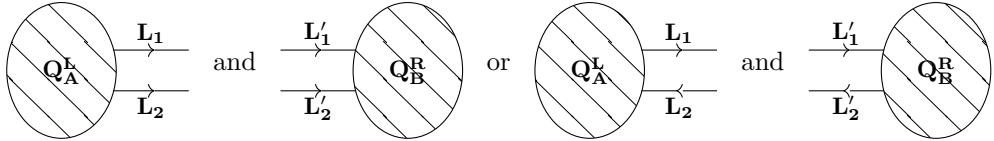
\[\begin{tikzpicture}[baseline={(current bounding box.center)}, scale=0.9, transform shape]

\draw[->] (10.76,0.3) -- (11.36,0.3);
\draw[-] (11.36,0.3) -- (11.86,0.3);

\draw[->] (10.76,-0.3) -- (11.36,-0.3);
\draw[-] (11.36,-0.3) -- (11.86,-0.3);

\draw[pattern=northwest, hatch distance=15pt, hatch thickness = 0.3pt] (10,0) ellipse (0.8cm and 1cm);
\node (Z) at (10,0) {$\mathbf{Q^L_A}$};

\node (X) at (11.31,0.55) {$\mathbf{L_1}$};
\node (Y) at (11.31,-0.55) {$\mathbf{L_2}$};

\end{tikzpicture}\;\; \text{ and } \;\;
\begin{tikzpicture}[baseline={(current bounding box.center)}, scale=0.9, transform shape]

\draw[->] (8.14,0.3) -- (8.74,0.3);
\draw[-] (8.74,0.3) -- (9.24,0.3);

\draw[->] (8.14,-0.3) -- (8.74,-0.3);
\draw[-] (8.74,-0.3) -- (9.24,-0.3);

\draw[pattern=northwest, hatch distance=15pt, hatch thickness = 0.3pt] (10,0) ellipse (0.8cm and 1cm);
\node (Z) at (10,0) {$\mathbf{Q^R_B}$};

\node (X) at (8.69,0.55) {$\mathbf{L_1'}$};
\node (Y) at (8.69,-0.55) {$\mathbf{L_2'}$};
\end{tikzpicture}\;\; \text{ or } \;\;
\begin{tikzpicture}[baseline={(current bounding box.center)}, scale=0.9, transform shape]

\draw[->] (10.76,0.3) -- (11.36,0.3);
\draw[-] (11.36,0.3) -- (11.86,0.3);

\draw[-<] (10.76,-0.3) -- (11.36,-0.3);
\draw[-] (11.36,-0.3) -- (11.86,-0.3);

\draw[pattern=northwest, hatch distance=15pt, hatch thickness = 0.3pt] (10,0) ellipse (0.8cm and 1cm);
\node (Z) at (10,0) {$\mathbf{Q^L_A}$};

\node (X) at (11.31,0.55) {$\mathbf{L_1}$};
\node (Y) at (11.31,-0.55) {$\mathbf{L_2}$};

\end{tikzpicture}\;\; \text{ and } \;\;
\begin{tikzpicture}[baseline={(current bounding box.center)}, scale=0.9, transform shape]

\draw[->] (8.14,0.3) -- (8.74,0.3);
\draw[-] (8.74,0.3) -- (9.24,0.3);

\draw[-<] (8.14,-0.3) -- (8.74,-0.3);
\draw[-] (8.74,-0.3) -- (9.24,-0.3);

\draw[pattern=northwest, hatch distance=15pt, hatch thickness = 0.3pt] (10,0) ellipse (0.8cm and 1cm);
\node (Z) at (10,0) {$\mathbf{Q^R_B}$};

\node (X) at (8.69,0.55) {$\mathbf{L_1'}$};
\node (Y) at (8.69,-0.55) {$\mathbf{L_2'}$};
\end{tikzpicture}\]
where $\mathbf{L_1}=\mathbf{L_1'}=\overrightarrow{A}_{h_1}$, $\mathbf{L_2}=\mathbf{L_2'}=\overrightarrow{A}_{h_2}$ and there are no relations in the subquivers $\mathbf{L_1}$, $\mathbf{L_1'}$, $\mathbf{L_2}$ and $\mathbf{L_2'}$ and returning a new bound quiver algebra $\tilde{\La}=\K \tilQ/\tilR$ where $\tilQ$ is
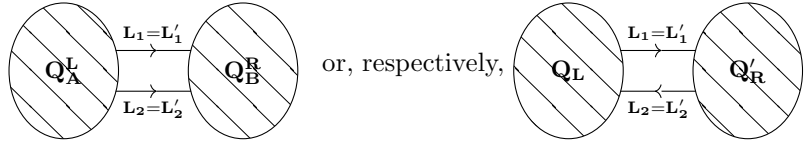
\[\begin{tikzpicture}[baseline={(current bounding box.center)}, scale=0.9, transform shape]

\draw[->] (8.14,0.3) -- (8.74,0.3);
\draw[-] (8.74,0.3) -- (9.24,0.3);

\draw[->] (8.14,-0.3) -- (8.74,-0.3);
\draw[-] (8.74,-0.3) -- (9.24,-0.3);

\draw[pattern=northwest, hatch distance=15pt, hatch thickness = 0.3pt] (10,0) ellipse (0.8cm and 1cm);
\draw[pattern=northwest, hatch distance=15pt, hatch thickness = 0.3pt] (7.38,0) ellipse (0.8cm and 1cm);
\node (Z) at (10,0) {$\mathbf{Q^R_B}$};
\node (ZZ) at (7.38,0) {$\mathbf{Q^L_A}$};

\node (X) at (8.69,0.55) {$\scriptstyle \mathbf{L_1}=\mathbf{L_1'}$};
\node (Y) at (8.69,-0.55) {$\scriptstyle \mathbf{L_2}=\mathbf{L_2'}$};
\end{tikzpicture}\;\; \text{ or, respectively, }
\begin{tikzpicture}[baseline={(current bounding box.center)}, scale=0.9, transform shape]

\draw[->] (8.14,0.3) -- (8.74,0.3);
\draw[-] (8.74,0.3) -- (9.24,0.3);

\draw[-<] (8.14,-0.3) -- (8.74,-0.3);
\draw[-] (8.74,-0.3) -- (9.24,-0.3);

\draw[pattern=northwest, hatch distance=15pt, hatch thickness = 0.3pt] (10,0) ellipse (0.8cm and 1cm);
\draw[pattern=northwest, hatch distance=15pt, hatch thickness = 0.3pt] (7.38,0) ellipse (0.8cm and 1cm);
\node (Z) at (10,0) {$\mathbf{Q'_R}$};
\node (ZZ) at (7.38,0) {$\mathbf{Q_L}$};

\node (X) at (8.69,0.55) {$\scriptstyle \mathbf{L_1}=\mathbf{L_1'}$};
\node (Y) at (8.69,-0.55) {$\scriptstyle \mathbf{L_2}=\mathbf{L_2'}$};
\end{tikzpicture}\]
and $\tilR$ is generated by $\cR_A\cup \cR_B$ together with all paths from $\mathbf{Q^L_A}$ to $\mathbf{Q^R_B}$ and all paths from $\mathbf{Q^R_B}$ to $\mathbf{Q^L_A}$. Our aim is to show that if $\m A$ and $\m B$ admit $n$-fractured subcategories that are sufficiently compatible with the above gluings, then $\m\tilLa$ admits an $n$-cluster tilting subcategory.

To make the above notions more precise, let $A=\K Q_A/\cR_A$ and $B=\K Q_B/\cR_B$ be two bound quiver algebras as above in the case where $\mathbf{L_2}$ has a source and $\mathbf{L_2'}$ has a sink; the other case can be dealt with similarly. Let us call $P_1$ and $P_2$ the left abutments of $A$ with support $\left(\mathbf{L_1}\right)_0$ and $\left(\mathbf{L_2}\right)_0$ respectively and $I_1$ and $I_2$ the right abutments of $B$ with support $\left(\mathbf{L_1'}\right)_0$ and $\left(\mathbf{L_2'}\right)_0$. Notice that these abutments are in general not maximal; let us call $W_1$ and $W_2$ the maximal left abutments of $A$ with $P_1\leq W_1$ and $P_2\leq W_2$ and let us call $J_1$ and $J_2$ the maximal right abutments of $B$ with $I_1\leq J_1$ and $I_2\leq J_2$. By the assumption $\mathbf{L_1}=\mathbf{L_1'}$, we can perform the gluing $\La_1=B\glue[P_1][I_1] A$ where $\La_1=\K Q_{\La_1}/\cR_{\La_1}$ with $Q_{\La_1}$ the quiver
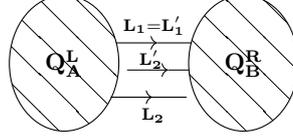
\[\begin{tikzpicture}[baseline={(current bounding box.center)}, scale=0.9, transform shape]

\draw[->] (8.14,0.3) -- (8.74,0.3);
\draw[-] (8.74,0.3) -- (9.24,0.3);

\draw[->] (8.06,-0.5) -- (8.67,-0.5);
\draw[-] (8.67,-0.5) -- (9.17,-0.5);

\draw[->] (8.30,-0.1) -- (8.90,-0.1);
\draw[-] (8.90,-0.1) -- (9.20,-0.1);

\draw[pattern=northwest, hatch distance=15pt, hatch thickness = 0.3pt] (10,0) ellipse (0.8cm and 1cm);
\draw[pattern=northwest, hatch distance=15pt, hatch thickness = 0.3pt] (7.38,0) ellipse (0.8cm and 1cm);
\node (Z) at (10,0) {$\mathbf{Q^R_B}$};
\node (ZZ) at (7.38,0) {$\mathbf{Q^L_A}$};

\node (X) at (8.69,0.55) {$\scriptstyle \mathbf{L_1}=\mathbf{L_1'}$};
\node (Y) at (8.68,-0.77) {$\scriptstyle \mathbf{L_2}$};
\node (Y) at (8.65,0.05) {$\scriptstyle \mathbf{L_2'}$};
\end{tikzpicture}\]
and where $\cR_{\La_1}$ is generated by $\cR_A\cup \cR_B$ together with all paths from $\mathbf{Q_A^L}$ to $\mathbf{Q_B^R}$.

Assume now that $A$ has a fracturing $(T^L_A,T^R_A)$ and that $\m A$ admits a $(T^L_A,T^R_A,n)$-fractured subcategory $\cM_A$ such that $T^R=D(A)$ and the only nonprojective indecomposable summands of $T^L_A$ appear in $\cF_{P_1}$ or $\cF_{P_2}$. Dually assume that $B$ has a fracturing $(T^L_B,T^R_B)$ and that $\m B$ admits a $(T^L_B,T^R_B,n)$-fractured subcategory $\cM_B$ such that $T^L_B=B$ and the only noninjective indecomposable summands of $T^R_B$ appear in $\cG_{I_1}$ or $\cG_{I_2}$. Moreover assume that the two fracturings are compatible in the sense of Theorem \ref{thrm:fractsubcat}, that is we have
\begin{equation}
\prescript{\La_1}{A}{\pi}_{\ast}\left(T_A^{(P_1)}\right) \isom \prescript{\La_1}{B}{\pi}_{\ast} \left(T_B^{(I_1)}\right), \text{ and}
\end{equation} 
\begin{equation}\label{eq:fractures in double gluing} 
\prescript{\La_1}{A}{\pi}_{\ast}\left(T_A^{(P_2)}\right) \isom \prescript{\La_1}{B}{\pi}_{\ast} \left(T_B^{(I_2)}\right). \hphantom{\text{ and}}
\end{equation}
In particular, it follows by Theorem \ref{thrm:fractsubcat} that $\left(T_{\La_1}^L,T_{\La_1}^R\right)=\left(T_B^L,T_B^R\right)\glue[P_1][I_1]\left(T_A^L,T_A^R\right)$ is a fracturing of $\La_1$ and $\cM_{\La_1}=\add\left\{\prescript{\La_1}{A}{\pi}_{\ast}\left(\cM_A\right),\prescript{\La_1}{B}{\pi}_{\ast}\left(\cM_2\right)\right\}\subseteq \m\La_1$ is a $\left(T_{\La_1}^L,T_{\La_2}^R,n\right)$-fractured subcategory. 

With these assumption we have the following.

\begin{proposition}\label{prop:double gluing}
The algebra $\La_1$ is $n$-self-gluable and $\m\tilde{\La}_1$ admits an $n$-cluster tilting subcategory.
\end{proposition}

\begin{proof}
For clarity, let us redraw the quiver $Q_{\La_1}$ as
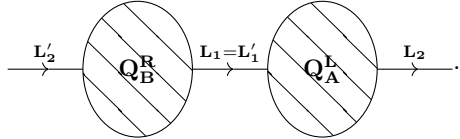
\[\begin{tikzpicture}[baseline={(current bounding box.center)}, scale=0.9, transform shape]

\draw[pattern=northwest, hatch distance=15pt, hatch thickness = 0.3pt] (7.3,0) ellipse (0.8cm and 1cm);
\node (Z) at (7.3,0) {$\mathbf{Q^R_B}$};

\draw[->] (8.1,0) -- (8.7,0);
\draw[-] (8.7,0) -- (9.2,0);

\draw[->] (5.4,0) -- (6,0);
\draw[-] (6,0) -- (6.5,0);

\draw[->] (10.8,0) -- (11.4,0);
\draw[-] (11.4,0) -- (11.9,0);

\draw[pattern=northwest, hatch distance=15pt, hatch thickness = 0.3pt] (10,0) ellipse (0.8cm and 1cm);
\node (ZZ) at (10,0) {$\mathbf{Q^L_A}$};
\node (X) at (8.65,0.25) {$\scriptstyle \mathbf{L_1}=\mathbf{L_1'}$};
\node (Y) at (5.95,0.25) {$\scriptstyle \mathbf{L_2'}$};
\node (W) at (11.35,0.25) {$\scriptstyle \mathbf{L_2}$};
\end{tikzpicture}.\]
Set $P_{2\ast}\coloneqq\prescript{\La_1}{A}{\pi}_{\ast}(P_2)$ and $I_{2\ast}\coloneqq\prescript{\La_1}{B}{\pi}_{\ast}(I_2)$. Then by the assumptions on the fractures $(T^L_A,D(A))$ and $(B, T^R_B)$ it follows that the only nonprojective indecomposable summands of $T^L_{\La_1}$ appear in $\cF_{P_2\ast}$ and the only noninjective indecomposable summands of $T^R_{\La_1}$ appear in $\cG_{I_2\ast}$. This shows that $(W_2,J_2)$ is a fractured pair where $W_2$ is a maximal left abutment with $P_{2\ast}\leq W_2$ and $J_2$ is a maximal right abutment with $I_{2\ast}\leq J_2$. Moreover, since $\mathbf{L_2}=\mathbf{L_2'}$ we have $\height\left(P_{2\ast}\right)=\height\left(I_{2\ast}\right)$ and together with equation (\ref{eq:fractures in double gluing}) this shows that $(P_{2\ast},I_{2\ast})$ is a compatible pair. Hence $\La_1$ is $n$-self-gluable. The result follows by Corollary \ref{cor:La infinity n-ct}.
\end{proof}

\begin{remark}\label{rem:k-simultaneous gluing}
We note here that many of the results of ordinary gluing hold for this type of gluing. In particular the Aus\-lan\-der--Rei\-ten quiver of $\tilde{\La}$ can be computed by identifying the subquivers of the Aus\-lan\-der--Rei\-ten quivers of $A$ and $B$ corresponding to the abutments $P_1$ and $I_1$ and $P_2$ and $I_2$. However, the resulting algebra is not always rep\-re\-sen\-ta\-tion-di\-rect\-ed. If it turns out that the resulting algebra is rep\-re\-sen\-ta\-tion-di\-rect\-ed, then we can iterate this result to simultaneously glue along more than two linearly oriented subquivers of Dynkin type $A$. That is, let $A=\K Q_A/\cR_A$ and $B=\K Q_B/\cR_B$ be two bound quiver algebras where the quivers are of the form
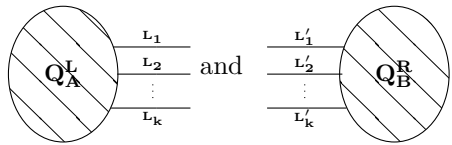
\[\begin{tikzpicture}[baseline={(current bounding box.center)}, scale=0.9, transform shape]

\draw[-] (10.73,0.4) -- (11.36,0.4);
\draw[-] (11.36,0.4) -- (11.86,0.4);

\draw[-] (10.7,-0.5) -- (11.36,-0.5);
\draw[-] (11.36,-0.5) -- (11.86,-0.5);

\draw[-] (10.79,0) -- (11.36,0);
\draw[-] (11.36,0) -- (11.86,0);

\draw[pattern=northwest, hatch distance=15pt, hatch thickness = 0.3pt] (10,0) ellipse (0.8cm and 1cm);
\node (Z) at (10,0) {$\mathbf{Q^L_A}$};

\node (X) at (11.31,0.55) {$\scriptscriptstyle \mathbf{L_1}$};
\node (XX) at (11.31,0.15) {$\scriptscriptstyle \mathbf{L_2}$};
\node (Y) at (11.31,-0.65) {$\scriptscriptstyle \mathbf{L_k}$};

\draw[dotted] (11.31,-0.15) -- (11.31, -0.43);
\end{tikzpicture}\text{ and } \;\;
\begin{tikzpicture}[baseline={(current bounding box.center)}, scale=0.9, transform shape]

\draw[-] (8.14,0.4) -- (8.74,0.4);
\draw[-] (8.74,0.4) -- (9.27,0.4);

\draw[-] (8.14,0) -- (8.74,0);
\draw[-] (8.74,0) -- (9.24,0);

\draw[-] (8.14,-0.5) -- (8.74,-0.5);
\draw[-] (8.74,-0.5) -- (9.3,-0.5);

\draw[pattern=northwest, hatch distance=15pt, hatch thickness = 0.3pt] (10,0) ellipse (0.8cm and 1cm);
\node (Z) at (10,0) {$\mathbf{Q^R_B}$};

\node (X) at (8.69,0.55) {$\scriptscriptstyle \mathbf{L_1'}$};
\node (XX) at (8.69,0.15) {$\scriptscriptstyle \mathbf{L_2'}$};
\node (Y) at (8.69,-0.65) {$\scriptscriptstyle \mathbf{L_k'}$};

\draw[dotted] (8.69,-0.15) -- (8.69, -0.43);
\end{tikzpicture}\]
where $\mathbf{L_i}=\mathbf{L_i'}=\overrightarrow{A}_{h_i}$ and there are no relations in the subquivers $\mathbf{L_i}$, $\mathbf{L_i'}$. For $i\in\{1,\dots,k\}$ define $\La_i=\K Q_{\La_i} /\cR_{\La_i}$ to be the bound quiver algebra where $Q_{\La_i}$ is
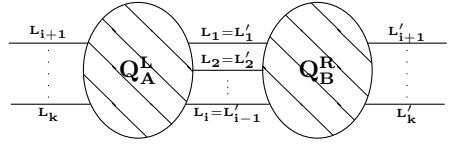
\[\begin{tikzpicture}[baseline={(current bounding box.center)}, scale=0.9, transform shape]

\draw[-] (8.11,0.4) -- (8.74,0.4);
\draw[-] (8.74,0.4) -- (9.27,0.4);

\draw[-] (6.65,0.4) -- (5.49,0.4);
\draw[-] (10.73,0.4) -- (11.89,0.4);

\draw[-] (8.17,0) -- (8.74,0);
\draw[-] (8.74,0) -- (9.21,0);

\draw[-] (8.08,-0.5) -- (8.74,-0.5);
\draw[-] (8.74,-0.5) -- (9.3,-0.5);

\draw[-] (6.68,-0.5) -- (5.52,-0.5);
\draw[-] (10.7,-0.5) -- (11.86,-0.5);

\node (X) at (8.69,0.55) {$\scriptscriptstyle \mathbf{L_1}=\mathbf{L_1'}$};
\node (XX) at (8.69,0.15) {$\scriptscriptstyle \mathbf{L_2}=\mathbf{L_2'}$};
\node (Y) at (8.69,-0.65) {$\scriptscriptstyle \mathbf{L_i}=\mathbf{L_{i-1}'}$};

\node (Z) at (6.07,0.55) {$\scriptscriptstyle \mathbf{L_{i+1}}$};
\node (YY) at (6.07,-0.65) {$\scriptscriptstyle \mathbf{L_k}$};
\node (W) at (11.31,0.55) {$\scriptscriptstyle \mathbf{L_{i+1}'}$};
\node (U) at (11.31,-0.65) {$\scriptscriptstyle \mathbf{L_k'}$};

\draw[dotted] (8.69,-0.15) -- (8.69, -0.43);
\draw[loosely dotted] (6.07,0.3) -- (6.07, -0.43);
\draw[loosely dotted] (11.31,0.3) -- (11.31, -0.43);

\draw[pattern=northwest, hatch distance=15pt, hatch thickness = 0.3pt] (10,0) ellipse (0.8cm and 1cm);
\draw[pattern=northwest, hatch distance=15pt, hatch thickness = 0.3pt] (7.38,0) ellipse (0.8cm and 1cm);
\node (Z) at (10,0) {$\mathbf{Q^R_B}$};
\node (ZZ) at (7.38,0) {$\mathbf{Q^L_A}$};
\end{tikzpicture}\]
and $\cR_{\La_i}$ is generated by $\cR_A\cup \cR_B$ together with all paths from $\mathbf{Q^L_A}$ to $\mathbf{Q^R_B}$ and all paths from $\mathbf{Q^R_B}$ to $\mathbf{Q^L_A}$. Assume moreover that $\m A$ and $\m B$ admit $n$-fractured subcategories that are compatible with the above gluings and such that all nonprojective left fractures are supported in $\mathbf{L_i}$ or $\mathbf{L_{k+1}'}$ and all noninjective right fractures are supported in $\mathbf{L_i'}$ or $\mathbf{L_{k+1}}$. Then if the algebra $\La_i$ is rep\-re\-sen\-ta\-tion-di\-rect\-ed, it follows by Proposition \ref{prop:double gluing} that the module category $\m\La_{i+1}$ admits an $n$-fractured subcategory. If in particular $\La_{k-1}$ is rep\-re\-sen\-ta\-tion-di\-rect\-ed, then $\m\La_k$ admits an $n$-cluster tilting subcategory. For an example of this type see Example \ref{ex:k-simultaneous gluing}.
\end{remark}

\section{Concrete examples}\label{sec:concrete examples}

In this section we provide some families of algebras whose module categories admit $n$-cluster tilting subcategories which we can obtain using the methods of Section \ref{sec:applications}. 

\subsection{Quivers with any number of sinks and sources}\label{subsection:quivers with any number of sinks and sources}

To apply Proposition \ref{prop:gluing system of n-ct} we need to provide rep\-re\-sen\-ta\-tion-di\-rect\-ed algebras whose module categories admit an $n$-cluster tilting subcategory and such that their quivers have an arbitrary number of sinks and sources. From this point of view, the following question is natural.

\begin{question}\label{question:sinks and sources}
Let $\left(s,t,n\right)$ be a triple of positive integers. Does there exist a rep\-re\-sen\-ta\-tion-di\-rect\-ed (connected) algebra $\La$ such that $\sources_{\La}=s$, $\sinks_{\La}=t$ and $\m\La$ admits an $n$-cluster tilting subcategory? 
\end{question}

For $n=1$ it is clear that Question \ref{question:sinks and sources} has a positive answer. For $n=2$ we answer Question \ref{question:sinks and sources} affirmatively in \cite[Remark 3.19]{VAS2}. In this section show that the same result holds for $n\geq 3$. The construction will be given by the gluing algebra of a gluing system over a finite graph $G$. We first describe the algebras that we need for the gluing.

\begin{definition}\label{def:starlike algebra}
Let $k\in\ZZ_{\geq 1}$ with $k\neq 2$ and let $m_1,\dots,m_{k}\in\ZZ_{\geq 2}$ such that $m_1\geq m_2 \geq \cdots \geq m_k$. The \emph{starlike tree}\index[definitions]{starlike tree} $T(k,m_1,m_2,\dots,m_{k})$\index[symbols]{T(k,m1,m2)@$T(k,m_1,m_2,\dots,m_{k})$} is the graph
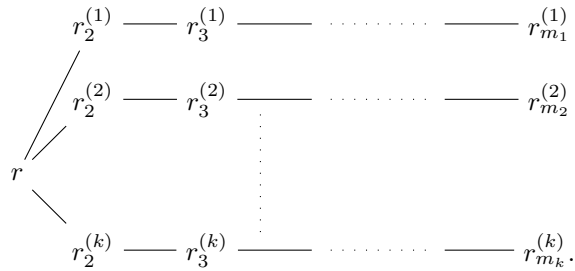
\[\begin{tikzpicture}
    \node (R) at (0,0) {$r$};

    \node (1M1) at (7,2) {$r_{m_1}^{(1)}$};
    \node (1M1-) at (5.5,2) {};
    \node (13+) at (4,2) {};
    \node (13) at (2.5,2) {$r_3^{(1)}$};
    \node (12) at (1,2) {$r_2^{(1)}$};
    
    \draw[-] (1M1) -- (1M1-);
    \draw[loosely dotted] (1M1-) -- (13+);
    \draw[-] (13+) -- (13);
    \draw[-] (13) -- (12);
    \draw[-] (12) -- (R);
    
    \node (2M1) at (7,1) {$r_{m_2}^{(2)}$};
    \node (2M1-) at (5.5,1) {};
    \node (23+) at (4,1) {};
    \node (23) at (2.5,1) {$r_3^{(2)}$};
    \node (22) at (1,1) {$r_2^{(2)}$};
    
    \draw[-] (2M1) -- (2M1-);
    \draw[loosely dotted] (2M1-) -- (23+);
    \draw[-] (23+) -- (23);
    \draw[-] (23) -- (22);
    \draw[-] (22) -- (R);

    \node (kM1) at (7,-1) {$r_{m_k}^{(k)}$.};
    \node (kM1-) at (5.5,-1) {};
    \node (k3+) at (4,-1) {};
    \node (k3) at (2.5,-1) {$r_3^{(k)}$};
    \node (k2) at (1,-1) {$r_2^{(k)}$};
    
    \draw[-] (kM1) -- (kM1-);
    \draw[loosely dotted] (kM1-) -- (k3+);
    \draw[-] (k3+) -- (k3);
    \draw[-] (k3) -- (k2);
    \draw[-] (k2) -- (R);
    
    \draw[loosely dotted] (3.2,0.8) -- (3.2,-0.8);
\end{tikzpicture}\]
We say that an algebra $\La=\K Q/\cR$ is a \emph{starlike algebra}\index[definitions]{starlike algebra} if the underlying graph of $Q$ is a starlike tree. In this case we denote by $A_{m_i}$ the full subquiver of $Q$ with vertices $\{r,r_2^{(i)},\dots,r_{m_i}^{(i)}\}$.
\end{definition}

The condition $k\neq 2$ is included so that each tuple $(k,m_1,\dots,m_k)$ gives rise to a unique starlike tree. 

\begin{proposition}\label{prop:starlike algebras}
Let $\La=\K Q/R_Q^2$ be a starlike algebra where the underlying graph of $Q$ is a starlike tree $T(k,m_1,m_2,\dots,m_{k})$ and let $n\geq 2$. Then $\m\La$ admits an $n$-cluster tilting subcategory $\cC$ if and only if $A_{m_i}$ is linearly oriented for every $i\in\{1,\dots,k\}$ and one of the following conditions holds.
\begin{enumerate}[label=(\alph*)]
    \item $k=1$ and there exists an $x\in\ZZ_{\geq 0}$ such that $m_1=n+1+xn$.\label{case:k=1}
    \item $k=3$, and there exist $x_i\in\ZZ_{\geq 0}$ for $i\in\{1,2,3\}$ such that one of the following conditions holds.
    \begin{enumerate}
        \item[(b1)] Two of $\{r_{m_1}^{(1)},r_{m_2}^{(2)},r_{m_3}^{(3)}\}$ are sinks, in which case $m_i=n+1+x_i n$ and one of them is a source, in which case $m_i=n+x_i n$.
        \item[(b2)] Two of $\{r_{m_1}^{(1)},r_{m_2}^{(2)},r_{m_3}^{(3)}\}$ are sources, in which case $m_i=n+1+x_i n$ and one of them is a sink, in which case $m_i=n+x_i n$.
    \end{enumerate}\label{case:k=3}
    \item $k=4$ and $n=2$, there exist $x_i\in\ZZ_{\geq 0}$ for $i\in\{1,2,3,4\}$ such that $m_{i}=3+2x_i$ and exactly two of $\{r_{m_1}^{(1)},r_{m_2}^{(2)},r_{m_3}^{(3)},r_{m_4}^{(4)}\}$ are sources.\label{case:k=4} 
\end{enumerate}
In this case the algebra $\La$ is rep\-re\-sen\-ta\-tion-di\-rect\-ed, the $n$-cluster tilting subcategory $\cC$ is unique and the global dimension $\gldim(\La)=d$ of $\La$ is $d=(x+1)n$ in case \ref{case:k=1}, it is $d=\left(2+x_3+\max\{x_1,x_2\}\right)n-1$ in case \ref{case:k=3}, and it is $d=4+2\left(\max\{x_1,x_2\} + \max\{x_3,x_4\}\right)$ in case \ref{case:k=4}.
\end{proposition}

\begin{proof}
First notice that in all cases \ref{case:k=1}, \ref{case:k=3}, and \ref{case:k=4} the algebra $\La$ is rep\-re\-sen\-ta\-tion-fi\-nite since it is a radical square zero algebra and its separated quiver is a disjoint union of Dynkin diagrams, see \cite{Gab1}. Moreover, it satisfies the separation condition and so it is rep\-re\-sen\-ta\-tion-di\-rect\-ed, see \cite[Section IX.4]{ASS}.

Let us first show that the conditions are necessary for the existence of an $n$-cluster tilting subcategory. Let $\cC$ be an $n$-cluster tilting subcategory of $\m\La$. Assume to a contradiction that $A_{m_i}$ is not linearly oriented for some $i\in\{1,\dots,k\}$. Then there exists a vertex $p\in V_{A_{m_i}}$ such that $p$ is a sink or a source and such that the full subquiver of $Q$ containing the vertices $Q_0\setminus\{p\}$ is disconnected. But this contradicts \cite[Proposition 4.1]{VAS}. 

Since $\cC$ is an $n$-cluster tilting subcategory, we have $\La\in\cC$ and $\D(\La)\in \cC$. Let $f$ be the number of arrows in $Q$ with $r$ as a target and let $g$ be the number of arrows in $Q$ with $r$ as a source. A direct computation shows that $\om(I(r))=S(r)^{\oplus(f-1)}$ and $\om^{-}(P(r))=S(r)^{\oplus(g-1)}$ where $I(r)$, $P(r)$ and $S(r)$ are the indecomposable injective, projective and simple modules corresponding to the vertex $r$ respectively. If $f>2$, then $I(r)\in\cC$ is not projective but has decomposable syzygy, which contradicts \cite[Corollary 3.3]{VAS}. It follows $f\leq 2$ and similarly $g\leq 2$. In particular, since $f+g=k$, we have that $k\leq 4$. We proceed by showing that in each of the cases $k=1,3,4$ the conditions of the proposition are necessary and sufficient.

If $k=1$, then $Q=\overrightarrow{A}_{m_1}$ and it follows by \cite[Proposition 6.2]{JAS} that $\m\left(\K \overrightarrow{A}_{m_1}/R^2_{\overrightarrow{A}_{m_1}}\right)$ admits an $n$-cluster tilting subcategory if and only if $n \mid (m_1-1)$. That $\gldim(\La)=m_1-1=(x+1)n$ can be shown by an easy computation; for the proof of a more general case see \cite[Proposition 5.2]{VAS}.

If $k=3$, and since $f,g\leq 2$, it follows that $r$ is the target of two arrows and the source of one or vice versa. Since $A_{m_1}$, $A_{m_2}$ and $A_{m_3}$ are linearly oriented, it follows that either exactly two of $\{r_{m_1}^{(1)},r_{m_2}^{(2)},r_{m_3}^{(3)}\}$ are sinks or exactly two of them are sources. Let us consider the case where $r_{m_1}^{(1)}$ and $r_{m_2}^{(2)}$ are both sources and $r_{m_3}^{(3)}$ is a sink; the other cases are similar. In this case we can directly compute that
\begin{align}
    \label{eq:decomposable cosyzygy for starlike tree}
    \om^{-}\left(S(r)\right) &\isom S\left(r_2^{(1)}\right)\oplus S\left(r_2^{(2)}\right), \\
    \label{eq:cosyzygy for starlike tree}
    \om^{-j}\left(S\left(r_l^{(3)}\right)\right) &\isom S\left(r^{(3)}_{l-j}\right), \text{ for $2\leq l\leq m_3$ and $0\leq j\leq l-1$, and} \\ 
    \label{eq:inverse tau for starlike tree}
    \tau^{-}\left(S\left(r^{(3)}_l\right)\right) &\isom S\left(r^{(3)}_{l-1}\right), \text{ for $2\leq l\leq m_3$,}
\end{align}
under the convention $r^{(3)}_1=r$. Since $r_{m_3}^{(3)}$ is a sink, we have $P\left(r_{m_3}^{(3)}\right)\isom S\left(r_{m_3}^{(3)}\right)\in \cC$ and $S\left(r_{m_3}^{(3)}\right)$ is not injective. By \cite[Theorem 1]{VAS} we have that $\tn^{-q}\left(S\left(r_{m_3}^{(3)}\right)\right)\in\cC$. By (\ref{eq:cosyzygy for starlike tree}) and (\ref{eq:inverse tau for starlike tree}) we have $\tn^{-q}\left(S\left(r_{m_3}^{(3)}\right)\right)\isom S\left(r_{qn}^{(3)}\right)$, as long as $qn\leq 1$. In particular, if $x_3$ is maximal such that $x_3 n\leq 1$, then there exists a $0\leq p\leq n-1$ such that $\om^{-p}\left(S\left(r_{x_3 n}^{(3)}\right)\right)\isom S\left(r_1^{(3)}\right)=S(r)$. If $p<n-1$, then $\om^{-(p+1)}\left(S\left(r_{x_3 n}^{(3)}\right)\right)\isom S\left(r_2^{(1)}\right)\oplus S\left(r_2^{(2)}\right)$ by (\ref{eq:decomposable cosyzygy for starlike tree}), contradicting the fact that $\cC$ is an $n$-cluster tilting subcategory since $\tn^{-x_3}\left(S\left(r_{m_3}^{(3)}\right)\right)\in \cC$ has a decomposable $(p+1)$-th cosyzygy. Therefore $p=n-1$ and so 
\[m_3 -x_3 n-(n-1) = 1\]
or $m_3=n+x_3 n$, as claimed. By similar arguments applied to $I\left(r_{m_1}^{(1)}\right)$ and $I\left(r_{m_2}^{(3)}\right)$ we get that $m_1=n+1+x_1n$ and $m_2=n+1+x_2n$. That these conditions are sufficient follows by a direct computation and using \cite[Theorem 1]{VAS}. Finally, for the global dimension we have for $i=1,2$ that
\[\om^{-j}\left(S\left(r_2^{(i)}\right)\right) \isom S\left(r_{2+j}^{(i)}\right), \text{for $j\leq m_i-2$,}\]
and since moreover
\[\om^{-(m_3-2)}\left(S\left(r_{m_3}^{(3)}\right)\right) \isom \om^{-}\left(S\left(r_{m_3-(m_3-1)}^{(3)}\right)\right) \isom \om^{-}\left(S\left(r_{1}^{(3)}\right)\right)=\om^{-}\left(S(r)\right)\isom S\left(r_2^{(1)}\right)\oplus S\left(r_2^{(2)}\right),\]  
every simple module appears as a summand of the $j$-th cosyzygy of $S\left(r_{m_3}^{(3)}\right)$. Hence
\begin{align*}
    \gldim(\La) &= \idim \left(S\left(r_{m_3}^{(3)}\right)\right) \\
    &=m_3 + \idim \left(S\left(r_2^{(1)}\right)\oplus S\left(r_2^{(2)}\right)\right) \\
    &=m_3 + \max\left\{\idim \left(S\left(r_2^{(1)}\right)\right), \idim\left(S\left(r_2^{(2)}\right)\right)\right\} \\
    &=m_3 + \max\{m_1-2,m_2-2\} \\
    &=n+x_3 n + \max\{n+x_1n-1,n+x_2n-1\} \\
    &=\left(2+x_3+\max\{x_1,x_2\}\right)n-1
\end{align*}
as claimed.

Finally, for the case $k=4$ we have that $g\leq 2$ and $f\leq 2$ and $g+f=4$ which implies $g=f=2$ and so two of $\{r_{m_1}^{(1)},r_{m_2}^{(2)},r_{m_3}^{(3)},r_{m_4}^{(4)}\}$ are sources and the other two are sinks. Let us assume that $r_{m_1}^{(1)}$ and $r_{m_2}^{(2)}$ are both sinks and $r_{m_3}^{(3)}$ and $r_{m_4}^{(4)}$ are both sources; the other cases are similar. It is easy to check that there is an exact sequence 
\[0 \to P(r) \to I\left(r_2^{(3)}\right)\oplus I\left(r_2^{(4)}\right) \to P\left(r_2^{(1)}\right)\oplus P\left(r_2^{(2)}\right) \to I(r) \to 0,\]
and hence $\Ext^2_{\La}\left(I(r),P(r)\right)\neq 0$, showing that $n\leq 2$. Moreover, if $n=2$, very similar computations to the case $k=3$ show that there exists a $2$-cluster tilting subcategory $\cC$ if and only if $m_i=3+2x_i$ for some $x_1,x_2,x_3,x_4\in\ZZ_{\geq 0}$. The global dimension is then again computed similarly to be
\begin{align*}
    \gldim(\La) &= \max\left\{\idim(S\left(r_{m_1}^{(1)}\right),\idim\left(S\left(r_{m_2}^{(2)}\right)\right)\right\}\\
    &=\max\left\{m_1+\max\{m_3-2,m_4-2\}, m_2+\max\{m_3-2,m_4-2\}\right\} \\ 
    &=\max\left\{m_1,m_2\right\} + \max\{m_3-2,m_4-2\} \\
    &=\max\{3+2x_1,3+2x_2\} + \max\{1+2x_3,1+2x_4\} \\
    &=4+2\left(\max\{x_1,x_2\} + \max\{x_3,x_4\}\right).\qedhere
\end{align*}
\end{proof}

For an example coming from Proposition \ref{prop:starlike algebras} see Figure \ref{fig:starlike example}. Using the algebras of Proposition \ref{prop:starlike algebras} we can now answer Question \ref{question:sinks and sources}.

\begin{figure}[htb]
    \centering
    \begin{tikzpicture}[baseline={(current bounding box.center)}]
        \tikzstyle{nct3}=[circle, minimum width=6pt, draw=white, inner sep=0pt, scale=0.9]
        \node (name) at (-0.3,0) {$Q:$};
        
        \node[nct3] (1) at (0,0) {$1$};
        
        \node[nct3] (21) at (0.7,0.7) {$2_1$};
        \node[nct3] (31) at (1.4,0.7) {$3_1$};
        \node[nct3] (41) at (2.1,0.7) {$4_1$};
        \node[nct3] (51) at (2.8,0.7) {$5_1$};
        
        \node[nct3] (22) at (0.7,0) {$2_2$};
        \node[nct3] (32) at (1.4,0) {$3_2$};
        \node[nct3] (42) at (2.1,0) {$4_2$};
        \node[nct3] (52) at (2.8,0) {$5_2$\nospacepunct{,}};

        \node[nct3] (23) at (0.7,-0.7) {$2_3$};
        \node[nct3] (33) at (1.4,-0.7) {$3_3$};
        \node[nct3] (43) at (2.1,-0.7) {$4_3$};
        
        \draw[->] (1) -- (21);
        \draw[->] (21) -- (31);
        \draw[->] (31) -- (41);
        \draw[->] (41) -- (51);
        
        \draw[->] (1) -- (22);
        \draw[->] (22) -- (32);
        \draw[->] (32) -- (42);
        \draw[->] (42) -- (52);

        \draw[<-] (1) -- (23);
        \draw[<-] (23) -- (33);
        \draw[<-] (33) -- (43);
    \end{tikzpicture}\;\;\;\;
    \begin{tikzpicture}[scale=0.9, transform shape, baseline={(current bounding box.center)}]
        \tikzstyle{nct2}=[circle, minimum width=0.6cm, draw, inner sep=0pt, text centered, scale=0.9]
        \tikzstyle{nct22}=[circle, minimum width=0.6cm, draw, inner sep=0pt, text centered, scale=0.8]
        \tikzstyle{nct3}=[circle, minimum width=6pt, draw=white, inner sep=0pt, scale=0.9]
        
        \node[scale=0.8] (name) at (4.9,-0.3) {$\Gamma\left(\K Q/R_Q^2\right)$};
        
        \node[nct2] (12) at (0,0.7) {$\qthree{}[5_2][]$};
        \node[nct2] (21) at (0.7,0) {$\qthree{4_2}[5_2]$};
        \node[nct3] (32) at (1.4,0.7) {$\qthree{}[4_2][]$};
        \node[nct2] (41) at (2.1,0) {$\qthree{3_2}[4_2]$};
        \node[nct3] (52) at (2.8,0.7) {$\qthree{}[3_2][]$};
        \node[nct2] (61) at (3.5,0) {$\qthree{2_2}[3_2]$};
        \node[nct3] (72) at (4.2,0.7) {$\qthree{}[2_2][]$};
        
        \draw[->] (12) -- (21);
        \draw[->] (21) -- (32);
        \draw[->] (32) -- (41);
        \draw[->] (41) -- (52);
        \draw[->] (52) -- (61);
        \draw[->] (61) -- (72);
        
        \draw[loosely dotted] (12.east) -- (32);
        \draw[loosely dotted] (32.east) -- (52);
        \draw[loosely dotted] (52.east) -- (72);
        
        \node[nct2] (14) at (0,2.1) {$\qthree{}[5_1][]$};
        \node[nct2] (25) at (0.7,2.8) {$\qthree{4_1}[5_1]$};
        \node[nct3] (34) at (1.4,2.1) {$\qthree{}[4_1][]$};
        \node[nct2] (45) at (2.1,2.8) {$\qthree{3_1}[4_1]$};
        \node[nct3] (54) at (2.8,2.1) {$\qthree{}[3_1][]$};
        \node[nct2] (65) at (3.5,2.8) {$\qthree{2_1}[3_1]$};
        \node[nct3] (74) at (4.2,2.1) {$\qthree{}[2_1][]$};
        
        \draw[->] (14) -- (25);
        \draw[->] (25) -- (34);
        \draw[->] (34) -- (45);
        \draw[->] (45) -- (54);
        \draw[->] (54) -- (65);
        \draw[->] (65) -- (74);
        
        \draw[loosely dotted] (14.east) -- (34);
        \draw[loosely dotted] (34.east) -- (54);
        \draw[loosely dotted] (54.east) -- (74);
        
        \node[nct2] (83) at (4.9,1.4) {$\begin{smallmatrix}  & 1 & \\ & 2_1 2_2 &\end{smallmatrix} $};
        
        \draw[->] (72) -- (83);
        \draw[->] (74) -- (83);
        
        \node[nct2] (94) at (5.6,2.1) {$\qthree{1}[2_2]$};
        \node[nct2] (92) at (5.6,0.7) {$\qthree{1}[2_1]$};
        \node[nct3] (103) at (6.3,1.4) {$\qthree{}[1][]$};
        \node[nct2] (114) at (7,2.1) {$\qthree{2_3}[1]$};
        \node[nct3] (123) at (7.7,1.4) {$\qthree{}[2_3][]$};
        \node[nct2] (134) at (8.4,2.1) {$\qthree{3_3}[2_3]$};
        \node[nct3] (143) at (9.1,1.4) {$\qthree{}[3_3][]$};
        \node[nct2] (154) at (9.8,2.1) {$\qthree{4_3}[3_3]$};
        \node[nct2] (163) at (10.5,1.4) {$\qthree{}[4_3][]$};
        
        \draw[->] (83) -- (94);
        \draw[->] (83) -- (92);
        \draw[->] (94) -- (103);
        \draw[->] (92) -- (103);
        \draw[->] (103) -- (114);
        \draw[->] (114) -- (123);
        \draw[->] (123) -- (134);
        \draw[->] (134) -- (143);
        \draw[->] (143) -- (154);
        \draw[->] (154) -- (163);
        
        \draw[loosely dotted] (83.east) -- (103);
        \draw[loosely dotted] (103.east) -- (123);
        \draw[loosely dotted] (123.east) -- (143);
        \draw[loosely dotted] (143.east) -- (163);
        \draw[loosely dotted] (74.east) -- (94);
        \draw[loosely dotted] (72.east) -- (92);
    \end{tikzpicture}
    \caption{The Aus\-lan\-der--Rei\-ten quiver of the algebra $\La=\K Q/R_Q^2$ where $Q$ is the quiver on the left and where the additive closure of the encircled modules is a $4$-cluster tilting subcategory.} 
\label{fig:starlike example}
\end{figure}

\begin{proposition}\label{prop:sinks and sources exist}
Let $\left(s,t,n\right)$ be a triple of positive integers. Then there exists a rep\-re\-sen\-ta\-tion-di\-rect\-ed algebra $\La$ with $\sources_{\La}=s$ and $\sinks_{\La}=t$ such that $\m\La$ admits an $n$-cluster tilting subcategory.
\end{proposition}

\begin{proof}
If $n=1$ the result is clear. Otherwise, consider the graph $G$ given by
\[-(s-1)\overset{f_{-(s-1)}}{\longto} \cdots \overset{f_{-3}}{\longto} -2 \overset{f_{-2}}{\longto}-1\overset{f_{-1}}{\longto}0\overset{g_1}{\longto}1\overset{g_2}{\longto}2\overset{g_3}{\longto}\cdots \overset{g_{t-1}}{\longto}t-1.\]
We set $\La_0$ to be any starlike algebra satisfying the conditions of Proposition \ref{prop:starlike algebras}\ref{case:k=1} and set $P_{f_{-1}}$ to be the unique simple left abutment of $\La_0$ and $I_{e_1}$ to be the unique simple right abutment of $\La_0$. Let $v\in V_G\setminus\{0\}$. We set $\La_{v}$ to be any starlike algebra satisfying the conditions of Proposition \ref{prop:starlike algebras}\ref{case:k=3} such that $\La_v$ has two sources if $v<0$ and $\La_v$ has two sinks if $v>0$. For an arrow $f_{i}\in E_G$ we set $I_{f_i}$ to be the unique simple right abutment of $\La_i$ and $P_{f_i}$ to be any of the two simple left abutments of $\La_i$; dually we define $I_{e_i}$ and $P_{e_i}$. It follows by Remark \ref{rem:gluing system needs sources and sinks} that $(\La_{v},P_{e},I_{e})_{v\in V_G,e\in E_G}$ is a gluing system on $G$. Moreover, by Proposition \ref{prop:starlike algebras} we have that for all $v\in V_G$ there exists an $n$-cluster tilting subcategory $\cM_v\subseteq \m\La_v$ and so by Proposition \ref{prop:gluing system of n-ct}(c) we have that $\cM_G$ is an $n$-cluster tilting subcategory of $\m\La_G$. Now notice that by the definition of gluing and a simple induction, the algebra $\La_{\langle -i,\dots,0,\dots,j\rangle}$ has $\sources_{\La_{\langle -i,\dots,0,\dots,j\rangle}}=1+i$ and $\sinks_{\La_{\langle -i,\dots,0,\dots,j\rangle}}=j+i$. It follows that $\La_G$ has $\sources_{\La_G}=s$ and $\sinks_{\La_G}=t$. 
\end{proof}

\begin{example}\label{ex:4 sources 3 sinks}
The algebra $\La=\K Q/R_Q^2$ where $Q$ is the quiver
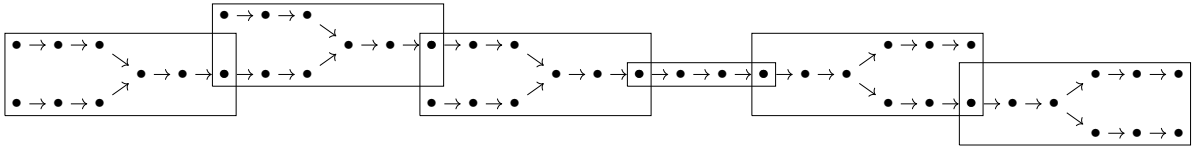
\[\resizebox {\columnwidth} {!} {\begin{tikzpicture}
\node (1) at (1,1) {$\bullet$};
\node (1') at (1,2) {$\bullet$};
\node (2) at (1.7,1) {$\bullet$};
\node (2') at (1.7,2) {$\bullet$};
\node (3) at (2.4,1) {$\bullet$};
\node (3') at (2.4,2) {$\bullet$};
\node (4) at (3.1,1.5) {$\bullet$};
\node (5) at (3.8,1.5) {$\bullet$};
\node (6) at (4.5,1.5) {$\bullet$};

\draw[->] (1) -- (2);
\draw[->] (2) -- (3);
\draw[->] (3) -- (4);
\draw[->] (1') -- (2');
\draw[->] (2') -- (3');
\draw[->] (3') -- (4);
\draw[->] (4) -- (5);
\draw[->] (5) -- (6);

\draw[line width=0.05mm] (0.8,0.8) -- (0.8, 2.2) -- (4.7,2.2) -- (4.7, 0.8) -- (0.8,0.8);

\node (12) at (4.5,1.5) {$\bullet$};
\node (12') at (4.5,2.5) {$\bullet$};
\node (22) at (5.2,1.5) {$\bullet$};
\node (22') at (5.2,2.5) {$\bullet$};
\node (32) at (5.9,1.5) {$\bullet$};
\node (32') at (5.9,2.5) {$\bullet$};
\node (42) at (6.6,2) {$\bullet$};
\node (52) at (7.3,2) {$\bullet$};
\node (62) at (8,2) {$\bullet$};

\draw[->] (12) -- (22);
\draw[->] (22) -- (32);
\draw[->] (32) -- (42);
\draw[->] (12') -- (22');
\draw[->] (22') -- (32');
\draw[->] (32') -- (42);
\draw[->] (42) -- (52);
\draw[->] (52) -- (62);

\draw[line width=0.05mm] (4.3,1.3) -- (4.3, 2.7) -- (8.2,2.7) -- (8.2, 1.3) -- (4.3,1.3);

\node (13) at (8,2) {$\bullet$};
\node (13') at (8,1) {$\bullet$};
\node (23) at (8.7,2) {$\bullet$};
\node (23') at (8.7,1) {$\bullet$};
\node (33) at (9.4,2) {$\bullet$};
\node (33') at (9.4,1) {$\bullet$};
\node (43) at (10.1,1.5) {$\bullet$};
\node (53) at (10.8,1.5) {$\bullet$};
\node (63) at (11.5,1.5) {$\bullet$};

\draw[->] (13) -- (23);
\draw[->] (23) -- (33);
\draw[->] (33) -- (43);
\draw[->] (13') -- (23');
\draw[->] (23') -- (33');
\draw[->] (33') -- (43);
\draw[->] (43) -- (53);
\draw[->] (53) -- (63);

\draw[line width=0.05mm] (7.8,0.8) -- (7.8, 2.2) -- (11.7,2.2) -- (11.7, 0.8) -- (7.8,0.8);

\node (14) at (11.5,1.5) {$\bullet$};
\node (24) at (12.2,1.5) {$\bullet$};
\node (34) at (12.9,1.5) {$\bullet$};
\node (44) at (13.6,1.5) {$\bullet$};

\draw[->] (14) -- (24);
\draw[->] (24) -- (34);
\draw[->] (34) -- (44);

\draw[line width=0.05mm] (11.3,1.3) -- (11.3, 1.7) -- (13.8,1.7) -- (13.8, 1.3) -- (11.3,1.3);

\node (15) at (13.6,1.5) {$\bullet$};
\node (25) at (14.3,1.5) {$\bullet$};
\node (35) at (15,1.5) {$\bullet$};
\node (45) at (15.7,2) {$\bullet$};
\node (45') at (15.7,1) {$\bullet$};
\node (55) at (16.4,2) {$\bullet$};
\node (55') at (16.4,1) {$\bullet$};
\node (65) at (17.1,2) {$\bullet$};
\node (65') at (17.1,1) {$\bullet$};

\draw[->] (15) -- (25);
\draw[->] (25) -- (35);
\draw[->] (35) -- (45);
\draw[->] (35) -- (45');
\draw[->] (45) -- (55);
\draw[->] (45') -- (55');
\draw[->] (55) -- (65);
\draw[->] (55') -- (65');

\draw[line width=0.05mm] (13.4,0.8) -- (13.4, 2.2) -- (17.3,2.2) -- (17.3, 0.8) -- (13.4,0.8);

\node (16) at (17.1,1) {$\bullet$};
\node (26) at (17.8,1) {$\bullet$};
\node (36) at (18.5,1) {$\bullet$};
\node (46) at (19.2,1.5) {$\bullet$};
\node (46') at (19.2,0.5) {$\bullet$};
\node (56) at (19.9,1.5) {$\bullet$};
\node (56') at (19.9,0.5) {$\bullet$};
\node (66) at (20.6,1.5) {$\bullet$};
\node (66') at (20.6,0.5) {$\bullet$};

\draw[->] (16) -- (26);
\draw[->] (26) -- (36);
\draw[->] (36) -- (46);
\draw[->] (36) -- (46');
\draw[->] (46) -- (56);
\draw[->] (46') -- (56');
\draw[->] (56) -- (66);
\draw[->] (56') -- (66');

\draw[line width=0.05mm] (16.9,0.3) -- (16.9, 1.7) -- (20.8,1.7) -- (20.8, 0.3) -- (16.9,0.3);
\end{tikzpicture}}\]
has $\sources_{\La}=4$, $\sinks_{\La}=3$. The quiver $Q$ was obtained by gluing the path algebras of the quivers inside the rectangles modulo the ideal generated by all paths of length $2$, and where the gluing is done along the simple modules corresponding to the vertices in the intersection of the boxes. Since the module categories of all the algebras inside the rectangles admit a $3$-cluster tilting subcategory by Proposition \ref{prop:starlike algebras}, it follows that $\m\La$ also admits a $3$-cluster tilting subcategory. Indeed, $\cM=\add\left\{ \tau_3^{-k}(\La) \mid k\geq 0\right\} \subseteq \m\La$ is a $3$-cluster tilting subcategory. Notice also that we could have omitted the rectangle containing a quiver of type $\overrightarrow{A}_4$ and get another algebra with the same properties.
\end{example}

As a corollary of Proposition \ref{prop:sinks and sources exist}, we have that for every directed tree $G$ we can find a gluing system $\cL$ on $G$ and an $n$-fractured system of $\cL$ which gives rise to an $n$-cluster tilting subcategory.

\begin{corollary}
Let $G$ be a directed tree and $n\geq 1$. Then there exists a gluing system $\cL=(\La_{v},P_{e},I_{e})_{v\in V_G,e\in E_G}$ on $G$ and an $n$-fractured system $(\cM_v)_{v\in V_G}$ of $\cL$ such that if $(\Cks,\Pks_H,\Tks_{HK})$ is the firm source of the induced $\Cat$-inverse system $\left(\cC_H,F_{HK},\theta_{HKL}\right)$, then $\Mks$ is an $n$-cluster tilting subcategory of $\m\Cks$.
\end{corollary}

\begin{proof}
Let $\La_v$ be a rep\-re\-sen\-ta\-tion-di\-rect\-ed algebra such that $\sources_{\La_v}\geq\delta^{+}(v)$ and $\sinks_{\La_v}\geq\delta^{-}(v)$ and $\m\La_v$ admits an $n$-cluster tilting subcategory; such an algebra exists by Proposition \ref{prop:sinks and sources exist}. The result follows by Proposition \ref{prop:gluing system of n-ct} and Remark \ref{rem:gluing system needs sources and sinks}.
\end{proof}

Notice that Question \ref{question:sinks and sources} has a negative answer if $s=0$ or $t=0$. Indeed, if $\sources_{\La}=0$ or $\sinks_{\La}=0$, then the quiver of $\La$ is not acyclic and in particular $\La$ is not rep\-re\-sen\-ta\-tion-di\-rect\-ed. Hence we pose the following more general question.

\begin{question}\label{question:sinks and sources 0}
Let $\left(s,t,n\right)$ be a triple of nonnegative integers with $n\geq 1$. Does there exist a connected algebra $\La$ such that $\sources_{\La}=s$, $\sinks_{\La}=t$ and $\m\La$ admits an $n$-cluster tilting subcategory? 
\end{question}

To construct such an algebra we use $n$-self-gluable algebras as described in Section \ref{sec:gluings and orbit categories}. We have the following result which answers Question \ref{question:sinks and sources 0} affirmatively.

\begin{theorem}
Let $\left(s,t,n\right)$ be a triple of nonnegative integers and $n\geq 1$. Then there exists a bound quiver algebra $\tilLa$ with $\sources_{\tilLa}=s$ and $\sinks_{\tilLa}=t$ such that $\m\tilLa$ admits an $n$-cluster tilting subcategory.
\end{theorem}

\begin{proof}
If $n=1$ the result is clear. Otherwise, let $\La$ be a rep\-re\-sen\-ta\-tion-di\-rect\-ed algebra with $\sources_{\La}=s+1$ and $\sinks_{\La}=t+1$ such that $\m\La$ admits an $n$-cluster tilting subcategory; such an algebra exists by Proposition \ref{prop:sinks and sources exist} since $s,t\geq 0$. Let $P$ be a simple projective $\La$-module and let $I$ be a simple injective $\La$-module. Let $W$ be a maximal right abutment of $\La$ with $P\leq W$ and $J$ be a maximal right abutment of $\La$ with $I\leq J$. Then $(W,J)$ is a fractured pair and $(P,J)$ is a compatible pair for the fractured pair $(W,J)$. Hence $\La$ is $n$-self-gluable and $\m\tilLa$ admits an $n$-cluster tilting subcategory by Corollary \ref{cor:La n-ct implies La infinity n-ct}. Moreover, by (\ref{picture:self-gluing quiver}) we have $\sources_{\tilLa}=\sources_{\La}-1=s$ and $\sinks_{\tilLa}=\sinks_{\La}-1=t$ as required.
\end{proof}

\begin{example}\label{ex:0 sources 3 sinks}
The algebra $\La=\K Q/R_Q^2$ where $Q$ is the quiver
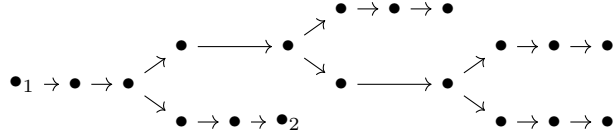
\[\begin{tikzpicture}
\node (1) at (1,1) {$\bullet_1$};
\node (2) at (1.7,1) {$\bullet$};
\node (3) at (2.4,1) {$\bullet$};
\node (4) at (3.1,0.5) {$\bullet$};
\node (4') at (3.1,1.5) {$\bullet$};
\node (5) at (3.8,0.5) {$\bullet$};
\node (6) at (4.5,0.5) {$\bullet_2$};
\node (6') at (4.5,1.5) {$\bullet$};
\node (7) at (5.2,1) {$\bullet$};
\node (7') at (5.2,2) {$\bullet$};
\node (8') at (5.9,2) {$\bullet$};
\node (9) at (6.6,1) {$\bullet$};
\node (9') at (6.6,2) {$\bullet$};
\node (10) at (7.3,0.5) {$\bullet$};
\node (10') at (7.3,1.5) {$\bullet$};
\node (11) at (8,0.5) {$\bullet$};
\node (11') at (8,1.5) {$\bullet$};
\node (12) at (8.7,0.5) {$\bullet$};
\node (12') at (8.7,1.5) {$\bullet$};

\draw[->] (1) -- (2);
\draw[->] (2) -- (3);
\draw[->] (3) -- (4);
\draw[->] (3) -- (4');
\draw[->] (4) -- (5);
\draw[->] (4') -- (6');
\draw[->] (5) -- (6);
\draw[->] (6') -- (7);
\draw[->] (6') -- (7');
\draw[->] (7) -- (9);
\draw[->] (7') -- (8');
\draw[->] (8') -- (9');
\draw[->] (9) -- (10);
\draw[->] (9) -- (10');
\draw[->] (10) -- (11);
\draw[->] (10') -- (11');
\draw[->] (11) -- (12);
\draw[->] (11') -- (12');
\end{tikzpicture}\]
has $\sources_{\La}=1$, $\sinks_{\La}=4$ and $\cM=\add\left\{ \tau_3^{-k}(\La)\mid k\geq 0\right\}\subseteq \m\La$ is a $3$-cluster tilting subcategory. Notice that this algebra is not obtained by gluing starlike algebras as in the proof of Proposition \ref{prop:sinks and sources exist}, but a direct computation using \cite[Theorem 1]{VAS} shows that $\cM$ is indeed a $3$-cluster tilting subcategory. By self-gluing at the simple injective module $I(\bullet_1)$ and the simple projective module $P(\bullet_2)$ we obtain the algebra $\tilde{\La}=\K \tilQ/R_{\tilQ}^2$ where $\tilQ$ is the quiver
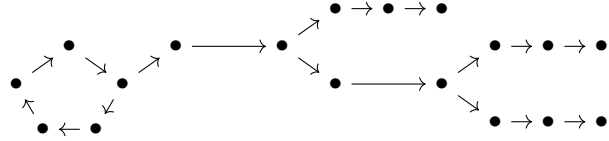
\[\begin{tikzpicture}
\node (1) at (1,1) {$\bullet$};
\node (2) at (1.7,1.5) {$\bullet$};
\node (3) at (2.4,1) {$\bullet$};
\node (4) at (2.05,0.4) {$\bullet$};
\node (4') at (3.1,1.5) {$\bullet$};
\node (5) at (1.35,0.4) {$\bullet$};
\node (6') at (4.5,1.5) {$\bullet$};
\node (7) at (5.2,1) {$\bullet$};
\node (7') at (5.2,2) {$\bullet$};
\node (8') at (5.9,2) {$\bullet$};
\node (9) at (6.6,1) {$\bullet$};
\node (9') at (6.6,2) {$\bullet$};
\node (10) at (7.3,0.5) {$\bullet$};
\node (10') at (7.3,1.5) {$\bullet$};
\node (11) at (8,0.5) {$\bullet$};
\node (11') at (8,1.5) {$\bullet$};
\node (12) at (8.7,0.5) {$\bullet$};
\node (12') at (8.7,1.5) {$\bullet$};

\draw[->] (1) -- (2);
\draw[->] (2) -- (3);
\draw[->] (3) -- (4);
\draw[->] (3) -- (4');
\draw[->] (4) -- (5);
\draw[->] (4') -- (6');
\draw[->] (5) -- (1);
\draw[->] (6') -- (7);
\draw[->] (6') -- (7');
\draw[->] (7) -- (9);
\draw[->] (7') -- (8');
\draw[->] (8') -- (9');
\draw[->] (9) -- (10);
\draw[->] (9) -- (10');
\draw[->] (10) -- (11);
\draw[->] (10') -- (11');
\draw[->] (11) -- (12);
\draw[->] (11') -- (12');
\end{tikzpicture}\]
with $\sources_{\tilde{\La}}=0$ and $\sinks_{\tilde{\La}}=3$ and $\Mks/\ZZ$ is a $3$-cluster tilting subcategory of $\m\tilde{\La}$.
\end{example}

\begin{remark}
If $Q$ is an acyclic connected quiver such that $\abs{Q_0}<\infty$, then $Q$ has at least one sink and at least one source. However, if we allow $\abs{Q_0}=\infty$, this is no longer true. By letting $G=\overrightarrow{A}^{\infty}_{\infty}$ we can adapt the proof of Proposition \ref{prop:sinks and sources exist} to show that if $(s,t,n)$ is a triple of nonnegative integers with $n\geq 1$, then there exists an infinite connected acyclic quiver $Q$ with $s$ sources and $t$ sinks and a two-sided ideal $\cR$ of the path category $\K Q$ such that $\m\left(\K Q/\cR\right)$ admits an $n$-cluster tilting subcategory. 
\end{remark}

\subsection{Examples from algebras with \texorpdfstring{$n$}{n}-cluster tilting subcategories} Given a rep\-re\-sen\-ta\-tion-di\-rect\-ed algebra whose module category admits an $n$-cluster tilting subcategory, Corollary \ref{cor:La n-ct implies La infinity n-ct} produces a new algebra whose module category admits an $n$-cluster tilting subcategory. Proposition \ref{prop:starlike algebras} gives a list of rep\-re\-sen\-ta\-tion-di\-rect\-ed algebras whose module categories admit $n$-cluster tilting subcategories. We may hence apply Corollary \ref{cor:La n-ct implies La infinity n-ct} to these algebras.

The algebras of Proposition \ref{prop:starlike algebras}\ref{case:k=1} are special cases of \emph{Nakayama algebras}\index[definitions]{Nakayama algebra}. Many more examples of Nakayama algebras whose module categories admit $n$-cluster tilting subcategories are already known, so we may apply Corollary \ref{cor:La n-ct implies La infinity n-ct} to an even wider collection of examples.

Let us start by recalling some facts about Nakayama algebras. Throughout we denote by $\tilde{A}_h$ the quiver
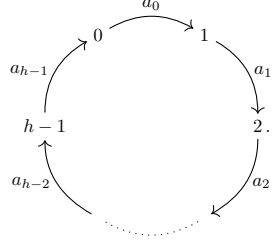
\[
\begin{tikzpicture}[scale=0.7, transform shape]
\node (1) at (1,1.732050808) {$1$};
\node (0) at (-1,1.732050808) {$0$};
\node (m-1) at (-2,0) {$h-1$};
\node (m-2) at (-1,-1.732050808) {$ $};
\node (dots) at (1,-1.732050808) {$ $};
\node (2) at (2,0) {$2$\nospacepunct{.}};
\draw[->] (0) to [out=30, in=150]  node[draw=none, above] {$a_{0}$} (1);
\draw[->] (1) to [out=-30, in=90] node[draw=none, midway, right] {$a_{1}$} (2);
\draw[->] (m-1) to [out=90, in=-150] node[draw=none, midway, left] {$a_{h-1}$} (0);
\draw[->] (m-2) to [out=150, in=-90] node[draw=none, midway, left] {$a_{h-2}$} (m-1);
\draw[->] (2) to [out=-90, in=30] node[draw=none, midway, right] {$a_{2}$} (dots);
\draw[dotted] (dots) to [out=-150,in=-30] (m-2);
\end{tikzpicture}.
\]

Recall that bound quiver algebras where the quiver is $\overrightarrow{A}_h$ (respectively $\tilde{A}_h$) are called \emph{acyclic}\index[definitions]{acyclic Nakayama algebra} (respectively \emph{cyclic}\index[definitions]{cyclic Nakayama algebra}) \emph{Nakayama algebras}\index[definitions]{Nakayama algebra}. Recall also that every indecomposable projective module over a Nakayama algebra is \emph{uniserial}, that is its radical series is its unique composition series. In particular the
dimensions of the indecomposable projective modules over a Nakayama algebra define the Nakayama algebra uniquely. By encoding this information we can combinatorially describe a Nakayama algebra. The \emph{Kupisch series}\index[definitions]{Kupisch series of Nakayama algebra} \cite{KUP} of an acyclic Nakayama algebra $\K \overrightarrow{A}_h/\cR$ is the vector
\[\left(\dim(P(1)),\dim(P(2)),\dots,\dim(P(h))\right),\]
and the Kupisch series of a cyclic Nakayama algebra $\K\tilde{A}_h/\cR$ is the vector
\[\left(\dim(P(1)),\dim(P(2)),\dots,\dim(P(h-1),\dim(P(0))\right),\]
Moreover, every vector $(d_1,d_2,\dots,d_h)\in\ZZ_{\geq 1}^h$ with $d_h=1$ and $d_i\geq 2$ for $1\leq i\leq h-1$ as well as $d_{i-1}-1\leq d_i$ defines uniquely an acyclic Nakayama algebra and every vector $(d_1,d_2,\dots,d_{h-1},d_0)\in\ZZ_{\geq 2}^h$ with $d_{i-1}-1\leq d_i$ defines uniquely, up to cyclic permutation, a cyclic Nakayama algebra. 

\subsubsection{Self-gluing of Nakayama algebras} We start by obtaining some examples of cyclic Nakayama algebras whose module categories admit $n$-cluster tilting subcategories coming from acyclic Nakayama algebras whose module categories admit $n$-cluster tilting subcategories. Notice in particular that acyclic Nakayama algebras are always rep\-re\-sen\-ta\-tion-di\-rect\-ed while cyclic Nakayama algebras are never rep\-re\-sen\-ta\-tion-di\-rect\-ed. Recall also that selfinjective Nakayama algebras whose module categories admit $n$-cluster tilting subcategories have been classified in \cite{DI}.

\begin{corollary}\label{cor:Nakayama self-glue at simple}
Let $\La=\K\overrightarrow{A}_h/\cR$ be an acyclic Nakayama algebra with Kupisch series $(d_1,\dots,d_h)$ and assume that $\m\La$ admits an $n$-cluster tilting subcategory. Then the module category of the cyclic Nakayama algebra $\tilde{\La}$ with Kupisch series $(d_1,\dots,d_{h-1})$ admits an $n$-cluster tilting subcategory
\end{corollary}

\begin{proof}
Notice that $P(h)=P$ is a simple projective $\La$-module and $I(1)=I$ is a simple injective $\La$-module. Hence the algebra $\tilLa=\K\tilQ/\tilR$ is well defined and $\m\tilLa$ admits an $n$-cluster tilting subcategory by Corollary \ref{cor:La infinity n-ct}. By the description of $\tilQ$ and $\tilR$ it readily follows that $\tilQ\isom\tilde{A}_{h-2}$ via the identification $i\mapsto i$ for $i\in\{1,\dots,h-2\}$ and $h-1\mapsto 0$. Moreover, with this identification, the ideal $\tilR$ is generated by all paths in $\cR$ together with the additional path of length two from $0$ to $2$. A direct computation shows that the Kupisch series of $\tilLa$ is then $(d_1,\dots,d_{h-1})$.
\end{proof}

In Sections \ref{subsubsec:exactly one fracture} and \ref{subsubsec:exactly two fractures} we provide examples of self-gluing of Nakayama algebras which is not over a simple module.

\begin{example}
Acyclic Nakayama algebras with homogeneous relations whose module categories admit $n$-cluster tilting subcategories where classified in \cite{VAS}. More examples of acyclic Nakayama algebras whose module categories admit $n$-cluster tilting subcategories were produced in \cite[Part IV]{VAS2}. In both cases Corollary \ref{cor:Nakayama self-glue at simple} applies.

For example by \cite[Proposition 4.9(c)]{VAS2} the module category of the acyclic Nakayama algebra with Kupisch series $(2,2,3,3,3,3,2,1)$ admits a $3$-cluster tilting subcategory, see Figure \ref{fig:acyclic to cyclic example 1}. Then by Corollary \ref{cor:Nakayama self-glue at simple} the module category of the cyclic Nakayama algebra with Kupisch series $(2,2,3,3,3,3,2)$ admits a $3$-cluster tilting subcategory, see Figure \ref{fig:acyclic to cyclic example 2}.

\begin{figure}[htb]
    \centering
    \begin{tikzpicture}[scale=0.9, transform shape, baseline={(current bounding box.center)}]
        \tikzstyle{nct2}=[circle, minimum width=0.5cm, draw, inner sep=0pt, text centered]
        
        \node[nct2] (21) at (2,1) {\tiny $8$};
        \node (31) at (3,1) {\tiny $7$};
        \node (41) at (4,1) {\tiny $6$};
        \node (51) at (5,1) {\tiny $5$};
        \node[nct2] (61) at (6,1) {\tiny $4$};
        \node (71) at (7,1) {\tiny $3$};
        \node (81) at (8,1) {\tiny $2$};
        \node[nct2] (91) at (9,1) {\tiny $1$\nospacepunct{,}};
        
        \node[nct2] (22) at (2.5,2) {\tiny $\begin{matrix} 7 \\ 8 \end{matrix}$};
        \node (32) at (3.5,2) {\tiny $\begin{matrix} 6 \\ 7 \end{matrix}$};
        \node (42) at (4.5,2) {\tiny $\begin{matrix} 5 \\ 6 \end{matrix}$};
        \node (52) at (5.5,2) {\tiny $\begin{matrix} 4 \\ 5 \end{matrix}$};
        \node[nct2] (62) at (6.5,2) {\tiny $\begin{matrix} 3 \\ 4 \end{matrix}$};
        \node[nct2] (72) at (7.5,2) {\tiny $\begin{matrix} 2 \\ 3 \end{matrix}$};
        \node[nct2] (82) at (8.5,2) {\tiny $\begin{matrix} 1 \\ 2 \end{matrix}$};
            
        \node[nct2] (23) at (3,3) {\tiny $\begin{matrix} 6 \\ 7 \\ 8 \end{matrix}$};
        \node[nct2] (33) at (4,3) {\tiny $\begin{matrix} 5 \\ 6 \\ 7 \end{matrix}$};
        \node[nct2] (43) at (5,3) {\tiny $\begin{matrix} 4 \\ 5 \\ 6 \end{matrix}$};
        \node[nct2] (53) at (6,3) {\tiny $\begin{matrix} 3 \\ 4 \\ 5 \end{matrix}$};
            
        \draw[->] (21) -- (22);
        \draw[->] (31) -- (32);
        \draw[->] (41) -- (42);
        \draw[->] (51) -- (52);
        \draw[->] (61) -- (62);
        \draw[->] (71) -- (72);
        \draw[->] (81) -- (82);
            
        \draw[->] (22) -- (23);
        \draw[->] (32) -- (33);
        \draw[->] (42) -- (43);
        \draw[->] (52) -- (53);

        \draw[->] (22) -- (31);
        \draw[->] (32) -- (41);
        \draw[->] (42) -- (51);
        \draw[->] (52) -- (61);
        \draw[->] (62) -- (71);
        \draw[->] (72) -- (81);
        \draw[->] (82) -- (91);
            
        \draw[->] (23) -- (32);
        \draw[->] (33) -- (42);
        \draw[->] (43) -- (52);
        \draw[->] (53) -- (62);

        \draw[loosely dotted] (21.east) -- (31);
        \draw[loosely dotted] (31.east) -- (41);
        \draw[loosely dotted] (41.east) -- (51);
        \draw[loosely dotted] (51.east) -- (61);
        \draw[loosely dotted] (61.east) -- (71);
        \draw[loosely dotted] (71.east) -- (81);
        \draw[loosely dotted] (81.east) -- (91);
            
        \draw[loosely dotted] (22.east) -- (32);
        \draw[loosely dotted] (32.east) -- (42);
        \draw[loosely dotted] (42.east) -- (52);
        \draw[loosely dotted] (52.east) -- (62);
        \draw[loosely dotted] (62.east) -- (72);
        \draw[loosely dotted] (72.east) -- (82);
            
    \end{tikzpicture}
    \caption{The Aus\-lan\-der--Rei\-ten quiver of the acyclic Nakayama algebra with Kupisch series $(2,2,3,3,3,3,2,1)$ where the additive closure of the encircled modules is a $3$-cluster tilting subcategory $\cM$.}
    \label{fig:acyclic to cyclic example 1}
\end{figure}

\begin{figure}[htb]
    \centering
    \begin{tikzpicture}[scale=0.9, transform shape, baseline={(current bounding box.center)}]
        \tikzstyle{nct2}=[circle, minimum width=0.5cm, draw, inner sep=0pt, text centered]
        
        \node[nct2] (21) at (2,1) {\tiny $1$};
        \node (31) at (3,1) {\tiny $0$};
        \node (41) at (4,1) {\tiny $6$};
        \node (51) at (5,1) {\tiny $5$};
        \node[nct2] (61) at (6,1) {\tiny $4$};
        \node (71) at (7,1) {\tiny $3$};
        \node (81) at (8,1) {\tiny $2$};
        \node[nct2] (91) at (9,1) {\tiny $1$\nospacepunct{,}};
        
        \node[nct2] (22) at (2.5,2) {\tiny $\begin{matrix} 0 \\ 1 \end{matrix}$};
        \node (32) at (3.5,2) {\tiny $\begin{matrix} 6 \\ 0 \end{matrix}$};
        \node (42) at (4.5,2) {\tiny $\begin{matrix} 5 \\ 6 \end{matrix}$};
        \node (52) at (5.5,2) {\tiny $\begin{matrix} 4 \\ 5 \end{matrix}$};
        \node[nct2] (62) at (6.5,2) {\tiny $\begin{matrix} 3 \\ 4 \end{matrix}$};
        \node[nct2] (72) at (7.5,2) {\tiny $\begin{matrix} 2 \\ 3 \end{matrix}$};
        \node[nct2] (82) at (8.5,2) {\tiny $\begin{matrix} 1 \\ 2 \end{matrix}$};
            
        \node[nct2] (23) at (3,3) {\tiny $\begin{matrix} 6 \\ 0 \\ 1 \end{matrix}$};
        \node[nct2] (33) at (4,3) {\tiny $\begin{matrix} 5 \\ 6 \\ 0 \end{matrix}$};
        \node[nct2] (43) at (5,3) {\tiny $\begin{matrix} 4 \\ 5 \\ 6 \end{matrix}$};
        \node[nct2] (53) at (6,3) {\tiny $\begin{matrix} 3 \\ 4 \\ 5 \end{matrix}$};
            
        \draw[->] (21) -- (22);
        \draw[->] (31) -- (32);
        \draw[->] (41) -- (42);
        \draw[->] (51) -- (52);
        \draw[->] (61) -- (62);
        \draw[->] (71) -- (72);
        \draw[->] (81) -- (82);
            
        \draw[->] (22) -- (23);
        \draw[->] (32) -- (33);
        \draw[->] (42) -- (43);
        \draw[->] (52) -- (53);

        \draw[->] (22) -- (31);
        \draw[->] (32) -- (41);
        \draw[->] (42) -- (51);
        \draw[->] (52) -- (61);
        \draw[->] (62) -- (71);
        \draw[->] (72) -- (81);
        \draw[->] (82) -- (91);
            
        \draw[->] (23) -- (32);
        \draw[->] (33) -- (42);
        \draw[->] (43) -- (52);
        \draw[->] (53) -- (62);

        \draw[loosely dotted] (21.east) -- (31);
        \draw[loosely dotted] (31.east) -- (41);
        \draw[loosely dotted] (41.east) -- (51);
        \draw[loosely dotted] (51.east) -- (61);
        \draw[loosely dotted] (61.east) -- (71);
        \draw[loosely dotted] (71.east) -- (81);
        \draw[loosely dotted] (81.east) -- (91);
            
        \draw[loosely dotted] (22.east) -- (32);
        \draw[loosely dotted] (32.east) -- (42);
        \draw[loosely dotted] (42.east) -- (52);
        \draw[loosely dotted] (52.east) -- (62);
        \draw[loosely dotted] (62.east) -- (72);
        \draw[loosely dotted] (72.east) -- (82);
        
        \draw[loosely dotted] (2,3.5) -- (2,0.5);    
        \draw[loosely dotted] (9,3.5) -- (9,0.5);    
    \end{tikzpicture}
    \caption{The Aus\-lan\-der--Rei\-ten quiver of the cyclic Nakayama algebra with Kupisch series $(2,2,3,3,3,3,2)$ where the additive closure of the encircled modules is a $3$-cluster tilting subcategory equivalent to $\Mks/\ZZ$.}
    \label{fig:acyclic to cyclic example 2}
\end{figure}
\end{example}

\subsubsection{Self-gluing of starlike algebras with three rays}\label{subsubsec:starlike algebras with three rays} We continue by applying Corollary \ref{cor:La infinity n-ct} to algebras as in Proposition \ref{prop:starlike algebras}\ref{case:k=3}.

\begin{corollary}\label{cor:cyclic starlike with 3 rays}
Let $Q$ be one of the quivers
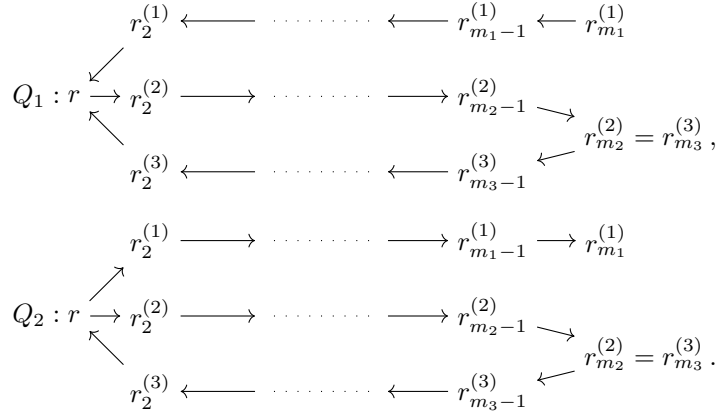
\[\begin{tikzpicture}
    \node (Q) at (-0.5,0) {$Q_1:$};

    \node (R) at (0,0) {$r$};

    \node (1M1) at (7,1) {$r_{m_1}^{(1)}$};
    \node (1M1-) at (5.5,1) {$r_{m_1-1}^{(1)}$};
    \node (13+) at (4,1) {};
    \node (13) at (2.5,1) {};
    \node (12) at (1,1) {$r_2^{(1)}$};
    
    \draw[->] (1M1) -- (1M1-);
    \draw[->] (1M1-) -- (13+);
    \draw[loosely dotted] (13+) -- (13);
    \draw[->] (13) -- (12);
    \draw[->] (12) -- (R);
    
    \node (2M1-) at (5.5,0) {$r_{m_2-1}^{(2)}$};
    \node (23+) at (4,0) {};
    \node (23) at (2.5,0) {};
    \node (22) at (1,0) {$r_2^{(2)}$};
    
    \draw[<-] (2M1-) -- (23+);
    \draw[loosely dotted] (23+) -- (23);
    \draw[<-] (23) -- (22);
    \draw[<-] (22) -- (R);

    \node (3M1) at (7.5,-0.5) {$r_{m_2}^{(2)}=r_{m_3}^{(3)}$\nospacepunct{,}};
    \node (3M1-) at (5.5,-1) {$r_{m_3-1}^{(3)}$};
    \node (33+) at (4,-1) {};
    \node (33) at (2.5,-1) {};
    \node (32) at (1,-1) {$r_2^{(3)}$};
    
    \draw[<-] (3M1) -- (2M1-);
    \draw[->] (3M1) -- (3M1-);
    \draw[->] (3M1-) -- (33+);
    \draw[loosely dotted] (33+) -- (33);
    \draw[->] (33) -- (32);
    \draw[->] (32) -- (R);
\end{tikzpicture}
\]
\[\begin{tikzpicture}
    \node (Q) at (-0.5,0) {$Q_2:$};

    \node (R) at (0,0) {$r$};

    \node (1M1) at (7,1) {$r_{m_1}^{(1)}$};
    \node (1M1-) at (5.5,1) {$r_{m_1-1}^{(1)}$};
    \node (13+) at (4,1) {};
    \node (13) at (2.5,1) {};
    \node (12) at (1,1) {$r_2^{(1)}$};
    
    \draw[<-] (1M1) -- (1M1-);
    \draw[<-] (1M1-) -- (13+);
    \draw[loosely dotted] (13+) -- (13);
    \draw[<-] (13) -- (12);
    \draw[<-] (12) -- (R);
    
    \node (2M1-) at (5.5,0) {$r_{m_2-1}^{(2)}$};
    \node (23+) at (4,0) {};
    \node (23) at (2.5,0) {};
    \node (22) at (1,0) {$r_2^{(2)}$};
    
    \draw[<-] (2M1-) -- (23+);
    \draw[loosely dotted] (23+) -- (23);
    \draw[<-] (23) -- (22);
    \draw[<-] (22) -- (R);

    \node (3M1) at (7.5,-0.5) {$r_{m_2}^{(2)}=r_{m_3}^{(3)}$\nospacepunct{.}};
    \node (3M1-) at (5.5,-1) {$r_{m_3-1}^{(3)}$};
    \node (33+) at (4,-1) {};
    \node (33) at (2.5,-1) {};
    \node (32) at (1,-1) {$r_2^{(3)}$};
    
    \draw[<-] (3M1) -- (2M1-);
    \draw[->] (3M1) -- (3M1-);
    \draw[->] (3M1-) -- (33+);
    \draw[loosely dotted] (33+) -- (33);
    \draw[->] (33) -- (32);
    \draw[->] (32) -- (R);
\end{tikzpicture}\]
If $m_1=n+1+x_1n,m_2=n+1+x_2n$ and $m_3=n+x_3 n$ for some $x_i\in\ZZ_{\geq 0}$ and some $n\geq 2$, then $\m(\K Q/R_Q^2)$ admits an $n$-cluster tilting subcategory.
\end{corollary}

\begin{proof}
Assume that $m_1\geq m_2\geq m_3$; the other cases are similar. Let $\La=\K T^2/\rad(\K T)$ where $T=T(3,m_1,m_2,m_3)$ for some $A_{m_1},A_{m_2}$ and $A_{m_3}$ linearly oriented. Assume first that $r^{(1)}_{m_1}$ and $r^{(2)}_{m_2}$ are both sources and $r_{m_3}^{(3)}$ is a sink. Then $\m\La$ admits an $n$-cluster tilting subcategory by Proposition \ref{prop:starlike algebras}\ref{case:k=3}. Let $P=P\left(r_{m_3}^{(3)}\right)$ and $I=I\left(r^{(2)}_{m_2}\right)$. Then $\La$ is $n$-self-gluable and so $\m\tilde{\La}$ admits an $n$-cluster tilting subcategory by Corollary \ref{cor:La infinity n-ct}. Moreover, it readily follows by (\ref{picture:self-gluing quiver}) that $\tilde{\La}\isom \K Q_1/R_{Q_1}^2$, which proves that $\m(\K Q_1/R_{Q_1}^2)$ admits an $n$-cluster tilting subcategory. Assuming that $r^{(1)}_{m_1}$ and $r^{(2)}_{m_2}$ are both sinks and $r_{m_3}^{(3)}$ is a source we can similarly prove that $\m(\K Q_2/R_{Q_2}^2)$ admits an $n$-cluster tilting subcategory.
\end{proof}

For an example coming from Corollary \ref{cor:cyclic starlike with 3 rays} see Figure \ref{fig:cyclic starlike with 3 rays example}.

To the same class of algebras, we can also apply Proposition \ref{prop:double gluing}. In this case we have the following results.

\begin{corollary}\label{cor:double gluing of starlike three} Let $Q$ be one of the quivers
    \[\begin{tikzpicture}
    \node (R) at (0,-0.5) {$r$};
    \node (Q) at (-5,-0.5) {$Q_1:$};

    \node (1M1-) at (-4,-0.5) {$r_{m_3}^{(3)}$};
    \node (13+) at (-3,-0.5) {};
    \node (13) at (-2,-0.5) {};
    \node (12) at (-1,-0.5) {$r_2^{(3)}$};
    
    \draw[->] (1M1-) -- (13+);
    \draw[loosely dotted] (13+) -- (13);
    \draw[->] (13) -- (12);
    \draw[->] (12) -- (R);
    
    \node (2M1-) at (4,0) {$r_{M_2-1}^{(2)}$};
    \node (23+) at (3,0) {};
    \node (23) at (2,0) {};
    \node (22) at (1,0) {$r_2^{(2)}$};
    
    \draw[<-] (2M1-) -- (23+);
    \draw[loosely dotted] (23+) -- (23);
    \draw[<-] (23) -- (22);
    \draw[<-] (22) -- (R);

    \node (3M1) at (5.5,-0.5) {$r'$};
    \node (3M1-) at (4,-1) {$r_{M_1-1}^{(1)}$};
    \node (33+) at (3,-1) {};
    \node (33) at (2,-1) {};
    \node (32) at (1,-1) {$r_2^{(1)}$};
    
    \draw[<-] (3M1) -- (2M1-);
    \draw[<-] (3M1) -- (3M1-);
    \draw[<-] (3M1-) -- (33+);
    \draw[loosely dotted] (33+) -- (33);
    \draw[<-] (33) -- (32);
    \draw[<-] (32) -- (R);
    
    \node (1M1-') at (9.5,-0.5) {$r_{m_3'}^{(3')}$\nospacepunct{,}};
    \node (13+') at (8.5,-0.5) {};
    \node (13') at (7.5,-0.5) {};
    \node (12') at (6.5,-0.5) {$r_2^{(3')}$};

    \draw[<-] (1M1-') -- (13+');
    \draw[loosely dotted] (13+') -- (13');
    \draw[<-] (13') -- (12');
    \draw[<-] (12') -- (3M1);
\end{tikzpicture}
\]
\[\begin{tikzpicture}
    \node (R) at (0,-0.5) {$r$};
    \node (Q) at (-5,-0.5) {$Q_2:$};

    \node (1M1-) at (-4,-0.5) {$r_{m_3}^{(3)}$};
    \node (13+) at (-3,-0.5) {};
    \node (13) at (-2,-0.5) {};
    \node (12) at (-1,-0.5) {$r_2^{(3)}$};
    
    \draw[->] (1M1-) -- (13+);
    \draw[loosely dotted] (13+) -- (13);
    \draw[->] (13) -- (12);
    \draw[->] (12) -- (R);
    
    \node (2M1-) at (4,0) {$r_{M_2-1}^{(2)}$};
    \node (23+) at (3,0) {};
    \node (23) at (2,0) {};
    \node (22) at (1,0) {$r_2^{(2)}$};
    
    \draw[<-] (2M1-) -- (23+);
    \draw[loosely dotted] (23+) -- (23);
    \draw[<-] (23) -- (22);
    \draw[<-] (22) -- (R);

    \node (3M1) at (5.5,-0.5) {$r'$};
    \node (3M1-) at (4,-1) {$r_{M_1-1}^{(1)}$};
    \node (33+) at (3,-1) {};
    \node (33) at (2,-1) {};
    \node (32) at (1,-1) {$r_2^{(1)}$};
    
    \draw[<-] (3M1) -- (2M1-);
    \draw[->] (3M1) -- (3M1-);
    \draw[->] (3M1-) -- (33+);
    \draw[loosely dotted] (33+) -- (33);
    \draw[->] (33) -- (32);
    \draw[->] (32) -- (R);
    
    \node (1M1-') at (9.5,-0.5) {$r_{m_3'}^{(3')}$\nospacepunct{.}};
    \node (13+') at (8.5,-0.5) {};
    \node (13') at (7.5,-0.5) {};
    \node (12') at (6.5,-0.5) {$r_2^{(3')}$};

    \draw[<-] (1M1-') -- (13+');
    \draw[loosely dotted] (13+') -- (13');
    \draw[<-] (13') -- (12');
    \draw[<-] (12') -- (3M1);
\end{tikzpicture}
\]
If $M_1=2n+1+x_1n$, $M_2=2n+1+x_2n$, $m_3=n+x_3n$ and $m_3'=n+x_3'n$ for some $n\geq 2$ and some $x_1,x_2,x_3,x_3'\in\ZZ_{\geq 0}$, then $\m(\K Q/R^2_{Q})$ admits an $n$-cluster tilting subcategory.
\end{corollary}

\begin{proof}
Assume $m_1\geq m_2\geq m_3$ and $m_1'\geq m_2'\geq m_3'$; the other cases are similar. Let $m_1=n+1+x_1n$, $m_2=n+1+x_2n$, $m_1'=n+1$ and $m_2'=n+1$. Let $A=\K T/R_T^2$ where $T=T(3,m_1,m_2,m_3)$ for some $A_{m_1},A_{m_2}$ and $A_{m_3}$ linearly oriented and let $B=\K T'/R^2_{T'}$ where $T'=T(3,m_1',m_2',m_3')$ for some $A_{m_1},A_{m_2}$ and $A_{m_3}$ linearly oriented. Assume first that $r_{m_1}^{(1)}$, $r_{m_2}^{(2)}$ and $r_{m_3'}^{(3')}$ are sinks, while $r_{m_3}^{(3)}$, $r_{m_1'}^{(1')}$  and $r_{m_2'}^{(2')}$ are sources. Then $\m A$ and $\m B$ admit an $n$-cluster tilting subcategory by Proposition \ref{prop:starlike algebras}\ref{case:k=3}. In this case the quivers of $A$ and $B$ are of the form
\[\begin{tikzpicture}[baseline={(current bounding box.center)}, scale=0.8, transform shape]
    \node (R) at (0,-0.5) {$r$};
    
    \node (1M1-) at (-4,-0.5) {$r_{m_3}^{(3)}$};
    \node (13+) at (-3,-0.5) {};
    \node (13) at (-2,-0.5) {};
    \node (12) at (-1,-0.5) {$r_2^{(3)}$};
    
    \draw[->] (1M1-) -- (13+);
    \draw[loosely dotted] (13+) -- (13);
    \draw[->] (13) -- (12);
    \draw[->] (12) -- (R);
    
    \node (2M1-) at (4,0) {$r_{m_2}^{(2)}$};
    \node (23+) at (3,0) {};
    \node (23) at (2,0) {};
    \node (22) at (1,0) {$r_2^{(2)}$};
    
    \draw[<-] (2M1-) -- (23+);
    \draw[loosely dotted] (23+) -- (23);
    \draw[<-] (23) -- (22);
    \draw[<-] (22) -- (R);
    
    \node (3M1-) at (4,-1) {$r_{m_1}^{(1)}$};
    \node (33+) at (3,-1) {};
    \node (33) at (2,-1) {};
    \node (32) at (1,-1) {$r_2^{(1)}$};
    
    \draw[<-] (3M1-) -- (33+);
    \draw[loosely dotted] (33+) -- (33);
    \draw[<-] (33) -- (32);
    \draw[<-] (32) -- (R);
    
    \node (Q) at (0,-1.5) {$T$};
\end{tikzpicture}
\text{ and }
\begin{tikzpicture}[baseline={(current bounding box.center)}, scale=0.8, transform shape]
    \node (2M1-) at (4,0) {$r_{2}^{(2')}$};
    \node (23+) at (3,0) {};
    \node (23) at (2,0) {};
    \node (22) at (1,0) {$r_{n+1}^{(2')}$};
    
    \draw[<-] (2M1-) -- (23+);
    \draw[loosely dotted] (23+) -- (23);
    \draw[<-] (23) -- (22);
    
    \node (3M1-) at (4,-1) {$r_{2}^{(1')}$};
    \node (33+) at (3,-1) {};
    \node (33) at (2,-1) {};
    \node (32) at (1,-1) {$r_{n+1}^{(1')}$};
    \node (3M1) at (5.5,-0.5) {$r'$};
    
    \draw[<-] (3M1) -- (2M1-);
    \draw[<-] (3M1) -- (3M1-);
    \draw[<-] (3M1-) -- (33+);
    \draw[loosely dotted] (33+) -- (33);
    \draw[<-] (33) -- (32);
    
    \node (1M1-') at (9.5,-0.5) {$r_{m_3'}^{(3')}$};
    \node (13+') at (8.5,-0.5) {};
    \node (13') at (7.5,-0.5) {};
    \node (12') at (6.5,-0.5) {$r_2^{(3')}$};
    
    \draw[<-] (1M1-') -- (13+');
    \draw[loosely dotted] (13+') -- (13');
    \draw[<-] (13') -- (12');
    \draw[<-] (12') -- (3M1);
        
    \node (Q) at (5.5,-1.5) {$T'$};
\end{tikzpicture}\]
and so by performing a double gluing as described in Section \ref{sec:double gluing} we get the quiver $Q_1$. Then $\m(\K Q_1/R_{Q_1}^2)$ admits an $n$-cluster tilting subcategory by Proposition \ref{prop:double gluing}. Assuming that $r_{m_2}^{(2)}$, $r_{m_1'}^{(1')}$ and $r_{m_3'}^{(3')}$ are sinks, while $r_{m_1}^{(1)}$, $r_{m_3}^{(3)}$ and $r_{m_2'}^{(2')}$ are sources we can similarly prove that $\m(\K Q_2/R_{Q_2}^2)$ admits an $n$-cluster tilting subcategory.
\end{proof}

\begin{figure}[htb]
    \centering
    \begin{tikzpicture}[baseline={(current bounding box.center)}]
        \tikzstyle{nct3}=[circle, minimum width=6pt, draw=white, inner sep=0pt, scale=0.9]
        \node (name) at (-0.3,0) {$Q:$};
        
        \node[nct3] (1) at (0,0) {$1$};
        
        \node[nct3] (21) at (0.7,0.7) {$2_1$};
        \node[nct3] (31) at (1.4,0.7) {$3_1$};
        \node[nct3] (41) at (2.1,0.7) {$4_1$};
        \node[nct3] (51) at (2.8,0.7) {$5_1$};
        
        \node[nct3] (22) at (0.7,0) {$2_2$};
        \node[nct3] (32) at (1.4,0) {$3_2$};
        \node[nct3] (42) at (2.1,0) {$4_2$};
        
        \node[nct3] (52) at (3.2,-0.35) {$5_2=4_3$\nospacepunct{,}};

        \node[nct3] (23) at (0.7,-0.7) {$2_3$};
        \node[nct3] (33) at (1.4,-0.7) {$3_3$};

        \draw[->] (1) -- (21);
        \draw[->] (21) -- (31);
        \draw[->] (31) -- (41);
        \draw[->] (41) -- (51);
        
        \draw[->] (1) -- (22);
        \draw[->] (22) -- (32);
        \draw[->] (32) -- (42);
        \draw[->] (42) -- (52);

        \draw[<-] (1) -- (23);
        \draw[<-] (23) -- (33);
        \draw[<-] (33) -- (52);
    \end{tikzpicture}\;\;\;\;
    \begin{tikzpicture}[scale=0.9, transform shape, baseline={(current bounding box.center)}]
        \tikzstyle{nct2}=[circle, minimum width=0.6cm, draw, inner sep=0pt, text centered, scale=0.9]
        \tikzstyle{nct22}=[circle, minimum width=0.6cm, draw, inner sep=0pt, text centered, scale=0.8]
        \tikzstyle{nct3}=[circle, minimum width=6pt, draw=white, inner sep=0pt, scale=0.9]
        
        \node[scale=0.8] (name) at (4.9,-0.3) {$\Gamma\left(\K Q/\rad(\K Q)^2\right)$};
        
        \node[nct2] (12) at (0,0.7) {$\qthree{}[4_3][]$};
        \node[nct2] (21) at (0.7,0) {$\qthree{4_2}[4_3]$};
        \node[nct3] (32) at (1.4,0.7) {$\qthree{}[4_2][]$};
        \node[nct2] (41) at (2.1,0) {$\qthree{3_2}[4_2]$};
        \node[nct3] (52) at (2.8,0.7) {$\qthree{}[3_2][]$};
        \node[nct2] (61) at (3.5,0) {$\qthree{2_2}[3_2]$};
        \node[nct3] (72) at (4.2,0.7) {$\qthree{}[2_2][]$};
        
        \draw[->] (12) -- (21);
        \draw[->] (21) -- (32);
        \draw[->] (32) -- (41);
        \draw[->] (41) -- (52);
        \draw[->] (52) -- (61);
        \draw[->] (61) -- (72);
        
        \draw[loosely dotted] (12.east) -- (32);
        \draw[loosely dotted] (32.east) -- (52);
        \draw[loosely dotted] (52.east) -- (72);
        
        \node[nct2] (14) at (0,2.1) {$\qthree{}[5_1][]$};
        \node[nct2] (25) at (0.7,2.8) {$\qthree{4_1}[5_1]$};
        \node[nct3] (34) at (1.4,2.1) {$\qthree{}[4_1][]$};
        \node[nct2] (45) at (2.1,2.8) {$\qthree{3_1}[4_1]$};
        \node[nct3] (54) at (2.8,2.1) {$\qthree{}[3_1][]$};
        \node[nct2] (65) at (3.5,2.8) {$\qthree{2_1}[3_1]$};
        \node[nct3] (74) at (4.2,2.1) {$\qthree{}[2_1][]$};
        
        \draw[->] (14) -- (25);
        \draw[->] (25) -- (34);
        \draw[->] (34) -- (45);
        \draw[->] (45) -- (54);
        \draw[->] (54) -- (65);
        \draw[->] (65) -- (74);
        
        \draw[loosely dotted] (14.east) -- (34);
        \draw[loosely dotted] (34.east) -- (54);
        \draw[loosely dotted] (54.east) -- (74);
        
        \node[nct2] (83) at (4.9,1.4) {$\begin{smallmatrix}  & 1 & \\ & 2_1 2_2 &\end{smallmatrix} $};
        
        \draw[->] (72) -- (83);
        \draw[->] (74) -- (83);
        
        \node[nct2] (94) at (5.6,2.1) {$\qthree{1}[2_2]$};
        \node[nct2] (92) at (5.6,0.7) {$\qthree{1}[2_1]$};
        \node[nct3] (103) at (6.3,1.4) {$\qthree{}[1][]$};
        \node[nct2] (114) at (7,2.1) {$\qthree{2_3}[1]$};
        \node[nct3] (123) at (7.7,1.4) {$\qthree{}[2_3][]$};
        \node[nct2] (134) at (8.4,2.1) {$\qthree{3_3}[2_3]$};
        \node[nct3] (143) at (9.1,1.4) {$\qthree{}[3_3][]$};
        \node[nct2] (154) at (9.8,2.1) {$\qthree{4_3}[3_3]$};
        \node[nct2] (163) at (10.5,1.4) {$\qthree{}[4_3][]$};
        
        \draw[->] (83) -- (94);
        \draw[->] (83) -- (92);
        \draw[->] (94) -- (103);
        \draw[->] (92) -- (103);
        \draw[->] (103) -- (114);
        \draw[->] (114) -- (123);
        \draw[->] (123) -- (134);
        \draw[->] (134) -- (143);
        \draw[->] (143) -- (154);
        \draw[->] (154) -- (163);
        
        \draw[loosely dotted] (83.east) -- (103);
        \draw[loosely dotted] (103.east) -- (123);
        \draw[loosely dotted] (123.east) -- (143);
        \draw[loosely dotted] (143.east) -- (163);
        \draw[loosely dotted] (74.east) -- (94);
        \draw[loosely dotted] (72.east) -- (92);
    \end{tikzpicture}
    \caption{The Aus\-lan\-der--Rei\-ten quiver of the algebra $\tilLa=\K \tilQ/R^2_{\tilQ}$ where $\tilQ$ is the quiver on top and where the additive closure of the encircled modules is a $4$-cluster tilting subcategory. Notice that the module $4_3$ appears twice and so $\tilLa$ is not rep\-re\-sen\-ta\-tion-di\-rect\-ed. Compare with Figure \ref{fig:starlike example}.} 
\label{fig:cyclic starlike with 3 rays example}
\end{figure}

\begin{corollary}\label{cor:triple gluing of starlike} Let $Q$ be the quiver
\[\begin{tikzpicture}
    \node (R) at (0,-0.5) {$r$};

    \node (1M1-) at (4,-0.5) {$r_{M_3-1}^{(3)}$};
    \node (13+) at (3,-0.5) {};
    \node (13) at (2,-0.5) {};
    \node (12) at (1,-0.5) {$r_2^{(3)}$};
    
    \draw[->] (1M1-) -- (13+);
    \draw[loosely dotted] (13+) -- (13);
    \draw[->] (13) -- (12);
    \draw[->] (12) -- (R);
    
    \node (2M1-) at (4,0.2) {$r_{M_2-1}^{(2)}$};
    \node (23+) at (3,0.2) {};
    \node (23) at (2,0.2) {};
    \node (22) at (1,0.2) {$r_2^{(2)}$};
    
    \draw[<-] (2M1-) -- (23+);
    \draw[loosely dotted] (23+) -- (23);
    \draw[<-] (23) -- (22);
    \draw[<-] (22) -- (R);

    \node (3M1) at (5.5,-0.5) {$r'$\nospacepunct{.}};
    \node (3M1-) at (4,-1.2) {$r_{M_1-1}^{(1)}$};
    \node (33+) at (3,-1.2) {};
    \node (33) at (2,-1.2) {};
    \node (32) at (1,-1.2) {$r_2^{(1)}$};
    
    \draw[->] (3M1) -- (1M1-);
    \draw[<-] (3M1) -- (2M1-);
    \draw[<-] (3M1) -- (3M1-);
    \draw[<-] (3M1-) -- (33+);
    \draw[loosely dotted] (33+) -- (33);
    \draw[<-] (33) -- (32);
    \draw[<-] (32) -- (R);
\end{tikzpicture}\]
If $M_1=2n+1+x_1n, M_2=2n+1+x_2n$ and $M_3=2n-1+x_3n$ for some $n\geq 2$ and some $x_i\in \ZZ_{\geq 0}$, then $\m(\K Q/R_Q^2)$ admits an $n$-cluster tilting subcategory.
\end{corollary}

\begin{proof}
Let $m_3=n+x_3n$ and $m_3'=n$. Let $\La=\K Q_1/R_{Q_1}^2$, where $Q_1$ is as in Corollary \ref{cor:double gluing of starlike three} with $M_1,M_2,m_3,m_3'$ as given. Then $\m\La$ admits an $n$-cluster tilting subcategory and a straightforward computation of the Aus\-lan\-der--Rei\-ten quiver of $\La$ shows that $\La$ is a rep\-re\-sen\-ta\-tion-di\-rect\-ed algebra. Then performing the self-gluing along the unique simple projective module $P$ and the unique simple injective module $I$ we get that the module category of the algebra $\tilLa\isom \K Q/R_Q^2$ admits an $n$-cluster tilting subcategory by Corollary \ref{cor:La infinity n-ct}.
\end{proof}

For an example of Corollary \ref{cor:double gluing of starlike three} see Figure \ref{fig:double gluing example}; for an example of Corollary \ref{cor:triple gluing of starlike} see Figure \ref{fig:triple gluing example}.

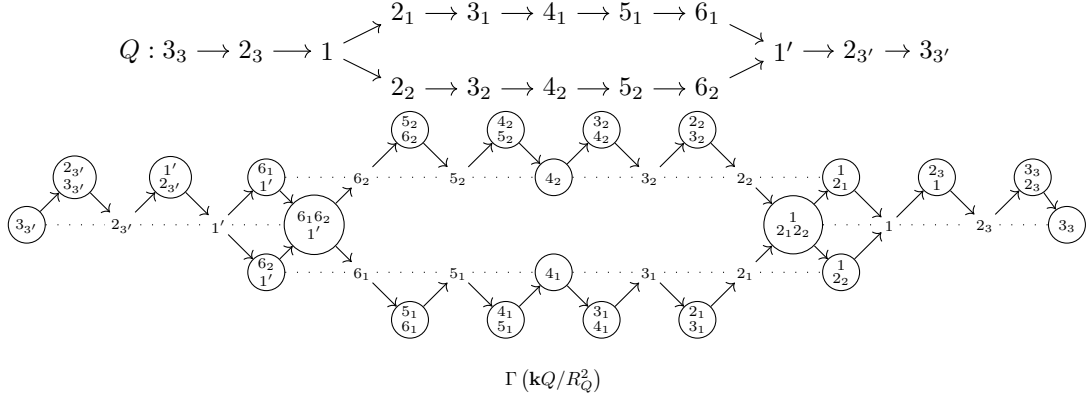
\begin{figure}[htb]
    \centering
    \begin{tikzpicture}
        \node (R) at (-1,-0.5) {$1$};
        \node (Q) at (-3.5,-0.5) {$Q:$};

        \node (13) at (-3,-0.5) {$3_3$};
        \node (12) at (-2,-0.5) {$2_3$};
        
        \draw[->] (13) -- (12);
        \draw[->] (12) -- (R);
        
        \node (2M1-) at (4,0) {$6_1$};
        \node (23+) at (3,0) {$5_1$};
        \node (23) at (2,0) {$4_1$};
        \node (22) at (1,0) {$3_1$};
        \node (21) at (0,0) {$2_1$};
        
        \draw[<-] (2M1-) -- (23+);
        \draw[<-] (23+) -- (23);
        \draw[<-] (23) -- (22);
        \draw[<-] (22) -- (21);
        \draw[<-] (21) -- (R);
    
        \node (3M1) at (5,-0.5) {$1'$};
        \node (3M1-) at (4,-1) {$6_2$};
        \node (33+) at (3,-1) {$5_2$};
        \node (33) at (2,-1) {$4_2$};
        \node (32) at (1,-1) {$3_2$};
        \node (31) at (0,-1) {$2_2$};
        
        \draw[<-] (3M1) -- (2M1-);
        \draw[<-] (3M1) -- (3M1-);
        \draw[<-] (3M1-) -- (33+);
        \draw[<-] (33+) -- (33);
        \draw[<-] (33) -- (32);
        \draw[<-] (32) -- (31);
        \draw[<-] (31) -- (R);
        
        \node (13') at (7,-0.5) {$3_{3'}$};
        \node (12') at (6,-0.5) {$2_{3'}$};
    
        \draw[<-] (13') -- (12');
        \draw[<-] (12') -- (3M1);
    \end{tikzpicture}\;\;\;\;
    \begin{tikzpicture}[scale=0.9, transform shape, baseline={(current bounding box.center)}]
        \tikzstyle{nct2}=[circle, minimum width=0.6cm, draw, inner sep=0pt, text centered, scale=0.9]
        \tikzstyle{nct22}=[circle, minimum width=0.6cm, draw, inner sep=0pt, text centered, scale=0.8]
        \tikzstyle{nct3}=[circle, minimum width=6pt, draw=white, inner sep=0pt, scale=0.9]
        
        \node[scale=0.8] (name) at (7.7,-2.3) {$\Gamma\left(\K Q/R_Q^2\right)$};
        
        \node[nct2] (13) at (0,0) {$\qthree{}[3_{3'}][]$};
        \node[nct2] (24) at (0.7,0.7) {$\qthree{2_{3'}}[3_{3'}]$};
        \node[nct3] (33) at (1.4,0) {$\qthree{}[2_{3'}][]$};
        \node[nct2] (44) at (2.1,0.7) {$\qthree{1'}[2_{3'}]$};
        \node[nct3] (53) at (2.8,0) {$\qthree{}[1'][]$};
        \node[nct2] (64) at (3.5,0.7) {$\qthree{6_1}[1']$};
        \node[nct2] (62) at (3.5,-0.7) {$\qthree{6_2}[1'][]$};
        \node[nct2] (73) at (4.2,0) {$\begin{smallmatrix} & 6_1  6_2 & \\ & 1' &\end{smallmatrix} $};
        
        \draw[->] (13) -- (24);
        \draw[->] (24) -- (33);
        \draw[->] (33) -- (44);
        \draw[->] (44) -- (53);
        \draw[->] (53) -- (64);
        \draw[->] (53) -- (62);
        \draw[->] (64) -- (73);
        \draw[->] (62) -- (73);
        
        \draw[loosely dotted] (13.east) -- (33);
        \draw[loosely dotted] (33.east) -- (53);
        \draw[loosely dotted] (53.east) -- (73);
        
        \node[nct3] (84) at (4.9,0.7) {$\qthree{}[6_2][]$};
        \node[nct2] (95) at (5.6,1.4) {$\qthree{5_2}[6_2][]$};
        \node[nct3] (104) at (6.3,0.7) {$\qthree{}[5_2][]$};
        \node[nct2] (115) at (7,1.4) {$\qthree{4_2}[5_2][]$};
        \node[nct2] (124) at (7.7,0.7) {$\qthree{}[4_2][]$};
        \node[nct2] (135) at (8.4,1.4) {$\qthree{3_2}[4_2][]$};
        \node[nct3] (144) at (9.1,0.7) {$\qthree{}[3_2][]$};
        \node[nct2] (155) at (9.8,1.4) {$\qthree{2_2}[3_2][]$};
        \node[nct3] (164) at (10.5,0.7) {$\qthree{}[2_2][]$};
        
        \draw[->] (73) -- (84);
        \draw[->] (84) -- (95);
        \draw[->] (95) -- (104);
        \draw[->] (104) -- (115);
        \draw[->] (115) -- (124);
        \draw[->] (124) -- (135);
        \draw[->] (135) -- (144);
        \draw[->] (144) -- (155);
        \draw[->] (155) -- (164);
        
        \draw[loosely dotted] (84.east) -- (104);
        \draw[loosely dotted] (104.east) -- (124);
        \draw[loosely dotted] (124.east) -- (144);
        \draw[loosely dotted] (144.east) -- (164);
        
        \node[nct3] (82) at (4.9,-0.7) {$\qthree{}[6_1][]$};
        \node[nct2] (91) at (5.6,-1.4) {$\qthree{5_1}[6_1][]$};
        \node[nct3] (102) at (6.3,-0.7) {$\qthree{}[5_1][]$};
        \node[nct2] (111) at (7,-1.4) {$\qthree{4_1}[5_1][]$};
        \node[nct2] (122) at (7.7,-0.7) {$\qthree{}[4_1][]$};
        \node[nct2] (131) at (8.4,-1.4) {$\qthree{3_1}[4_1][]$};
        \node[nct3] (142) at (9.1,-0.7) {$\qthree{}[3_1][]$};
        \node[nct2] (151) at (9.8,-1.4) {$\qthree{2_1}[3_1][]$};
        \node[nct3] (162) at (10.5,-0.7) {$\qthree{}[2_1][]$};
        
        \draw[->] (73) -- (82);
        \draw[->] (82) -- (91);
        \draw[->] (91) -- (102);
        \draw[->] (102) -- (111);
        \draw[->] (111) -- (122);
        \draw[->] (122) -- (131);
        \draw[->] (131) -- (142);
        \draw[->] (142) -- (151);
        \draw[->] (151) -- (162);
        
        \draw[loosely dotted] (82.east) -- (102);
        \draw[loosely dotted] (102.east) -- (122);
        \draw[loosely dotted] (122.east) -- (142);
        \draw[loosely dotted] (142.east) -- (162);
        
        \node[nct2] (233) at (15.2,0) {$\qthree{}[3_3][]$};
        \node[nct2] (224) at (14.7,0.7) {$\qthree{3_3}[2_3]$};
        \node[nct3] (213) at (14,0) {$\qthree{}[2_3][]$};
        \node[nct2] (204) at (13.3,0.7) {$\qthree{2_3}[1]$};
        \node[nct3] (193) at (12.6,0) {$\qthree{}[1][]$};
        \node[nct2] (184) at (11.9,0.7) {$\qthree{1}[2_1]$};
        \node[nct2] (182) at (11.9,-0.7) {$\qthree{1}[2_2][]$};
        \node[nct2] (173) at (11.2,0) {$\begin{smallmatrix} & 1 & \\ & 2_1 2_2 &\end{smallmatrix} $};
        
        \draw[->] (224) -- (233);
        \draw[->] (213) -- (224);
        \draw[->] (204) -- (213);
        \draw[->] (193) -- (204);
        \draw[->] (182) -- (193);
        \draw[->] (184) -- (193);
        \draw[->] (173) -- (182);
        \draw[->] (173) -- (184);
        \draw[->] (162) -- (173);
        \draw[->] (164) -- (173);
        
        \draw[loosely dotted] (173.east) -- (193);
        \draw[loosely dotted] (193.east) -- (213);
        \draw[loosely dotted] (213.east) -- (233);
        \draw[loosely dotted] (64.east) -- (84);
        \draw[loosely dotted] (62.east) -- (82);
        \draw[loosely dotted] (164.east) -- (184);
        \draw[loosely dotted] (162.east) -- (182);
    \end{tikzpicture}
    \caption{The Aus\-lan\-der--Rei\-ten quiver of the algebra $\La=\K Q/R_Q^2$ where $Q$ is the quiver on top and where the additive closure of the encircled modules is a $3$-cluster tilting subcategory.}
\label{fig:double gluing example}
\end{figure}

\begin{figure}[htb]
    \centering
    \begin{tikzpicture}
        \node (R) at (-1,-0.5) {$1$};
        \node (Q) at (-1.5,-0.5) {$\tilQ:$};

        \node (14) at (3.5,-0.5) {$4_3$};
        \node (13) at (2,-0.5) {$3_3$};
        \node (12) at (0.5,-0.5) {$2_3$};
        
        \draw[->] (14) -- (13);
        \draw[->] (13) -- (12);
        \draw[->] (12) -- (R);
        
        \node (2M1-) at (4,0) {$6_1$};
        \node (23+) at (3,0) {$5_1$};
        \node (23) at (2,0) {$4_1$};
        \node (22) at (1,0) {$3_1$};
        \node (21) at (0,0) {$2_1$};
        
        \draw[<-] (2M1-) -- (23+);
        \draw[<-] (23+) -- (23);
        \draw[<-] (23) -- (22);
        \draw[<-] (22) -- (21);
        \draw[<-] (21) -- (R);
    
        \node (3M1) at (5,-0.5) {$1'$};
        \node (3M1-) at (4,-1) {$6_2$};
        \node (33+) at (3,-1) {$5_2$};
        \node (33) at (2,-1) {$4_2$};
        \node (32) at (1,-1) {$3_2$};
        \node (31) at (0,-1) {$2_2$};
        
        \draw[<-] (3M1) -- (2M1-);
        \draw[<-] (3M1) -- (3M1-);
        \draw[<-] (3M1-) -- (33+);
        \draw[<-] (33+) -- (33);
        \draw[<-] (33) -- (32);
        \draw[<-] (32) -- (31);
        \draw[<-] (31) -- (R);
        
        \draw[->] (3M1) -- (14);
    \end{tikzpicture}\;\;\;\;
    \begin{tikzpicture}[scale=0.9, transform shape, baseline={(current bounding box.center)}]
        \tikzstyle{nct2}=[circle, minimum width=0.6cm, draw, inner sep=0pt, text centered, scale=0.9]
        \tikzstyle{nct22}=[circle, minimum width=0.6cm, draw, inner sep=0pt, text centered, scale=0.8]
        \tikzstyle{nct3}=[circle, minimum width=6pt, draw=white, inner sep=0pt, scale=0.9]
        
        \node[scale=0.8] (name) at (7.7,-2.3) {$\Gamma\left(\K \tilQ/R_{\tilQ}^2\right)$};
        
        \node[nct2] (13) at (0,0) {$\qthree{}[3_{3}][]$};
        \node[nct2] (24) at (0.7,0.7) {$\qthree{4_3}[3_{3}]$};
        \node[nct3] (33) at (1.4,0) {$\qthree{}[4_3][]$};
        \node[nct2] (44) at (2.1,0.7) {$\qthree{1'}[4_3]$};
        \node[nct3] (53) at (2.8,0) {$\qthree{}[1'][]$};
        \node[nct2] (64) at (3.5,0.7) {$\qthree{6_1}[1']$};
        \node[nct2] (62) at (3.5,-0.7) {$\qthree{6_2}[1'][]$};
        \node[nct2] (73) at (4.2,0) {$\begin{smallmatrix} & 6_1  6_2 & \\ & 1' &\end{smallmatrix} $};
        
        \draw[->] (13) -- (24);
        \draw[->] (24) -- (33);
        \draw[->] (33) -- (44);
        \draw[->] (44) -- (53);
        \draw[->] (53) -- (64);
        \draw[->] (53) -- (62);
        \draw[->] (64) -- (73);
        \draw[->] (62) -- (73);
        
        \draw[loosely dotted] (13.east) -- (33);
        \draw[loosely dotted] (33.east) -- (53);
        \draw[loosely dotted] (53.east) -- (73);
        
        \node[nct3] (84) at (4.9,0.7) {$\qthree{}[6_2][]$};
        \node[nct2] (95) at (5.6,1.4) {$\qthree{5_2}[6_2][]$};
        \node[nct3] (104) at (6.3,0.7) {$\qthree{}[5_2][]$};
        \node[nct2] (115) at (7,1.4) {$\qthree{4_2}[5_2][]$};
        \node[nct2] (124) at (7.7,0.7) {$\qthree{}[4_2][]$};
        \node[nct2] (135) at (8.4,1.4) {$\qthree{3_2}[4_2][]$};
        \node[nct3] (144) at (9.1,0.7) {$\qthree{}[3_2][]$};
        \node[nct2] (155) at (9.8,1.4) {$\qthree{2_2}[3_2][]$};
        \node[nct3] (164) at (10.5,0.7) {$\qthree{}[2_2][]$};
        
        \draw[->] (73) -- (84);
        \draw[->] (84) -- (95);
        \draw[->] (95) -- (104);
        \draw[->] (104) -- (115);
        \draw[->] (115) -- (124);
        \draw[->] (124) -- (135);
        \draw[->] (135) -- (144);
        \draw[->] (144) -- (155);
        \draw[->] (155) -- (164);
        
        \draw[loosely dotted] (84.east) -- (104);
        \draw[loosely dotted] (104.east) -- (124);
        \draw[loosely dotted] (124.east) -- (144);
        \draw[loosely dotted] (144.east) -- (164);
        
        \node[nct3] (82) at (4.9,-0.7) {$\qthree{}[6_1][]$};
        \node[nct2] (91) at (5.6,-1.4) {$\qthree{5_1}[6_1][]$};
        \node[nct3] (102) at (6.3,-0.7) {$\qthree{}[5_1][]$};
        \node[nct2] (111) at (7,-1.4) {$\qthree{4_1}[5_1][]$};
        \node[nct2] (122) at (7.7,-0.7) {$\qthree{}[4_1][]$};
        \node[nct2] (131) at (8.4,-1.4) {$\qthree{3_1}[4_1][]$};
        \node[nct3] (142) at (9.1,-0.7) {$\qthree{}[3_1][]$};
        \node[nct2] (151) at (9.8,-1.4) {$\qthree{2_1}[3_1][]$};
        \node[nct3] (162) at (10.5,-0.7) {$\qthree{}[2_1][]$};
        
        \draw[->] (73) -- (82);
        \draw[->] (82) -- (91);
        \draw[->] (91) -- (102);
        \draw[->] (102) -- (111);
        \draw[->] (111) -- (122);
        \draw[->] (122) -- (131);
        \draw[->] (131) -- (142);
        \draw[->] (142) -- (151);
        \draw[->] (151) -- (162);
        
        \draw[loosely dotted] (82.east) -- (102);
        \draw[loosely dotted] (102.east) -- (122);
        \draw[loosely dotted] (122.east) -- (142);
        \draw[loosely dotted] (142.east) -- (162);
        
        \node[nct2] (233) at (15.2,0) {$\qthree{}[3_3][]$};
        \node[nct2] (224) at (14.7,0.7) {$\qthree{3_3}[2_3]$};
        \node[nct3] (213) at (14,0) {$\qthree{}[2_3][]$};
        \node[nct2] (204) at (13.3,0.7) {$\qthree{2_3}[1]$};
        \node[nct3] (193) at (12.6,0) {$\qthree{}[1][]$};
        \node[nct2] (184) at (11.9,0.7) {$\qthree{1}[2_1]$};
        \node[nct2] (182) at (11.9,-0.7) {$\qthree{1}[2_2][]$};
        \node[nct2] (173) at (11.2,0) {$\begin{smallmatrix} & 1 & \\ & 2_1 2_2 &\end{smallmatrix} $};
        
        \draw[->] (224) -- (233);
        \draw[->] (213) -- (224);
        \draw[->] (204) -- (213);
        \draw[->] (193) -- (204);
        \draw[->] (182) -- (193);
        \draw[->] (184) -- (193);
        \draw[->] (173) -- (182);
        \draw[->] (173) -- (184);
        \draw[->] (162) -- (173);
        \draw[->] (164) -- (173);
        
        \draw[loosely dotted] (173.east) -- (193);
        \draw[loosely dotted] (193.east) -- (213);
        \draw[loosely dotted] (213.east) -- (233);
        
        \draw[loosely dotted] (0,1.6) -- (0,-1.6);
        \draw[loosely dotted] (15.2,1.6) -- (15.2,-1.6);
        \draw[loosely dotted] (64.east) -- (84);
        \draw[loosely dotted] (62.east) -- (82);
        \draw[loosely dotted] (164.east) -- (184);
        \draw[loosely dotted] (162.east) -- (182);
    \end{tikzpicture}
    \caption{The Aus\-lan\-der--Rei\-ten quiver of the algebra $\La=\K \tilQ/R_{\tilQ}^2$ where $\tilQ$ is the quiver on top and where the additive closure of the encircled modules is a $3$-cluster tilting subcategory. Compare with Figure \ref{fig:double gluing example}.}
\label{fig:triple gluing example}
\end{figure}

\subsubsection{Self-gluing of starlike algebras with four rays} We continue by applying Corollary \ref{cor:La infinity n-ct} to algebras as in Proposition \ref{prop:starlike algebras}\ref{case:k=4}. 

\begin{corollary}\label{cor:cyclic starlike with 4 rays}
Let $Q$ be the quiver
\[\begin{tikzpicture}
    \node (R) at (0,0) {$r$};

    \node (1M1-) at (5.5,0.5) {$r_{m_3-1}^{(3)}$};
    \node (13+) at (4,0.5) {};
    \node (13) at (2.5,0.5) {};
    \node (12) at (1,0.5) {$r_2^{(3)}$};
    
    \draw[<-] (1M1-) -- (13+);
    \draw[loosely dotted] (13+) -- (13);
    \draw[<-] (13) -- (12);
    \draw[<-] (12) -- (R);
    
    \node (2M1-) at (-5.5,0.5) {$r_{m_2}^{(2)}$};
    \node (23+) at (-4,0.5) {};
    \node (23) at (-2.5,0.5) {};
    \node (22) at (-1,0.5) {$r_2^{(2)}$};
    
    \draw[->] (2M1-) -- (23+);
    \draw[loosely dotted] (23+) -- (23);
    \draw[->] (23) -- (22);
    \draw[->] (22) -- (R);

    \node (4M1-) at (-5.5,-0.5) {$r_{m_4}^{(4)}$};
    \node (43+) at (-4,-0.5) {};
    \node (43) at (-2.5,-0.5) {};
    \node (42) at (-1,-0.5) {$r_4^{(4)}$};
    
    \draw[<-] (4M1-) -- (43+);
    \draw[loosely dotted] (43+) -- (43);
    \draw[<-] (43) -- (42);
    \draw[<-] (42) -- (R);
    
    \node (3M1) at (7.5,0) {$r_{m_3}^{(3)}=r_{m_1}^{(1)}$\nospacepunct{.}};
    \node (3M1-) at (5.5,-0.5) {$r_{m_1-1}^{(1)}$};
    \node (33+) at (4,-0.5) {};
    \node (33) at (2.5,-0.5) {};
    \node (32) at (1,-0.5) {$r_2^{(1)}$};
    
    \draw[<-] (3M1) -- (1M1-);
    \draw[->] (3M1) -- (3M1-);
    \draw[->] (3M1-) -- (33+);
    \draw[loosely dotted] (33+) -- (33);
    \draw[->] (33) -- (32);
    \draw[->] (32) -- (R);
\end{tikzpicture}\]
If $m_i=3+2x_i$ for $i\in\{1,2,3,4\}$ and $x_i\in\ZZ_{\geq 0}$, then $\m(\K Q/R_Q^2)$ admits a $2$-cluster tilting subcategory.
\end{corollary}

\begin{proof}
Assume that $m_1\geq m_2\geq m_3\geq m_4$; the other cases are similar. Let $\La=\K T/R_T^2$ where $T=T(4,m_1,m_2,m_3,m_4)$ for some $A_{m_1}$, $A_{m_2}$, $A_{m_3}$ and $A_{m_4}$ linearly oriented and such that $r_{m_1}^{(1)}$ and $r_{m_2}^{(2)}$ are sources and $r_{m_3}^{(3)}$ and $r_{m_4}^{(4)}$ are sinks. Then $\m\La$ admits a $2$-cluster tilting subcategory by Proposition \ref{prop:starlike algebras}(d). Moreover, the pair $(W,J)$ where $W=P\left(r_{m_3-1}^{(3)}\right)$ and $J=I\left(r_{m_1-1}^{(1)}\right)$ is a fractured pair and the pair $(P,I)$ where $P=P\left(r_{m_3}^{(3)}\right)$ and $I=I\left(r_{m_1}^{(1)}\right)$ is a compatible pair for $(W,J)$. Hence $\La$ is $n$-self-gluable and by a direct computation we get that $\tilLa\isom \K Q/R_Q^2$. Finally, the module category $\m\tilLa$ admits a $2$-cluster tilting subcategory by Corollary \ref{cor:La infinity n-ct}. 
\end{proof}

We can also perform double gluings with algebras of this type. To avoid cumbersome notation and since the reader should be familiar with the methods described in this section by now, let us denote by
\[\begin{tikzpicture}[baseline={(current bounding box.center)}, scale=0.8, transform shape, decoration={
    markings,
    mark=at position 0.6 with {\arrow{>}}}]
  \node (r) at (0,0) {$\bullet$};
  
  \node (s1) at (-0.7,0.7) {$\bullet$};
  \node (s2) at (-0.7,-0.7) {$\bullet$};
  \node (s3) at (0.7,0.7) {$\bullet$};
  \node (s4) at (0.7,-0.7) {$\bullet$};
  
  \draw[postaction={decorate}] (s1) -- (r);
  \draw[postaction={decorate}] (s2) -- (r);
  \draw[postaction={decorate}] (r) -- (s3);
  \draw[postaction={decorate}] (r) -- (s4);
  
  \node (Q) at (0,-1) {$T_1$};
\end{tikzpicture}
\]
the quiver of a starlike algebra $\La_1$ satisfying the conditions of Proposition \ref{prop:starlike algebras}\ref{case:k=4}. That is, each
$\begin{tikzpicture}[scale=0.8, transform shape, decoration={markings, mark=at position 0.7 with {\arrow{>}}}] 
    \node (A) at (0,0) {$\bullet$}; 
    \node (B) at (0.7,0) {$\bullet$}; 
    \draw[postaction={decorate}] (A) -- (B);
\end{tikzpicture}$ 
is a subquiver of type $\overrightarrow{A}_{3+2x}$ for some $x\in\ZZ_{\geq 0}$, each of them for a possibly different value of $x$. Then by performing a double gluing with another starlike algebra $\La_2$ with a quiver $T_2$ denoted in the same way, we get a new algebra with a quiver
\[\begin{tikzpicture}[baseline={(current bounding box.center)}, scale=0.8, transform shape, decoration={
    markings,
    mark=at position 0.6 with {\arrow{>}}}]
  \node (r) at (0,0) {$\bullet$};
  
  \node (s1) at (-0.7,0.7) {$\bullet$};
  \node (s2) at (-0.7,-0.7) {$\bullet$};
  \node (t3) at (0.7,0.7) {$\bullet$};
  \node (t4) at (0.7,-0.7) {$\bullet$};
  
  \draw[postaction={decorate}] (s1) -- (r);
  \draw[postaction={decorate}] (s2) -- (r);
  \draw[postaction={decorate}] (r) -- (t3);
  \draw[postaction={decorate}] (r) -- (t4);
  
  \node (r') at (1.4,0) {$\bullet$};
  
  \node (t5) at (2.1,0.7) {$\bullet$};
  \node (t6) at (2.1,-0.7) {$\bullet$};
  
  \draw[postaction={decorate}] (t3) -- (r');
  \draw[postaction={decorate}] (t4) -- (r');
  \draw[postaction={decorate}] (r') -- (t5);
  \draw[postaction={decorate}] (r') -- (t6);
  
  \node (Q) at (0,-1) {$T_1$};
  \node (Q2) at (1.4,-1) {$T_2$};
\end{tikzpicture}
\]
and radical square zero relations whose module category admits a $2$-cluster tilting subcategory by Proposition \ref{prop:double gluing}. This algebra is also rep\-re\-sen\-ta\-tion-di\-rect\-ed and so continuing inductively we get that the module category of the algebra $\La=\K T/R_T^2$ where $T$ is a quiver of the form
\[\begin{tikzpicture}[baseline={(current bounding box.center)}, scale=0.8, transform shape, decoration={
    markings,
    mark=at position 0.6 with {\arrow{>}}}]
  \node (r) at (0,0) {$\bullet$};
  
  \node (s1) at (-0.7,0.7) {$\bullet$};
  \node (s2) at (-0.7,-0.7) {$\bullet$};
  \node (t3) at (0.7,0.7) {$\bullet$};
  \node (t4) at (0.7,-0.7) {$\bullet$};
  
  \draw[postaction={decorate}] (s1) -- (r);
  \draw[postaction={decorate}] (s2) -- (r);
  \draw[postaction={decorate}] (r) -- (t3);
  \draw[postaction={decorate}] (r) -- (t4);
  
  \node (r') at (1.4,0) {$\bullet$};
  
  \node (t5) at (2.1,0.7) {$\bullet$};
  \node (t6) at (2.1,-0.7) {$\bullet$};
  
  \draw[postaction={decorate}] (t3) -- (r');
  \draw[postaction={decorate}] (t4) -- (r');
  \draw[postaction={decorate}] (r') -- (t5);
  \draw[postaction={decorate}] (r') -- (t6);
  
  \node (s7) at (3.5,0.7) {$\bullet$};
  \node (s8) at (3.5,-0.7) {$\bullet$}; 
  \node (t9) at (4.9,0.7) {$\bullet$};  
  \node (t10) at (4.9,-0.7) {$\bullet$};
  
  \node (r'') at (4.2,0) {$\bullet$};
  
  \draw[postaction={decorate}] (s7) -- (r'');
  \draw[postaction={decorate}] (s8) -- (r'');
  \draw[postaction={decorate}] (r'') -- (t9);
  \draw[postaction={decorate}] (r'') -- (t10);
  
  \draw[loosely dotted] (2.2,0) -- (3.4,0);
\end{tikzpicture}
\]
admits a $2$-cluster tilting subcategory. Going one step further, by taking $T_1=T_2=\cdots=T_i$ we can apply a similar orbit construction as in Section \ref{sec:gluings and orbit categories} to get that the module category of the algebra $\tilde{\La}=\K\tilde{T}/R_{\tilde{T}}^2$ where $\tilde{T}$ is a quiver of the form
\[\begin{tikzpicture}[baseline={(current bounding box.center)}, scale=0.8, transform shape, decoration={
    markings,
    mark=at position 0.6 with {\arrow{>}}}]
  \node (r) at (0,0) {$\bullet$};
  
  \node (s1) at (0,1) {$\bullet$};
  \node (s2) at (0,-1) {$\bullet$};
  
  \draw[postaction={decorate}] (s1) to [bend right] (r);
  \draw[postaction={decorate}] (r) to [bend right] (s1);
  \draw[postaction={decorate}] (r) to [bend right] (s2);
  \draw[postaction={decorate}] (s2) to [bend right] (r);
  
\end{tikzpicture}
\]
admits a $2$-cluster tilting subcategory. For an example, see Figure \ref{fig:double gluing orbit four}.

\begin{figure}[htb]
    \centering
    \begin{tikzpicture}[baseline={(current bounding box.center)}, scale=0.9]
        \node (Q) at (-2,0) {$\tilQ:$};

        \node (1) at (0,0) {$1$};
        \node (2) at (0.7,0.7) {$5$};
        \node (3) at (1.4,0) {$3$};
        \node (4) at (0.7,-0.7) {$4$};
        \node (5) at (-0.7,-0.7) {$5$};
        \node (6) at (-1.4,0) {$6$};
        \node (7) at (-0.7,0.7) {$7$};
        
        \draw[->] (1) -- (2);
        \draw[->] (2) -- (3);
        \draw[->] (3) -- (4);
        \draw[->] (4) -- (1);
        \draw[->] (1) -- (5);
        \draw[->] (5) -- (6);
        \draw[->] (6) -- (7);
        \draw[->] (7) -- (1);
    \end{tikzpicture}\;\;\;\;
    \begin{tikzpicture}[scale=0.9, transform shape, baseline={(current bounding box.center)}]
        \tikzstyle{nct2}=[circle, minimum width=0.6cm, draw, inner sep=0pt, text centered, scale=0.9]
        \tikzstyle{nct22}=[circle, minimum width=0.6cm, draw, inner sep=0pt, text centered, scale=0.8]
        \tikzstyle{nct3}=[circle, minimum width=6pt, draw=white, inner sep=0pt, scale=0.9]
        
        \node[scale=0.8] (name) at (3.5,-2.3) {$\Gamma\left(\K \tilQ/R_{\tilQ}^2\right)$};
        
        \node[nct3] (13) at (0,0) {$\qthree{}[1][]$};
        \node[nct2] (24) at (0.7,0.7) {$\qthree{4}[1]$};
        \node[nct2] (22) at (0.7,-0.7) {$\qthree{7}[1]$};
        \node[nct22] (33) at (1.4,0) {$\begin{smallmatrix} 4 & & 7\\ & 1 &\end{smallmatrix} $};
        \node[nct3] (44) at (2.1,0.7) {$\qthree{}[7][]$};
        \node[nct3] (42) at (2.1,-0.7) {$\qthree{}[4][]$};
        \node[nct2] (55) at (2.8,1.4) {$\qthree{6}[7]$};
        \node[nct2] (51) at (2.8,-1.4) {$\qthree{3}[4]$};
        \node[nct2] (64) at (3.5,0.7) {$\qthree{}[6][]$};
        \node[nct2] (62) at (3.5,-0.7) {$\qthree{}[3][]$};
        \node[nct2] (75) at (4.2,1.4) {$\qthree{5}[6]$};
        \node[nct2] (71) at (4.2,-1.4) {$\qthree{2}[3]$};
        \node[nct3] (84) at (4.9,0.7) {$\qthree{}[5][]$};
        \node[nct3] (82) at (4.9,-0.7) {$\qthree{}[2][]$};
        \node[nct22] (93) at (5.6,0) {$\begin{smallmatrix}  & 1 & \\ 2 &  & 5\end{smallmatrix} $};
        \node[nct2] (104) at (6.3,0.7) {$\qthree{1}[2]$};
        \node[nct2] (102) at (6.3,-0.7) {$\qthree{1}[5]$};
        \node[nct3] (113) at (7,0) {$\qthree{}[1][]$};
    
        \draw[->] (13) -- (24);
        \draw[->] (13) -- (22);
        \draw[->] (24) -- (33);
        \draw[->] (22) -- (33);
        \draw[->] (33) -- (44);
        \draw[->] (33) -- (42);
        \draw[->] (44) -- (55);
        \draw[->] (55) -- (64);
        \draw[->] (64) -- (75);
        \draw[->] (75) -- (84);
        \draw[->] (84) -- (93);
        \draw[->] (42) -- (51);
        \draw[->] (51) -- (62);
        \draw[->] (62) -- (71);
        \draw[->] (71) -- (82);
        \draw[->] (82) -- (93);
        \draw[->] (93) -- (104);
        \draw[->] (93) -- (102);
        \draw[->] (104) -- (113);
        \draw[->] (102) -- (113);
        
        \draw[loosely dotted] (13.east) -- (33);
        \draw[loosely dotted] (24.east) -- (44);
        \draw[loosely dotted] (44.east) -- (64);
        \draw[loosely dotted] (64.east) -- (84);
        \draw[loosely dotted] (84.east) -- (104);
        \draw[loosely dotted] (22.east) -- (42);
        \draw[loosely dotted] (42.east) -- (62);
        \draw[loosely dotted] (62.east) -- (82);
        \draw[loosely dotted] (82.east) -- (102);
        \draw[loosely dotted] (93.east) -- (113);
        
        \draw[loosely dotted] (0,1.6) -- (0,-1.6);
        \draw[loosely dotted] (7,1.6) -- (7,-1.6);
    \end{tikzpicture}
    \caption{The Aus\-lan\-der--Rei\-ten quiver of the algebra $\La=\K \tilQ/R_{\tilQ}^2$ where $\tilQ$ is the quiver on the left and where the additive closure of the encircled modules is a $2$-cluster tilting subcategory.}
\label{fig:double gluing orbit four}
\end{figure}

\subsubsection{Self-gluing of quivers with arbitrary number of sinks and sources} Let $\La$ be a rep\-re\-sen\-ta\-tion-di\-rect\-ed algebra with $\sinks_{\La}=t$. Then the opposite algebra $\La^{\text{op}}$ has $\sources_{\La^{\text{op}}}=t$. Assume that we can perform a $t$-simultaneous gluing between $\La$ and $\La^{\text{op}}$ as described in Remark \ref{rem:k-simultaneous gluing} to obtain an algebra $\La'$. Then, if $\m\La$ admits an $n$-cluster tilting subcategory, the module category $\m\La'$ also admits an $n$-cluster tilting subcategory. We illustrate with an example.

\begin{example}\label{ex:k-simultaneous gluing}
Let $\La=\K Q/R_Q^2$ be as in Example \ref{ex:0 sources 3 sinks}. Then $\sources_{\La}=1$ and $\sinks_{\La}=4$ and $\m\La$ admits a $3$-cluster tilting subcategory. We can draw the quivers of $\La$ and $\La^{\text{op}}=\K Q^{\text{op}}/R^2_{Q^{\text{op}}}$ as
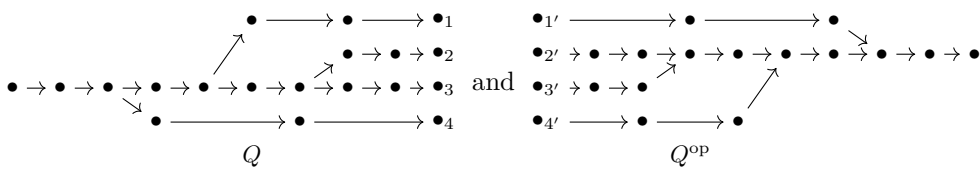
\[\begin{tikzpicture}[baseline={(current bounding box.center)}, scale=0.9, transform shape]
\node (1) at (1,1) {$\bullet$};
\node (2) at (1.7,1) {$\bullet$};
\node (3) at (2.4,1) {$\bullet$};
\node (4) at (3.1,0.5) {$\bullet$};
\node (4') at (3.1,1) {$\bullet$};
\node (5) at (5.2,0.5) {$\bullet$};
\node (5') at (3.8,1) {$\bullet$};
\node (6) at (7.3,0.5) {$\bullet_4$};
\node (6') at (4.5,1) {$\bullet$};
\node (7) at (5.2,1) {$\bullet$};
\node (7') at (4.5,2) {$\bullet$};
\node (8) at (5.9,1) {$\bullet$};
\node (8') at (5.9,2) {$\bullet$};
\node (9) at (6.6,1) {$\bullet$};
\node (9') at (7.3,2) {$\bullet_1$};
\node (10) at (7.3,1) {$\bullet_3$};
\node (10') at (5.9,1.5) {$\bullet$};
\node (11') at (6.6,1.5) {$\bullet$};
\node (12') at (7.3,1.5) {$\bullet_2$};

\node (A) at (4.5,0) {$Q$};

\draw[->] (1) -- (2);
\draw[->] (2) -- (3);
\draw[->] (3) -- (4);
\draw[->] (3) -- (4');
\draw[->] (4) -- (5);
\draw[->] (4') -- (5');
\draw[->] (5) -- (6);
\draw[->] (5') -- (6');
\draw[->] (6') -- (7);
\draw[->] (5') -- (7');
\draw[->] (7) -- (8);
\draw[->] (7') -- (8');
\draw[->] (8) -- (9);
\draw[->] (8') -- (9');
\draw[->] (9) -- (10);
\draw[->] (7) -- (10');
\draw[->] (10') -- (11');
\draw[->] (11') -- (12');
\end{tikzpicture}\text{ and }
\begin{tikzpicture}[baseline={(current bounding box.center)}, scale=0.9, transform shape]
\node (12') at (2.4,1) {$\bullet_{3'}$};
\node (11') at (3.1,1) {$\bullet$};
\node (10') at (3.8,1) {$\bullet$};
\node (10) at (2.4,1.5) {$\bullet_{2'}$};
\node (9) at (3.1,1.5) {$\bullet$};
\node (9') at (2.4,0.5) {$\bullet_{4'}$};
\node (8) at (3.8,1.5) {$\bullet$};
\node (8') at (3.8,0.5) {$\bullet$};
\node (7) at (4.5,1.5) {$\bullet$};
\node (7') at (5.2,0.5) {$\bullet$};
\node (6) at (5.2,1.5) {$\bullet$};
\node (6') at (2.4,2) {$\bullet_{1'}$};
\node (5) at (5.9,1.5) {$\bullet$};
\node (5') at (4.5,2) {$\bullet$};
\node (4) at (6.6,1.5) {$\bullet$};
\node (4') at (6.6,2) {$\bullet$};
\node (3) at (7.3,1.5) {$\bullet$};
\node (2) at (8,1.5) {$\bullet$};
\node (1) at (8.7,1.5) {$\bullet$.};

\draw[<-] (1) -- (2);
\draw[<-] (2) -- (3);
\draw[<-] (3) -- (4);
\draw[<-] (3) -- (4');
\draw[<-] (4) -- (5);
\draw[<-] (4') -- (5');
\draw[<-] (5) -- (6);
\draw[<-] (5') -- (6');
\draw[<-] (6) -- (7);
\draw[<-] (5) -- (7');
\draw[<-] (7) -- (8);
\draw[<-] (7') -- (8');
\draw[<-] (8) -- (9);
\draw[<-] (8') -- (9');
\draw[<-] (9) -- (10);
\draw[<-] (7) -- (10');
\draw[<-] (10') -- (11');
\draw[<-] (11') -- (12');

\node (A) at (4.5,0) {$Q^{\text{op}}$};
\end{tikzpicture}\]
By gluing simultaneously along the vertices $\bullet_i=\bullet_{i'}$ for $i\in\{1,2,3,4\}$ we obtain the quiver $Q'$:
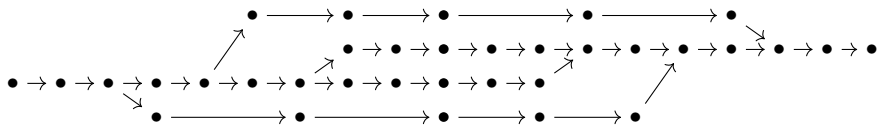
\[\begin{tikzpicture}[baseline={(current bounding box.center)}, scale=0.9, transform shape]
\node (1) at (1,1) {$\bullet$};
\node (2) at (1.7,1) {$\bullet$};
\node (3) at (2.4,1) {$\bullet$};
\node (4) at (3.1,0.5) {$\bullet$};
\node (4') at (3.1,1) {$\bullet$};
\node (5) at (5.2,0.5) {$\bullet$};
\node (5') at (3.8,1) {$\bullet$};
\node (6) at (7.3,0.5) {$\bullet$};
\node (6') at (4.5,1) {$\bullet$};
\node (7) at (5.2,1) {$\bullet$};
\node (7') at (4.5,2) {$\bullet$};
\node (8) at (5.9,1) {$\bullet$};
\node (8') at (5.9,2) {$\bullet$};
\node (9) at (6.6,1) {$\bullet$};
\node (9') at (7.3,2) {$\bullet$};
\node (10) at (7.3,1) {$\bullet$};
\node (10') at (5.9,1.5) {$\bullet$};
\node (11') at (6.6,1.5) {$\bullet$};
\node (12') at (7.3,1.5) {$\bullet$};

\draw[->] (1) -- (2);
\draw[->] (2) -- (3);
\draw[->] (3) -- (4);
\draw[->] (3) -- (4');
\draw[->] (4) -- (5);
\draw[->] (4') -- (5');
\draw[->] (5) -- (6);
\draw[->] (5') -- (6');
\draw[->] (6') -- (7);
\draw[->] (5') -- (7');
\draw[->] (7) -- (8);
\draw[->] (7') -- (8');
\draw[->] (8) -- (9);
\draw[->] (8') -- (9');
\draw[->] (9) -- (10);
\draw[->] (7) -- (10');
\draw[->] (10') -- (11');
\draw[->] (11') -- (12');

\node (12'R) at (7.3,1) {$\bullet$};
\node (11'R) at (8,1) {$\bullet$};
\node (10'R) at (8.7,1) {$\bullet$};
\node (10R) at (7.3,1.5) {$\bullet$};
\node (9R) at (8,1.5) {$\bullet$};
\node (9'R) at (7.3,0.5) {$\bullet$};
\node (8R) at (8.7,1.5) {$\bullet$};
\node (8'R) at (8.7,0.5) {$\bullet$};
\node (7R) at (9.4,1.5) {$\bullet$};
\node (7'R) at (10.1,0.5) {$\bullet$};
\node (6R) at (10.1,1.5) {$\bullet$};
\node (6'R) at (7.3,2) {$\bullet$};
\node (5R) at (10.8,1.5) {$\bullet$};
\node (5'R) at (9.4,2) {$\bullet$};
\node (4R) at (11.5,1.5) {$\bullet$};
\node (4'R) at (11.5,2) {$\bullet$};
\node (3R) at (12.2,1.5) {$\bullet$};
\node (2R) at (12.9,1.5) {$\bullet$};
\node (1R) at (13.6,1.5) {$\bullet$.};

\draw[<-] (1R) -- (2R);
\draw[<-] (2R) -- (3R);
\draw[<-] (3R) -- (4R);
\draw[<-] (3R) -- (4'R);
\draw[<-] (4R) -- (5R);
\draw[<-] (4'R) -- (5'R);
\draw[<-] (5R) -- (6R);
\draw[<-] (5'R) -- (6'R);
\draw[<-] (6R) -- (7R);
\draw[<-] (5R) -- (7'R);
\draw[<-] (7R) -- (8R);
\draw[<-] (7'R) -- (8'R);
\draw[<-] (8R) -- (9R);
\draw[<-] (8'R) -- (9'R);
\draw[<-] (9R) -- (10R);
\draw[<-] (7R) -- (10'R);
\draw[<-] (10'R) -- (11'R);
\draw[<-] (11'R) -- (12'R);
\end{tikzpicture}
\]
and one can check that the algebra $\La'=\K Q'/R_{Q'}^2$ is rep\-re\-sen\-ta\-tion-di\-rect\-ed and so its module category admits a $3$-cluster tilting subcategory $\cM'$ by Remark \ref{rem:k-simultaneous gluing}.

Next let $P$ be the indecomposable projective $\La'$-module corresponding to the unique source of $\La'$ and let $W\geq P$ be a maximal left abutment. Moreover let $I$ be the indecomposable injective $\La'$-module corresponding to the unique sink of $\La'$ and let $J\geq I$ be a maximal right abutment. Then $\La'$ is $3$-self-gluable by Corollary \ref{cor:La n-ct implies La infinity n-ct}. Hence by gluing along the unique source and the unique sink of $Q'$, we obtain the quiver $\tilQ'$:
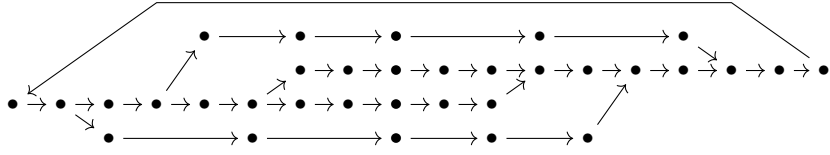
\[\begin{tikzpicture}[baseline={(current bounding box.center)}, scale=0.9, transform shape]
\node (2) at (1.7,1) {$\bullet$};
\node (3) at (2.4,1) {$\bullet$};
\node (4) at (3.1,0.5) {$\bullet$};
\node (4') at (3.1,1) {$\bullet$};
\node (5) at (5.2,0.5) {$\bullet$};
\node (5') at (3.8,1) {$\bullet$};
\node (6) at (7.3,0.5) {$\bullet$};
\node (6') at (4.5,1) {$\bullet$};
\node (7) at (5.2,1) {$\bullet$};
\node (7') at (4.5,2) {$\bullet$};
\node (8) at (5.9,1) {$\bullet$};
\node (8') at (5.9,2) {$\bullet$};
\node (9) at (6.6,1) {$\bullet$};
\node (9') at (7.3,2) {$\bullet$};
\node (10) at (7.3,1) {$\bullet$};
\node (10') at (5.9,1.5) {$\bullet$};
\node (11') at (6.6,1.5) {$\bullet$};
\node (12') at (7.3,1.5) {$\bullet$};

\draw[->] (2) -- (3);
\draw[->] (3) -- (4);
\draw[->] (3) -- (4');
\draw[->] (4) -- (5);
\draw[->] (4') -- (5');
\draw[->] (5) -- (6);
\draw[->] (5') -- (6');
\draw[->] (6') -- (7);
\draw[->] (5') -- (7');
\draw[->] (7) -- (8);
\draw[->] (7') -- (8');
\draw[->] (8) -- (9);
\draw[->] (8') -- (9');
\draw[->] (9) -- (10);
\draw[->] (7) -- (10');
\draw[->] (10') -- (11');
\draw[->] (11') -- (12');

\node (12'R) at (7.3,1) {$\bullet$};
\node (11'R) at (8,1) {$\bullet$};
\node (10'R) at (8.7,1) {$\bullet$};
\node (10R) at (7.3,1.5) {$\bullet$};
\node (9R) at (8,1.5) {$\bullet$};
\node (9'R) at (7.3,0.5) {$\bullet$};
\node (8R) at (8.7,1.5) {$\bullet$};
\node (8'R) at (8.7,0.5) {$\bullet$};
\node (7R) at (9.4,1.5) {$\bullet$};
\node (7'R) at (10.1,0.5) {$\bullet$};
\node (6R) at (10.1,1.5) {$\bullet$};
\node (6'R) at (7.3,2) {$\bullet$};
\node (5R) at (10.8,1.5) {$\bullet$};
\node (5'R) at (9.4,2) {$\bullet$};
\node (4R) at (11.5,1.5) {$\bullet$};
\node (4'R) at (11.5,2) {$\bullet$};
\node (3R) at (12.2,1.5) {$\bullet$};
\node (2R) at (12.9,1.5) {$\bullet$};
\node (1R) at (13.6,1.5) {$\bullet$.};

\draw[<-] (1R) -- (2R);
\draw[<-] (2R) -- (3R);
\draw[<-] (3R) -- (4R);
\draw[<-] (3R) -- (4'R);
\draw[<-] (4R) -- (5R);
\draw[<-] (4'R) -- (5'R);
\draw[<-] (5R) -- (6R);
\draw[<-] (5'R) -- (6'R);
\draw[<-] (6R) -- (7R);
\draw[<-] (5R) -- (7'R);
\draw[<-] (7R) -- (8R);
\draw[<-] (7'R) -- (8'R);
\draw[<-] (8R) -- (9R);
\draw[<-] (8'R) -- (9'R);
\draw[<-] (9R) -- (10R);
\draw[<-] (7R) -- (10'R);
\draw[<-] (10'R) -- (11'R);
\draw[<-] (11'R) -- (12'R);

\draw[->] (1R) -- (12.2,2.5) -- (3.8,2.5) -- (2);
\end{tikzpicture}
\]
and the algebra $\tilLa'=\K \tilQ'/R_{\tilQ'}^2$ is not longer rep\-re\-sen\-ta\-tion-di\-rect\-ed, but its module category still admits a $3$-cluster tilting subcategory $\tilde{M}$ by Corollary \ref{cor:La n-ct implies La infinity n-ct}.
\end{example}

\subsection{Examples from algebras with \texorpdfstring{$n$}{n}-fractured subcategories} We continue by obtaining examples of algebras whose module categories admit $n$-cluster tilting subcategories, starting from algebras whose module categories admit $n$-fracture subcategories. For this section, let $n\geq 2$ and let $\La=\K Q/\cR$ be a rep\-re\-sen\-ta\-tion-di\-rect\-ed bound quiver algebra with a fracturing $(T^L,T^R)$ and a $\left(T^L,T^R,n\right)$-fractured subcategory $\cM$. 

\subsubsection{Exactly one fracture}\label{subsubsec:exactly one fracture} In this section a special class of tilting $\K\overrightarrow{A}_h$-modules is important. Recall that the Aus\-lan\-der--Rei\-ten quiver of $\Gamma(\K\overrightarrow{A}_h)$ is of the form $\PD$ where $P=P(1)$. Since $\D(\K\overrightarrow{A}_h)=\K\overrightarrow{A}^{\text{op}}_h$, there exists a unique isomorphism of Aus\-lan\-der--Rei\-ten quivers $f:\Gamma(\K\overrightarrow{A}_h)\overset{\sim}{\longto}\Gamma( \D(\K\overrightarrow{A}_h))$. Moreover, using the $\K$-dual functor $\D=\Hom_{\K}(-,\K)$, we also have a bijection $\D:\Gamma( \D(\K\overrightarrow{A}_h))_0 \to \Gamma(\K \overrightarrow{A}_h)_0$ given by $\D([N])=[\D(N)]$. For every $[M]\in \Gamma(\K \overrightarrow{A}_h)$ we define $S([M])\coloneqq \D\circ f([M])$. By unraveling the definitions of the functors involved, it follows that $S([M])$ is the vertex in $\Gamma(\K\overrightarrow{A}_h)$ such that $[M]$ and $[S(M)]$ are symmetric along the perpendicular bisector of the bottom line of $\PD$.

\begin{definition}
Let $M=\bigoplus_{i=1}^k M_i$ be a $\K\overrightarrow{A}_h$-module where all $M_i$ are indecomposable and $M_i\not\isom M_j$ for $i\neq j$. We say that $M$ is \emph{mirrored}\index[definitions]{mirrored $\K\overrightarrow{A}_h$-module} if $\{[M_i]\}_{i=1}^k=\{S([M_i])\}_{i=1}^k$.
\end{definition}

\begin{example}\label{ex:mirrored tilting module}
Let $h=5$. Then the Aus\-lan\-der--Rei\-ten quiver $\Gamma(\K \overrightarrow{A}_5)$ is
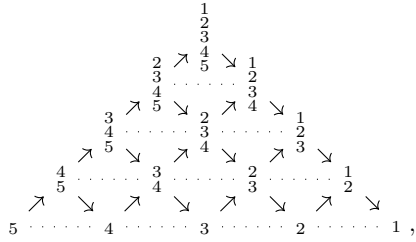
\[\begin{tikzpicture}[scale=0.9, transform shape, baseline={(current bounding box.center)}]

\node (5) at (0,0) {$\qthree{}[5][]$};
\node (4) at (1.4,0) {$\qthree{}[4][]$};
\node (3) at (2.8,0) {$\qthree{}[3][]$};
\node (2) at (4.2,0) {$\qthree{}[2][]$};
\node (1) at (5.6,0) {$\qthree{}[1][]$\nospacepunct{,}};

\draw[loosely dotted] (5.east) -- (4);
\draw[loosely dotted] (4.east) -- (3);
\draw[loosely dotted] (3.east) -- (2);
\draw[loosely dotted] (2.east) -- (1);

\node (45) at (0.7,0.7) {$\qtwo{4}[5]$};
\node (34) at (2.1,0.7) {$\qtwo{3}[4]$};
\node (23) at (3.5,0.7) {$\qtwo{2}[3]$};
\node (12) at (4.9,0.7) {$\qtwo{1}[2]$};

\draw[->] (5) to (45);
\draw[->] (4) to (34);
\draw[->] (3) to (23);
\draw[->] (2) to (12);
\draw[->] (12) to (1);
\draw[->] (23) to (2);
\draw[->] (34) to (3);
\draw[->] (45) to (4);

\draw[loosely dotted] (45.east) -- (34);
\draw[loosely dotted] (34.east) -- (23);
\draw[loosely dotted] (23.east) -- (12);

\node (345) at (1.4,1.4) {$\qthree{3}[4][5]$};
\node (234) at (2.8,1.4) {$\qthree{2}[3][4]$};
\node (123) at (4.2,1.4) {$\qthree{1}[2][3]$};

\draw[->] (45) to (345);
\draw[->] (34) to (234);
\draw[->] (23) to (123);
\draw[->] (123) to (12);
\draw[->] (234) to (23);
\draw[->] (345) to (34);

\draw[loosely dotted] (345.east) -- (234);
\draw[loosely dotted] (234.east) -- (123);

\node (2345) at (2.1,2.1) {$\qfour{2}[3][4][5]$};
\node (1234) at (3.5,2.1) {$\qfour{1}[2][3][4]$};

\draw[->] (345) to (2345);
\draw[->] (234) to (1234);
\draw[->] (1234) to (123);
\draw[->] (2345) to (234);

\draw[loosely dotted] (2345.east) -- (1234);

\node (12345) at (2.8,2.8) {$\qfive{1}[2][3][4][5]$};

\draw[->] (2345) to (12345);
\draw[->] (12345) to (1234);
\end{tikzpicture}\]
from which we see for example that $S\left(\qtwo{4}[5]\right)=\qtwo{1}[2]$ and $S\left(\qthree{2}[3][4]\right)=\qthree{2}[3][4]$. It follows that if 
\[T_1={\scriptstyle 5}\oplus \qtwo{4}[5] \oplus \qfive{1}[2][3][4][5] \oplus \qtwo{1}[2]\oplus {\scriptstyle 1}\]
then $T_1$ is mirrored. On the other hand, if 
\[T_2 = {\scriptstyle 3} \oplus \qtwo{3}[4] \oplus \qthree{2}[3][4] \oplus \qfour{1}[2][3][4] \oplus \qfive{1}[2][3][4][5]\]
then $T_2$ is not mirrored. Indeed, the vertices of $\Gamma(\K \overrightarrow{A}_h)$ corresponding to the indecomposable summands of $T_2$ are
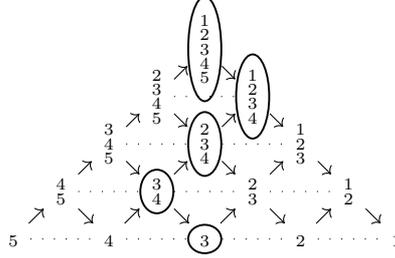
\[\begin{tikzpicture}[scale=0.9, transform shape, baseline={(current bounding box.center)}]

\node (5) at (0,0) {$\qthree{}[5][]$};
\node (4) at (1.4,0) {$\qthree{}[4][]$};
\node (3) at (2.8,0) {\enumber{$\qthree{}[3][]$}};
\node (2) at (4.2,0) {$\qthree{}[2][]$};
\node (1) at (5.6,0) {$\qthree{}[1][]$};

\draw[loosely dotted] (5.east) -- (4);
\draw[loosely dotted] (4.east) -- (3);
\draw[loosely dotted] (3.east) -- (2);
\draw[loosely dotted] (2.east) -- (1);

\node (45) at (0.7,0.7) {$\qtwo{4}[5]$};
\node (34) at (2.1,0.7) {\enumber{$\qtwo{3}[4]$}};
\node (23) at (3.5,0.7) {$\qtwo{2}[3]$};
\node (12) at (4.9,0.7) {$\qtwo{1}[2]$};

\draw[->] (5) to (45);
\draw[->] (4) to (34);
\draw[->] (3) to (23);
\draw[->] (2) to (12);
\draw[->] (12) to (1);
\draw[->] (23) to (2);
\draw[->] (34) to (3);
\draw[->] (45) to (4);

\draw[loosely dotted] (45.east) -- (34);
\draw[loosely dotted] (34.east) -- (23);
\draw[loosely dotted] (23.east) -- (12);

\node (345) at (1.4,1.4) {$\qthree{3}[4][5]$};
\node (234) at (2.8,1.4) {\enumber{$\qthree{2}[3][4]$}};
\node (123) at (4.2,1.4) {$\qthree{1}[2][3]$};

\draw[->] (45) to (345);
\draw[->] (34) to (234);
\draw[->] (23) to (123);
\draw[->] (123) to (12);
\draw[->] (234) to (23);
\draw[->] (345) to (34);

\draw[loosely dotted] (345.east) -- (234);
\draw[loosely dotted] (234.east) -- (123);

\node (2345) at (2.1,2.1) {$\qfour{2}[3][4][5]$};
\node (1234) at (3.5,2.1) {\enumber{$\qfour{1}[2][3][4]$}};

\draw[->] (345) to (2345);
\draw[->] (234) to (1234);
\draw[->] (1234) to (123);
\draw[->] (2345) to (234);

\draw[loosely dotted] (2345.east) -- (1234);

\node (12345) at (2.8,2.8) {\enumber{$\qfive{1}[2][3][4][5]$}};

\draw[->] (2345) to (12345);
\draw[->] (12345) to (1234);
\end{tikzpicture}\]
while the vertices symmetrical to the perpendicular bisector of the base of $\Gamma(\K\overrightarrow{A}_5)$ are
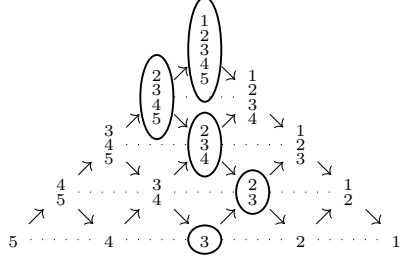
\[\begin{tikzpicture}[scale=0.9, transform shape, baseline={(current bounding box.center)}]

\node (5) at (0,0) {$\qthree{}[5][]$};
\node (4) at (1.4,0) {$\qthree{}[4][]$};
\node (3) at (2.8,0) {\enumber{$\qthree{}[3][]$}};
\node (2) at (4.2,0) {$\qthree{}[2][]$};
\node (1) at (5.6,0) {$\qthree{}[1][]$};

\draw[loosely dotted] (5.east) -- (4);
\draw[loosely dotted] (4.east) -- (3);
\draw[loosely dotted] (3.east) -- (2);
\draw[loosely dotted] (2.east) -- (1);

\node (45) at (0.7,0.7) {$\qtwo{4}[5]$};
\node (34) at (2.1,0.7) {$\qtwo{3}[4]$};
\node (23) at (3.5,0.7) {\enumber{$\qtwo{2}[3]$}};
\node (12) at (4.9,0.7) {$\qtwo{1}[2]$};

\draw[->] (5) to (45);
\draw[->] (4) to (34);
\draw[->] (3) to (23);
\draw[->] (2) to (12);
\draw[->] (12) to (1);
\draw[->] (23) to (2);
\draw[->] (34) to (3);
\draw[->] (45) to (4);

\draw[loosely dotted] (45.east) -- (34);
\draw[loosely dotted] (34.east) -- (23);
\draw[loosely dotted] (23.east) -- (12);

\node (345) at (1.4,1.4) {$\qthree{3}[4][5]$};
\node (234) at (2.8,1.4) {\enumber{$\qthree{2}[3][4]$}};
\node (123) at (4.2,1.4) {$\qthree{1}[2][3]$};

\draw[->] (45) to (345);
\draw[->] (34) to (234);
\draw[->] (23) to (123);
\draw[->] (123) to (12);
\draw[->] (234) to (23);
\draw[->] (345) to (34);

\draw[loosely dotted] (345.east) -- (234);
\draw[loosely dotted] (234.east) -- (123);

\node (2345) at (2.1,2.1) {\enumber{$\qfour{2}[3][4][5]$}};
\node (1234) at (3.5,2.1) {$\qfour{1}[2][3][4]$};

\draw[->] (345) to (2345);
\draw[->] (234) to (1234);
\draw[->] (1234) to (123);
\draw[->] (2345) to (234);

\draw[loosely dotted] (2345.east) -- (1234);

\node (12345) at (2.8,2.8) {\enumber{$\qfive{1}[2][3][4][5]$}};

\draw[->] (2345) to (12345);
\draw[->] (12345) to (1234);
\end{tikzpicture}\]
\end{example}

\begin{remark}\label{rem:mirrored tilting}
If $T$ is a tilting $\K\overrightarrow{A}_h$-module then it is easy to see that a necessary condition for $T$ to be mirrored is that $\coker(P(l)\hookrightarrow P(m))$ is a direct summand of $T$ if and only if $\ker(P(h+1-m)\twoheadrightarrow P(h+1-l))$ is a direct summand of $T$, from which it follows that $h$ must be odd. Moreover, if $h$ is odd, it is easy to construct a mirrored tilting $\K\overrightarrow{A}_h$-module. 
\end{remark}

Assume now that the left fracturing $T^L=\hspace{-0.5em}\bigoplus\limits_{W \in \cP_{\text{ind}}^{\text{mab}}}\hspace{-0.5em}T^{(W)}$ of $\La$ has exactly one nonprojective summand $T^{(W)}$, while $T^R=D(\La)$; the dual case where $T^R$ has exactly one noninjective summand $T^{(J)}$ and $T^L=\La$ is similar. Let $P\leq W$ be a left abutment such that $\underline{T}^{(W)}\in \mathcal{F}_P$ and assume that $\height(P)=h$. Let $T$ be the corresponding tilting $\K \overrightarrow{A}_{h}$-module. Moreover, since $\La$ is rep\-re\-sen\-ta\-tion-di\-rect\-ed, there exists a simple injective $\La$-module $I$. 

Consider now the opposite algebra $\La^{\text{op}}$. By the above properties of $\La$, it follows that we have a fracturing $\left(D(T^L),D(T^R)\right)$ of $\La^{\text{op}}$ and a $\left(D(T^L),D(T^R),n\right)$-fractured subcategory $\D(\cM)$. Moreover we have that the right fracturing $\D(T^R)=\hspace{-0.5em}\bigoplus\limits_{W \in \cP_{\text{ind}}^{\text{mab}}}\hspace{-0.5em}T^{(D(W))}$ has exactly one noninjective summand $T^{(D(W))}$ and that the right abutment $\D(P)$ satisfies $\overline{T}^{(D(W))}\in\mathcal{G}_{D(P)}$ and $\height(D(P))=h$. Moreover, we have that $\D(I)$ is a simple projective $\La$-module.

In general it is not true that the modules $\pi_{P\ast}\left(T^{(P)}\right)$ and $\pi_{D(P)\ast}\left(T^{(D(P))}\right)$ are isomorphic. However, if $T$ is mirrored it easily follows by the definitions involved that indeed $\pi_{P\ast}\left(T^{(P)}\right)$ and $\pi_{D(P)\ast}\left(T^{(D(P))}\right)$ are isomorphic. In this case we can perform two different gluings. Performing the gluing $\La^{\text{op}}\glue[P][D(P)]\La$ we obtain by Theorem \ref{thrm:fractsubcat} an algebra $\La$ whose module category admits an $n$-cluster tilting subcategory. For an example of this kind we refer \cite[Example 4.21]{VAS2}. On the other hand, performing the gluing $\La\glue[D(I)][I]\La^{\text{op}}$ we have the following Corollary.

\begin{corollary}\label{cor:one fracture self glue}
The algebra $\La'=\La\glue[D(I)][I]\La^{\text{op}}$ is $n$-self-gluable. In particular $\m\tilLa'$ admits an $n$-cluster tilting subcategory.
\end{corollary}

\begin{proof}
One can easily verify that $\La'$ is $n$-self-gluable by construction. The result follows by Corollary \ref{cor:La infinity n-ct}. 
\end{proof}

In the particular case where $\La$ is an acyclic Nakayama algebra, we can formulate the previous results in the following Corollary.

\begin{corollary}\label{cor:one fracture Nakayama}
Let $\La$ be an acyclic Nakayama algebra with Kupisch series $(d_1,\dots,d_l)$ and let $\left(b_1,\dots,b_{l}\right)$ be the Kupisch series of $\La^{\text{op}}$. Let $(T^L,T^R)$ be a fracturing of $\La$ and $\cM\subseteq \m\La$ be a $\left(T^L,T^R,n\right)$-fractured subcategory for some $n\geq 2$. In particular, we have that $T^L=T^{(W)}$ for some maximal left abutment $W$ and $T^R=T^{(J)}$ for some maximal right abutment $J$. Let $P\leq W$ be such that $\underline{T}^{(W)}\in\cF_P$ and $\pi_{P\ast}\left(T^{(P)}\right)$ is a mirrored tilting $\K\overrightarrow{A}_{\height(P)}$-module. Moreover, let $1\leq p \leq l$ be such that $P\isom P(p)$. If $T^R$ is injective, then 
\begin{enumerate}[label=(\alph*)]
    \item the module category of the acyclic Nakayama algebra with Kupisch series \[\left(d_1,\dots,d_{p-1},b_1,\dots,b_l\right)\]
    admits an $n$-cluster tilting subcategory, and 
    \item the module category of the cyclic Nakayama algebra with Kupisch series \[\left(b_1,\dots,b_{l-1},d_1,\dots,d_{p-1}\right)\] 
    admits an $n$-cluster tilting subcategory.
\end{enumerate} 
\end{corollary}

\begin{proof}
In the first case, the resulting Nakayama algebra is the gluing $\La^{\text{op}}\glue[P][D(P)]\La$ and the result follows by Theorem \ref{thrm:fractsubcat}. In the second case, the resulting Nakayama algebra is the self-gluing of the gluing of $\La$ with $\La^{\text{op}}$ along the unique simple injective $\La$-module and the result follows by by Corollary \ref{cor:La n-ct implies La infinity n-ct}.
\end{proof}

In the rest of this article, when drawing the Aus\-lan\-der--Rei\-ten quiver of a Nakayama algebra we will only draw the vertices; the arrows are implied by the relative positions of the vertices. Moreover, in a Kupisch series we will denote by $h^{(k)}$ a sequence $h,h,\dots,h$ where $h$ appears $k$ consecutive times.

\begin{example}
Let $\La$ be the acyclic Nakayama algebra with Kupisch series $\left(3^{(7)},5^{(6)},5,4,3,2,1\right)$. Then the Aus\-lan\-der--Rei\-ten quiver of $\La$ is
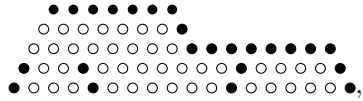
\[\begin{tikzpicture}[scale=1.3, transform shape]
        \foreach \x in {0,0.2,...,3.4}
            \mct{\x}{0};
        \foreach \x in {0.1,0.3,...,3.3}
            \mct{\x}{0.2};
        \foreach \x in {0.2,0.4,...,3.2}
            \mct{\x}{0.4};
        \foreach \x in {0.3,0.5,...,1.7}
            \mct{\x}{0.6};
        \foreach \x in {0.4,0.6,...,1.8}
            \nct{\x}{0.8};
        \nct{0}{0};
        \nct{0.8}{0};        
        \nct{0.1}{0.2};
        \nct{0.7}{0.2};
        \nct{1.7}{0.6};
        \foreach \x in {1.8,2,...,3.2}
            \nct{\x}{0.4};
        \nct{3.3}{0.2};
        \nct{3.4}{0};
        \nct{2.2}{0}:
        \nct{2.3}{0.2};
        \node (A) at (3,0.85) {};
    \end{tikzpicture},\]
where the additive closure of the indecomposable modules corresponding to the bold vertices is a $4$-fractured subcategory. The acyclic Nakayama algebra $\La^{\text{op}}$ has Kupisch series $\left(5^{(7)},4,3^{(8)},2,1\right)$. Gluing $\La$ with $\La^{\text{op}}$ at the unique left maximal abutment of $\La$ we obtain the Nakayama algebra $\La_1$ with Kupisch series $\left(3^{(7)},5^{(6)},5^{(7)},4,3^{(8)},2,1\right)$. The Aus\-lan\-der--Rei\-ten quiver of $\La_1$ is
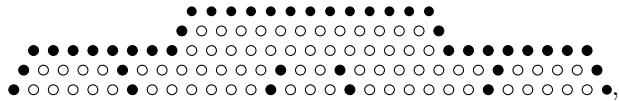
\[\begin{tikzpicture}[scale=1.3, transform shape]
        \foreach \x in {-2.6,-2.4,...,3.4}
            \mct{\x}{0};
        \foreach \x in {-2.5,-2.3,...,3.3}
            \mct{\x}{0.2};
        \foreach \x in {-2.4,-2.2,...,3.2}
            \mct{\x}{0.4};
        \foreach \x in {-0.9,-0.7,...,1.7}
            \mct{\x}{0.6};
        \foreach \x in {-0.8,-0.6,...,1.6}
            \nct{\x}{0.8};
        \nct{0.7}{0.2};
        \nct{0.8}{0};        
        \nct{1.7}{0.6};
        \foreach \x in {1.8,2,...,3.2}
            \nct{\x}{0.4};
        \nct{3.3}{0.2};
        \nct{3.4}{0};
        \nct{2.2}{0}:
        \nct{2.3}{0.2};
        
        \nct{0.1}{0.2};
        \nct{0}{0};
        \nct{-0.9}{0.6}; 
        \foreach \x in {-1,-1.2,...,-2.4}
            \nct{\x}{0.4};
        \nct{-2.5}{0.2};
        \nct{-2.6}{0};
        \nct{-1.4}{0};
        \nct{-1.5}{0.2};
        \node (A) at (3,0.85) {};
    \end{tikzpicture},\]
where the additive closure of the indecomposable modules corresponding to the bold vertices is a $4$-cluster tilting subcategory. Gluing $\La$ with $\La^{\text{op}}$ at the unique simple injective $\La$-module we obtain the Nakayama algebra $\La_2$ with Kupisch series $\left(5^{(7)},4,3^{(8)},2,3^{(7)},5^{(6)},5,4,3,2,1\right)$. The Aus\-lan\-der--Rei\-ten quiver of $\La_2$ is
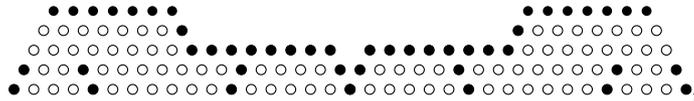
\[\begin{tikzpicture}[scale=1.3, transform shape]
        \foreach \x in {0,0.2,...,3.4}
            \mct{\x}{0};
        \foreach \x in {0.1,0.3,...,3.3}
            \mct{\x}{0.2};
        \foreach \x in {0.2,0.4,...,3.2}
            \mct{\x}{0.4};
        \foreach \x in {0.3,0.5,...,1.7}
            \mct{\x}{0.6};
        \foreach \x in {0.4,0.6,...,1.8}
            \nct{\x}{0.8};
        \nct{0}{0};
        \nct{0.8}{0};        
        \nct{0.1}{0.2};
        \nct{0.7}{0.2};
        \nct{1.7}{0.6};
        \foreach \x in {1.8,2,...,3.2}
            \nct{\x}{0.4};
        \nct{3.3}{0.2};
        \nct{3.4}{0};
        \nct{2.2}{0}:
        \nct{2.3}{0.2};
        
        \foreach \x in {3.4,3.6,...,7}
            \mct{\x}{0};
        \foreach \x in {3.5,3.7,...,6.7}
            \mct{\x}{0.2};
        \foreach \x in {3.6,3.8,...,6.6}
            \mct{\x}{0.4};
        \foreach \x in {5.1,5.3,...,6.5}
            \mct{\x}{0.6};
        \foreach \x in {5.2,5.4,...,6.4}
            \nct{\x}{0.8};
        \nct{3.5}{0.2};        
        \foreach \x in {3.6,3.8,...,5}
            \nct{\x}{0.4};
        \nct{5.1}{0.6};
        \nct{6.8}{0};
        \nct{6.7}{0.2};
        \nct{4.6}{0};
        \nct{4.5}{0.2};
        \nct{6.1}{0.2};
        \nct{6}{0};
    \end{tikzpicture},\]
    where the additive closure of the indecomposable modules corresponding to the bold vertices is a $4$-fractured subcategory. In particular, the algebra $\La_2$ is $4$-self-gluable. By self-gluing $\La_2$ we obtain the cyclic Nakayama algebra $\tilLa_2$ with Kupisch series $\left(5^{(7)},4,3^{(8)},2,3^{(7)},5^{(6)}\right)$. The Aus\-lan\-der--Rei\-ten quiver of $\tilLa_2$ is 
    
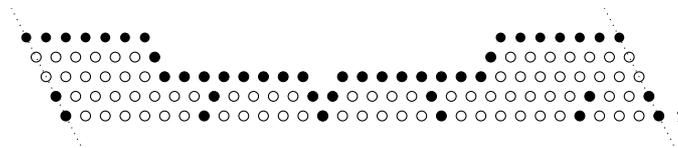
\[\begin{tikzpicture}[scale=1.3, transform shape]
        \foreach \x in {0.8,1,...,3.4}
            \mct{\x}{0};
        \foreach \x in {0.7,0.9,...,3.3}
            \mct{\x}{0.2};
        \foreach \x in {0.6,0.8,...,3.2}
            \mct{\x}{0.4};
        \foreach \x in {0.5,0.7,...,1.7}
            \mct{\x}{0.6};
        \foreach \x in {0.4,0.6,...,1.8}
            \nct{\x}{0.8};
        \nct{0.8}{0};        
        \nct{0.7}{0.2};
        \nct{1.7}{0.6};
        \foreach \x in {1.8,2,...,3.2}
            \nct{\x}{0.4};
        \nct{3.3}{0.2};
        \nct{3.4}{0};
        \nct{2.2}{0}:
        \nct{2.3}{0.2};
        
        \foreach \x in {3.4,3.6,...,7}
            \mct{\x}{0};
        \foreach \x in {3.5,3.7,...,6.7}
            \mct{\x}{0.2};
        \foreach \x in {3.6,3.8,...,6.6}
            \mct{\x}{0.4};
        \foreach \x in {5.1,5.3,...,6.5}
            \mct{\x}{0.6};
        \foreach \x in {5.2,5.4,...,6.4}
            \nct{\x}{0.8};
        \nct{3.5}{0.2};        
        \foreach \x in {3.6,3.8,...,5}
            \nct{\x}{0.4};
        \nct{5.1}{0.6};
        \nct{6.8}{0};
        \nct{6.7}{0.2};
        \nct{4.6}{0};
        \nct{4.5}{0.2};
        \nct{6.1}{0.2};
        \nct{6}{0};
        
        \draw[dotted] (0.95,-0.3) -- (0.25,1.1);
        \draw[dotted] (6.95,-0.3) -- (6.25,1.1);
        \node (A) at (6.9,0) {\nospacepunct{,}};
    \end{tikzpicture}\]
    where the additive closure of the indecomposable modules corresponding to the bold vertices is a $4$-cluster tilting subcategory.
\end{example}

\subsubsection{Exactly two fractures}\label{subsubsec:exactly two fractures} Assume now that the left fracturing $T^L$ has exactly one nonprojective summand $T^{(W)}$ while the right fracturing $T^R$ has exactly one noninjective summand $T^{(J)}$ and moreover the conditions of Definition \ref{def:self-gluable} are satisfied. In this case the algebra $\La$ is $n$-self gluable and so $\m\tilLa$ admits an $n$-cluster tilting subcategory by Corollary \ref{cor:La infinity n-ct}.

For Nakayama algebras we have the following corollary.

\begin{corollary}\label{cor:n-self gluable Nakayama}
Let $\La$ be an acyclic Nakayama algebra with Kupisch series $(d_1,\dots,d_k)$. Assume that $\m\La$ admits a fracturing $(T^L, T^R)$ and a $(T^L,T^R,n)$-fractured subcategory $\cM$. If $\La$ is $n$-self gluable, then there exist a right abutment $I$ and a left abutment $P$ of $\La$, both of the same height $h$, and such that $\underline{T}^L\in\cF_P$ and $\overline{T}^R\in\cG_I$ and moreover the module category of the cyclic Nakayama algebra with Kupisch series $(d_1,\dots,d_{k-h})$ admits an $n$-cluster tilting subcategory.
\end{corollary}

\begin{proof}
The existence of $P$ and $I$ follows directly from Definition \ref{def:self-gluable} since an acyclic Nakayama algebra has a unique maximal left abutment and a unique maximal right abutment. The fact that the module category of the cyclic Nakayama algebra with Kupisch series $(d_1,\dots,d_{k-h})$ admits an $n$-cluster tilting subcategory follows by Corollary \ref{cor:La infinity n-ct}.
\end{proof}

Using acyclic Nakayama algebras with homogeneous relations, we can systematically construct a class of $n$-self gluable acyclic Nakayama algebras. However, applying Corollary \ref{cor:n-self gluable Nakayama} to them, we obtain selfinjective Nakayama algebras whose module categories admit $n$-cluster tilting subcategories. Since selfinjective Nakayama algebras whose module categories admit $n$-cluster tilting subcategories are completely classified in \cite{DI} we do not include this special case. We give a different example instead.

\begin{example}\label{ex:two fractures Nakayama}
Let $\La$ be the acyclic Nakayama algebra with Kupisch series \[\left(5^{(7)},4,3^{(10)},5^{(6)},5,4,3,2,1\right).\] 
Then the Aus\-lan\-der--Rei\-ten quiver of $\La$ is
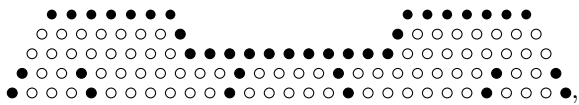
\[\begin{tikzpicture}[scale=1.3, transform shape]
        \foreach \x in {0,0.2,...,5.6}
            \mct{\x}{0};
        \foreach \x in {0.1,0.3,...,5.5}
            \mct{\x}{0.2};
        \foreach \x in {0.2,0.4,...,5.4}
            \mct{\x}{0.4};
        \foreach \x in {0.3,0.5,...,1.7}
            \mct{\x}{0.6};
        \foreach \x in {3.9,4.1,...,5.5}
            \mct{\x}{0.6};
        \foreach \x in {0.4,0.6,...,1.8}
            \nct{\x}{0.8};
        \foreach \x in {4,4.2,...,5.2}
            \nct{\x}{0.8};
        \nct{0}{0};
        \nct{0.8}{0};        
        \nct{0.1}{0.2};
        \nct{0.7}{0.2};
        \nct{1.7}{0.6};
        \foreach \x in {1.8,2,...,3.8}
            \nct{\x}{0.4};
        \nct{3.3}{0.2};
        \nct{3.4}{0};
        \nct{2.2}{0}:
        \nct{2.3}{0.2};
        \nct{3.9}{0.6};
        \nct{5.6}{0};
        \nct{5.5}{0.2};
        \nct{4.8}{0};
        \nct{4.9}{0.2};
        \node (A) at (3,0.85) {};
    \end{tikzpicture},\]
where the additive closure of the indecomposable modules corresponding to the bold vertices is a $4$-fractured subcategory. Gluing along the unique maximal left and right fracture, we get the cyclic Nakayama algebra $\tilLa$ with Kupisch series $\left(5^{(7)},4,3^{(10)},5^{(6)}\right)$. The Aus\-lan\-der--Rei\-ten quiver of $\tilLa$ is
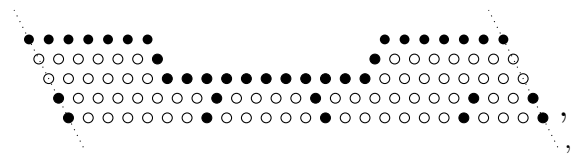
\[\begin{tikzpicture}[scale=1.3, transform shape]
        \foreach \x in {0.8,1,...,5.6}
            \mct{\x}{0};
        \foreach \x in {0.7,0.9,...,5.5}
            \mct{\x}{0.2};
        \foreach \x in {0.6,0.8,...,5.4}
            \mct{\x}{0.4};
        \foreach \x in {0.5,0.7,...,1.7}
            \mct{\x}{0.6};
        \foreach \x in {3.9,4.1,...,5.5}
            \mct{\x}{0.6};
        \foreach \x in {0.4,0.6,...,1.8}
            \nct{\x}{0.8};
        \foreach \x in {4,4.2,...,5.2}
            \nct{\x}{0.8};
        \nct{0.8}{0};        
        \nct{0.7}{0.2};
        \nct{1.7}{0.6};
        \foreach \x in {1.8,2,...,3.8}
            \nct{\x}{0.4};
        \nct{3.3}{0.2};
        \nct{3.4}{0};
        \nct{2.2}{0}:
        \nct{2.3}{0.2};
        \nct{3.9}{0.6};
        \nct{5.6}{0};
        \nct{5.5}{0.2};
        \nct{4.8}{0};
        \nct{4.9}{0.2};
        
        \draw[dotted] (0.95,-0.3) -- (0.25,1.1);
        \draw[dotted] (5.75,-0.3) -- (5.05,1.1);
        \node (A) at (5.7,0) {\nospacepunct{,}};
    \end{tikzpicture},\]
    where the additive closure of the indecomposable modules corresponding to the bold vertices is a $4$-cluster tilting subcategory.
\end{example}

We end this section with an example that exhibits many interesting properties.

\begin{example}\label{ex:interesting self-gluing}
Let $\La$ be the algebra given by the quiver with relations
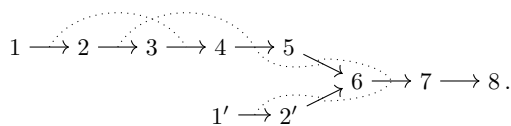
\[\begin{tikzpicture}[scale=0.9, transform shape]
\node (1) at (0,0.5) {$1$};
\node (2) at (1,0.5) {$2$};
\node (3) at (2,0.5) {$3$};
\node (4) at (3,0.5) {$4$};
\node (5) at (4,0.5) {$5$};
\node (6) at (5,0) {$6$};
\node (2') at (4,-0.5) {$2'$};
\node (1') at (3,-0.5) {$1'$};
\node (7) at (6,0) {$7$};
\node (8) at (7,0) {$8$\nospacepunct{.}};

\draw[->] (1) to (2);
\draw[->] (2) to (3);
\draw[->] (3) to (4);
\draw[->] (4) to (5);
\draw[->] (5) to (6);
\draw[->] (1') to (2');
\draw[->] (2') to (6);
\draw[->] (6) to (7);
\draw[->] (7) to (8);

\draw[dotted] (0.5,0.5) to [out=60,in=120] (2.5,0.5);
\draw[dotted] (1.5,0.5) to [out=60,in=120] (3.5,0.5);
\draw[dotted] (3.5,0.5) to [out=-60,in=190] (4.5,0.25);
\draw[dotted] (3.5,-0.5) to [out=60,in=170] (4.5,-0.25);
\draw[dotted] (4.5,0.25) to [out=160,in=120] (5.5,0);
\draw[dotted] (4.5,-0.25) to [out=190,in=-120] (5.5,0);
\end{tikzpicture}\]
The Aus\-lan\-der--Rei\-ten quiver $\Gamma(\La)$ of $\La$ is
\[\begin{tikzpicture}[scale=0.9, transform shape, baseline={(current bounding box.center)}]
    \tikzstyle{nct2}=[circle, minimum width=0.6cm, draw, inner sep=0pt, text centered, scale=0.9]
    \tikzstyle{nct22}=[circle, minimum width=0.6cm, draw, inner sep=0pt, text centered, scale=0.8]
    \tikzstyle{nct33}=[circle, minimum width=0.6cm, draw=white, inner sep=0pt, text centered, scale=0.8]
    \tikzstyle{nct3}=[circle, minimum width=6pt, draw=white, inner sep=0pt, scale=0.9]
    
    \node[nct3] (02) at (0,0) {$\qthree{}[8][]$};
    \node[nct3] (13) at (0.7,0.7) {$\qthree{7}[8][]$};
    \node[nct3] (24) at (1.4,1.4) {$\qthree{6}[7][8]$};
    \node[nct3] (22) at (1.4,0) {$\qthree{}[7][]$};
    \node[nct3] (33) at (2.1,0.7) {$\qthree{6}[7][]$};
    \node[nct3] (42) at (2.8,0) {$\qthree{}[6][]$};
    \node[nct3] (53) at (3.5,0.7) {$\qthree{2'}[6][]$};
    \node[nct3] (51) at (3.5,-0.7) {$\qthree{5}[6][]$};
    \node[nct3] (62) at (4.2,0) {$\begin{smallmatrix} & 2'  5 &  \\ & 6 &\end{smallmatrix} $};
    \node[nct3] (73) at (4.9,0.7) {$\qthree{}[5][]$};
    \node[nct3] (71) at (4.9,-0.7) {$\qthree{}[2'][]$};
    \node[nct3] (84) at (5.6,1.4) {$\qthree{4}[5][]$};
    \node[nct3] (80) at (5.6,-1.4) {$\qthree{1'}[2'][]$};
    \node[nct3] (95) at (6.3,2.1) {$\qthree{3}[4][5]$};
    \node[nct3] (93) at (6.3,0.7) {$\qthree{}[4][]$};
    \node[nct3] (91) at (6.3,-0.7) {$\qthree{}[1'][]$};
    \node[nct3] (104) at (7,1.4) {$\qthree{3}[4][]$};
    \node[nct3] (115) at (7.7,2.1) {$\qthree{2}[3][4]$};
    \node[nct3] (113) at (7.7,0.7) {$\qthree{}[3][]$};
    \node[nct3] (124) at (8.4,1.4) {$\qthree{2}[3][]$};
    \node[nct3] (135) at (9.1,2.1) {$\qthree{1}[2][3]$};
    \node[nct3] (133) at (9.1,0.7) {$\qthree{}[2][]$};
    \node[nct3] (144) at (9.8,1.4) {$\qthree{1}[2][]$};
    \node[nct3] (153) at (10.5,0.7) {$\qthree{}[1][]$\nospacepunct{,}};
        
    \draw[->] (02) -- (13);
    \draw[->] (13) -- (24);
    \draw[->] (13) -- (22);
    \draw[->] (24) -- (33);
    \draw[->] (22) -- (33);
    \draw[->] (33) -- (42);
    \draw[->] (42) -- (53);
    \draw[->] (42) -- (51);
    \draw[->] (53) -- (62);
    \draw[->] (51) -- (62);
    \draw[->] (62) -- (73);
    \draw[->] (62) -- (71);
    \draw[->] (73) -- (84);
    \draw[->] (71) -- (80);
    \draw[->] (84) -- (95);
    \draw[->] (84) -- (93);
    \draw[->] (80) -- (91);
    \draw[->] (95) -- (104);
    \draw[->] (93) -- (104);
    \draw[->] (104) -- (115);
    \draw[->] (104) -- (113);
    \draw[->] (115) -- (124);
    \draw[->] (113) -- (124);
    \draw[->] (124) -- (135);
    \draw[->] (124) -- (133);
    \draw[->] (135) -- (144);
    \draw[->] (133) -- (144);
    \draw[->] (144) -- (153);
    
    \draw[loosely dotted] (02) -- (22);
    \draw[loosely dotted] (13) -- (33);
    \draw[loosely dotted] (22) -- (42);
    \draw[loosely dotted] (42) -- (62);
    \draw[loosely dotted] (53) -- (73);
    \draw[loosely dotted] (51) -- (71);
    \draw[loosely dotted] (73) -- (93);
    \draw[loosely dotted] (71) -- (91);
    \draw[loosely dotted] (84) -- (104);
    \draw[loosely dotted] (93) -- (113);
    \draw[loosely dotted] (104) -- (124);
    \draw[loosely dotted] (113) -- (133);
    \draw[loosely dotted] (124) -- (144);
    \draw[loosely dotted] (133) -- (153);
\end{tikzpicture}\]
and we see that the unique maximal left abutment is $\qthree{6}[7][8]$ while the two maximal right abutments are $\qthree{1}[2][3]$ and $\qthree{1'}[2'][]$. We set 
\[ W\coloneqq\qthree{6}[7][8], \;\; T^L=T^{\left(W\right)} \coloneqq W \oplus \qthree{6}[7][] \oplus \qthree{}[7][],\;\; J\coloneqq \qthree{1}[2][3],\;\; T^{\left(J\right)} \coloneqq J \oplus \qthree{1}[2][] \oplus \qthree{}[2][],\;\; T^{\left(\qthree{1'}[2'][]\right)}\coloneqq\qthree{1'}[2'][]\oplus \qthree{}[1'][],\;\; T^R=T^{\left(J\right)}\oplus T^{\left(\qthree{1'}[2'][]\right)}.\]
Then $(T^L,T^R)$ is a fracturing of $\La$. By closing $\cP^L$ under the action of $\tau^-_2$ as in Proposition \ref{prop:n-fractured is self-orthogonal}(a), we obtain the $2$-fractured subcategory $\cM\subseteq \m\La$ which is given as the additive closure of the following encircled indecomposable modules: 
\[\begin{tikzpicture}[scale=0.9, transform shape, baseline={(current bounding box.center)}]
    \tikzstyle{nct2}=[circle, minimum width=0.6cm, draw, inner sep=0pt, text centered, scale=0.9]
    \tikzstyle{nct22}=[circle, minimum width=0.6cm, draw, inner sep=0pt, text centered, scale=0.8]
    \tikzstyle{nct33}=[circle, minimum width=0.6cm, draw=white, inner sep=0pt, text centered, scale=0.8]
    \tikzstyle{nct3}=[circle, minimum width=6pt, draw=white, inner sep=0pt, scale=0.9]
    
    \node[nct3] (02) at (0,0) {$\qthree{}[8][]$};
    \node[nct3] (13) at (0.7,0.7) {$\qthree{7}[8][]$};
    \node[nct2] (24) at (1.4,1.4) {$\qthree{6}[7][8]$};
    \node[nct2] (22) at (1.4,0) {$\qthree{}[7][]$};
    \node[nct2] (33) at (2.1,0.7) {$\qthree{6}[7][]$};
    \node[nct3] (42) at (2.8,0) {$\qthree{}[6][]$};
    \node[nct2] (53) at (3.5,0.7) {$\qthree{2'}[6][]$};
    \node[nct2] (51) at (3.5,-0.7) {$\qthree{5}[6][]$};
    \node[nct2] (62) at (4.2,0) {$\begin{smallmatrix} & 2'  5 &  \\ & 6 &\end{smallmatrix} $};
    \node[nct3] (73) at (4.9,0.7) {$\qthree{}[5][]$};
    \node[nct3] (71) at (4.9,-0.7) {$\qthree{}[2'][]$};
    \node[nct2] (84) at (5.6,1.4) {$\qthree{4}[5][]$};
    \node[nct2] (80) at (5.6,-1.4) {$\qthree{1'}[2'][]$};
    \node[nct2] (95) at (6.3,2.1) {$\qthree{3}[4][5]$};
    \node[nct2] (93) at (6.3,0.7) {$\qthree{}[4][]$};
    \node[nct2] (91) at (6.3,-0.7) {$\qthree{}[1'][]$};
    \node[nct3] (104) at (7,1.4) {$\qthree{3}[4][]$};
    \node[nct2] (115) at (7.7,2.1) {$\qthree{2}[3][4]$};
    \node[nct3] (113) at (7.7,0.7) {$\qthree{}[3][]$};
    \node[nct3] (124) at (8.4,1.4) {$\qthree{2}[3][]$};
    \node[nct2] (135) at (9.1,2.1) {$\qthree{1}[2][3]$};
    \node[nct2] (133) at (9.1,0.7) {$\qthree{}[2][]$};
    \node[nct2] (144) at (9.8,1.4) {$\qthree{1}[2][]$};
    \node[nct3] (153) at (10.5,0.7) {$\qthree{}[1][]$\nospacepunct{.}};
        
    \draw[->] (02) -- (13);
    \draw[->] (13) -- (24);
    \draw[->] (13) -- (22);
    \draw[->] (24) -- (33);
    \draw[->] (22) -- (33);
    \draw[->] (33) -- (42);
    \draw[->] (42) -- (53);
    \draw[->] (42) -- (51);
    \draw[->] (53) -- (62);
    \draw[->] (51) -- (62);
    \draw[->] (62) -- (73);
    \draw[->] (62) -- (71);
    \draw[->] (73) -- (84);
    \draw[->] (71) -- (80);
    \draw[->] (84) -- (95);
    \draw[->] (84) -- (93);
    \draw[->] (80) -- (91);
    \draw[->] (95) -- (104);
    \draw[->] (93) -- (104);
    \draw[->] (104) -- (115);
    \draw[->] (104) -- (113);
    \draw[->] (115) -- (124);
    \draw[->] (113) -- (124);
    \draw[->] (124) -- (135);
    \draw[->] (124) -- (133);
    \draw[->] (135) -- (144);
    \draw[->] (133) -- (144);
    \draw[->] (144) -- (153);
    
    \draw[loosely dotted] (02) -- (22);
    \draw[loosely dotted] (13) -- (33);
    \draw[loosely dotted] (22) -- (42);
    \draw[loosely dotted] (42) -- (62);
    \draw[loosely dotted] (53) -- (73);
    \draw[loosely dotted] (51) -- (71);
    \draw[loosely dotted] (73) -- (93);
    \draw[loosely dotted] (71) -- (91);
    \draw[loosely dotted] (84) -- (104);
    \draw[loosely dotted] (93) -- (113);
    \draw[loosely dotted] (104) -- (124);
    \draw[loosely dotted] (113) -- (133);
    \draw[loosely dotted] (124) -- (144);
    \draw[loosely dotted] (133) -- (153);
\end{tikzpicture}\]
Notice in particular that $\qthree{6}[7][]\in\cI^R\cap\cP^L$, which is an extreme case. Notice also that $\om^{-2}(7) = 5\oplus 2$ which shows that conditions (a3) and (a4) in Proposition \ref{prop:n-fractured is self-orthogonal} do not need to hold for $i=n$.

Next, we have that $\La$ is $2$-self gluable with $\left(W,J\right)$ as a fractured pair and $(W,J)$ as a compatible pair too. It follows by Corollary \ref{cor:La infinity n-ct} that the module category of the algebra $\tilLa=\K \tilQ / \tilR$ where $\tilQ$ is the quiver 
\[\begin{tikzpicture}[scale=0.9, transform shape, baseline={(current bounding box.center)}]
\node (1) at (0,0) {$1$};
\node (2) at (1,0.5) {$2$};
\node (3) at (2,0.5) {$3$};
\node (4) at (2,-0.5) {$4$\nospacepunct{,}};
\node (5) at (1,-0.5) {$5$};
\node (2') at (-1,0) {$2'$};
\node (1') at (-2,0) {$1'$};

\draw[->] (1) -- node[above] {$\gamma$} (2);
\draw[->] (2) -- node[above] {$\delta$} (3);
\draw[->] (3) -- node[right] {$\epsilon$} (4);
\draw[->] (4) -- node[below] {$\zeta$} (5);
\draw[->] (5) -- node[below] {$\eta$} (1);
\draw[->] (1') -- node[above] {$\alpha$} (2');
\draw[->] (2') -- node[above] {$\beta$} (1);
\end{tikzpicture}\]
and $\tilR =\langle \alpha\beta, \beta\gamma, \gamma\delta\epsilon, \delta\epsilon\zeta, \zeta\eta, \eta\gamma\rangle$ admits a $2$-cluster tilting subcategory. Indeed, the Aus\-lan\-der--Rei\-ten quiver $\Gamma(\tilLa)$ of $\tilLa$ is
\[\begin{tikzpicture}[baseline={(current bounding box.center)}]
    \tikzstyle{nct2}=[circle, minimum width=0.6cm, draw, inner sep=0pt, text centered, scale=0.9]
    \tikzstyle{nct22}=[circle, minimum width=0.6cm, draw, inner sep=0pt, text centered, scale=0.8]
    \tikzstyle{nct33}=[circle, minimum width=0.6cm, draw=white, inner sep=0pt, text centered, scale=0.8]
    \tikzstyle{nct3}=[circle, minimum width=6pt, draw=white, inner sep=0pt, scale=0.9]
    
    \node[nct3] (42) at (2.8,0) {$\qthree{}[1][]$};
    \node[nct2] (53) at (3.5,0.7) {$\qthree{2'}[1][]$};
    \node[nct2] (51) at (3.5,-0.7) {$\qthree{5}[1][]$};
    \node[nct2] (62) at (4.2,0) {$\begin{smallmatrix} & 2'  5 &  \\ & 1 &\end{smallmatrix} $};
    \node[nct3] (73) at (4.9,0.7) {$\qthree{}[5][]$};
    \node[nct3] (71) at (4.9,-0.7) {$\qthree{}[2'][]$};
    \node[nct2] (84) at (5.6,1.4) {$\qthree{4}[5][]$};
    \node[nct2] (80) at (5.6,-1.4) {$\qthree{1'}[2'][]$};
    \node[nct2] (95) at (6.3,2.1) {$\qthree{3}[4][5]$};
    \node[nct2] (93) at (6.3,0.7) {$\qthree{}[4][]$};
    \node[nct2] (91) at (6.3,-0.7) {$\qthree{}[1'][]$};
    \node[nct3] (104) at (7,1.4) {$\qthree{3}[4][]$};
    \node[nct2] (115) at (7.7,2.1) {$\qthree{2}[3][4]$};
    \node[nct3] (113) at (7.7,0.7) {$\qthree{}[3][]$};
    \node[nct3] (124) at (8.4,1.4) {$\qthree{2}[3][]$};
    \node[nct2] (135) at (9.1,2.1) {$\qthree{1}[2][3]$};
    \node[nct2] (133) at (9.1,0.7) {$\qthree{}[2][]$};
    \node[nct2] (144) at (9.8,1.4) {$\qthree{1}[2][]$};
    \node[nct3] (153) at (10.5,0.7) {$\qthree{}[1][]$\nospacepunct{,}};
        
    \draw[->] (42) -- (53);
    \draw[->] (42) -- (51);
    \draw[->] (53) -- (62);
    \draw[->] (51) -- (62);
    \draw[->] (62) -- (73);
    \draw[->] (62) -- (71);
    \draw[->] (73) -- (84);
    \draw[->] (71) -- (80);
    \draw[->] (84) -- (95);
    \draw[->] (84) -- (93);
    \draw[->] (80) -- (91);
    \draw[->] (95) -- (104);
    \draw[->] (93) -- (104);
    \draw[->] (104) -- (115);
    \draw[->] (104) -- (113);
    \draw[->] (115) -- (124);
    \draw[->] (113) -- (124);
    \draw[->] (124) -- (135);
    \draw[->] (124) -- (133);
    \draw[->] (135) -- (144);
    \draw[->] (133) -- (144);
    \draw[->] (144) -- (153);
    
    \draw[loosely dotted] (42) -- (62);
    \draw[loosely dotted] (53) -- (73);
    \draw[loosely dotted] (51) -- (71);
    \draw[loosely dotted] (73) -- (93);
    \draw[loosely dotted] (71) -- (91);
    \draw[loosely dotted] (84) -- (104);
    \draw[loosely dotted] (93) -- (113);
    \draw[loosely dotted] (104) -- (124);
    \draw[loosely dotted] (113) -- (133);
    \draw[loosely dotted] (124) -- (144);
    \draw[loosely dotted] (133) -- (153);
    
    \draw[loosely dotted] (2.8,2.1) -- (2.8,-1.4);
    \draw[loosely dotted] (10.5,2.1) -- (10.5,-1.4);
\end{tikzpicture}\]
and the additive closure of the encircled modules is a $2$-cluster tilting subcategory $\tilde{M}\subseteq\m\tilLa$. Notice in particular that $\La$ has infinite global dimension but is not an Iwanaga--Gorenstein algebra since $\idim\left(\qthree{2'}[1][]\right)=\infty$.
\end{example}

\section*{Acknowledgements}

The author wishes to thank his advisor Martin Herschend for the constant support and many helpful suggestions during the preparation of this article.

\bibliography{algebra}

\begin{thebibliography}{JKPK19}

\bibitem[AR74]{AR}
Maurice Auslander and Idun Reiten.
\newblock Stable equivalence of dualizing {R}-varieties.
\newblock {\em Advances in Mathematics}, 12(3):306–366, 1974.

\bibitem[ARS95]{ARS}
Maurice Auslander, Idun Reiten, and Svere~Olaf Smal{\o}.
\newblock {\em {Representation theory of {A}rtin algebras}}, volume~36 of {\em
  {Cambridge Studies in Advanced Mathematics}}.
\newblock Cambridge University Press, Cambridge, 1995.

\bibitem[ASS06]{ASS}
Ibrahim Assem, Daniel Simson, and Andrzej Skowro{\'n}ski.
\newblock {\em {Elements of the Representation Theory of Associative Algebras:
  Volume 1: Techniques of Representation Theory}}.
\newblock {Elements of the Representation Theory of Associative Algebras}.
  Cambridge University Press, 2006.

\bibitem[Aus74]{AUS}
Maurice Auslander.
\newblock {Representation Theory of Artin Algebras I}.
\newblock {\em Communications in Algebra}, 1(3):177–268, 1974,
  https://doi.org/10.1080/00927877408548230.

\bibitem[BR81]{BR}
Klaus Bongartz and Claus~Michael Ringel.
\newblock Representation-finite tree algebras.
\newblock In {\em Representations of algebras ({P}uebla, 1980)}, volume 903 of
  {\em Lecture Notes in Math.}, pages 39--54. Springer, Berlin-New York, 1981.

\bibitem[CIM19]{CIM}
Aaron Chan, Osamu Iyama, and René Marczinzik.
\newblock Auslander–gorenstein algebras from serre-formal algebras via
  replication.
\newblock {\em Advances in Mathematics}, 345:222 -- 262, 2019.

\bibitem[DI20]{DI}
Erik Darpö and Osamu Iyama.
\newblock d-representation-finite self-injective algebras.
\newblock {\em Advances in Mathematics}, 362:106932, 2020.

\bibitem[EH08]{EH}
Karin Erdmann and Thorsten Holm.
\newblock {Maximal n-orthogonal modules for selfinjective algebras}.
\newblock {\em Proc. Amer. Math. Soc.}, 136(9):3069--3078, 2008.

\bibitem[Fis86]{Fis}
Urs Fischbacher.
\newblock The representation-finite algebras with at most {$3$} simple modules.
\newblock In {\em Representation theory, {I} ({O}ttawa, {O}nt., 1984)}, volume
  1177 of {\em Lecture Notes in Math.}, pages 94--114. Springer, Berlin, 1986.

\bibitem[Gab72]{Gab1}
Peter Gabriel.
\newblock {Unzerlegbare Darstellungen I}.
\newblock {\em manuscripta mathematica}, 6(1):71–103, Mar 1972.

\bibitem[HI10]{HI}
Martin Herschend and Osamu Iyama.
\newblock {Selfinjective quivers with potential and 2-representation-finite
  algebras}.
\newblock {\em Compos. Math. 147 (2011), no. 6, 1885--1920}, July 2010,
  1006.1917v2.

\bibitem[HJV20]{HJV}
Martin {Herschend}, Peter {Jorgensen}, and Laertis {Vaso}.
\newblock {Wide subcategories of $d$-cluster tilting subcategories}.
\newblock {\em {Trans. Amer. Math. Soc. 373 (2020), 2281-2309}}, jan 2020.

\bibitem[IO13]{IO}
Osamu Iyama and Steffen Oppermann.
\newblock {Stable categories of higher preprojective algebras}.
\newblock {\em Advances in Mathematics}, 244:23--68, 2013.

\bibitem[Iya06]{IYA4}
Osamu Iyama.
\newblock {Auslander correspondence}.
\newblock {\em Adv. Math. 210 (2007), no. 1, 51–82}, May 2006,
  math/0411631v2.

\bibitem[Iya07]{IYA2}
Osamu Iyama.
\newblock {Higher-dimensional Auslander--Reiten theory on maximal orthogonal
  subcategories}.
\newblock {\em Advances in Mathematics}, 210(1):22--50, 2007.

\bibitem[Iya08]{IYA1}
Osamu Iyama.
\newblock {Auslander-{R}eiten theory revisited}.
\newblock In {\em {Trends in representation theory of algebras and related
  topics}}, {EMS Ser. Congr. Rep.}, pages 349--397. Eur. Math. Soc.,
  Z{\"u}rich, 2008.

\bibitem[Jas16]{JAS}
Gustavo Jasso.
\newblock {n-Abelian and n-exact categories}.
\newblock {\em Mathematische Zeitschrift}, 283(3):703--759, 2016.

\bibitem[JKPK19]{JK}
Gustavo Jasso, Julian Külshammer, Chrysostomos Psaroudakis, and Sondre Kvamme.
\newblock {Higher Nakayama algebras I: Construction}.
\newblock {\em Advances in Mathematics}, 351:1139 -- 1200, 2019.

\bibitem[Kra99]{KRA}
Henning Krause.
\newblock Functors on locally finitely presented additive categories.
\newblock {\em Colloq. Math.}, 75, 05 1999.

\bibitem[Kup58]{KUP}
Herbert Kupisch.
\newblock {\em {Beiträge zur Theorie nichthalbeinfacher Ringe mit
  Minimalbedingung}}.
\newblock PhD thesis, NA Heidelberg, 1958.

\bibitem[Mac98]{CWM}
Saunders MacLane.
\newblock {\em {Categories for the Working Mathematician}}.
\newblock Springer, Berlin, 2. auflage edition, 10 1998.

\bibitem[Psa14]{PSA}
Chrysostomos Psaroudakis.
\newblock Homological theory of recollements of abelian categories.
\newblock {\em Journal of Algebra}, 398:63–110, 2014.

\bibitem[Rie80]{Rid}
Christine Riedtmann.
\newblock Algebren, {D}arstellungsk\"{o}cher, \"{U}berlagerungen und
  zur\"{u}ck.
\newblock {\em Comment. Math. Helv.}, 55(2):199--224, 1980.

\bibitem[Rin16]{RIN}
Claus~Michael Ringel.
\newblock Representation theory of dynkin quivers. three contributions.
\newblock {\em Frontiers of Mathematics in China}, 11(4):765--814, Aug 2016.

\bibitem[Vas18]{VAS2}
Laertis Vaso.
\newblock {Gluing of $n$-cluster tilting subcategories for
  representation-directed algebras}.
\newblock {\em arXiv e-prints}, page arXiv:1805.12180, May 2018, 1805.12180.

\bibitem[Vas19]{VAS}
Laertis Vaso.
\newblock n-cluster tilting subcategories of representation-directed algebras.
\newblock {\em Journal of Pure and Applied Algebra}, 223(5):2101 -- 2122, 2019.

\end{thebibliography}
\bibliographystyle{halpha}

\printindex[definitions]
\printindex[symbols]

\end{document}